\documentclass[A4,10pt]{article}
\usepackage[margin=1in,dvips]{geometry}
\usepackage[T1]{fontenc}
\usepackage{graphicx,cancel}

\usepackage{graphics}
\usepackage{color, soul}
\usepackage{amsthm}
\usepackage{amsmath, slashed}
\usepackage{amssymb}
\usepackage{marvosym}
\usepackage{rotating}
\usepackage{mathrsfs}
\usepackage{epsfig}
\usepackage[affil-it]{authblk}
\usepackage{rotating}
\usepackage{graphicx}
\usepackage{accents}
\usepackage{harmony, slashed}
\usepackage{tikz}
\usetikzlibrary{patterns}
\usepackage{scalerel}

\usepackage{enumerate}
\usepackage{stackengine,wasysym}

\newtheorem{definition}{Definition}[section]
\newtheorem{theorem}{Theorem}[section]
\newtheorem{conjecture}{Conjecture}
\newtheorem{proposition}[theorem]{Proposition}
\newtheorem{lemma}[theorem]{Lemma}
\newtheorem{remark}{Remark}[section]

\newtheorem{convention}{Convention}[section]

\newtheorem{bigtheorem}{Theorem}

\numberwithin{equation}{section}

\newcommand{\Lb}{\underline{L}}

\DeclareMathAlphabet\mathbfcal{OMS}{cmsy}{b}{n}

\title{Naked Singularities for the Einstein Vacuum Equations: \\ The Exterior Solution}

\date\today
\author[1]{Igor Rodnianski}
\author[2,3]{Yakov Shlapentokh-Rothman}

\affil[1]{\small Princeton University, Department of Mathematics, Fine~Hall,~Washington~Road,~Princeton,~NJ~08544,~United~States~of~America\vskip.2pc \ }
\affil[2]{\small University of Toronto, Department of Mathematics, 40 St.~George~Street, Toronto, ON, Canada\vskip.2pc \ }
\affil[3]{\small University of Toronto Mississauga, Department of Mathematical and Computational Sciences, 3359 Mississauga Road, Mississauga, ON, Canada\vskip.2pc \ }
\begin{document}

\maketitle
\begin{abstract}In this work we initiate the mathematical study of naked singularities for the Einstein vacuum equations in $3+1$ dimensions by constructing solutions which correspond to the exterior region of a naked singularity. A key element is our introduction of a new type of self-similarity for the Einstein vacuum equations. Connected to this is a new geometric twisting phenomenon which plays the leading role in singularity formation.

Prior to this work, the only known examples of naked singularities were the solutions constructed by Christodoulou for the spherically symmetric Einstein-scalar-field system, as well as other solutions explored numerically for either the spherically symmetric Einstein equations coupled to suitable matter models or for the Einstein equations in higher dimensions.
\end{abstract}
\tableofcontents

\section{Introduction}

Already in the earliest investigations of the Einstein vacuum equations,

\begin{equation}\label{EVE}
{\rm Ric}_{\mu\nu}\left(g\right) = 0,
\end{equation}
the existence of singular solutions forced theorists to confront fundamental questions concerning the domain of validity for General Relativity. Namely, since (severe) singularities do occur in some solutions, e.g., the Schwarzschild solution~\cite{schwarzschild1916}, what is the relevance/predictive power of a non-singular portion of a given solution? Though originally the possibility was entertained that generic asymmetric perturbations of a spacetime $\left(\mathcal{M},g_{\mu\nu}\right)$ satisfying~\eqref{EVE} would be regular, the incompleteness theorem of Penrose~\cite{Penroseincomp} showed that when a trapped surface is present, some degree of pathology is in fact a stable feature. Fortunately, Penrose also suggested a way out of this problem, at least for isolated self-gravitating systems:

\begin{conjecture}\label{cosmicconj}(Weak Cosmic Censorship Conjecture \underline{Original} Version~\cite{Penroseergo}) For asymptotically flat solutions to the Einstein vacuum equations, singularities are always hidden behind an event horizon.
\end{conjecture}
In particular, if the weak cosmic censorship conjecture holds, and if we are only interested in gravitational physics outside the event horizon, we do not need to concern ourselves with the structure of singularities! 

A singular solution which is not confined within an event horizon is known as a ``naked singularity.''  Informally, a naked singularity may be thought of as a singular solution where the future light cone of the singularity extends to an asymptotic region in such a way that arbitrarily far away observers may still intersect the light cone in finite time and thus ``see'' the singularity. (See Definition~\ref{nakeddef} below for a precise definition.) In addition to being visible to far away observers, another important quality of a naked singularity is that the singular point represents a ``genuine'' loss of regularity relative to the initial data. Finally, we note that the ``exterior region'' of a naked singularity refers to region of a naked singularity which is in the future of the past light cone of the singular point. 

Previously, Christodoulou constructed naked singularities for the spherically symmetric Einstein-scalar-field system~\cite{ChristNaked}. In this case, the loss of regularity referred to in the previous paragraph may be seen as follows: Let $m$ denote the Hawking mass of a sphere and $r$ denote the area radius of a given sphere. Then $\frac{m}{r}$ is a scale-invariant quantity which, for regular solutions, must vanish when $r= 0$. However, at the singular point of Christodoulou's solutions, we have that $r = 0$  but $\frac{m}{r}$ does not converge to $0$. In contrast, the Cauchy data for Christodoulou's solutions lie in the so-called ``absolutely continuous'' class of data which is \emph{more} regular than the scale-invariant class of bounded variation data. On a more technical level, we recall that a key role in the construction is played by a reduction of the self-similar spherically symmetric Einstein-scalar-field system to a two dimensional autonomous system. The existence of such a system cannot be expected outside of spherical symmetry, and thus the study of naked singularities for the Einstein vacuum equations (where the assumption of spherical symmetry would eliminate the dynamics) must take a different approach. (For a more thorough discussion of Christodoulou's solutions, see Sections~\ref{christisnakedsphsymmwhat}-\ref{isitSig}.) 

Christodoulou has also constructed naked singularities for the Einstein dust model~\cite{dustnaked}, and there has been further numerical analysis and construction of naked singularities for the spherically symmetric Einstein equations coupled to fluid models which allow for pressure~\cite{nakedfluidoripiran,nakedfluidjoshidwivedi}. Additionally, though we will not survey this here, we note that there is a large numerical literature concerning other types of naked singularities; see, for example, solutions associated with critical phenomena~\cite{chop,critrotate,critsurv} and higher dimensional black holes~\cite{blackstringnakedsing,nakedsinganzhang}.

We now give precise definitions of a spacetime not possessing a complete null infinity and a ``naked singularity,'' that is, a singularity which is \emph{not} hidden behind an event horizon. 
\begin{definition}\label{nakeddef}Let $(\mathcal{M},g)$ posses an asymptotically flat null hypersurface $\mathcal{H}$, let $L'$ be a geodesic outgoing null normal vector for $\mathcal{H}$ with affine parameter $v$, let $\mathcal{S}_v$ denote a surface of constant $v$ on $\mathcal{H}$, and define $\underline{L}'$ along $\mathcal{H}$ to be the unique future directed null vector transversal to $\mathcal{H}$ which satisfies $g\left(\underline{L}',L'\right) = -2$. 

Then we say that $\left(\mathcal{M},g_{\mu\nu}\right)$ \underline{does not posses a complete future null infinity} if there exists a constant $A > 0$, a sequence $\{v_i\}_{i=1}^{\infty}$ with $v_i \to \infty$, and a sequence $\{p_i\}$ with $p_i \in \mathcal{S}_{v_i}$, such that the each maximal null geodesic $\gamma_i$, with tangent vector $\underline{L}'$ at $p_i$, has affine length less than $A$. 

If $(\mathcal{M},g_{\mu\nu})$ does not posses a complete future null infinity and is a maximal globally hyperbolic development of suitably regular\footnote{We do not here give an explicit definition of ``suitably regular,'' but simply note that any given choice of functional framework must be justified. Cf.~the discussion of Christodoulou's naked singularities in Section~\ref{christisnakedsphsymmwhat} and the discussion of the solutions constructed in this paper in Section~\ref{letscompareyayaya}.} and complete initial data, then we say that $(\mathcal{M},g_{\mu\nu})$ contains a \underline{naked singularity}.
\end{definition}
(The original version of) weak cosmic censorship can then be understood as the statement that naked singularities \emph{do not} arise from the maximal globally hyperbolic developments of complete asymptotically flat initial data. We note that Definition~\ref{nakeddef}, which is in the spirit of the definition given in~\cite{Chrmil}, has the benefit of not relying on an explicit conformal compactification of the spacetime. 

It is important to note that the ``maximality'' of a globally hyperbolic development may depend on the regularity class which the spacetimes are a priori restricted to lie in. As with other fundamental questions in general relativity, different regularity frameworks could, in principle, lead to different outcomes for the weak cosmic censorship conjecture. (Compare with, for example, the role played by the regularity of the Cauchy horizon in the strong cosmic censorship conjecture~\cite{DafLuk1,LukOh2017one,LukOh2017two}.) Later in the paper we will discuss the relevant precise notion of maximal globally hyperbolic development.

We now state our main theorem.
\begin{bigtheorem}\label{themainextresult} Let $N \gg 1$ be a sufficiently large integer and $0 < \epsilon \ll \gamma \ll 1$ be sufficiently small, potentially depending on $N$. Then there exists a spacetime $\left(\mathcal{M},g_{\mu\nu}\right)$ solving the Einstein vacuum equations so that
\begin{enumerate}
\item  $\left(\mathcal{M},g_{\mu\nu}\right)$ is covered by coordinates  $\left(u,\hat{v},\theta^A\right) \in [-\underline{v}^2,0) \times [0,\infty) \times \mathbb{S}^2$ where $\underline{v} > 0$, and $g_{\mu\nu}$ takes the following double-null form:
\begin{equation}\label{thisisgdouble}
g = -2\hat{\Omega}^2\left(du\otimes d\hat{v} + d\hat{v} \otimes du\right) + \slashed{g}_{AB}\left(d\theta^A - b^Adu\right)\otimes\left(d\theta^B - b^Bdu\right).
\end{equation}

\item $\left(\mathcal{M},g_{\mu\nu}\right)$ has the following Penrose diagram:\footnote{For the reader unfamiliar with the Penrose diagram notation, we recommend the discussion in the lecture notes~\cite{Mihalisnotes} and the references therein.}
\begin{center}
\begin{tikzpicture}[scale = 1]
\fill[lightgray] (0,0)--(2,-2)--(4,0)  -- (2,2) -- (0,0);
\draw (0,0) -- (2,-2) node[sloped,below,midway]{\footnotesize $\{\hat{v} = 0\}$};
\draw (2,-2) -- (4,0) node[sloped,below,midway]{\footnotesize $\{u=-\underline{v}^2\}$};
\draw[dashed] (4,0) -- (2,2) node[sloped,above,midway]{\footnotesize $\mathcal{I}^+$};
\draw[dashed] (2,2) -- (0,0) node[sloped,above,midway]{\footnotesize $\{u = 0\}$};
\path [draw=black,fill=white] (0,0) circle (1/16); 
\draw (-.1,0) node[above]{\footnotesize $\mathcal{O}$};
\end{tikzpicture}
\end{center}

\item\label{lp3} There exists a constant $c > 0$, independent of $\epsilon$ and $N$, so that we have $g_{\mu\nu} \in C^N\left(\mathcal{M}\setminus\{\hat{v}=0\}\right)$, $g_{\mu\nu} \in C^{1,c\epsilon^2}\left(\mathcal{M}\right)$, and $(g_{\mu\nu},\partial_{\hat{v}}g_{\mu\nu}) \in C^N\left(\{\hat{v}=0\}\right)$.

\item The null hypersurface $\{u = -\underline{v}^2\}$ is asymptotically flat as $\hat{v} \to \infty$, and future null infinity $\mathcal{I}^+$ is incomplete in the sense of Definition~\ref{nakeddef}.

\item Along $\{\hat{v} = 0\}$ we have that  in a Lie-propagated coordinate frame
\begin{equation}\label{choices}
\hat{\Omega}\left(u,\theta^A\right) = (-u)^{\kappa}\tilde\Omega\left(\theta^A\right),\qquad \slashed{g}_{AB} = u^2\tilde{\slashed{g}}_{AB}\left(\theta^C\right),\qquad b_A = u\tilde{b}_A\left(\theta^B\right),
\end{equation}
where $\kappa$, $\tilde\Omega$, $\tilde{\slashed{g}}_{AB}$, and $\tilde{b}_A$ are a suitable positive constant, positive function, Riemannian metric, and $1$-form on $\mathbb{S}^2$. 
\item There exists a vector field $S=u\partial_u$ -- generator of scaling symmetry -- which is tangent to $\{\hat{v} = 0\}$ and conformally Killing along $\{\hat{v} = 0\}$.
\item Let $m(u)$ denote the Hawking mass of a sphere $\mathbb{S}^2_{u,0} \subset \{\hat{v} = 0\}$. Then we have a constant $C > 0$, independent of $\epsilon$, such that
\begin{equation}\label{masskfow}
\epsilon ^2 C^{-1} \leq \frac{m(u)}{\sqrt{{\rm Area}\left(\mathbb{S}^2_{u,0}\right)}} \leq C\epsilon^2,
\end{equation}
and so that the shear\footnote{This is the shear defined with respect to the geodesic null frame $e_3 \doteq \mathcal{L}_{\partial_u} + \mathcal{L}_b$ and $e_4 \doteq \hat{\Omega}^{-2}\mathcal{L}_{\partial_{\hat{v}}}$.} $\hat{\underline{\chi}}_{AB}\doteq {\rm tf}\left(\left(\mathcal{L}_{\partial_u}+\mathcal{L}_b\right)\slashed{g}_{AB}\right)$ satisfies the following bound along $\{\hat{v} = 0\}$:
\begin{equation}\label{shearkofw}
\frac{\epsilon|a\left(\theta^A\right)|}{-u} C^{-1} \leq \left|\hat{\underline{\chi}}\left(u,\theta^A\right)\right|_{\slashed{g}} \leq C \frac{\epsilon|a\left(\theta^A\right)|}{-u},
\end{equation}
where ${\rm tf}$ denotes the trace-free part, and the function $a: \mathbb{S}^2 \to \mathbb{R}$ satisfies ${\rm Area}\left(\{\theta : |a(\theta)| \leq 1/2\}\right) \lesssim \gamma$. 
\end{enumerate}
\end{bigtheorem}
\begin{remark}
The key mechanism behind the statement of the Theorem and specifically that it describes the exterior region of a naked singularity $\mathcal O$ with the behavior of the Hawking mass in {\it 7.}, is the deformation tensor of the shift vector
$b^A$ with respect to the metric $\slashed{g}_{AB}$ on $\Bbb S^2$:
$$
\slashed\nabla_Ab_B + \slashed\nabla_Bb_A  = \left(\mathcal L_b \slashed g\right)_{AB}
$$
The parameter $\epsilon$ is related to the size of
$$
|\slashed{\nabla}\hat{\otimes} b|\sim \epsilon,\qquad |\slashed {\rm div} b|\sim \epsilon^2
$$
In view of this, the phenomenon described in the Theorem is highly non-symmetric and can be thought of as generated 
by a rotation of the incoming cone $\hat v=0$, see Remark \ref{cone}. For a further discussion of the importance of $b$ having a non-trivial deformation tensor, see the discussion after Lemma~\ref{thatissuchaconst}.
\end{remark}
\begin{remark}It follows from the method of our proof that, after a suitable rescaling of the double-null coordinates, the metric extends to $\{u = 0\}\setminus \{\hat{v} = 0\}$ in a H\"{o}lder continuous fashion and that the area of the spheres $\mathbb{S}^2_{0,v}$ converge to infinity as $v\to \infty$. Since it is not needed for the interpretation of the solution as the exterior region of a naked singularity, we will not here pursue sharp regularity statements for the solution along $\{u = 0\}$. Nevertheless it would be interesting to systematically study and determine the precise behavior of the solution along $\{u=0\}$.
\end{remark}
\begin{remark}
One should contrast the regularity of the outgoing data described in {\it 3.} with the behavior of the Hawking mass and 
the shear along the incoming cone $\hat v=0$ as $u\to 0$ in {\it 7}.  See further discussion in Section \ref{letscompareyayaya}.
\end{remark}

\begin{remark}\label{othercoeffawokf}We will have that $\left\vert\left\vert \hat{\chi}\right\vert\right\vert_{L^2\left(\mathbb{S}^2_{-\underline{v}^2,0}\right)} \sim \epsilon^{-1}$. In particular, we are in a ``large data regime.'' Furthermore, $\hat{\chi}_{AB}$ is only H\"{o}lder continuous as $\hat{v} \to 0$ and the curvature component $\alpha$ satisfies $\left\vert\left\vert \alpha\right\vert\right\vert_{L^2\left(\mathbb{S}^2_{-\underline{v}^2,\hat{v}}\right)}|_{\{u= -\underline{v}^2\}} \sim \hat{v}^{-1+c\epsilon^2}$ as $\hat{v} \to 0$. We note that since the pioneering work~\cite{Chr}, there have been various works which treat the Einstein equations in large data regimes, for example, ~\cite{klainrodtrap,anluk} . However, in contrast to our situation, these other large data regimes concerned solutions with initial data which was Minkowskian along (the analogue of) $\{\hat{v} = 0\}$ and whose evolution typically ended in trapped surface formation.
\end{remark}
\begin{remark}\label{stabilityofsolsls}It is a consequence of the method of proof used that the qualitative behavior of solutions described in Theorem~\ref{themainextresult} is stable to perturbations of the outgoing characteristic data along $\{u =-\underline{v}^2\}$ which vanish sufficiently quickly as $\hat{v} \to 0$ and $\hat{v} \to \infty$. However, without the vanishing condition at $\hat{v} = 0$, one expects a generic sufficiently regular asymptotically flat perturbation to create an instability and result in trapped surface formation.
\end{remark}
\begin{remark}It is a straightforward consequence of method of proof that as $\epsilon \to 0$, $g_{\mu\nu}$ converges to the Minkowski metric in $H^1_{\rm loc}$. However, in view of Remark~\ref{othercoeffawokf}, this convergence to the Minkowski metric does \underline{not} hold, already, in $C^1$. 
\end{remark}
\begin{remark}\label{cone}
The following schematic diagram may help the reader to visualize the null geometry of the cone $\{\hat{v} = 0\}$:
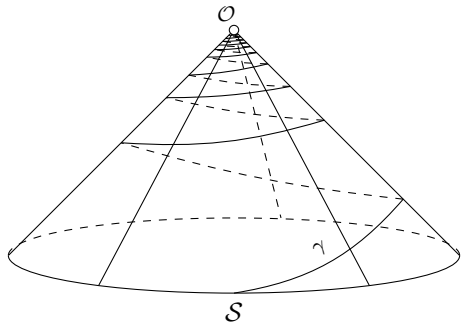
\begin{center}
\begin{tikzpicture}[scale = 1]

\draw (0,0) -- (3,-3);
\draw (0,0) -- (-1.8,-3.4);
\draw (0,0) -- (1.8,-3.4);
\draw [dashed] (0,0) -- (.625,-2.5);
\draw (0,0) -- (-3,-3) ;
\draw (3,-3) arc(0:-180: 3 and .5) node[sloped,below,midway]{$\mathcal{S}$};
\draw [dashed] (3,-3) arc(0:180: 3 and .5);
\draw (0,-3.5) to [bend right = 20] node[sloped,above,midway]{\footnotesize $\gamma$} (2.25,-2.25);
\draw [dashed] (2.25,-2.25) to [bend left = 5] (-1.5,-1.5);
\draw (-1.5,-1.5) to [bend right = 10] (1.2,-1.2);
\draw [dashed] (1.2,-1.2) to [bend left = 5] (-.9,-.9);
\draw (-.9,-.9) to [bend right = 5]  (.75,-.75);
\draw [dashed] (.75,-.75) to [bend left = 5] (-.6,-.6);
\draw (-.6,-.6) to [bend right = 5]  (.45,-.45);
\draw [dashed] (.45,-.45) to [bend left = 5] (-.36,-.36);
\draw (-.36,-.36) to [bend right = 5]  (.3,-.3);
\draw [dashed] (.3,-.3) to [bend left = 5] (-.26,-.26);
\draw (-.26,-.26) to [bend right = 5]  (.22,-.22);
\draw [dashed] (.22,-.22) to [bend left = 5] (-.2,-.2);
\draw (-.2,-.2) to [bend right = 5]  (.17,-.17);
\draw [dashed] (.17,-.17) to [bend left = 5] (-.15,-.15);
\draw (-.15,-.15) to [bend right = 5]  (.13,-.13);
\draw [dashed] (.13,-.13) to [bend left = 5] (-.11,-.11);
\draw (-.11,-.11) to [bend right = 5]  (.09,-.09);
\draw [dashed] (.09,-.09) to [bend left = 5] (-.07,-.07);
\draw (-.07,-.07) to [bend right = 5]  (.06,-.06);
\draw (-.1,0) node[above]{\footnotesize $\mathcal{O}$};
\path [draw=black,fill=white] (0,0) circle (1/16); 
\end{tikzpicture}
\end{center}
Here $\mathcal{S}$ represents an $\mathbb{S}^2$-section of the cone, the straight lines represent various null normal lines, and the curve $\gamma$ represents an orbit of the vector field $S$ which generates the scaling symmetry along the cone $\{\hat{v} = 0\}$. In particular, $\gamma$ winds around infinitely often as it approaches the point $\mathcal{O}$. (Note that this diagram is drawn with respect to a different set of coordinates than those used in the statement of Theorem~\ref{themainextresult}; in fact, in the coordinate system of Theorem~\ref{themainextresult}, the orbits of $S$ appear as straight lines and instead the null normal lines twist around the cone. In those coordinates $S=u\partial_u$ and 
the null generator $\Lb=\partial_u + b^A\slashed{\nabla}_A$.) This twisting of the cone can be considered to be the key mechanism behind the formation of the naked singularity and serves as replacement for the role played by the logarithmic growth of the scalar field in Christodoulou's solutions. 
\end{remark}
\begin{remark}Lastly, we remark that we actually construct a large family of spacetimes which satisfy the conclusions of Theorem~\ref{themainextresult}. The various possible choices of incoming data along the null hypersurface $\{\hat{v} = 0\}$  are parametrized by choices of  ``$\left(\epsilon,\gamma,\delta,N_0,M_0,M_1\right)$-regular $4$-tuples,'' see Definition~\ref{Mreg} for the specifics.  
\end{remark}

We now quickly outline the rest of the introductory part of the paper. In Section~\ref{christisnakedsphsymmwhat} we will review the naked singularities of Christodoulou. Then, in Sections~\ref{letscompareyayaya} and~\ref{sectionpast} we will compare Christodoulou's solutions with the solutions constructed in our Theorem~\ref{themainextresult}. \emph{In particular, we will see that the solutions of Theorem~\ref{themainextresult} correspond to the exterior region of a naked singularity.} Finally, in Section~\ref{connectFeffGrah} we will discuss the relations between (the proof of our) Theorem~\ref{themainextresult} and the formal power series for self-similar solutions derived by Fefferman--Graham~\cite{FG1,FG2}. \emph{In particular, we will see that underlining the proof of Theorem~\ref{themainextresult} is a fundamentally new type of self-similarity of the Einstein vacuum equations.}

\subsection{Christodoulou's Naked Singularities for the Spherically Symmetric Einstein-Scalar-Field System}\label{christisnakedsphsymmwhat}
In the work~\cite{ChristNaked}, Christodoulou studied the spherically symmetric Einstein-scalar-field system and constructed examples of naked singularities. Thus, the original formulation of weak cosmic censorship conjectures fails for this system! (Of course, due to the rigidities imposed by Birkhoff's theorem, we cannot hope to set the scalar field to be $0$, and thus Christodoulou's constructions do not yield naked singularity solutions to the Einstein vacuum equations.) Despite the existence of these naked singularities, in a later work~\cite{Christodoulou4}, Christodoulou showed that \emph{generically} naked singularities do not occur for the spherically symmetric Einstein-scalar-field system, and thus weak cosmic censorship holds if we relax the statement to the requirement that naked singularities do not occur for \emph{generic} initial data. (It is in fact this relaxed version of weak cosmic censorship which is the currently accepted formulation.)

 In this section we will review in detail the solutions constructed by Christodoulou.

\subsubsection{The Spherically Symmetric Einstein-Scalar-Field System and $k$-Self-Similarity}
A solution to the Einstein-scalar-field system consists of a $3+1$ dimensional Lorentzian manifold $(\mathcal{M},h_{\mu\nu})$ and a real-valued scalar field $\phi : \mathcal{M} \to \mathbb{R}$ which satisfy
\begin{equation}\label{einscaleqn}
{\rm Ric}_{\mu\nu}\left(h\right) - \frac{1}{2}h_{\mu\nu}{\rm R}\left(h\right) = \partial_{\mu}\phi\partial_{\nu}\phi - \frac{1}{2}h_{\mu\nu}h^{\gamma\delta}\partial_{\gamma}\phi\partial_{\delta}\phi,\qquad h^{\mu\nu}D_{\mu}D_{\nu}\phi = 0,
\end{equation}
where ${\rm Ric}$ and ${\rm R}$ denote the Ricci tensor and scalar curvature respectively, and $D$ denotes the covariant derivative associated to $h$. 

Under the assumption of spherical symmetry, we may define the quotient manifold $\left(\mathscr{Q},g_{\mu\nu}\right) \doteq \left(\mathcal{M},h_{\mu\nu}\right)/SO(3)$, which will be a $1+1$ dimensional Lorentzian manifold with boundary, where the boundary consists of the fixed points of the $SO(3)$ action. We then have the area radius function $r : \mathscr{Q} \to [0,\infty)$ which gives the area of the corresponding $SO(3)$ orbit. Finally, the scalar field descends to a function $\phi : \mathscr{Q} \to \mathbb{R}$. The equations~\eqref{einscaleqn} then reduce to
\begin{equation}\label{einscalesdf2}
r\nabla_{\mu}\nabla_{\nu}r = \frac{1}{2}g_{\mu\nu}\left(1-\partial^{\gamma}r\partial_{\gamma}r\right) - r^2\left(\partial_{\mu}\phi\partial_{\nu}\phi - \frac{1}{2}g_{\mu\nu}g^{\gamma\delta}\partial_{\gamma}\phi\partial_{\delta}\phi\right),
\end{equation}
\begin{equation}\label{dijowmodkwo3}
\nabla^{\mu}\left(r^2\partial_{\mu}\phi\right) = 0,\qquad K\left(g\right) = r^{-2}\left(1-\partial^{\mu}r\partial_{\mu}r\right) + \partial^{\mu}\phi\partial_{\mu}\phi,
\end{equation}
where $K$ denotes the Gaussian curvature and $\nabla_{\mu}$ denotes the covariant derivative associated to $g_{\mu\nu}$. We refer to the system \eqref{einscalesdf2}-\eqref{dijowmodkwo3} as the spherically symmetric Einstein-scalar-field system.

There are two important symmetries of the spherically symmetric Einstein-scalar-field system:
\begin{enumerate}
	\item Given a triple $\left(g_{\mu\nu},r,\phi\right)$ solving the system~\eqref{einscalesdf2}-\eqref{dijowmodkwo3} and a constant $a > 0$, the triple $\left(a^2g_{\mu\nu},ar,\phi\right)$ will also solve \eqref{einscalesdf2}-\eqref{dijowmodkwo3}.
	\item Given a triple $\left(g_{\mu\nu},r,\phi\right)$ solving the system~\eqref{einscalesdf2}-\eqref{dijowmodkwo3} and a real constant $b$, the triple $\left(g_{\mu\nu},r,\phi+b\right)$ will also solve \eqref{einscalesdf2}-\eqref{dijowmodkwo3}.
\end{enumerate}

This leads to the definition of $k$-self-similarity:
\begin{definition}Let $k \in \mathbb{R}$. We say that a triple $\left(g_{\mu\nu},r,\phi\right)$ solving the spherically symmetric Einstein-scalar-field system is $k$-self-similar if there exists a $1$-parameter group of diffeomorphisms $\{f_a\}_{a>0}$ of $\mathscr{Q}$ such that 
\[f_a^*g_{\mu\nu} = a^2g_{\mu\nu},\qquad f_a^*r = ar,\qquad f_a^*\phi = \phi -k\log a.\] 
If a triple $\left(g_{\mu\nu},r,\phi\right)$ is $k$-self-similar with $k = 0$ then we say that the solution is scale-invariant.
\end{definition}
\subsubsection{Solutions of Bounded Variation}\label{bvbvbvbvbvbvbvbv}
The work~\cite{ChristBV} established well-posedness for the spherically symmetric Einstein-scalar-field system in the class of solutions of bounded variation. We will not give here a full review of bounded variation solutions; however it will be useful to recall the following facts about the behavior of the Hawking mass $m \doteq \frac{r}{2}\left(1-\left|\nabla r\right|^2\right)$ and the scalar field $\phi$ for any solution of bounded variation:
\begin{enumerate}
	\item For every outgoing null hypersurface $\mathcal{C}_{\rm out} \subset \mathscr{Q}$ with affine parameter $v$, the scalar field $\phi$ is required to be absolutely continuous along $\mathcal{C}_{\rm out}$, and $r\frac{\partial \phi}{\partial v}$ is required to be a function of bounded variation along $\mathcal{C}_{\rm out}$. Similarly, for every incoming null hypersurface $\mathcal{C}_{\rm in} \subset \mathscr{Q}$ with affine parameter $u$, the scalar field $\phi$ is required to be absolutely continuous along $\mathcal{C}_{\rm in}$, and $r\frac{\partial \phi}{\partial u}$ is required to be a function of bounded variation along $\mathcal{C}_{\rm in}$. Finally, we must also have that for each outgoing null hypersurface $\mathcal{C}_{\rm out}$ with a compact closure and  incoming null hypersurface $\mathcal{C}_{\rm in}$ with a compact closure that
	\begin{equation}\label{finiteflux}
	\int_{\mathcal{C}_{\rm out}}\left|\frac{\partial \phi}{\partial v}\right| dv < \infty,\qquad \int_{\mathcal{C}_{\rm in}}\left|\frac{\partial \phi}{\partial u}\right| du < \infty.
	\end{equation}
	Note that these integrals are invariant under a reparametrization of the affine parameters $u$ and $v$!
	\item Let $\Gamma$ denote the projection to $\mathscr{Q}$ of the fixed points of the $SO(3)$ action on $\mathcal{M}$. Then, for every null hypersurface $\mathcal{C}$ intersecting $\Gamma$, we must have that 
	\begin{equation}\label{finitemasssss}
	\lim_{r\to 0}\left(\frac{m}{r}\right)|_{\mathcal{C}} = 0.
	\end{equation}
\end{enumerate}
\subsubsection{Global Structure of $k$-Self-Similar Solutions}\label{nakedsingchristchrist}
In Section 2 of~\cite{ChristBV} Christodoulou analyzed a natural class of scale-invariant solutions and showed that it is possible to write them all down explicitly (cf.~\cite{robertsselfsimilar}). However, none of the solutions thus obtained are relevant for the construction of naked singularities.

More important for us will be the case when $k \neq 0$. These solutions, however, are significantly more complicated and the bulk of the work~\cite{ChristNaked} is concerned with a thorough analysis of these. For the current paper, the most relevant part of Christodoulou's analysis is the following:
\begin{theorem}\cite{ChristNaked}\label{k13yay} Let $0 < k^2 < 1/3$. Then there exist $k$-self-similar solutions such that the following properties hold:
\begin{enumerate}
	\item The $1+1$ Lorentzian manifold $\left(\mathscr{Q},g_{\mu\nu}\right)$ has a global expression in ``self-similar Bondi coordinates''
	\[g = -e^{2\nu}du^2 - 2e^{\nu+\lambda}dudr,\qquad \mathscr{Q} \doteq \{\left(u,r\right) \in (0,-\infty) \times [0,\infty)\},\] 
	where $\nu\left(u,r\right) = \tilde\nu\left(-\frac{r}{u}\right)$ and $\lambda\left(u,r\right) = \tilde\lambda\left(-\frac{r}{u}\right)$ for suitable functions $\tilde\nu$ and $\tilde\lambda$. 
	\item The Penrose diagram of $\left(\mathscr{Q},g_{\mu\nu}\right)$ is given by 
	
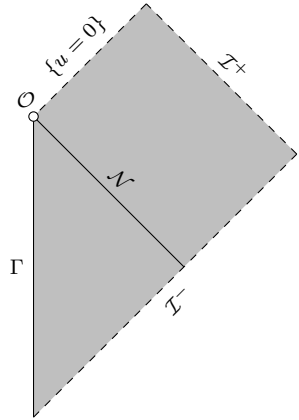
\begin{center}
\begin{tikzpicture}[scale = 1]
\fill[lightgray] (0,0) -- (0,-4) -- (3.5,-.5) -- (1.5,1.5) -- (0,0);
\draw (0,0) -- (2,-2) node[sloped,above,midway]{\footnotesize $\mathcal{N}$};
\draw[dashed] (0,-4) -- (3.5,-.5) node[sloped,below,midway]{\footnotesize $\mathcal{I}^-$};
\draw (0,-4) -- (0,0) node[left,midway]{\footnotesize $\Gamma$};
\draw[dashed] (3.5,-.5) -- (1.5,1.5) node[sloped,above,midway]{\footnotesize $\mathcal{I}^+$};
\draw[dashed] (1.5,1.5) -- (0,0) node[sloped,above,midway]{\footnotesize $\{u = 0\}$};
\path [draw=black,fill=white] (0,0) circle (1/16); 
\draw (-.1,0) node[above]{\footnotesize $\mathcal{O}$};

\end{tikzpicture}
\end{center}
Here $\Gamma$ denotes the boundary of $\mathscr{Q}$ where $r = 0$ and which corresponds to the projection of the fixed points of the $SO(3)$ action on $\mathcal{M}$. The point $\mathcal{O}$ corresponds to $(u,r) = (0,0)$ and is a terminal singularity (see point~\ref{terminal} below). Lastly, $\mathcal{N}$ denotes the past light cone of the singular point $\mathcal{O}$, and $\mathcal{I}^{\pm}$ (future/past null infinity) corresponds to the ideal endpoints of complete future/past oriented null geodesics.
\item\label{terminal}  The hypersurface $\mathcal{N}$ is future null geodesically incomplete, yet the solution cannot be extended to $\mathcal{O}$ and remain a solution of bounded variation. This is a consequence of the requirements~\eqref{finiteflux} and~\eqref{finitemasssss} and either of the following two facts:
\begin{equation}\label{thismakesitnotinBV}
\frac{2m}{r}|_{\mathcal{N}} = \frac{k^2}{1+k^2} \neq 0,
\end{equation}
\begin{equation}\label{thismakesitnotinBV2}
n(\phi)|_{\mathcal{N}} = \frac{\left(1+k\right)^{1/2}k}{r},
\end{equation}
where $n|_{\mathcal{N}} = \left(2e^{-\nu}\frac{\partial}{\partial u} - e^{-\lambda}\frac{\partial}{\partial r}\right)|_{\mathcal{N}}$  denotes an ingoing null vector normal to $\mathcal{N}$, and we note that since it may be shown that $\lambda$ and $\nu$ are constant along $\mathcal{N}$, there exists $c_0, c_1\in (0,\infty)$ so that $\left(u(s),r(s)\right)_{\{s>0\}} = \left(-c_0s,c_1s\right)$ is an integral curve of $n$ along $\mathcal{N}$ with $\mathcal{O}$ corresponding to $s = 0$.
\item Along any null geodesic terminating on $\mathcal{I}^-$ or $\mathcal{I}^+$ we have that $r\to \infty$. 
\item The triple $\left(g_{\mu\nu},\phi,r\right)$ forms a solution of bounded variation (where we emphasize that $\mathcal{O}$ is \underline{not} included in the spacetime). We have that the radius function $r$ is in $C^2\left(\mathscr{Q}\right)$, the Gauss curvature  $K\left(g\right)$ is in $C^0\left(\mathscr{Q}\right)$, and the scalar field $\phi$ is in $C^2\left(\mathscr{Q}\setminus \mathcal{N}\right)$ and in $C^{1,\frac{k^2}{1-k^2}}\left(\mathscr{Q}\right)$. 
\end{enumerate}
\end{theorem}
We note that solutions of this type were also studied numerically in the work~\cite{bradyssc}.

It is worth emphasizing that because the gradient of the scalar field is H\"{o}lder continuous instead of being merely a function of bounded variation, the triple $\left(g_{\mu\nu},r,\phi\right)$ may be considered as being \underline{more regular} than a solution of bounded variation (we again remind the reader that the point $\mathcal{O}$ is \underline{not} included in the spacetime). (In fact the solution is also more regular than the AC-class of absolutely continuous data, cf.~\cite{Chrmil,Christodoulou4}.)

\subsubsection{Asymptotically Flat Truncations}\label{isitSig}
The $k$-self-similar solutions constructed in Theorem~\ref{k13yay} are not asymptotically flat; in particular, all of the solutions constructed by Theorem~\ref{k13yay} have
\[\lim_{r\to\infty}m\left(u,r\right) = \infty.\]
However, the solutions may be ``truncated'' along an outgoing null hypersurface to construct an asymptotically flat solution:
\begin{theorem}\cite{ChristNaked}\label{k13yay2} Let $0 < k^2 < 1/3$. There exists a $1+1$ dimensional Lorentzian manifold $\left(\mathscr{Q},g_{\mu\nu}\right)$ and functions $r,\phi : \mathscr{Q} \to \mathbb{R}$ so that $\left(g_{\mu\nu},r,\phi\right)$ solves the spherically symmetric Einstein-scalar-field system and $\left(\mathscr{Q},g_{\mu\nu}\right)$ has the following Penrose diagram:
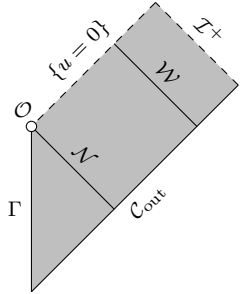
\begin{center}
\begin{tikzpicture}[scale = 1.1]
\fill[lightgray] (0,0) -- (0,-2) -- (2.5,.5) -- (1.5,1.5) -- (0,0);
\draw (0,0) -- (1,-1) node[sloped,above,midway]{\footnotesize $\mathcal{N}$};
\draw (0,-2) -- (2.5,.5) node[sloped,below,midway]{\footnotesize $\mathcal{C}_{\rm out}$};
\draw (0,-2) -- (0,0) node[left,midway]{\footnotesize $\Gamma$};
\draw[dashed] (2.5,.5) -- (1.5,1.5) node[sloped,above,midway]{\footnotesize $\mathcal{I}^+$};
\draw[dashed] (1.5,1.5) -- (0,0) node[sloped,above,midway]{\footnotesize $\{u = 0\}$};
\path [draw=black,fill=white] (0,0) circle (1/16); 
\draw (-.1,0) node[above]{\footnotesize $\mathcal{O}$};
\draw (2,0)--(1,1) node[sloped,above,midway]{\footnotesize $\mathcal{W}$};
\end{tikzpicture}
\end{center}
In the past of $\mathcal{W}$ the spacetime is identical to that produced by Theorem~\ref{k13yay}. As with the solutions of Theorem~\ref{k13yay}, we have that $r\to\infty$ for any null geodesic terminating on $\mathcal{I}^+$. However, in contrast to the solutions produced by Theorem~\ref{k13yay}, the scalar field $\phi|_{\mathcal{C}_{\rm out}}$ vanishes for sufficiently large $r$, and the solution is asymptotically flat. Lastly, future null infinity is incomplete in the sense of Definition~\ref{nakeddef}.
\end{theorem}

We emphasize that despite the limited regularity of the solutions of Theorem~\ref{k13yay2}, these solutions may be considered naked singularities for the following three reasons:
\begin{enumerate}
\item The initial data along $\mathcal{C}_{\rm out}$ is more regular than that of solutions of bounded variation.
\item The work~\cite{ChristBV} established well-posedness for the spherically symmetric Einstein-scalar-field system in the class of solutions of bounded variation.
\item The solutions of Theorem~\ref{k13yay2} cannot be extended to the point $\mathcal{O}$ as a solution of bounded variation.
\end{enumerate}

\subsection{Comparison of the solutions of Theorem~\ref{themainextresult} with the Naked Singularities of Christodoulou}\label{letscompareyayaya}
In this section we will compare the spacetimes constructed by this paper's Theorem~\ref{themainextresult} with those constructed by Christodoulou's Theorem~\ref{k13yay2}, and we will see that the solutions of Theorem~\ref{themainextresult} correspond to the exterior region of a naked singularity. 

There is, of course, the obvious difference that we do not show in this paper that the spacetimes of Theorem~\ref{themainextresult} contain a past complete extension to the past of $\{\hat{v} = 0\}$. We will discuss the problem of constructing such an extension in Section~\ref{sectionpast}. Thus we now focus on the region in Christodoulou's solutions to the future of the hypersurface $\mathcal{N}$. It will be useful to keep in mind the schematic rule that when comparing the spherically symmetric Einstein-scalar-field system with the Einstein vacuum equations one should identify $\partial_u\phi$ with the ingoing shear $\hat{\underline{\chi}}_{AB}$ and $\partial_v\phi$ with the outgoing shear $\hat{\chi}_{AB}$. 
\begin{enumerate}
	\item (Regularity of Initial Outgoing Data) Under the correspondence of $\partial_v\phi$ and $\hat{\chi}_{AB}$ we find that the data along $\mathcal{C}_{\rm out}$ is analogous to the data along $\{u = -\underline{v}^2\}$ in that both $\partial_v\phi$ and $\hat{\chi}_{AB}$ are H\"{o}lder continuous.  Of course, for the Einstein vacuum equations, we cannot appeal to Christodoulou's well-posedness for bounded variation solutions. However, the works~\cite{impulsivefirst,impulsive} have established a \emph{local} well posedness result for data where $\hat{\chi}_{AB}$ (and a suitable number of angular derivatives), though required to vanish near the tip of the cone, are otherwise allowed to only lie in $L^2$ along an outgoing null hypersurface. The theory developed by~\cite{impulsivefirst,impulsive} does not concern itself with the behavior near the ``axis''; however  we conjecture that a well-posedness result including the axis may be established for initial data where $\hat{\chi}_{AB}$ and $\hat{\underline{\chi}}_{AB}$ and a suitable number of angular derivatives thereof are H\"{o}lder continuous.
	\item\label{9090099090128jnomi2} (Singular Boundary) The asymptotic behavior of the Hawking mass and $\hat{\underline{\chi}}_{AB}$ along $\{\hat{v} = 0\}$ given by~\eqref{masskfow} and~\eqref{shearkofw} are analogous to the behavior of the Hawking mass and $n(\phi)$ along $\mathcal{N}$ given by~\eqref{thismakesitnotinBV} and~\eqref{thismakesitnotinBV2}.  Outside of spherical symmetry the Hawking mass is not invariant under a change of foliation of the cone $\{\hat{v} = 0\}$. However, the blow-up of the shear $\hat{\underline{\chi}}_{AB}$ can be re-phrased in a more invariant fashion as follows. Let $\gamma(s)$ be the parameterization of a future oriented null geodesic $\gamma$ along $\{\hat{v} = 0\}$ induced by the  normal vector field $\partial_u + b^A\slashed{\nabla}_A$. Then one may show that there exists such $\gamma$ with $\int_{\gamma}\left|\hat{\underline{\chi}}\right|_{\slashed{g}}\, ds  = \infty$. This statement is reparametrization invariant and corresponds to a logarithmic singularity for $\slashed{g}_{AB}$ along $\gamma$ (which in turn formally corresponds to the logarithmic blow-up of the scalar-field in Christodoulou's spacetimes). Furthermore, it is also possible to show in this case the existence of a Jacobi field $J$ which blows-up along $\gamma$. (See also point~\ref{point2} in Section~\ref{sectionpast} below.)
	\item (Asymptotic Flatness and Incompleteness of Future Null Infinity) Both solutions possess an asymptotically flat outgoing null hypersurface and do not posses a complete future null infinity.
	\item\label{dwasevbndw} (Underlying Self-Similar Solution) As explained above, Christodoulou's solutions are obtained by ``truncation'' of an underlying self-similar solution. In contrast, there is no self-similar solution produced at an intermediate stage during the proof of Theorem~\ref{themainextresult}. However, via an amalgamation of the techniques and estimates of this paper along with the methods developed in our previous work~\cite{scaleinvariant}, it is possible to construct an underlying solution which is self-similar in the sense of possessing a vector field $S$ with $\mathcal{L}_Sg_{\mu\nu} = 2g_{\mu\nu}$, and so that, in analogy with the relation of Theorem~\ref{k13yay} and Theorem~\ref{k13yay2}, a suitable truncation yields spacetimes as in Theorem~\ref{themainextresult}. We will not pursue this line of approach in this paper since going though a self-similar solution does not lead to any essential simplification of the proof. Nevertheless, the solutions of Theorem~\ref{themainextresult} will be ``self-similar as $\frac{v}{-u} \to 0$'' in the sense that
	\[\mathcal{L}_Sg_{\mu\nu} - 2g_{\mu\nu} \to 0\text{ as }\frac{v}{-u} \to 0,\qquad S \doteq u\frac{\partial}{\partial u} + \left(1-2\kappa\right)\hat{v}\frac{\partial}{\partial \hat{v}},\]
where $\kappa$ is a positive constant which satisfies $\kappa \sim \epsilon^2$. In terms of the normal vector $e_3$, we will have, in particular, that $\left(u\mathcal{L}_{e_3}- u^{-1}\tilde{b}^A\mathcal{L}_{\partial_{\theta^A}}\right)g_{\mu\nu}|_{\hat{v} = 0} = 2g_{\mu\nu}|_{\hat{v} = 0}$. The logarithmic twisting induced by the flow of $u^{-1}\tilde{b}^A$ can be considered an analogue to the $k$-self-similar actions on the scalar field $\phi \mapsto \phi + k\log(u)$.
\end{enumerate}

\subsection{Constructing the Interior Solution}\label{sectionpast}
In this paper we will not establish any results concerning extensions of the spacetime to the past of the hypersurface $\{\hat{v} = 0\}$. However, in a current work in progress we construct a past extension of the spacetime from Theorem~\ref{themainextresult} where the new spacetime $\left(\tilde{\mathcal{M}},\tilde{g}_{\mu\nu}\right)$ takes the double-null form~\eqref{thisisgdouble} for $\left(u,\hat{v},\theta^A\right) \in [-\underline{v}^2,0) \times [u,\infty) \times \mathbb{S}^2$ and 
\begin{enumerate}
	\item The spacetimes $\left(\tilde{\mathcal{M}},\tilde{g}_{\mu\nu}\right)$ and $\left(\mathcal{M},g_{\mu\nu}\right)$ coincide in  the region $\left(u,\hat{v},\theta^A\right) \in [-\underline{v}^2,0) \times [0,\infty) \times \mathbb{S}^2$.
	\item $\left(\tilde{\mathcal{M}},\tilde{g}_{\mu\nu}\right)$ has the Penrose diagram:
	
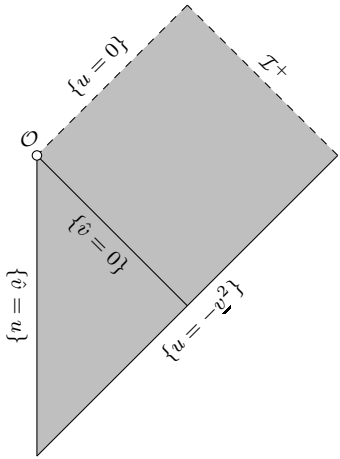
\begin{center}
\begin{tikzpicture}[scale = 1]
\fill[lightgray] (0,0)--(0,-4)--(4,0)  -- (2,2) -- (0,0);
\draw (0,0) -- (0,-4) node[sloped,below,midway]{\footnotesize $\{\hat{v} = u\}$};
\draw (0,0) -- (2,-2) node[sloped,below,midway]{\footnotesize $\{\hat{v} = 0\}$};
\draw (0,-4) -- (4,0) node[sloped,below,midway]{\footnotesize $\{u=-\underline{v}^2\}$};
\draw[dashed] (4,0) -- (2,2) node[sloped,above,midway]{\footnotesize $\mathcal{I}^+$};
\draw[dashed] (2,2) -- (0,0) node[sloped,above,midway]{\footnotesize $\{u = 0\}$};
\path [draw=black,fill=white] (0,0) circle (1/16); 
\draw (-.1,0) node[above]{\footnotesize $\mathcal{O}$};
\end{tikzpicture}
\end{center}
\item\label{pk23} We have that $\tilde{g}_{\mu\nu} \in C^N\left(\mathcal{M}\setminus\left(\{\hat{v}=0\}\cup \{u=\hat{v}\}\right)\right)$, $\tilde{g}_{\mu\nu} \in C^{1,c\epsilon^2}\left(\mathcal{M}\right)$ (where $c$ is the same constant from point~\ref{lp3} of Theorem~\ref{themainextresult}), and $(\tilde{g}_{\mu\nu},\partial_{\hat{v}}\tilde{g}_{\mu\nu}) \in C^N\left(\{\hat{v}=0\}\right)$. In a neighborhood of any point on the ``axis'' $\{ \hat{v} = u\}$ there exists a new coordinate system so that $\tilde{g}_{\mu\nu}$ is $C^N$. (We emphasize that clearly the point $\mathcal{O}$ does not lie on $\{u=\hat{v}\}$.)
\item\label{point2} The timelike curve $\pi(s)$ defined by $s \mapsto \left(u,\hat{v}\right) = (s,s)$ for $s \in [-\underline{v}^2,0)$ corresponds to a smooth timelike curve in $\left(\tilde{M},\tilde{g}_{\mu\nu}\right)$ (see point~\ref{pk23} above), has a finite length $2\underline{v}^2$, is future inextendible, and there does not exist any continuous Lorentzian extension of the spacetime $\left(\tilde{\mathcal{M}},\tilde{g}_{\mu\nu}\right)$ where $\pi$ is extendible to a curve with length greater than $2\underline{v}^2$. 
\end{enumerate}
We remark that in order to glue in these interior solutions with our exterior solutions, is it important that in Theorem~\ref{themainextresult} we actually have considerable flexibility in the choices of the lapse $\tilde\Omega$ and metric $\tilde{\slashed{g}}_{AB}$ (see~\eqref{choices}).

\subsection{Connections to Fefferman--Graham Theory}\label{connectFeffGrah}
As we will review in more detail in Section~\ref{asss}, in the works~\cite{FG1,FG2} Fefferman and Graham classified formal power series expansions corresponding to a certain type of self-similar solution, and in the work~\cite{scaleinvariant} we showed that all of these expansions correspond to true solutions of the Einstein vacuum equations. The solutions considered by Fefferman and Graham all share the property that there exists a null hypersurface $\mathcal{H}$ such that the conformal Killing field $K^{\mu}$ is normal to $\mathcal{H}$. Among other things, this implies that the cone $\mathcal{H}$ is shear free. In contrast, the underlying self-similar solution for the spacetimes of Theorem~\ref{themainextresult} (see point~\ref{dwasevbndw} in Section~\ref{letscompareyayaya}) posses a null hypersurface $\mathcal{H}$ where the conformal Killing field $K^{\mu}$ is tangent but \emph{not} normal, and, in particular, the cone is not shear free. Thus, this provides a genuinely new local model for self-similar solutions.  One may draw an analogy for the relation between these new solutions and the solutions of Fefferman--Graham with the relation of the rotating Kerr black hole solutions and the Schwarzschild solutions. Finally, we note that one expects analogues of this construction to work also in higher dimensions.

Beyond the generation of new local models for self-similar solutions, the (proof of) Theorem~\ref{themainextresult} is also relevant for the \emph{global} study of Fefferman and Graham's self-similar solutions. We briefly explain: Fefferman and Graham's solutions in $3+1$ dimensions are parametrized by two choices of data, $\slashed{g}_{AB}|_{(u,v) = (-1,0)}$ and ${\rm tf}\partial_v\slashed{g}_{AB}|_{(u,v) = (-1,0)}$, where ${\rm tf}$ denotes the trace-free part of a symmetric $(0,2)$-$\mathbb{S}^2_{u,v}$ tensor. (See Section~\ref{asss}.) In~\cite{scaleinvariant} we showed that given such a pair, there is some $\epsilon > 0$ so that a corresponding self-similar solution exists in the region $(u,v) \in \{0 \leq \frac{v}{-u} < \epsilon\} \cap \{u \in (-\infty,0)\}$. It is thus natural to ask about the global behavior of this solution. In particular, if the data $\slashed{g}_{AB}|_{(u,v) = (-1,0)}$ is close to the round metric and ${\rm tf}\partial_v\slashed{g}_{AB}|_{(u,v) = (-1,0)}$ is suitably small (where ${\rm tf}$ denotes the trace-free part of a symmetric $(0,2)$-tensor), do we obtain a ``global'' solution in the region $\left(\{v \geq 0\} \cap \{u \leq 0\}\right) \setminus\{(u,v) = (0,0)\}$?\footnote{One can also look in the ``interior region'' corresponding to $\{v \leq 0\}$ where one expects the problem to be elliptic as opposed to hyperbolic. However, in $3+1$ dimensions one does not expect to find any non-trivial interior solutions corresponding to the Fefferman--Graham data along the cone $\{v = 0\}$. In dimensions strictly higher than $3+1$ this rigidity disappears, and the problem of constructing global interior solutions for small data was positively resolved in work of Graham--Lee~\cite{grahamlee}. Finally, we note that in dimensions strictly larger than $3+1$, the global behavior in the exterior region $\{v \geq 0\}$, for a certain restricted class of small data, reduces to the problem of the stability of de-Sitter space to small perturbations from $\mathcal{I}^-$, various aspects of which have been positively resolved in the works~\cite{friedrich1986existence,andsitter,ringstab}.} It follows from a combination of the techniques of~\cite{scaleinvariant} and the estimates behind the proof of Theorem~\ref{themainextresult} that one has existence in a region  $\left(\{v \geq 0\} \cap \{u < 0\}\right) \setminus\{(u,v) = (0,0)\}$, and, after a suitable change of coordinates, the metric $g$ extends to the cone $\{u = 0\}$ in a H\"{o}lder continuous fashion. It would be very interesting, however, to determine whether or not the cone $\{u = 0\}$ is shear free and thus is itself locally modeled on a Fefferman--Graham solution. We note that the analogous statement for scale-invariant solutions to the spherically symmetric Einstein-scalar-field system is true and is related to Christodoulou's proof of well-posedness for solutions of bounded variation~\cite{ChristBV}.

\subsection{Acknowledgements} The authors would like to thank Mihalis Dafermos for valuable discussions and comments on the manuscript. IR is partially supported by the NSF 
grant DMS-1709270 and a Simons Investigator Award. YS acknowledges support from NSF grant DMS-1900288, from an Alfred P. Sloan Fellowship in Mathematics, and from NSERC discovery grants RGPIN-2021-02562 and DGECR-2021-00093.

\section{Preliminaries}
\subsection{Equations of the Double-Null Foliation}
In this section we will recall the form of the Einstein equations in a double null foliation (see~\cite{KN} for detailed derivations). We start with a $3+1$ dimensional Lorentzian manifold $\left(\mathcal{M},g_{\mu\nu}\right)$ solving the Einstein vacuum equations. We let $\{\theta^A\}$ denote local coordinates\footnote{Unless said otherwise, we will assume it understood by the reader that each coordinate function $\theta^A$ is only defined on a suitable coordinate patch of $\mathbb{S}^2$.} on $\mathbb{S}^2$ and assume that for some open $\mathcal{U} \subset \mathbb{R}^2$, there exists coordinates $\left(u,v,\theta^A\right)\in \mathcal{U} \times \mathbb{S}^2 $ so that the metric $g_{\mu\nu}$ takes the form
\begin{equation}\label{metricform}
g = -2\Omega^2\left(du\otimes dv + dv\otimes du\right) + \slashed{g}_{AB}\left(d\theta^A - b^Adu\right)\otimes\left(d\theta^B - b^Bdu\right).
\end{equation}
Here $\Omega$ is a function on $\mathcal{U} \times \mathbb{S}^2$ called the ``lapse,'' for each $(u,v)$, $\slashed{g}_{AB}$ denotes the induced Riemannian metric on the corresponding copy of $\mathbb{S}^2$, and for each $(u,v)$, $b_A$, called the ``shift,'' is a $1$-form on the corresponding copy of $\mathbb{S}^2$. We will denote the copy of $\mathbb{S}^2$ at a particular $(u,v)$ by $\mathbb{S}^2_{u,v}$, and we refer to any tensor field on $\mathcal{M}$ which is tangential to each $\mathbb{S}^2_{u,v}$ as an ``$\mathbb{S}^2_{u,v}$-tensor'' (see~\cite{KN}). We use the standard convention that Latin indices are reserved for $\mathbb{S}^2_{u,v}$-tensors. The covariant derivative associated to $g_{\mu\nu}$ will be denoted by $D_{\mu}$, the projection of $D_4$ and $D_3$ to $\mathbb{S}^2_{u,v}$ will be denoted by $\nabla_4$ and $\nabla_3$ respectively, and the induced covariant derivative of $\mathbb{S}^2_{u,v}$ will be denoted by $\slashed{\nabla}_A$. We will assume that $\left(\mathcal{M},g_{\mu\nu}\right)$ is oriented, which allows us to define $\slashed{\epsilon}_{AB}$ to be the volume form corresponding to $\slashed{g}_{AB}$. Unless indicated otherwise, norms of $\mathbb{S}^2_{u,v}$-tensors are computed with respect to $\slashed{g}_{AB}$, indices of $\mathbb{S}^2_{u,v}$-tensors are raised and lowered with respect to $\slashed{g}_{AB}$, and the Hodge star operator ${}^*$ applied to an $\mathbb{S}^2_{u,v}$-tensor refers to a contraction with $\slashed{\epsilon}_{AB}$. 

One consequence of the form~\eqref{metricform} of the metric is that the level sets of $u$ and $v$ are null hypersurfaces. This leads us to define the following null pair:
\[e_4 \doteq \Omega^{-1}\partial_v,\qquad e_3 \doteq \Omega^{-1}\left(\partial_u + b^A\partial_{\theta^A}\right),\]
which will then satisfy 
\[g\left(e_4,e_3\right) = -2.\]
The Ricci coefficients are the following $\mathbb{S}^2_{u,v}$ tensors:
\[\chi_{AB} \doteq g\left(D_Ae_4,e_B\right),\qquad \underline{\chi}_{AB} = g\left(D_Ae_3,e_B\right),\]
\[\eta_A \doteq -\frac{1}{2}g\left(D_3e_A,e_4\right),\qquad \underline{\eta}_A \doteq -\frac{1}{2}g\left(D_4e_A,e_3\right),\]
\[\omega \doteq -\frac{1}{4}g\left(D_4e_3,e_4\right),\qquad \underline{\omega} \doteq -\frac{1}{4}g\left(D_3e_4,e_3\right),\]
\begin{equation}\label{okokokthisistrosion}
\zeta_A \doteq \frac{1}{2}g\left(D_Ae_4,e_3\right).
\end{equation}
We will use $\psi$ to denote an arbitrary Ricci coefficient.

We often refer to the $1$-form $\zeta_A$ as the ``torsion'' $1$-form. Note that $\chi_{AB}$ and $\underline{\chi}_{AB}$ are simply the second fundamental forms of the $\mathbb{S}^2_{u,v}$ and hence are symmetric tensors. It  will be  convenient to split $\chi_{AB}$ and $\underline{\chi}_{AB}$ into their trace and trace-free parts:
\[\chi_{AB} \doteq \hat{\chi}_{AB} + \frac{1}{2}{\rm tr}\chi\slashed{g}_{AB},\qquad \underline{\chi}_{AB} \doteq \hat{\underline{\chi}}_{AB} + \frac{1}{2}{\rm tr}\underline{\chi}\slashed{g}_{AB}.\]

The Ricci coefficients are related to the derivatives of the metric quantities $\Omega$, $\slashed{g}_{AB}$, and $b_A$ as follows:
\begin{equation}\label{osjf}
\omega = -\frac{1}{2}\nabla_4\log\Omega,\qquad \underline{\omega} = -\frac{1}{2}\nabla_3\log\Omega,\qquad \mathcal{L}_{\partial_v}b^A = -4\Omega^2\zeta^A,
\end{equation}
\[\eta_A = \zeta_A + \slashed{\nabla}_A\log\Omega,\qquad \underline{\eta}_A = -\zeta_A + \slashed{\nabla}_A\log\Omega,\]
\begin{equation}\label{okmdwqejodq12ed}
\mathcal{L}_{e_4}\slashed{g}_{AB} = 2\chi_{AB},\qquad \mathcal{L}_{e_3}\slashed{g}_{AB} = 2\underline{\chi}_{AB}.
\end{equation}
Here $\mathcal{L}$ denotes the Lie-derivative.

We let $R$ denote the curvature tensor of $g$ and then define the null curvature components as follows:
\[\alpha_{AB} \doteq R\left(e_A,e_4,e_B,e_4\right),\qquad \underline{\alpha}_{AB} \doteq R\left(e_A,e_3,e_B,e_3\right),\]
\[\beta_A \doteq \frac{1}{2}R\left(e_A,e_4,e_3,e_4\right),\qquad \underline{\beta}_A \doteq \frac{1}{2}R\left(e_A,e_3,e_3,e_4\right),\]
\[\rho \doteq \frac{1}{4}R\left(e_4,e_3,e_4,e_3\right),\qquad \sigma \doteq \frac{1}{4}\left({}^*R\right)\left(e_4,e_3,e_4,e_3\right).\]
Here ${}^*$ denotes the Hodge star operator. The symmetries of the curvature tensor and the Einstein vacuum equations imply that $\alpha$ and $\underline{\alpha}$ are symmetric trace-free tensors.  The curvature components listed above suffice to reconstruct the entire curvature tensor (this fact uses both that we are in $3+1$ dimensions and that the Einstein vacuum equations are satisfied). We will often use $\Psi$ to denote an arbitrary null curvature component.

Next we recall the notion of signature from~\cite{CK}.\footnote{Our definition is actually the negative of the definition from~\cite{CK}.}
\begin{definition}\label{signature}For a Ricci coefficient $\psi$ or a null curvature component $\Psi$ we define $s\left(\psi\right)$, the signature of $\psi$  or $s\left(\Psi\right)$, the signature of $\Psi$, by
\[s\left(\psi\right) \doteq \#\left(e_3\right) -\#\left(e_4\right),\qquad s\left(\Psi\right) \doteq \#\left(e_3\right) -\#\left(e_4\right),\]
where  $\#\left(e_3\right)$ denotes the number of $e_3$'s which show up in the definition of $\psi$ and $\#\left(e_4\right)$ denotes the number of $e_4$'s which show up in the definition. We have
\[s\left(\hat{\chi}\right) = -1,\qquad s\left({\rm tr}\chi\right) = -1,\qquad s\left(\omega\right) = -1,\]
\[s\left(\eta\right) = 0,\qquad s\left(\underline{\eta}\right) = 0,\qquad s\left(\zeta\right) = 0,\]
\[s\left(\hat{\underline{\chi}}\right) = 1,\qquad s\left({\rm tr}\underline{\chi}\right) = 1,\qquad s\left(\underline{\omega}\right) = 1,\]
\[s\left(\alpha\right) = -2,\qquad s\left(\beta\right) = -1,\]
\[s\left(\rho\right) = 0,\qquad s\left(\sigma\right) = 0,\]
\[s\left(\underline{\alpha}\right) = 2,\qquad s\left(\underline{\beta}\right) = 1.\]

\end{definition}
The derivatives of the Ricci coefficients are related to the null curvature components by the following set of ``null-structure equations'':
\begin{align}
\label{4trchi}\nabla_4{\rm tr}\chi + \frac{1}{2}\left({\rm tr}\chi\right)^2 &= -\left|\hat{\chi}\right|^2 - 2\omega{\rm tr}\chi,
\\ \label{4hatchi} \nabla_4\hat{\chi}_{AB} +{\rm tr}\chi \hat{\chi}_{AB} &= -\alpha_{AB}  -2\omega\hat{\chi}_{AB},
\\ \label{3truchi} \nabla_3{\rm tr}\underline{\chi} + \frac{1}{2}\left({\rm tr}\underline{\chi}\right)^2 &= -\left|\hat{\underline{\chi}}\right|^2 - 2\underline{\omega}{\rm tr}\underline{\chi},
\\ \label{3hatuchi} \nabla_3\underline{\hat{\chi}}_{AB} +{\rm tr}\underline{\chi} \underline{\hat{\chi}}_{AB} &= -\underline{\alpha}_{AB} -2\underline{\omega}\underline{\hat{\chi}}_{AB},
\\ \label{3hatchi} \nabla_3\hat{\chi}_{AB} +\frac{1}{2}{\rm tr}\underline{\chi}\hat{\chi}_{AB} &=  2\underline{\omega}\hat{\chi}_{AB} + \left(\slashed{\nabla}\hat\otimes \eta\right)_{AB} + \left(\eta\hat\otimes \eta\right)_{AB} - \frac{1}{2}{\rm tr}\chi \hat{\underline{\chi}}_{AB},
\\ \label{3trchi}\nabla_3{\rm tr}\chi + \frac{1}{2}{\rm tr}\chi{\rm tr}\underline{\chi} &= 2\rho+ 2\underline{\omega}{\rm tr}\chi + 2\slashed{\rm div}\eta + 2\left|\eta\right|^2 - \hat{\chi}\cdot\hat{\underline{\chi}},
\\ \label{4hatuchi} \nabla_4\hat{\underline{\chi}}_{AB} + \frac{1}{2}{\rm tr}\chi \hat{\underline{\chi}}_{AB}&=  2\omega\hat{\underline{\chi}}_{AB} + \left(\slashed{\nabla}\hat\otimes \underline\eta\right)_{AB} + \left(\underline\eta\hat\otimes \underline\eta\right)_{AB}- \frac{1}{2}{\rm tr}\underline{\chi} \hat{\chi}_{AB},
\\ \label{4truchi} \nabla_4{\rm tr}\underline{\chi} + \frac{1}{2}{\rm tr}\chi{\rm tr}\underline{\chi} &= 2\rho+ 2\omega{\rm tr}\underline{\chi} + 2\slashed{\rm div}\underline{\eta} + 2\left|\underline\eta\right|^2 - \hat{\chi}\cdot\hat{\underline{\chi}},
\end{align}
\begin{align}
\label{4eta} \nabla_4\eta_A &= -\left(\chi\cdot\left(\eta-\underline{\eta}\right)\right)_A -\beta_A,
\\ \label{3ueta} \nabla_3\underline{\eta}_A &= -\left(\underline{\chi}\cdot\left(\underline\eta-\eta\right)\right)_A +\underline{\beta}_B,
\\ \label{curleta} \slashed{\rm curl}\  \eta  &= \sigma + \frac{1}{2}\hat{\underline{\chi}}\wedge \hat{\chi}
\\ \label{curlueta}\slashed{\rm curl}\ \underline{\eta} &= -\sigma - \frac{1}{2}\hat{\underline{\chi}}\wedge \hat{\chi},
\\ \label{4uomega} \nabla_4\underline{\omega} &= \frac{1}{2}\rho + 2\underline{\omega}\omega + \frac{1}{2}\left|\eta\right|^2 - \eta\cdot\underline{\eta},
\\ \label{3omega} \nabla_3\omega &= \frac{1}{2}\rho + 2\underline{\omega}\omega + \frac{1}{2}\left|\underline\eta\right|^2 - \eta\cdot\underline{\eta},
\end{align}
and

\begin{align}\label{genGauss}
K &= -\rho +\frac{1}{2}\hat{\chi}\cdot\hat{\underline{\chi}} -\frac{1}{4}{\rm tr}\chi{\rm tr}\underline{\chi},
\\ \label{tcod1} \left(\slashed{\rm div}\ \hat{\chi}\right)_A-\frac{1}{2}\slashed{\nabla}_A{\rm tr}\chi &= -\beta_A+ \frac{1}{2}{\rm tr}\chi \zeta_A - \left(\zeta\cdot\hat{\chi}\right)_A,
\\ \label{tcod2} \left(\slashed{\rm div}\ \hat{\underline{\chi}}\right)_A -\frac{1}{2}\slashed{\nabla}_A{\rm tr}\underline\chi  &=\underline{\beta}_A- \frac{1}{2}{\rm tr}\underline\chi \zeta_A + \left(\zeta\cdot\underline{\hat{\chi}}\right)_A,
\end{align}
where $K$ denotes the Gaussian curvature of $\slashed{g}_{AB}$, and we have used the following definitions for $1$-forms $\psi_A$ and symmetric trace-free $(0,2)$-tensors $\phi_{AB}$:
\begin{align*}
\left(\psi^{(1)}\hat{\otimes}\psi^{(2)}\right)_{AB} &\doteq \psi^{(1)}_A\psi^{(2)}_B + \psi^{(1)}_B\psi^{(2)}_A - \slashed{g}^{CD}\psi^{(1)}_C\psi^{(2)}_D\slashed{g}_{AB},
\\ \nonumber \left(\slashed{\nabla}\hat{\otimes}\psi\right)_{AB} &\doteq \slashed{\nabla}_A\psi_B + \slashed{\nabla}_B\psi_A - \slashed{g}^{CD}\slashed{\nabla}_C\psi_D \slashed{g}_{AB},
\\ \nonumber \phi^{(1)}\wedge \phi^{(2)} &\doteq \slashed{\epsilon}^{AC}\slashed{g}^{BD}\phi^{(1)}_{AB}\phi^{(2)}_{CD},
\\ \nonumber  \psi^{(1)}\wedge \psi^{(2)} &\doteq \slashed{\epsilon}^{AB}\psi^{(1)}_A\psi^{(2)}_B,
\\ \nonumber \slashed{\rm div}\psi &\doteq \slashed{g}^{AB}\slashed{\nabla}_A\psi_B,
\\ \nonumber \slashed{\rm curl}\psi &\doteq \slashed{\epsilon}^{AB}\slashed{\nabla}_A\psi_B,
\\ \nonumber \left(\slashed{\rm div}\phi\right)_A &\doteq \slashed{g}^{BC}\slashed{\nabla}_B\phi_{CA}.
\end{align*}

The null curvature components satisfy the following consequence of the Bianchi identities:
\begin{align}\label{3alpha}
\nabla_3\alpha_{AB} +\frac{1}{2}{\rm tr}\underline{\chi} \alpha_{AB} &= \left(\slashed{\nabla}\hat{\otimes}\beta\right)_{AB} + 4\underline{\omega}\alpha_{AB} -3\left(\hat{\chi}_{AB}\rho + {}^*\hat{\chi}_{AB}\sigma\right) + \left(\left(\zeta+ 4\eta\right)\hat{\otimes}\beta\right)_{AB},
\\ \label{4beta} \nabla_4\beta_A + 2{\rm tr}\chi \beta_A &= \left(\slashed{\rm div}\alpha\right)_A - 2\omega\beta_A + \left(\left(2\zeta + \underline{\eta}\right)\cdot\alpha\right)_A,
\\ \label{3beta} \nabla_3\beta_A + {\rm tr}\underline{\chi}\beta_A &= \slashed{\nabla}_A\rho + 2\underline{\omega}\beta_A + \left({}^*\slashed{\nabla}\right)_A\sigma + 2\left(\hat{\chi}\cdot\underline{\beta}\right)_A+ 3\left(\eta_A\rho + {}^*\eta_A\sigma\right),
\\ \label{4sigma} \nabla_4\sigma + \frac{3}{2}{\rm tr}\chi \sigma &= -\slashed{\rm div}{}^*\beta + \frac{1}{2}\hat{\underline{\chi}}\wedge \alpha -\zeta\wedge \beta -2\underline{\eta}\wedge \beta,
\\ \label{3sigma} \nabla_3\sigma + \frac{3}{2}{\rm tr}\underline{\chi} \sigma &= -\slashed{\rm div}{}^*\underline{\beta} - \frac{1}{2}\hat{\chi}\wedge \underline{\alpha} +\zeta\wedge \underline{\beta} -2\eta\wedge \underline{\beta},
\\ \label{4rho} \nabla_4\rho + \frac{3}{2}{\rm tr}\chi \rho &= \slashed{\rm div}\beta -\frac{1}{2}\hat{\underline{\chi}}\cdot\alpha + \zeta\cdot\beta + 2\underline{\eta}\cdot\beta,
\\ \label{3rho} \nabla_3\rho + \frac{3}{2}{\rm tr}\underline{\chi} \rho &=  -\slashed{\rm div}\underline{\beta} -\frac{1}{2}\hat{\chi}\cdot\underline{\alpha} + \zeta\cdot\underline{\beta} - 2\eta\cdot\underline{\beta},
\\ \label{4undbeta} \nabla_4\underline\beta_A + {\rm tr}\chi\beta_A &= -\slashed{\nabla}_A\rho + 2\omega\underline{\beta}_A + \left({}^*\slashed{\nabla}\right)_A\sigma + 2\left(\hat{\underline{\chi}}\cdot\beta\right)_A+ 3\left(-\underline{\eta}_A\rho + {}^*\underline{\eta}_A\sigma\right),
\\ \label{3undbeta} \nabla_3\underline{\beta}_A + 2{\rm tr}\underline{\chi} \underline{\beta_A} &=- \left(\slashed{\rm div}\underline{\alpha}\right)_A - 2\underline{\omega}\underline{\beta}_A + \left(\left(2\zeta - \eta\right)\cdot\underline{\alpha}\right)_A,
\\ \label{4undalpha} \nabla_4\underline{\alpha}_{AB} +\frac{1}{2}{\rm tr}\chi \underline{\alpha}_{AB} &= -\left(\slashed{\nabla}\hat{\otimes}\underline{\beta}\right)_{AB} + 4\omega\underline{\alpha}_{AB} -3\left(\hat{\underline{\chi}}_{AB}\rho - {}^*\hat{\underline{\chi}}_{AB}\sigma\right) + \left(\left(\zeta- 4\underline{\eta}\right)\hat{\otimes}\underline{\beta}\right)_{AB}.
\end{align}
We refer to this set of equations as the Bianchi equations. 

We now recall a well-known lemma from~\cite{CK}.
\begin{lemma}\label{signatureconserve}Let us introduce the rule that 
\[s\left(\nabla_3\psi\right) \doteq 1 + s\left(\psi\right),\qquad s\left(\nabla_4\psi\right) \doteq -1 + s\left(\psi\right),\qquad s\left(\slashed{D}\psi\right) \doteq s\left(\psi\right),\]
\[s\left(\nabla_3\Psi\right) \doteq 1 + s\left(\Psi\right),\qquad s\left(\nabla_4\Psi\right) \doteq -1 + s\left(\Psi\right),\qquad s\left(\slashed{D}\Psi\right) \doteq  s\left(\Psi\right),\]
\[s\left(\psi\cdot\Psi\right) \doteq s\left(\psi\right) + s\left(\Psi\right),\]
where $\slashed{D}$ stands for any contraction of a $\slashed{g}_{AB}$-covariant derivative. Then, in any given null structure equation or Bianchi equation, each individual summand will have the same signature.
\end{lemma}

In the course of our construction we will study spacetimes where the $v$-dependence of $\hat{\chi}_{AB}$ is only H\"{o}lder continuous, and thus the curvature component $\alpha_{AB}$ may be only be understood distributionally (see~\eqref{4hatchi}) and, more importantly, does not lie in $L^2_{\rm loc}$. As a first step in the proof of our main theorem, we will use the results from~\cite{impulsivefirst,impulsive} which establish local existence for the characteristic initial value problem for data where $\hat{\chi}_{AB}$ has limited regularity in the $v$-direction. One key idea in the works~\cite{impulsivefirst,impulsive} is to introduce a renormalization scheme which serves to eliminate the curvature component $\alpha_{AB}$ from the Bianchi equations. One can also eliminate the curvature component $\underline{\alpha}_{AB}$, and this will also be useful for us when we study the solution near the cone $\{u = 0\}$ (see Section~\ref{bottIIIIIIII}). We now present these renormalized Bianchi equations. We first define
\begin{equation}\label{reallyaquitenicerenormalization}
\check{\rho} \doteq \rho - \frac{1}{2}\hat{\chi}\cdot\hat{\underline{\chi}},\qquad \check{\sigma} \doteq \sigma - \frac{1}{2}\hat{\chi}\wedge \hat{\underline{\chi}}.
\end{equation}
Then we have
\begin{align}\label{ren1}
\nabla_3\beta_A + {\rm tr}\underline{\chi}\beta_A &= \slashed{\nabla}_A\check{\rho} + \left({}^*\slashed{\nabla}\right)_A\check{\sigma} + 2\underline{\omega}\beta_A + 2\left(\hat{\chi}\cdot\underline{\beta}\right)_A + 3\left(\eta_A\check{\rho} + {}^*\eta_A\check{\sigma}\right)
\\ \nonumber &\qquad +\frac{1}{2}\left(\slashed{\nabla}_A\left(\hat{\chi}\cdot\hat{\underline{\chi}}\right) + \left({}^*\slashed{\nabla}\right)_A\left(\hat{\chi}\wedge \hat{\underline{\chi}}\right)\right) + \frac{3}{2}\left(\eta_A \hat{\chi}\cdot\hat{\underline{\chi}} + {}^*\eta_A \hat{\chi}\wedge \hat{\underline{\chi}}\right),
\\ \label{ren2} \nabla_4\check{\sigma} + \frac{3}{2}{\rm tr}\chi \check{\sigma} &= -\slashed{\rm div}{}^*\beta - \zeta\wedge \beta - 2\underline{\eta} \wedge \beta - \frac{1}{2}\hat{\chi}\wedge \left(\slashed{\nabla}\hat{\otimes}\underline{\eta}\right) - \frac{1}{2}\hat{\chi}\wedge \left(\underline{\eta}\hat{\otimes}\underline{\eta}\right),
\\ \label{ren3} \nabla_4\check{\rho} + \frac{3}{2}{\rm tr}\chi \check{\rho} &= \slashed{\rm div}\beta + \zeta\cdot\beta + 2\underline{\eta}\cdot\beta - \frac{1}{2}\hat{\chi}\cdot\left(\slashed{\nabla}\hat{\otimes}\underline{\eta}\right) -\frac{1}{2}\hat{\chi}\cdot\left(\underline{\eta}\hat{\otimes}\underline{\eta}\right) + \frac{1}{4}{\rm tr}\underline{\chi}\left|\hat{\chi}\right|^2,
\\ \label{ren4} \nabla_3\check{\sigma} + \frac{3}{2}{\rm tr}\underline{\chi} \check{\sigma} &= \slashed{\rm div}{}^*\underline{\beta} + \zeta\wedge \underline{\beta} - 2\eta \wedge \underline{\beta} +\frac{1}{2}\hat{\underline{\chi}}\wedge \left(\slashed{\nabla}\hat{\otimes}\eta\right) + \frac{1}{2}\hat{\underline{\chi}}\wedge \left(\eta\hat{\otimes}\eta\right),
\\ \label{ren5} \nabla_3\check{\rho} + \frac{3}{2}{\rm tr}\underline{\chi} \check{\rho} &= -\slashed{\rm div}\underline{\beta} + \zeta\cdot\underline{\beta} - 2\eta\cdot\underline{\beta} - \frac{1}{2}\hat{\underline{\chi}}\cdot\left(\slashed{\nabla}\hat{\otimes}\eta\right) -\frac{1}{2}\hat{\underline{\chi}}\cdot\left(\eta\hat{\otimes}\eta\right) + \frac{1}{4}{\rm tr}\chi\left|\hat{\underline{\chi}}\right|^2,
\\ \label{ren6} \nabla_4\underline{\beta}_A + {\rm tr}\chi\underline{\beta}_A &= -\slashed{\nabla}_A\check{\rho} + \left({}^*\slashed{\nabla}\right)_A\check{\sigma} + 2\omega\underline{\beta}_A + 2\left(\hat{\underline{\chi}}\cdot\beta\right)_A - 3\left(\underline{\eta}_A\check{\rho} - {}^*\underline{\eta}_A\check{\sigma}\right)
\\ \nonumber &\qquad -\frac{1}{2}\left(\slashed{\nabla}_A\left(\hat{\underline{\chi}}\cdot\hat{\chi}\right) + \left({}^*\slashed{\nabla}\right)_A\left(\hat{\underline{\chi}}\wedge \hat{\chi}\right)\right) - \frac{3}{2}\left(\underline{\eta}_A \hat{\underline{\chi}}\cdot\hat{\chi} + {}^*\underline{\eta}_A \hat{\underline{\chi}}\wedge \hat{\chi}\right).
\end{align}

Finally, we record the well-known expressions for the commutators $\left[\nabla_4,\slashed{\nabla}_A\right]$ and $\left[\nabla_3,\slashed{\nabla}_A\right]$.
\begin{lemma}\label{commlemm}We have
\begin{align}\label{com1}
\left[\Omega\nabla_4,\slashed{\nabla}_A\right]\phi_{B_1\cdots B_k} &= \Omega\sum_{i=1}^k\left(\left({}^*\beta\right)_A\slashed{\epsilon}_{B_i}^{\ \ C}-\chi_A^{\ \ C}\underline{\eta}_{B_i}+\chi_{B_iA}\underline{\eta}^C\right)\phi_{B_1\cdots \hat{B_i}C\cdots B_k}  - \Omega\chi_A^{\ \ C}\slashed{\nabla}_C\phi_{B_1\cdots B_k},
\end{align}
\begin{align}\label{com2}
\left[\Omega\nabla_3,\slashed{\nabla}_A\right]\phi_{B_1\cdots B_k} &= \Omega\sum_{i=1}^k\left(-\left({}^*\underline{\beta}\right)_A\slashed{\epsilon}_{B_i}^{\ \ C}-\underline{\chi}_A^{\ \ C}\eta_{B_i}+\underline{\chi}_{B_iA}\eta^C\right)\phi_{B_1\cdots \hat{B_i}C\cdots B_k} - \Omega\underline{\chi}_A^{\ \ C}\slashed{\nabla}_C\phi_{B_1\cdots B_k}.
\end{align}
\end{lemma}
\begin{remark}\label{signatureconserve2}We observe that the conclusions of Lemma~\ref{signatureconserve} extend both to the renormalized Bianchi equations~\eqref{ren1}-\eqref{ren6} and to the commutation formulas~\eqref{com1} and~\eqref{com2}.
\end{remark}

The following remark will be important later.
\begin{remark}\label{letusshifthteshift}One can also consider spacetimes $\left(\mathcal{M},g\right)$ where there exists coordinates $\left(u,v,\theta^A\right)\in \mathcal{U} \times \mathbb{S}^2 $ so that the metric $g$ takes the form
\begin{equation}\label{metricformotherway}
g = -2\Omega^2\left(du\otimes dv + dv\otimes du\right) + \slashed{g}_{AB}\left(d\theta^A - b^Adv\right)\otimes\left(d\theta^B - b^Bdv\right).
\end{equation}
We refer to coordinates where the metric takes the form~\eqref{metricform} as double-null coordinates with ``the shift in the $e_3$-direction'' and to coordinates with the metric takes the form~\eqref{metricformotherway} as double-null coordinates with ``the shift in the $e_4$-direction.'' For a spacetime $\left(\mathcal{M},g\right)$ with coordinates so that the metric takes the form~\eqref{metricformotherway}, we define
\[e_4 \doteq \Omega^{-1}\left(\partial_v + b^A\partial_{\theta^A}\right),\qquad e_3 \doteq \Omega^{-1}\partial_u,\]
and then all of the equations satisfied by the various double-null unknowns are the same with the exception of the shift which now satisfies 
\[\mathcal{L}_{\partial_u}b^A = 4\Omega^2\zeta^A.\]
\end{remark}

Lastly, it will be useful to introduce the following conventions.
\begin{convention}\label{llconvent}$b \ll 1$ means that one should take $b \leq c$ where $c$ is a small constant, independent of all other introduced parameters, but whose exact value may be determined at the end of the paper.  For two positive constants $a$ and $b$, $a \ll b$ means $\frac{a}{b} \ll 1$, and for a positive constant $b > 0$, $b \gg 1$ means that $b^{-1} \ll 1$. 
\end{convention}
\begin{convention}\label{contractconven}We will now describe a schematic notation for certain nonlinear tensorial expressions. Throughout we let $a^{(i)}_j$ denote a tensorial quantity. The schematic notation is defined as follows:

\begin{enumerate}
\item The notation $\left(a^{(1)}_1,\cdots,a^{(1)}_{i_1}\right)\cdots\left(a^{(j)}_1,\cdots,a^{(j)}_{i_j}\right)$ denotes an expression which could in principle represent an arbitrary linear combination of contractions of tensor products of $j$-tuples $\left(a^{(1)}_{k_1},\cdots,a^{(j)}_{k_j}\right)$. 

\item The notation $\left(a^{(1)}_1,\cdots,a^{(1)}_{i_1}\right)^k$ denotes $\underbrace{\left(a^{(1)}_1,\cdots,a^{(1)}_{i_1}\right)\cdots \left(a^{(1)}_1,\cdots,a^{(1)}_{i_1}\right)}_k$.

\item The notation $\slashed{\nabla}\left(a^{(1)}_1,\cdots,a^{(1)}_{i_1}\right)$ denotes an arbitrary linear combination of contractions of terms $\slashed{\nabla}a^{(i)}_j$.

\item The notation $\slashed{\nabla}^j\left(a^{(1)}_1,\cdots,a^{(1)}_{i_1}\right)$ denotes $\underbrace{\slashed{\nabla}(\slashed{\nabla}(\cdots \slashed{\nabla}}_j\left(a^{(1)}_1,\cdots,a^{(1)}_{i_1}\right) ))$
\end{enumerate}
\end{convention}

\subsection{The Characteristic Initial Value Problem}
In the course of the paper we will need to invoke local existence results for  suitable characteristic initial value problems. In this section we will quickly review the relevant theory. Throughout this section, we let $u_0$, $u_1$, $v_0$, and $v_1$ be real numbers which satisfy $u_0 < u_1$ and $v_0 < v_1$. 

We start with a definition of a characteristic initial data set.
\begin{definition}\label{indatasets}We say that two $1$-parameter families $\left(\Omega^{({\rm in})}\left(u,\theta^A\right),\left(b^A\right)^{(\rm in)}\left(u,\theta^B\right),\slashed{g}_{AB}^{(\rm in)}\left(u,\theta^C\right)\right)$ for $(u,\theta^A) \in [u_0,u_1]\times \mathbb{S}^2$ and $\left(\Omega^{(\rm out)}\left(v,\theta^A\right),\slashed{g}_{AB}^{(\rm out)}\left(v,\theta^C\right)\right)$ for $(v,\theta^A) \in [v_0,v_1]\times \mathbb{S}^2$ consisting of nowhere vanishing $C^1$ functions $\Omega^{(\rm in)}$ and $\Omega^{(\rm out)}$, a continuous vector field $\left(b^A\right)^{(\rm in)}$, and $C^1$ $1$-parameter families of Riemannian metrics $\slashed{g}_{AB}^{(\rm in)}$ and $\slashed{g}_{AB}^{(\rm out)}$ on $\mathbb{S}^2$, as well as a continuous $1$-form $\left(\zeta_A\right)_{u_0,v_0}$ on $\mathbb{S}^2$ form a ``characteristic initial data set'' if the following hold
\begin{enumerate}
	\item We have $\Omega^{(\rm in)}|_{u=u_0} = \Omega^{(\rm out)}|_{v=v_0}$ and $\slashed{g}_{AB}^{(\rm in)}|_{u=u_0} = \slashed{g}_{AB}^{(\rm out)}|_{v=v_0}$.
	\item After defining ${\rm tr}\underline{\chi}$, $\hat{\underline{\chi}}_{AB}$, and $\underline{\omega}$ for $u \in [u_0,u_1]$ by
	\[\left(\Omega^{(\rm in)}\right)^{-1}\mathcal{L}_{\partial_u + b^{(\rm in)}}\slashed{g}^{(\rm in)}_{AB} \doteq {\rm tr}\underline{\chi}\slashed{g}^{(\rm in)}_{AB} + 2\hat{\underline{\chi}}_{AB},\qquad \underline{\omega} \doteq -\frac{1}{2}\left(\Omega^{(\rm in)}\right)^{-1}\left(\partial_u + b^{(\rm in)}\right)\log\Omega^{(\rm in)},\]
	and the requirement that $\hat{\underline{\chi}}_{AB}$ be trace-free,
	we have that the following equation is satisfied:
	\begin{equation}\label{ray1cons}
	\left(\Omega^{(\rm in)}\right)^{-1}\left(\partial_u + b^{(\rm in)}\right){\rm tr}\underline{\chi} + \frac{1}{2}\left({\rm tr}\underline{\chi}\right)^2 = -2\underline{\omega}{\rm tr}\underline{\chi} - \left|\hat{\underline{\chi}}\right|^2.
	\end{equation}
	
	\item After defining ${\rm tr}\chi$, $\hat{\chi}_{AB}$, and $\omega$ for $v \in [v_0,v_1]$ by
	\[\left(\Omega^{(\rm out)}\right)^{-1}\mathcal{L}_{\partial_v}\slashed{g}^{(\rm out)}_{AB} \doteq {\rm tr}\chi\slashed{g}^{(\rm out)}_{AB} + 2\hat{\chi}_{AB},\qquad \omega \doteq -\frac{1}{2}\left(\Omega^{(\rm out)}\right)^{-1}\partial_v\log\Omega^{(\rm out)},\]
	and the requirement that $\hat{\chi}_{AB}$ be trace-free,
	we have that the following equation is satisfied:
	\begin{equation}\label{ray2cons}
	\left(\Omega^{(\rm out)}\right)^{-1}\partial_v{\rm tr}\chi + \frac{1}{2}\left({\rm tr}\chi\right)^2 = -2\omega{\rm tr}\chi - \left|\hat{\chi}\right|^2.
	\end{equation}
	
\end{enumerate}
\end{definition}
Note that the two differential equations that we require the characteristic data to satisfy are simply the two Raychaudhuri equations~\eqref{4trchi} and \eqref{3truchi}.

Now we recall the following local well-posedness result for the characteristic value problem.
\begin{theorem}\label{localexistencechar}~\cite{rendallchar,lukchar} Given any characteristic initial data set so that 
\[\left(\Omega^{(\rm in)},\Omega^{(\rm out)},\left(b^A\right)^{(\rm in)},\slashed{g}^{(\rm in)}_{AB},\slashed{g}^{(\rm out)}_{AB},\left(\zeta_{u_0,v_0}\right)_A\right),\]
are smooth, there exists a smooth spacetime $\left(\mathcal{M},g_{\mu\nu}\right)$ such that
\begin{enumerate}
	\item There exists a positive constant $c$, depending on the initial data set, so that the metric $g_{\mu\nu}$ takes the double-null form~\eqref{metricform} for $(u,v) \in \{\left([u_0,u_1]\times [v_0,v_0+c]\right) \cup \left([u_0,u_0+c]\times[v_0,v_1]\right)\}$.
	\item We have $\left(\Omega,b^A,\slashed{g}_{AB}\right)|_{v= v_0} = \left(\Omega^{(\rm in)},\left(b^A\right)^{(\rm in)},\slashed{g}_{AB}^{(\rm in)}\right)$, $\left(\Omega,\slashed{g}_{AB}\right)|_{u= u_0} = \left(\Omega^{(\rm out)},\slashed{g}_{AB}^{(\rm out)}\right)$, and $\zeta_A|_{(u,v) = (u_0,v_0)} = \left(\zeta_{u_0,v_0}\right)_A$.
\end{enumerate}
\end{theorem}

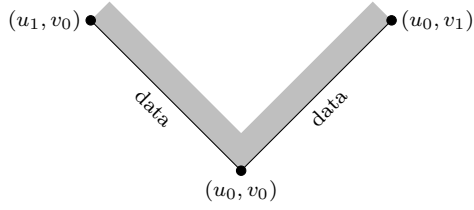
\begin{center}
\begin{tikzpicture}[scale = 1]
\fill[lightgray] (0,0)--(2,-2)--(4,0) --(3.75,.25) -- (2,-1.5) -- (.25,.25);
\draw (0,0) -- (2,-2) node[sloped,below,midway]{\footnotesize data};
\draw (2,-2) -- (4,0) node[sloped,below,midway]{\footnotesize data};
\path [draw=black,fill=black] (2,-2) circle (1/16); 
\draw (2,-2) node[below]{\footnotesize $(u_0,v_0)$};
\path [draw=black,fill=black] (0,0) circle (1/16); 
\draw (0,0) node[left]{\footnotesize $(u_1,v_0)$};
\path [draw=black,fill=black] (4,0) circle (1/16); 
\draw (4,0) node[right]{\footnotesize $(u_0,v_1)$};
\end{tikzpicture}
\end{center}

Since we will need to consider solutions where the $v$-dependence of $\hat{\chi}_{AB}$ is only H\"{o}lder continuous, it will be convenient to refer to the following generalization of Theorem~\ref{localexistencechar} (which can be deduced from the main results from~\cite{impulsivefirst,impulsive}):
\begin{theorem}\label{localexistencecharbetter}~\cite{impulsivefirst,impulsive} Let $N$ be a sufficiently large positive integer, and let $\mathring{\slashed{g}}_{AB}$ denote a round metric on $\mathbb{S}^2$ which we extend to $[u_0,u_1]\times\mathbb{S}^2$ and $[v_0,v_1]\times\mathbb{S}^2$ by Lie-propagation. Suppose we have a characteristic initial data set satisfying the following bounds for some $C > 0$: 
\[\left\vert\left\vert \left(b^{(\rm in)},\log\Omega^{(\rm in)}\right)\right\vert\right\vert_{L^{\infty}_u\mathring{H}^N} \leq C,\qquad \left\vert\left\vert \log\Omega^{(\rm out)}\right\vert\right\vert_{L^{\infty}_v\mathring{H}^N} \leq C,\]
\[\left\vert\left\vert \left(\underline{\chi},\mathcal{L}_{\partial_u}\underline{\chi},\underline{\omega}\right)\right\vert\right\vert_{L^{\infty}_u\mathring{H}^N} \leq C,\qquad \left\vert\left\vert \left(\chi,\omega\right)\right\vert\right\vert_{L^{\infty}_v\mathring{H}^N} \leq C,\qquad \left\vert\left\vert \zeta_{u_0,v_0} \right\vert\right\vert_{\mathring{H}^N} \leq C,\]
\[\left\vert\left\vert \left({\rm det}\slashed{g}^{(\rm out)}\right)^{-1}\right\vert\right\vert_{L^{\infty}_{v,\theta}} \leq C,\qquad \left\vert\left\vert \left({\rm det}\slashed{g}^{(\rm in)}\right)^{-1}\right\vert\right\vert_{L^{\infty}_{u,\theta}} \leq C, \qquad \left\vert\left\vert \slashed{g}^{(\rm out)}\right\vert\right\vert_{L^{\infty}_v\mathring{H}^N} \leq C,\qquad \left\vert\left\vert \slashed{g}^{(\rm in)}\right\vert\right\vert_{L^{\infty}_u\mathring{H}^N} \leq C,\]
where $\mathring{H}^i$ denotes (inhomogeneous) Sobolev norms defined with respect to the metric $\mathring{\slashed{g}}_{AB}$.

Then there exists a spacetime $\left(\mathcal{M},g_{\mu\nu}\right)$ such that
\begin{enumerate}
	\item There exists a positive constant $c$, depending on the constant $C$, so that the metric $g_{\mu\nu}$ takes the double-null form~\eqref{metricform} for $(u,v) \in \{\left([u_0,u_1]\times [v_0,v_0+c]\right) \cup \left([u_0,u_0+c]\times[v_0,v_1]\right)\}$.
	\item We have $\left(\Omega,b^A,\slashed{g}_{AB}\right)|_{v= v_0} = \left(\Omega^{(\rm in)},\left(b^A\right)^{(\rm in)},\slashed{g}_{AB}^{(\rm in)}\right)$, $\left(\Omega,\slashed{g}_{AB}\right)|_{u= u_0} = \left(\Omega^{(\rm out)},\slashed{g}_{AB}^{(\rm out)}\right)$, and $\zeta_A|_{\{(u,v) = (u_0,v_0)\}} = \left(\zeta_{u_0,v_0}\right)_A$.
	\item There exists $C'> 0$ so that in the region $(u,v) \in \{\left([u_0,u_1]\times [v_0,v_0+c]\right) \cup \left([u_0,u_0+c]\times[v_0,v_1]\right)\}$, the metric components, Ricci coefficients, and curvature components have the following regularity:
	\begin{equation}\label{bound1121}
	\left\vert\left\vert \left(\beta,\rho,\sigma,\underline{\beta}\right)\right\vert\right\vert_{L^{\infty}_uL^2_v\mathring{H}^{N-2}} \leq C',\qquad \left\vert\left\vert \left(\rho,\sigma,\underline{\beta},\underline{\alpha}\right)\right\vert\right\vert_{L^{\infty}_vL^2_u\mathring{H}^{N-2}} \leq C', 
	\end{equation}
	\begin{equation}\label{11212}
	\left\vert\left\vert \left(b,\underline{\chi},\chi,\omega,\underline{\omega},\slashed{\nabla}\log\Omega\right)\right\vert\right\vert_{L^{\infty}_uL^{\infty}_v\mathring{H}^{N-2}} \leq C'
	\end{equation}
	\begin{equation}\label{112123}
	\left\vert\left\vert \left({\rm det}\slashed{g}\right)^{-1}\right\vert\right\vert_{L^{\infty}_uL^{\infty}_vL^{\infty}_{\theta}} \leq C',\qquad \qquad \left\vert\left\vert \slashed{g}\right\vert\right\vert_{L^{\infty}_vL^{\infty}_u\mathring{H}^{N-2}} \leq C'.
		\end{equation}
	
\item $\left(\mathcal{M},g_{\mu\nu}\right)$ is a weak solution to ${\rm Ric}_{\mu\nu}(g) = 0$. In particular,~\eqref{4trchi},~\eqref{3truchi}-\eqref{tcod2}, and ~\eqref{3beta}-\eqref{ren6} all hold.
\end{enumerate}

Also, there exists a constant $Y$ (independent of $N$) so that we also have the following blow-up criterion: Let $0 < r_1 \leq v_1-v_0$ and $0 < r_2 \leq u_1-u_0$. Then one of the two possibilities must occur.
\begin{enumerate}
\item There exists a spacetime $\left(\mathcal{M},g_{\mu\nu}\right)$ in the double-null form~\eqref{metricform} which contains $(u,v) \in \{\left([u_0,u_1]\times [v_0,v_0+r_1]\right) \cup \left([u_0,u_0+r_2]\times[v_0,v_1]\right)\}$ , $\left(\Omega,b^A,\slashed{g}_{AB}\right)|_{v= v_0} = \left(\Omega^{(\rm in)},\left(b^A\right)^{(\rm in)},\slashed{g}_{AB}^{(\rm in)}\right)$, $\left(\Omega,\slashed{g}_{AB}\right)|_{u= u_0} = \left(\Omega^{(\rm out)},\slashed{g}_{AB}^{(\rm out)}\right)$,  $\zeta_A|_{\{(u,v) = (u_0,v_0)\}} = \left(\zeta_{u_0,v_0}\right)_A$, and the bounds~\eqref{bound1121}-\eqref{112123} hold for some $C' < \infty$.
\item For some $0 < s_1 < r_1$ and $0 < s_2 < r_2$ there exists a spacetime $\left(\mathcal{M},g_{\mu\nu}\right)$ in the double-null form~\eqref{metricform} which contains $(u,v) \in \{\left([u_0,u_1]\times [v_0,v_0+s_1)\right) \cup \left([u_0,u_0+s_2)\times[v_0,v_1]\right)\}\doteq \mathcal{U}_{s_1,s_2}$ , $\left(\Omega,b^A,\slashed{g}_{AB}\right)|_{v= v_0} = \left(\Omega^{(\rm in)},\left(b^A\right)^{(\rm in)},\slashed{g}_{AB}^{(\rm in)}\right)$, $\left(\Omega,\slashed{g}_{AB}\right)|_{u= u_0} = \left(\Omega^{(\rm out)},\slashed{g}_{AB}^{(\rm out)}\right)$, $\zeta_A|_{\{(u,v) = (u_0,v_0)\}} = \left(\zeta_{u_0,v_0}\right)_A$, and
\[\left\vert\left\vert \left(\chi,\underline{\chi},\underline{\omega},\omega,\slashed{\nabla}\log\Omega,\zeta,\slashed{g}\right)\right\vert\right\vert_{L^{\infty}\left(\mathcal{U}_{s_1,s_2}\right)\mathring{H}^Y} + \left\vert\left\vert \left({\rm det}\slashed{g}\right)^{-1}\right\vert\right\vert_{L^{\infty}\left(\mathcal{U}_{s_1,s_2}\right)L^{\infty}_{\theta^A}} = \infty.\]
\end{enumerate}

Finally, we know that the singularities of $\alpha$ must propagate in $e_3$-direction; more specifically,  whenever $\hat{v}_0 > v_0$, then
\begin{equation}\label{alphaisregularkdkow}
\left\vert\left\vert \alpha|_{u=u_0}\right\vert\right\vert_{L^2_{v\in [\hat{v}_0,v_1]}\mathring{H}^{N-2}} < \infty \Rightarrow  \sup_{u \in [u_0,u_1]}\left\vert\left\vert \alpha\right\vert\right\vert_{L^2_{v\in [\hat{v}_0,v_1]}\mathring{H}^{N-2}} < \infty.
\end{equation}

\end{theorem}
\begin{remark}One may, of course, provide explicit upper bounds on the optimum values of the constants $N$ and $Y$, but these will not play any role in this paper.
\end{remark}
\subsection{(Asymptotically) Self-Similar Solutions}\label{asss}
In this section we revisit the types of self-similarity (for $3+1$ dimensional solutions) considered in the works~\cite{FG1,FG2,scaleinvariant}.

\begin{definition}\label{scaleinvdef}We say that a solution $\left(\mathcal{M},g_{\mu\nu}\right)$ given in the double-null form~\eqref{metricform} and defined in the region $\{u < 0 \} \cap \{0 \leq \frac{v}{-u}\leq c\}$ for some $c  > 0$, is ``self-similar'' if it is smooth and in a coordinate frame we have
\begin{equation}\label{scaleinvrelations}
\Omega\left(u,v,\theta^A\right) = \check{\Omega}\left(\frac{v}{u},\theta^A\right),\qquad b_A\left(u,v,\theta^B\right) = u\check{b}_A\left(\frac{v}{u},\theta^B\right),\qquad \slashed{g}_{AB}\left(u,v,\theta^C\right) = u^2\check{\slashed{g}}_{AB}\left(\frac{v}{u},\theta^C\right),
\end{equation}
for some $\check{\Omega}$, $\check{b}_A$, and $\check{\slashed{g}}_{AB}$. Equivalently, the ``scaling vector field'' $K \doteq u\partial_u+v\partial_v$ satisfies
\[\mathcal{L}_Kg_{\mu\nu} = 2g_{\mu\nu}.\]
\end{definition}

Note, in particular, that if a solution $\left(\mathcal{M},g_{\mu\nu}\right)$ is self-similar then the restrictions of $\Omega$, $b_A$, and $\slashed{g}_{AB}$ to $\{v = 0\}$ must satisfy the following:
\begin{equation}\label{onthev0selflskw}
\Omega|_{v=0}\left(u,\theta^A\right) = \check{\Omega}\left(\theta^A\right),\qquad  b_A\left(u,\theta^B\right) = u\check{b}_A\left(\theta^B\right),\qquad \slashed{g}_{AB}\left(u,\theta^C\right) = u^2\check{\slashed{g}}_{AB}\left(\theta^C\right).
\end{equation}
For any Riemannian metric $\check{\slashed{g}}_{AB}$ on $\mathbb{S}^2$, one natural way to generate to generate a triple 
\[\left(\Omega^{(\rm in)}(u,\theta^A),(b^A)^{(\rm in)}(u,\theta^B),\slashed{g}_{AB}^{(\rm in)}(u,\theta^C)\right)\]
satisfying the requirement~\eqref{onthev0selflskw} and also the Raychaudhuri constraint equation~\eqref{ray1cons}  is to set
\begin{equation}\label{scalesetdata}
\Omega^{(\rm in)}\left(u,\theta^A\right) \doteq 1,\qquad (b^A)^{(\rm in)}\left(u,\theta^B\right) \doteq 0,\qquad \slashed{g}^{(\rm in)}_{AB}(u,\theta^C) \doteq u^2\check{\slashed{g}}_{AB}(\theta^C),\qquad \forall u \in (-\infty,0).
\end{equation}

The following result, due to Fefferman--Graham, classifies formal power series expansions which obtain the incoming data defined by~\eqref{scalesetdata}
\begin{theorem}\label{formalFG}\cite{FG1,FG2} Let $\check{\slashed{g}}_{AB}$ be an arbitrary smooth Riemannian metric on $\mathbb{S}^2$. Then there exists $\{\check\Omega^{(i)}\left(\theta\right)\}_{i=1}^{\infty}$, $\{(\check{b}^A)^{(i)}\left(\theta\right)\}_{i=1}^{\infty}$, and $\{\check{\slashed{g}}^{(i)}_{AB}\left(\theta\right)\}_{i=1}^{\infty}$ so that we obtain a formal metric $g_{\mu\nu}$ solving ${\rm Ric}_{\mu\nu}(g) = 0$ by defining the following formal power series expansions in $\frac{v}{u}$ and in coordinate frames
\begin{equation}\label{expand1}
\Omega\left(u,v,\theta^A\right) \doteq 1 + \sum_{i=1}^{\infty}\left(\frac{v}{u}\right)^i\check{\Omega}^{(i)}\left(\theta^A\right),\qquad b^A\left(u,v,\theta^B\right) \doteq u^{-1}\sum_{i=1}^{\infty}\left(\frac{v}{u}\right)^i\left(\check{b}^A\right)^{(i)}\left(\theta^B\right), 
\end{equation}
\begin{equation}\label{expand2}
\slashed{g}_{AB}\left(u,v,\theta^C\right) \doteq u^2\check{\slashed{g}}_{AB}\left(\theta^C\right) + u^2\sum_{i=1}^{\infty}\left(\frac{v}{u}\right)^i\check{\slashed{g}}^{(i)}_{AB}\left(\theta^C\right).
\end{equation}
By a formal solution, we mean that if one truncates these sums at some integer $N$, defines a corresponding self-similar double-null metric by using the truncated sums to define the metric components, and computes the corresponding Ricci tensor ${\rm Ric}^{(N)}_{\mu\nu}$, then, along any constant $u$-curve, one will have in a coordinate frame that ${\rm Ric}^{(N)}_{\mu\nu} = O_{u,N}\left(v^{p\left(N\right)}\right)$ where $p \to \infty$ as $N \to \infty$.  (In particular, may make no assertion about the convergence of the the infinite sums in~\eqref{expand1} and~\eqref{expand2}.)

Furthermore the expansions~\eqref{expand1} and~\eqref{expand2} are uniquely determined by the requirement that the incoming characteristic data satisfy~\eqref{scalesetdata} and a choice of ${\rm tf}\left(\check{\slashed{g}}^{(1)}\right)_{AB}$. (Here ${\rm tf}$ denotes the trace-free part.)
\end{theorem}

\begin{remark}\label{wavescale}As a point of comparison, we can consider the simpler situation of spherically symmetric solutions $\phi$ to the wave equation on Minkowski space, 
\[\partial_u\partial_v\left(r\phi\right) = 0,\]
which are self-similar in the sense that
\[\phi(u,v) = \mathring{\phi}\left(\frac{v}{u}\right),\]
for a suitable function $\mathring{\phi}$. It is straightforward to classify all such solutions which are of bounded variation in the sense of Christodoulou~\cite{ChristBV}. One finds that these comprise a two parameter class indexed by $(a,b) \in \mathbb{R}^2$ as follows:
\[ \phi\left(u,v\right) = \begin{cases} 
      a & \{v \leq 0\} \\
     a + b\frac{v}{v-u} & \{v > 0\} \cap \{u < 0\} \\
      a+b & \{u \geq 0\}
   \end{cases}.
\]
The parameter $a$ is, of course, a trivial freedom reflecting the ability to add a constant to any solution to the wave equation.

\end{remark}

In our previous work~\cite{scaleinvariant} we accomplished two main goals:
\begin{enumerate}
	\item To show that the formal power series expansions of Fefferman--Graham from Theorem~\ref{formalFG} correspond to true solutions.
	\item To identify a large class of characteristic initial data sets (see Definition~\ref{indatasets}) which lead to solutions which, while not self-similar, converge to a self-similar solution as the point $(u,v) = (0,0)$ is approached.
\end{enumerate}
We will not here undertake a full review of the proof of these results; however, it will be clarifying to revisit one of the preliminary steps in the analysis. We start with the following definition:
\begin{definition}\label{asyscaleinitdataset}Let $\left(\Omega^{(\rm in)}(u,\theta^A),(b^A)^{(\rm in)}(u,\theta^B),\slashed{g}_{AB}^{(\rm in)}(u,\theta^C)\right)$ be defined by~\eqref{scalesetdata} restricted to $u \in [-1,0)$ and $J$ be a non-negative integer. We say that a choice of $\tilde\zeta_A\left(\theta^B\right)$ and outgoing characteristic data \\ $\left(\Omega^{(\rm out)}(v,\theta^A),\slashed{g}_{AB}^{(\rm out)}(v,\theta^C)\right)$ defined for $v \in [0,\tilde v)$ and satisfying  $\Omega^{(\rm in)}\left(-1,\theta^A\right) = \Omega^{(\rm out)}\left(0,\theta^A\right)$, $\slashed{g}_{AB}^{(\rm in)}\left(-1,\theta^C\right) = \slashed{g}_{AB}^{(\rm out)}\left(0,\theta^C\right)$, and~\eqref{ray2cons}, is ``consistent with an asymptotic scale-invariance to order $J$'' if after applying Theorem~\ref{localexistencechar}, then along $\{v = 0\}$ we have the following bounds for any curvature component $\Psi$ and Ricci coefficient $\psi$:
\begin{equation}\label{theseboundscouldbeok}
\left|\nabla^j_4\slashed{\nabla}^i\Psi\right| \lesssim_{i,j} |u|^{-2-i-j},\qquad \left|\nabla^j_4\slashed{\nabla}^i\psi\right| \lesssim_{i,j} |u|^{-i-j-1},\qquad \forall i \geq 0\text{ and }0\leq j \leq J,
\end{equation}
where these norms are computed with respect to $\slashed{g}^{(\rm in)}_{AB}$.
\end{definition}
\begin{remark}As was shown in Section 3.11 of~\cite{scaleinvariant}, for an exactly self-similar solution, we will have $\left|\nabla_4^j\slashed{\nabla}^i\Psi\right| \sim_{i,j} |u|^{-2-i-j}$ and $\left|\nabla_4^j\slashed{\nabla}^i\psi\right| \sim_{i,j} |u|^{-1-i-j}$. Thus the bounds~\eqref{theseboundscouldbeok} are motivated by requiring that along $\{v = 0\}$ the solution is, at worst, as singular as an exactly self-similar solution.
\end{remark}

The following is a slight extension of a proposition proved in~\cite{scaleinvariant}.
\begin{proposition}\label{formalcalculationsscaleinv}We have that $\tilde\zeta_A$ and $\left(\Omega^{(out)},\slashed{g}_{AB}^{(\rm out)}\right)$ is ``consistent with an asymptotic scale-invariance of order $J$'' (see Definition~\ref{asyscaleinitdataset}) for any $J \in \mathbb{Z}_{\geq 0}$ if and only if we have
\begin{equation}\label{constrforscaleinv}
	\tilde \zeta_A = 0,\qquad \left(\slashed{g}^{(\rm out)}\right)^{AB}\mathcal{L}_{\partial_v}\left(\slashed{g}\right)^{({\rm out})}_{AB}|_{v=0} = 2K|_{v=0},
\end{equation}
where $K$ denotes the Gaussian curvature of $\slashed{g}_{AB}$. Furthermore, if we assume that~\eqref{constrforscaleinv} is satisfied, set 
\[\mathcal{N}_{AB} \doteq {\rm tf}\left(\mathcal{L}_{\partial_v}\left(\slashed{g}\right)^{({\rm out})}_{AB}\right)|_{v=0},\]
and then extend $\mathcal{N}_{AB}$ to all of $\{v = 0\}$ by Lie-propagation with respect to $\partial_u$, then we have the following bounds and identitites for the Ricci coefficients and curvature components along $\{v=0\}$:
\begin{equation}\label{firstsetofstuffonv0}
{\rm tr}\underline{\chi} = \frac{2}{u},\qquad \hat{\underline{\chi}}_{AB} = 0,\qquad \underline{\omega} = 0,\qquad \underline{\alpha}_{AB} = 0,\qquad \eta_A = 0,\qquad \underline{\eta}_A = 0,\qquad \underline{\beta}_A = 0.
\end{equation}
\begin{equation}\label{firstsetofstuffonv02}
\sigma = 0,\qquad \rho = 0,\qquad {\rm tr}\chi = -2uK,\qquad  \left|\omega\right| = O\left(1\right),\qquad \hat{\chi}_{AB} = -u\mathcal{N}_{AB},
\end{equation}
\begin{equation}\label{firstsetofstuffonv03}
\beta_A = u\left[\slashed{\rm div}\mathcal{N}_A- \slashed{\nabla}_AK\right],\qquad \left|\alpha + u\left(\slashed{\nabla}\hat{\otimes}\left(\slashed{\rm div}\mathcal{N} - \slashed{\nabla}K\right)\right)\right| \lesssim |u|^{-1}.
\end{equation}

\end{proposition}
\begin{proof}If $\Omega^{(\rm out)}$ is assumed to be identically $1$ and $J = 0$, then this is contained in  Proposition 4.1 from~\cite{scaleinvariant}; the modifications needed for the more general $\Omega^{(\rm out)}$ and $J$ are straightforward. (See also the discussion in Section 2.2 of~\cite{scaleinvariant}.)

\end{proof}

Note that due to the vanishing of the shear $\hat{\underline{\chi}}_{AB}$ these solutions will not, in particular, satisfy~\eqref{shearkofw}. Furthermore, it is an immediate consequence of the Gauss--Bonnet Theorem that the Hawking mass of any $\mathbb{S}^2_{u,0}$ sphere must vanish. In particular, we do not have~\eqref{masskfow}. In fact, solutions which are self-similar in the sense of Definition~\ref{scaleinvdef} or, in view of the main results of~\cite{scaleinvariant}, solutions corresponding to data which is ``consistent with an asymptotic scale-invariance of order $0$'' may be considered to be analogous to Christodoulou's solutions of bounded variation~\cite{ChristBV} (and also the solutions discussed in Remark~\ref{wavescale}). In particular, we should not consider the class of solutions from~\cite{scaleinvariant} as good models for naked singularities. 

Given the discussion in the previous paragraph, it is natural to wonder whether self-similar solutions can be built from characteristic initial data sets which, while satisfying~\eqref{onthev0selflskw}, are more general than those allowed by~\eqref{scalesetdata}; however, the following proposition puts a severe restriction on the behavior of $g$ along $\{v = 0\}$:
\begin{proposition}\label{restrictonv0yay}Let $\left(\mathcal{M},g_{\mu\nu}\right)$ be a solution to the Einstein vacuum equations which is self-similar in the sense of Definition~\ref{scaleinvdef}. Then we must have
\begin{equation}\label{jioavjoivjoav}
\mathcal{L}_b\slashed{g}_{AB}|_{v=0} = 0,\qquad \mathcal{L}_b\Omega|_{v=0} = 0.
\end{equation}
\end{proposition}
\begin{proof}Omitted. 
\end{proof}
Furthermore, using the techniques developed in~\cite{scaleinvariant} one can show that any self-similar solution satisfying~\eqref{jioavjoivjoav} is isometric to a self-similar solution whose restriction to $\{v =0\}$ satisfies~\eqref{scalesetdata}. Thus, it is clear that in order to construct naked singularities, we must leave this class of solutions.

\section{$\kappa$-Self-Similarity and an Outline of the Proof}
A key role in the proof of our main result will be played by a more general notion of a self-similar solution compared to the type discussed in Section~\ref{asss}. One way to think about this new self-similarity is as follows: The self-similarity discussed in Section~\ref{asss} assumes the existence of a double-null foliation whose domain of validity includes the important null hypersurface $\{v = 0\}$ and where the self-similar vector field $K$ takes the form $K = u\partial_u + v\partial_v$. However, there is no a priori reason to expect a solution to the Einstein vacuum equations with a conformal Killing vector field $K$ to admit such a coordinate system. 

Our new $\kappa$-self-similarity still starts with a double-null coordinate system where the self-similar vector field $K$ takes the form $K = u\partial_u + v\partial_v$. However, we no longer assume that the coordinates extend to $\{v = 0\}$. Instead, we require that for an alternative coordinate system $\left(\hat{v},u,\theta^A\right)$ where $\hat{v} \doteq v^{1-2\kappa}$, the metric extends to $\{\hat{v} = 0\}$. The coordinate system $\left(\hat{v},u,\theta^A\right)$ will still be a double-null coordinate system, however, the self-similar vector field $K$ will now take the form $K = u\partial_u + \left(1-2\kappa\right)\hat{v}\partial_{\hat{v}}$. Thus we may equivalently think of $\kappa$-self-similarity as relaxing (slightly) the requirement that $K$ takes the form $u\partial_u + v\partial_v$ in a double-null coordinate system. We now give the precise definition.

\begin{definition}\label{kapselfseim}
We say that a smooth solution $\left(\mathcal{M},g_{\mu\nu}\right)$ given in the double-null form~\eqref{metricform} and defined in the region $\{u < 0 \} \cap \{0 < \frac{v}{-u}\leq c\}$ for some $c  > 0$, is ``$\kappa$-self-similar'' if in a coordinate frame we have
\begin{equation}\label{scaleinvrelations2}
\Omega\left(u,v,\theta^A\right) = \check{\Omega}\left(\frac{v}{u},\theta^A\right),\qquad b_A\left(u,v,\theta^B\right) = u\check{b}_A\left(\frac{v}{u},\theta^B\right),\qquad \slashed{g}_{AB}\left(u,v,\theta^C\right) = u^2\check{\slashed{g}}_{AB}\left(\frac{v}{u},\theta^C\right),
\end{equation}
for some $\check{\Omega}$, $\check{b}_A$, and $\check{\slashed{g}}_{AB}$, and there exists $0< \kappa \ll 1$ such that in the coordinates $\left(\hat{v},u,\theta^A\right)$ defined by $\hat{v} \doteq v^{1-2\kappa}$, for every $i,|j| \geq 0  $, (where $j$ is a multi-index) there exists $\gamma(|j|) > 0$ such that $\mathcal{L}^i_{\partial_u}\mathcal{L}_{\partial_{\theta}}^{(j)}g_{\mu\nu}$ extends to $\{\hat{v} = 0\}$ as a $C^{1,\gamma(|j|)}$ tensor. 
\end{definition}
\begin{remark}\label{kappato0is}If $\kappa = 0$ and the metric in fact extends to $\{\hat{v} = 0\}$ as a smooth metric, then the metric will be self-similar in the sense of Definition~\ref{scaleinvdef}. We further emphasize that, as we will see later, the lapse $\Omega$ of a $\kappa$-self-similar solution will satisfy $\lim_{v\to 0}\left(v^{\kappa}\Omega\left(-1,v,\theta^A\right)\right) = h\left(\theta^A\right)$ for some function $h : \mathbb{S}^2 \to (0,\infty)$. In particular, a solution in a given double-null coordinate system $\left(u,v,\theta^A\right)$ can only possibly be $\kappa$-self-similar for a unique value of $\kappa$. 
\end{remark}
\begin{remark}\label{thisissharpforgamma}
We emphasize that the coordinates $(u,v,\theta^A)$ in which the metric is required to satisfy~\eqref{scaleinvrelations2} are \underline{not} regular as $v\to 0$. In particular, as we will see in more detail later, in these coordinates we have $\Omega^2 \sim \left(\frac{v}{-u}\right)^{-2\kappa}$. While, for the solutions we consider, we will not establish sharp estimates for $\gamma(|j|)$ for general $j$, it will follow from our analysis that we have the estimate
\begin{equation}\label{901j4ioj2}
\gamma(0) \leq \frac{2\kappa}{1-2\kappa}.
\end{equation}
In particular, even though a $\kappa$-self-similar solution is smooth for $\{\hat{v} > 0\}$ it will have limited regularity if we include the hypersurface $\{\hat{v} = 0\}$.
\end{remark}
\begin{remark}The fundamental reason we will need to restrict to $|\kappa| \ll 1$ in this paper is because, in view of the $\kappa$-constraint equation (see~\eqref{basicconstr2} below) the size of $\kappa$ is directly related to the size of $\left(v^{\kappa}\Omega,\slashed{g},b\right)$ along $\{v = 0\}$, and we will need to have the smallness of $\left(v^{\kappa}\Omega,\slashed{g},b\right)|_{v=0}$ in order to carry out our nonlinear analysis. However, this constraint is, in principle, just an artifact of our method of proof and, in view of the regularity constraint~\eqref{901j4ioj2}, it is an interesting problem to construct solutions where $\kappa$ is allowed to be as large as possible. 
\end{remark}

Despite not allowing for regular limits as $v\to 0$, the $(u,v,\theta^A)$ coordinates are useful because 
\begin{enumerate}
	\item The self-similar vector field takes the simple form $K = u\partial_u + v\partial_v$. This implies that algebraic identities induced by self-similarity take a relatively simple form (see, for example, Lemma~\ref{scalrelations2}) and also allows us to view the class of $\kappa$-self-similar solutions as a perturbation of the self-similar solutions of Definition~\ref{scaleinvdef}.
	\item Even though the Ricci coefficients $\psi$ and curvature components $\Psi$ will in general be singular as $v\to 0$, we will have a very simple procedure for weighting them appropriately (when $\Psi \neq \alpha_{AB}$ and $\psi \neq \omega$). Namely, letting $s$ denote the signature of $\psi$ and $\Psi$ (see Definition~\ref{signature}) we will have that $\Omega^s\psi$ and $\Omega^s\Psi$ have regular limits as $v\to 0$. (See Lemma~\ref{dkwowsss322} below.) 
\end{enumerate}
 One may think of the parameter $\kappa$ as being (roughly) the analogue of $k^2/2$ in Christodoulou's $k$-self-similar solutions. 
 
Finally, we note that while it is technically convenient for us to phrase Definition~\ref{kapselfseim} in terms of a coordinate system which is regular along $\{\hat{v} = 0\}$, it may be more useful conceptually to think of them as ``$b$-self-similar'' solutions. This is because the need to consider $\kappa$-self-similar solutions arises when one wants to construct a self-similar solution with a non-trivial shift $b_A$ along the incoming cone $\{\hat{v} = 0\}$, or equivalently, when one wants the scaling symmetry to induce a twist along the spheres in addition to rescaling in the null direction. See also Lemma~\ref{thatissuchaconst} and the following discussion.
 
\subsection{The Ingoing Raychaudhuri Equation and the Role of $\kappa$-Self-Similarity} 
 In this section we will explain how $\kappa$-self-similarity allows us to overcome the obstacle of Proposition~\ref{restrictonv0yay}.

The following lemma shows how to weight the Ricci coefficients and null curvature components in $(u,v,\theta^A)$ coordinates so as to have regular limits as $v\to 0$.
\begin{lemma}\label{dkwowsss322} Let $(\mathcal{M},g)$ be a $\kappa$-self-similar solution. Then, in the $\left(u,\hat{v},\theta^A\right)$ coordinates, the metric $g$ is of the form
\begin{equation}\label{newformhatv}
g = -2\Omega^2v^{2\kappa}\left(1-2\kappa\right)^{-1}\left(du\otimes d\hat{v} + d\hat{v}\otimes du\right) + \slashed{g}_{AB}\left(d\theta^A - b^Adu\right)\otimes \left(d\theta^B-b^Bdu\right).
\end{equation}
The self-similar vector field takes the form
\[K = u\partial_u + \left(1-2\kappa\right)\hat{v}\partial_{\hat{v}}.\]
Finally, using that, in the form~\eqref{newformhatv}, the tensor $\mathcal{L}^i_{\partial_u}\mathcal{L}_{\partial_{\theta}}^{(j)}g_{\mu\nu}$ must extend to $\{\hat{v} = 0\}$ as a $C^{1,p}$ tensor for $0 < p \leq \gamma(|j|)$, we find that the following double null quantities defined in $(u,v,\theta^A)$ coordinates must extend to $\{v = 0\}$ continuously and the restrictions to $\{v = 0\}$ are smooth tensors in $\left(u,\theta^A\right)$:

\begin{equation}\label{regeregerge1}
	v^{\kappa}\Omega,\qquad b^A,\qquad \slashed{g}_{AB},\qquad \Omega\underline{\omega},\qquad v\Omega\omega,\qquad \zeta_A,\qquad \eta_A,\qquad \underline{\eta}_A,\qquad \Omega\hat{\underline{\chi}}_{AB},\qquad \Omega{\rm tr}\underline{\chi},
	\end{equation}
	\begin{equation}\label{regeregerge2}
\Omega^{-1}\hat{\chi}_{AB},\qquad \Omega^{-1}{\rm tr}\chi,\qquad \Omega^2\underline{\alpha}_{AB},\qquad \Omega\underline{\beta}_A,\qquad \rho,\qquad \sigma,\qquad \Omega^{-1}\beta_A.
\end{equation}

Equivalently, for any Ricci coefficient $\psi$ not equal to $\omega$ or null curvature component $\Psi$ not equal to $\alpha$, we have that $\Omega^s\psi$ and $\Omega^s\Psi$ extend continuously to $\{v = 0\}$, where $s$ denotes the signature of $\psi$ or $\Psi$. Finally, we note that we will have
\[\left|\Omega^s\psi|_{v=0}\right| \lesssim |u|^{-1},\qquad \left|\Omega^s\Psi|_{v=0}\right| \lesssim u^{-2}.\]

\end{lemma}
\begin{proof}This follows in a straightforward manner from the definitions of the various metric components, Ricci coefficients, and curvature components and the fact that
\[\frac{\partial}{\partial v} = \left(1-2\kappa\right)v^{-2\kappa}\frac{\partial}{\partial \hat{v}}.\]
\end{proof}

\begin{convention}\label{thev0convention} In the $\left(u,v,\theta^A\right)$ coordinates, not all components of the metric $g$ extend continuously to $\{v = 0\}$. Nevertheless, we will refer to the hypersurface $\{v = 0\}$ with the understanding that along $\{v=0\}$ it only makes sense to consider the quantities listed in~\eqref{regeregerge1} and~\eqref{regeregerge2}.
\end{convention}

Before proceeding, it is useful to note that by differentiating the formulas in~\eqref{scaleinvrelations2} we can produce various relations between the Ricci coefficients. We list the most important ones in the lemma below.
\begin{lemma}\label{scalrelations2}Let $\left(\mathcal{M},g_{\mu\nu}\right)$ be a self-similar solution or a $\kappa$-self-similar solution. Then we have
\begin{equation}\label{0skm}
\Omega{\rm tr}\underline{\chi} + \Omega \frac{v}{u}{\rm tr}\chi = \frac{2}{u} + \slashed{\rm div}b,\qquad \Omega\hat{\underline{\chi}}_{AB} + \Omega \frac{v}{u}\hat{\chi}_{AB} = \frac{1}{2}\left(\slashed{\nabla}\hat{\otimes}b\right)_{AB},
\end{equation}
\begin{equation}\label{09sjn|}
\Omega\underline{\omega} + \frac{v}{u}\Omega\omega +\frac{1}{2}\mathcal{L}_b\log\Omega = 0.
\end{equation}
\end{lemma}
\begin{proof}See Lemma B.1 of~\cite{scaleinvariant} for the proof in the case of a self-similar solution. The same proof works for the $\kappa$-self-similar case.
\end{proof}

Observe that any $\kappa$-self-similar solution $g_{\mu\nu}$ in the $\left(u,\hat{v},\theta^A\right)$ coordinates must satisfy the $e_3$-Raychaudhuri equation~\eqref{3truchi} along $\{\hat{v} = 0\}$. This may be interpreted as a constraint equation for incoming characteristic $\kappa$-self-similar data. In the next lemma we compute the precise form of this constraint equation.
\begin{lemma}\label{thatissuchaconst}For any $\kappa$-self-similar solution or self-similar solution, along $\{v = 0\}$ the following equation must hold:
\begin{align}\label{basicconstr2}
&\frac{1}{-u}\slashed{\rm div}b- \mathcal{L}_b\slashed{\rm div}b- \frac{1}{2}\left(\slashed{\rm div}b\right)^2 =  \frac{1}{4}\left|\slashed{\nabla}\hat{\otimes}b\right|^2-\frac{4\kappa}{u^2}+\frac{4}{-u}\mathcal{L}_b\log\Omega -2\left(\mathcal{L}_b\log\Omega\right)\left(\slashed{\rm div}b\right),
\end{align}
where we take $\kappa = 0$ in~\eqref{basicconstr2} for a self-similar solution. (We note that~\eqref{basicconstr2} may be equivalently phrased as an equation for $\check{b}$ along $\mathbb{S}^2$ for $\check{b}$ as in Definition~\ref{scaleinvrelations2}.)
\end{lemma}
\begin{proof}See Appendix~\ref{askf}.
\end{proof}
We emphasize that it follows from Lemma~\ref{dkwowsss322} that $\mathcal{L}_b\log\Omega$ extends to $\{v = 0\}$. We call equation~\eqref{basicconstr2} the $\kappa$-constraint equation.

We can now explain the role of $\kappa$-self-similarity in the construction of naked singularities. Recall from the discussion in Sections~\ref{letscompareyayaya} and~\ref{asss} that in order for the singular point at $(u,v) = (0,0)$ to formally correspond to a non-BV singularity in the sense of Christodoulou~\cite{ChristBV}, we desire that the shear of the incoming cone is non-integrable as $u \to 0$ (see also the discussion in item~\ref{9090099090128jnomi2} of Section~\ref{letscompareyayaya}). We also desire (see the discussion in Section~\ref{nakedsingchristchrist}) for the Hawking mass $m\left(\mathbb{S}^2_{u,0}\right)$ to satisfy that $m\left(\mathbb{S}^2_{u,0}\right)\left({\rm Area}\left(\mathbb{S}^2_{u,0}\right)\right)^{-1/2}$ is uniformly bounded from below as $u\to 0$. For a self-similar solution, every Ricci coefficient $\psi$ will satisfy $\left|\psi|_{v=0}\right|_{\slashed{g}}\left(u,\theta^A\right) = u^{-1}h\left(\theta^A\right)$ for a suitable function $h$. Using this, it is possible to show that the above requirements will hold for a self-similar or $\kappa$-self-similar solution only if, for some $c > 0$, the restriction of the ingoing shear $\hat{\underline{\chi}}_{AB}$ to $\{u = -c\}\cap \{v = 0\}$ is non-vanishing along a suitable portion of $\mathbb{S}^2$ (cf.~the Hawking mass calculation later in Lemma~\ref{hawkcalccalc}).

For self-similar or $\kappa$-self-similar solutions, Lemma~\ref{scalrelations2} implies that $\Omega\hat{\underline{\chi}}_{AB}|_{v=0} = \frac{1}{2}\left(\slashed{\nabla}\hat{\otimes}b\right)_{AB}$.  Thus, in view of the previous paragraph, we are naturally lead to ask if we can find $\slashed{g}_{AB}$, $\Omega$, and $b_A$ so that $b_A$ has a non-trivial trace-free deformation tensor and satisfies the $\kappa$-constraint equation~\eqref{basicconstr2}. However,  Proposition~\ref{restrictonv0yay} implies that this is impossible for self-similar solutions with $\kappa = 0$!\footnote{\label{footnotenotexist}This is conceptually clearest in the case where $\Omega = 1$ and we assume that $b_A\sim \epsilon$. In this case~\eqref{basicconstr2} suggests that $\slashed{\rm div}b \sim \epsilon^2$ and, along $\{u =-1\} \cap \{v=0\}$, the equation~\eqref{basicconstr2} becomes of the form
\[\slashed{\rm div}b = \frac{1}{4}\left|\slashed{\nabla}\hat{\otimes}b\right|^2 + O\left(\epsilon^3\right) \Rightarrow \int_{\mathbb{S}^2}\left|\slashed{\nabla}\hat{\otimes}b\right|^2 = O\left(\epsilon^3\right).\]
This argument can be iterated and suggests that no solutions $b_A$ exist.} In contrast, if we allow $\kappa$ to be non-zero, then, as shown in Appendix~\ref{constructthetupleregul}, we have an infinite class of solutions to the $\kappa$-constraint equation where $\left(\slashed{\nabla}\hat{\otimes}b\right)_{AB}$ is non-trivial.\footnote{The difference with the discussion from footnote~\ref{footnotenotexist} is that if we allow $\kappa \neq 0$, then  we have 
\[\slashed{\rm div}b = \frac{1}{4}\left|\slashed{\nabla}\hat{\otimes}b\right|^2 -4\kappa+ O\left(\epsilon^3\right),\]
and for an appropriate value of $\kappa \sim \epsilon^2$, there are no obstructions to finding a solution with $\slashed{\nabla}\hat{\otimes}b \neq 0$.} \underline{This is the crucial place where $\kappa$-self-similarity plays a role in our construction.}
 
We close this section with one final remark about $\kappa$-self-similar solutions:
\begin{remark}\label{moregeneral}One may consider a generalization of the notion of $\kappa$-self-similar solutions from Definition~\ref{kapselfseim} where, instead of just being a constant along $\{v = 0\}$, $\kappa$ depends on the angular variable $\kappa\left(\theta^A\right)$, and we have $\mathcal{L}_b\kappa|_{v=0} = 0$. (If the condition $\mathcal{L}_b\kappa|_{v=0} = 0$ is violated, one may show that the $e_3$-Raychaudhuri equation cannot be satisfied along $\{v = 0\}$.)  We will not explicitly consider such solutions here; however, the possible existence of them will implicitly appear later in our analysis as we approach the cone $\{u=0\}$. (See Section~\ref{bottIIIIIIII}.)
\end{remark}

In the next seven subsections, we will discuss the main steps of the proof of our main result Theorem~\ref{themainextresult}.
\subsection{Degenerate Transport Equations on $\mathbb{S}^2$}
In Section~\ref{pdetheoryforformal} we will carry out an analysis of certain classes of degenerate transport equations on $\mathbb{S}^2$. The linear version of these equations come in two main forms:
\begin{enumerate}
	\item\label{123ijni299103} The first type is of the form $\left(u + \mathcal{L}_Xu + h\cdot u\right)_{A_1\cdots A_k} = F_{A_1\cdots A_k}$, where $X^A$ is a given vector field on $\mathbb{S}^2$, $h = \sum_{i=0}^k c_i \left(h^{(i)}\right)_{A_1\cdots A_i}^{\ \ \ \ \ \ \ B_1\cdots B_i}$ is a given linear combination of $(i,i)$-tensors on $\mathbb{S}^2$ for $i \in [0,k]$, and $F_{A_1\cdots A_k}$ is a given tensors on $\mathbb{S}^2$. We will furthermore have a smallness assumption on $X^A$ and $h$. We will show that these equations have unique solutions $u_{A_1\cdots A_k}$ which satisfy suitable a priori estimates. The basic idea is to use the smallness of $X^A$ and $h$ to treat these equations as perturbations of the identity.
	\item The second type of equation, called the ``$\kappa$-singular equation'' will be of the form \begin{equation}\label{kdlpwlpdwqacssing}\left(\mathcal{L}_bf_{AB}  - \left(\slashed{\nabla}\hat{\otimes}b\right)^C_{\ \ (A}f_{B)C} - \frac{1}{2}\slashed{\rm div}b f_{AB}\right) -2\kappa f_{AB}= H_{AB},\end{equation} where $b^A$ is a given vector field on $\mathbb{S}^2$, we have a Riemannian metric $\slashed{g}_{AB}$ on $\mathbb{S}^2$, $\kappa$ is as in Definition~\ref{kapselfseim}, and $H$ is a given trace-free symmetric $2$-tensor on $\mathbb{S}^2$. The goal will be to show that there exists a unique solution $f_{AB}$ satisfying appropriate a priori estimates. For a suitable $0 < \epsilon \ll 1$ we will have, schematically, that $b^A \sim \epsilon$ and $\kappa\sim \epsilon^2$. In particular, this equation cannot be treated as a perturbation of the identity. Our analysis will instead be based on exploiting an anti-symmetric structure. This anti-symmetric structure leads easily to $L^2$-estimates for $f_{AB}$, but we will have to work considerably harder for higher order estimates. We will also consider an evolutionary analogue of the $\kappa$-singular equation on $[0,\infty) \times \mathbb{S}^2$: 
	\begin{equation}\label{evolvemkwasf}
	\mathcal{L}_sf_{AB} + \left(\mathcal{L}_bf_{AB}  - \left(\slashed{\nabla}\hat{\otimes}b\right)^C_{\ \ (A}f_{B)C} - \frac{1}{2}\slashed{\rm div}b f_{AB}\right) -2\kappa f_{AB}= H_{AB}.
	\end{equation}
	This will be treated in a similar fashion to the $\kappa$-singular equation.
	\end{enumerate}
	
	Equations of these types arise in the following fashion:  For a $\kappa$-self-similar solution, let $\phi$ stand for one of expressions in~\eqref{regeregerge1} or~\eqref{regeregerge2}. Then we have that $\phi\left(u,v,\theta^A\right) = u^{-c}\check{\phi}\left(\frac{v}{u},\theta^A\right)$ for a suitable integer $c$ and some $\check{\phi}\left(x,\theta^A\right)$ which has a regular limit as $x \to 0$ . In view of this, we will have that 
	\[\Omega\nabla_3\phi = -\frac{v}{u}\Omega\nabla_4\phi + \mathcal{L}_b\phi + {\rm lower\ order\ terms}.\]
	Since we will also have that  $v\Omega\nabla_4\phi \to 0$ as $v\to 0$, it is thus clear that whenever we restrict a $\Omega\nabla_3\phi$ equation to $\{v = 0\}$ we will obtain an equation of the form 
	\[\mathcal{L}_b\phi + {\rm lower\ order\ terms} = 0.\]
	We will need to construct solutions to these equations in certain cases, and in all of these cases except when we study $\Omega^{-1}\hat{\chi}$, these equations will be of the type mentioned in item~\ref{123ijni299103}. When we carry out this same procedure for $\Omega^{-1}\hat{\chi}$, the restriction of the equation for $\Omega\nabla_3\left(\Omega^{-1}\hat{\chi}\right)$ leads to the study of the equation~\eqref{kdlpwlpdwqacssing}. Lastly, the equation~\eqref{evolvemkwasf} will arise when we construct outgoing characteristic initial data for our solutions. We will want the corresponding $\Omega^{-1}\hat{\chi}$ to be self-similar to leading order as $v\to 0$, and thus the  equation for $\Omega\nabla_3\left(\Omega^{-1}\hat{\chi}\right)$ will lead to an equation of the form~\eqref{evolvemkwasf} which must hold along this initial outgoing null hypersurface to leading order as $v\to 0$.

	Finally, we will use the linear theory developed to undertake a detailed analysis of certain classes of solutions to the nonlinear $\kappa$-constraint equation~\eqref{basicconstr2}.

\begin{remark}Though we will not pursue this direction, we could derive formal expansions for $\kappa$-self-similar metrics near $\{\hat{v} = 0\}$ in the spirit of Fefferman--Graham~\cite{FG1,FG2}. We note however a significant difference between our setting and that of Fefferman--Graham; namely, the formal expansions of~\cite{FG1,FG2} are algebraic in that successive terms in the expansion are given by rational functions of (angular derivatives of) previous terms in the expansions, while in our setting deriving even a formal expansion will require one to solve degenerate transport equations of the type discussed above.
\end{remark}

\subsection{Constructing the Characteristic Initial Data Sets}\label{previewofchardata}
In Section~\ref{setitupayya} we construct classes of characteristic initial data sets along the hypersurfaces $\{\hat{v} = 0\} \cap \{u \in [-1,0)\}$ and $\{u = -1\} \cap \{\hat{v} \in [0,v_0)\}$, for some $0 < \epsilon \ll v_0 \ll 1$. The desired solution for our main result Theorem~\ref{themainextresult} will be constructed from this characteristic initial data. As we have mentioned before, the solutions we construct will \emph{not} be globally $\kappa$-self-similar; nevertheless, they will become approximately $\kappa$-self-similar as $\frac{v}{-u}\to 0$. This approximate $\kappa$-self-similarity will be reflected in the construction of the characteristic initial data sets. 

Before we describe the construction we make the following point about working in the $\left(u,\hat{v},\theta^A\right)$ coordinates versus the $\left(u,v,\theta^A\right)$ coordinates: Strictly speaking, the local existence results must be applied to the metric in $\left(u,\hat{v},\theta^A\right)$ coordinates, since it is only in these coordinates the full metric $g$ extends regularly to $\{\hat{v} = 0\}$. However, any statement in the $\left(u,\hat{v},\theta^A\right)$ can be translated into an equivalent statement in the $\left(u,v,\theta^A\right)$ coordinates and, as we have mentioned before, it is often more convenient to work with metric and double-null quantities defined in the $(u,v,\theta^A)$ coordinates.

We return now to the discussion of characteristic initial data. Keeping Lemma~\ref{dkwowsss322} in mind, the incoming characteristic data we may prescribe are the restrictions to $\{v = 0\}$ of $v^{\kappa}\Omega$, $b^A$, and $\slashed{g}_{AB}$, where, $\kappa > 0$ is, for the moment, a free parameter. Since the change of variables $(u,v) \mapsto (\lambda u,\lambda v)$ leaves the hypersurface $\{v = 0\}$ invariant (keeping Convention~\ref{thev0convention} in mind), we immediately obtain a notion of being $\kappa$-self-similar along $\{v = 0\}$. Namely, the incoming characteristic initial data will be $\kappa$-self-similar along $\{v = 0\}$ if there exists a function $\check{\Omega}(\theta^A)$, a vector field $\check{b}^A(\theta^B)$, and a Riemannian metric $\check{\slashed{g}}_{AB}(\theta^C)$ on $\mathbb{S}^2$ such that
\[\left(v^{\kappa}\Omega,b^A,\slashed{g}_{AB}\right)|_{v=0} = \left((-u)^{\kappa}\check{\Omega},u^{-1}\check{b}^A,u^2\check{\slashed{g}}_{AB}\right).\]
Furthermore, our choice of incoming characteristic initial data must satisfy the $\kappa$-constraint equation~\eqref{basicconstr2}. We allow as our initial data any solution to~\eqref{basicconstr2} which satisfy certain regularity and smallness assumptions and if $b^A$ takes a certain specific form to leading order in the small parameter $\epsilon$ which will measure the size of $b|_{v=0}$. (See Definition~\ref{Mreg}.) In Appendix~\ref{constructthetupleregul} we construct an infinite class of such admissible solutions. For a small constant $0 < \epsilon \ll 1$ we will have, for all of these solutions, that
\begin{equation}\label{roughtsizeofthw}
|b| \sim \epsilon, \qquad \left|\slashed{\rm div}b\right| \lesssim \epsilon^2,\qquad |\varphi| \lesssim \epsilon^2, \qquad |\log\Omega| \lesssim \epsilon^2,\qquad \kappa \sim\epsilon^2,
\end{equation}
where $\slashed{g}_{AB} =u^2e^{2\varphi}\mathring{\slashed{g}}_{AB}$ for $\mathring{\slashed{g}}_{AB}$ denoting the round metric. (We will suppress in this introductory section certain $\epsilon^{-\delta}$ losses for $0 < \delta \ll 1$.) Having thus determined this incoming data, we set 
\[\left(\Omega^{(\rm in)},(b^A)^{(\rm in)},\slashed{g}_{AB}^{(\rm in)}\right) \doteq \left((-u)^{\kappa}\check{\Omega},u^{-1}\check{b}^A,u^2\check{\slashed{g}}_{AB}\right).\]

The next piece of characteristic data we need to define is the value of  the torsion $\zeta_A$, or equivalently, $\eta_A$ restricted to $(u,v) = (-1,0)$. From Lemma~\ref{bettereta} and the choice of incoming characteristic data above, in the eventual solution $\eta_A$ satisfies the following equation along $\{v=0\}$:
\begin{equation}\label{thisisteea}
\nabla_{\partial_u +b}\eta_A + \frac{3}{2}\left(\frac{2}{u} +\slashed{\rm div}b\right)\eta_A + \frac{1}{2}\left(\left(\slashed{\nabla}\hat{\otimes}b\right)\cdot\eta\right)_A = 2\slashed{\nabla}_A\left(\mathcal{L}_b\log\Omega\right) -\slashed{\rm div}\left(\slashed{\nabla}\hat{\otimes}b\right)_A +\frac{1}{2}\slashed{\nabla}_A\slashed{\rm div}b.
\end{equation} 
Note that $\partial_u +b^A\slashed{\nabla}_A$ has integral curves which are tangent to $\{v = 0\}$, and the right hand side of~\eqref{thisisteea} will be $O\left(\epsilon u^{-2}\right)$. In particular, keeping~\eqref{roughtsizeofthw} in mind, it is straightforward to see that solutions of~\eqref{thisisteea} with generic data posed at $\{u =-1\}$ will satisfy $\sup_{\mathbb{S}^2}\left|\eta\right| \gtrsim |u|^{-3+C\epsilon}$ for a suitable constant $C$, independent of $\epsilon$. This behavior is more singular than the $\kappa$-self-similar rate of $|u|^{-1}$ and we would not be able to effectively control the resulting solution on  a long enough time-scale. We will thus need to fine-tune the initial data for $\eta_A$ so as to arrange for $\eta_A$ to satisfy $|\eta| \sim |u|^{-1}$ (cf.~Proposition~\ref{formalcalculationsscaleinv}). To see how we might carry out this fine-tuning, it is useful to first note that if $\left(\mathcal{M},g_{\mu\nu}\right)$ is $\kappa$-self-similar then we would have both $\left|\eta\right| \sim |u|^{-1}$ and that
\begin{equation}\label{kdowdkodwkodwkowd}
\nabla_{\partial_u+b}\eta_A|_{v=0} = \mathcal{L}_b\eta_A - \frac{1}{2}\left(\left(\slashed{\nabla}\hat{\otimes}b\right)\cdot\eta\right)_A - \frac{1}{2}\left(\frac{2}{u} + \slashed{\rm div}b\right)\eta_A.
\end{equation}
(This last expression is derived by using that in a $\kappa$-self-similar spacetime and in a coordinate frame we would have that $\eta_A\left(v,u,\theta^B\right) = \check{\eta}_A\left(\frac{v}{-u},\theta^B\right)$.) Plugging in~\eqref{kdowdkodwkodwkowd} into~\eqref{thisisteea} and restricting to $\{u=-1\}$ leads to the equation
\begin{equation}\label{odwkwdokwdkowdok}
\mathcal{L}_b\eta_A + \left(-2 +\slashed{\rm div}b\right)\eta_A  = 2\slashed{\nabla}_A\left(\mathcal{L}_b\log\Omega\right) -\slashed{\rm div}\left(\slashed{\nabla}\hat{\otimes}b\right)_A +\frac{1}{2}\slashed{\nabla}_A\slashed{\rm div}b.
\end{equation}
Using the theory that we will develop in Section~\ref{pdetheoryforformal}, we will be able to show that this has a unique solution. This unique solution is the prescribed value that we use for $\eta$ at $(u,v) = (-1,0)$.\footnote{\label{agoodthingtoputinafootnoteyayaya}Note that if we solve the transport equation~\eqref{thisisteea} for $\eta$ from initial data solving~\eqref{odwkwdokwdkowdok}, then by uniqueness of solutions to transport equations, it must be the case that in the coordinate frame $\eta_A\left(u,\theta^A\right) = \check{\eta}_A\left(\theta^B\right)$ and that $\eta$ solves~\eqref{odwkwdokwdkowdok} when restricted to $\{u = -1\}$.}

Now we come to the outgoing characteristic data $\left(v^{\kappa}\Omega^{(\rm out)},\slashed{g}_{AB}^{(\rm out)}\right)$ along $\left(v,\theta^A\right) \in [0,\underline{v}] \times \mathbb{S}^2$, where $0 < \underline{v} \ll 1$ is a constant which we may freely prescribe. We will take $v^{\kappa}\Omega^{(\rm out)}$ simply to be constant in the $v$-direction. (The value of $v^{\kappa}\Omega^{(\rm out)}$ at $v = 0$ is fixed already by the value of $v^{\kappa}\Omega^{(\rm in)}$ at $u = -1$.)  Since the value of $\slashed{g}_{AB}$ is already determined at $v = 0$ and ${\rm tr}\chi$ is determined by the constraint~\eqref{ray2cons}, the prescription of $\slashed{g}_{AB}^{(\rm out)}$ may be considered roughly equivalent to the prescription of $\Omega^{-1}\hat{\chi}_{AB}$ along $v \in [0,\underline{v}]$ and $\Omega^{-1}{\rm tr}\chi$ at $v = 0$. (See the proof of Proposition~\ref{itexistsbutforalittlewhile} for details.) The problem of prescribing $\Omega^{-1}{\rm tr}\chi$ is similar to $\eta_A$; generic choices of $\Omega^{-1}{\rm tr}\chi$ along $(u,v) = (-1,0)$ would lead to $\Omega^{-1}{\rm tr}\chi$ having a singular behavior as $u\to 0$ which is more singular than the $\kappa$-self-similar rate. Just as with $\eta_A$ there is a specific value of $\Omega^{-1}{\rm tr}\chi$, obtained by solving the equation
\begin{equation}\label{trchiidwqqq}
 \mathcal{L}_b\left(\Omega^{-1}{\rm tr}\chi\right) + \left(\Omega^{-1}{\rm tr}\chi\right)\left(\frac{1}{u}+{\rm div}b+\frac{2\kappa}{u} + 2\mathcal{L}_b\log\Omega\right) = -2K  + 2\slashed{\rm div}\eta + 2\left|\eta\right|^2,
\end{equation}
which leads to a behavior consistent with $\kappa$-self-similarity. It is this unique value of $\Omega^{-1}{\rm tr}\chi$ determined by this equation and the analysis of Section~\ref{pdetheoryforformal} which we take for $\Omega^{-1}{\rm tr}\chi|_{(u,v) = (-1,0)}$. 

The prescription for $\Omega^{-1}\hat{\chi}_{AB}$ will be more complicated than that for $\eta_A$ and $\Omega^{-1}{\rm tr}\chi$. We need to pose this for $v \in [0,\underline{v}]$ but it will be conceptually clarifying to first focus on the prescribed value of $\Omega^{-1}\hat{\chi}_{AB}$ at $(u,v) = (-1,0)$. In the eventual solution, a consequence of~\eqref{3hatchi} is that along $\{v = 0\}$ we will have that $\Omega^{-1}\hat{\chi}_{AB}$ satisfies the following propagation equation:
\begin{align}\label{forjdiwkochuhar}
&\nabla_{\partial_u + b}\left(\Omega^{-1}\hat{\chi}\right)_{AB} + \frac{1}{2}\left(\frac{2}{u}+\slashed{\rm div}b\right)\left(\Omega^{-1}\hat{\chi}\right)_{AB} - \frac{2\kappa}{u}\left(\Omega^{-1}\hat{\chi}\right)_{AB} = 
\\ \nonumber &\qquad \qquad \qquad \left(\slashed{\nabla}\hat{\otimes}\eta\right)_{AB} + \left(\eta\hat{\otimes}\eta\right)_{AB} - \frac{1}{4}\left(\Omega^{-1}{\rm tr}\chi\right)\left(\slashed{\nabla}\hat{\otimes}b\right)_{AB},
\end{align}
where $\eta_A$ and $\Omega^{-1}{\rm tr}\chi$ have already been determined in the above analysis. This equation is schematically of the form 
\begin{equation}\label{kodwkowdkowdok}
\nabla_{\partial_u + b}\left(\Omega^{-1}\hat{\chi}\right)_{AB} + \frac{1+O\left(\epsilon\right)}{u}\left(\Omega^{-1}\hat{\chi}\right)_{AB} - \frac{2\kappa}{u}\left(\Omega^{-1}\hat{\chi}\right)_{AB} = O\left(\epsilon u^{-2}\right).
\end{equation}
Note that whether or not solutions $\Omega^{-1}\hat{\chi}_{AB}$ to the equation~\eqref{forjdiwkochuhar} with generic data at $u = -1$ satisfy the $\kappa$-self-similar bound $\left|\Omega^{-1}\hat{\chi}\right| \lesssim |u|^{-1}$ depends on the $O(\epsilon)$ term on the left hand side of~\eqref{kodwkowdkowdok}! Using the theory developed in Section~\ref{pdetheoryforformal}, one can in fact show that generic solutions \emph{do not} satisfy the $\kappa$-self-similar bound. Thus, as with $\eta_A$ and $\Omega^{-1}{\rm tr}\chi$ we need to fine-tune the value of $\Omega^{-1}\hat{\chi}_{AB}$ on $(u,v) = (-1,0)$. We start by observing that if $\left(\mathcal{M},g_{\mu\nu}\right)$ is $\kappa$-self-similar then we would have both $\left|\Omega^{-1}\hat{\chi}\right| \sim |u|^{-1}$ and that
\begin{equation}\label{kdowdkodwkodwkowd23}
\nabla_{\partial_u+b}\left(\Omega^{-1}\hat{\chi}\right)_{AB}|_{v=0} = \frac{1}{u}\left(\Omega^{-1}\hat{\chi}\right)_{AB} + \mathcal{L}_b\left(\Omega^{-1}\hat{\chi}\right)_{AB} - 2\left(\Omega\underline{\chi}\right)^C_{\ \ (A}\hat{\chi}_{B)C}.
\end{equation}
Plugging~\eqref{kdowdkodwkodwkowd23} into~\eqref{kodwkowdkowdok} leads to 
\begin{align}\label{knbhuikawefkf2}
&\mathcal{L}_b\left(\Omega^{-1}\hat{\chi}\right)_{AB} -\left(\frac{1}{2}\slashed{\rm div}b - \frac{2\kappa}{u} - 2\mathcal{L}_b\log\Omega\right)\Omega^{-1}\hat{\chi}_{AB} - \left(\slashed{\nabla}\hat{\otimes}b\right)^C_{\ \ (A}\left(\Omega^{-1}\hat{\chi}\right)_{B)C} = \\ \nonumber &\qquad \left(\slashed{\nabla}\hat{\otimes}\eta\right)_{AB} + \left(\eta\hat{\otimes}\eta\right)_{AB} - \frac{1}{2}\left(\Omega^{-1}{\rm tr}\chi\right)\left(\slashed{\nabla}\hat{\otimes}b\right)_{AB},
\end{align}
which is essentially the $\kappa$-singular equation~\eqref{kdlpwlpdwqacssing}. Using the theory developed in Section~\ref{pdetheoryforformal} we will be able to show that there is a unique solution $\Omega^{-1}\hat{\chi}_{AB}$ to this equation; however, we will only have the estimate $|\Omega^{-1}\hat{\chi}|_{v=0}| \sim \epsilon^{-1}|u|^{-1}$. (Conceptually, the source of the amplification can be easily understood if one simply drops the anti-symmetric operator term in parentheses in~\eqref{kdlpwlpdwqacssing}  and refers to~\eqref{roughtsizeofthw}.) In particular, this implies that for $0 \leq \frac{v}{-u} \lesssim 1$, the best $L^{\infty}$ bound we can hope to propagate for $\hat{\chi}_{AB}$ into our spacetime $\left(\mathcal{M},g_{\mu\nu}\right)$ is $\left\vert\left\vert \Omega^{-1}\hat{\chi}\right\vert\right\vert_{L^{\infty}\left(\mathbb{S}^2_{u,v}\right)} \lesssim \epsilon^{-1}|u|^{-1}$.

Given the largeness of $\hat{\chi}_{AB}$ in $L^{\infty}$, we can only hope to control the solution on a long time-scale if $\hat{\chi}_{AB}$ becomes small \emph{after integration in $v$}. In fact, we will eventually show something much stronger, that is, we will eventually show that $\left|\hat{\chi}\right| \sim \epsilon|u|^{-1}$ once $\frac{v}{-u} \gtrsim e^{-\left(C\epsilon\right)^{-1}}$! To understand the possible behavior of $\hat{\chi}_{AB}$ for $\frac{v}{-u} > 0$, and how that affects what characteristic data we should impose, we need to consider the equation~\eqref{forjdiwkochuhar} for $v > 0$. For $0 < \frac{v}{-u} \ll 1$, we expect (optimistically) that the right hand side and the coefficients of the $\Omega^{-1}\hat{\chi}_{AB}$ on the left hand side will be roughly constant in $\frac{v}{-u}$. Furthermore, just as along $\{v = 0\}$, we expect that the initial data for $\Omega^{-1}\hat{\chi}_{AB}$ along $\{u = -1\} \cap \{0 \leq v \ll 1\}$ needs to be fine-tuned in order for $\Omega^{-1}\hat{\chi}_{AB}$ to have a bound as $u\to 0$ which is consistent with $\kappa$-self-similarity. This suggests that we make a self-similar ansatz for $\Omega^{-1}\hat{\chi}_{AB}$, freeze the coefficients of~\eqref{forjdiwkochuhar} with their values at $\{v = 0\}$, and then try to understand the set of solutions to~\eqref{forjdiwkochuhar}. Since the resulting equation involves only $\kappa$-self-similar quantities, it suffices to consider the restriction to $\{u = -1\} \cap \{0 \leq v \ll 1\}$. We obtain then the following analogue of~\eqref{knbhuikawefkf2} along $\{u = -1\} \cap \{0 \leq v \ll 1\}$:
\begin{align}\label{knbhuikawefkf2348920}
&v\mathcal{L}_{\partial_v}\left(\Omega^{-1}\hat{\chi}\right)_{AB} + \mathcal{L}_{b|_{v=0}}\left(\Omega^{-1}\hat{\chi}\right)_{AB} 
\\ \nonumber &\qquad -\left(\frac{1}{2}\slashed{\rm div}b + 2\kappa - 2\mathcal{L}_b\log\Omega\right)|_{v=0}\left(\Omega^{-1}\hat{\chi}\right)_{AB} - \left(\left(\slashed{\nabla}\hat{\otimes}b\right)|_{v=0}\right)^C_{\ \ (A}\left(\Omega^{-1}\hat{\chi}\right)_{B)C} =
\\ \nonumber & \left(\slashed{\nabla}\hat{\otimes}\eta + \eta\hat{\otimes}\eta - \frac{1}{2}\left(\Omega^{-1}{\rm tr}\chi\right)\left(\slashed{\nabla}\hat{\otimes}b\right)\right)_{AB}|_{v=0}.
\end{align}

The equation~\eqref{knbhuikawefkf2348920} is a (degenerate) transport equation which, after the change of variables $s \doteq -\log(v)$, is of the same essential form as the equation~\eqref{evolvemkwasf}. The value of $\Omega^{-1}\hat{\chi}_{AB}$ that we will impose along $\{u = -1\}$ will be, to leading order as $v\to 0$, a suitable solution to the equation~\eqref{knbhuikawefkf2348920}. We determine this ``suitable solution'' as follows. Along integral curves of $\partial_s + b^A\slashed{\nabla}_A$, the equation~\eqref{knbhuikawefkf2348920} become an ordinary differential equation. In particular, solutions to~\eqref{knbhuikawefkf2348920} are in one to one correspondence with initial data for~\eqref{knbhuikawefkf2348920} prescribed at $\mathbb{S}^2_{-1,\underline{v}}$. Our desired $\Omega^{-1}\hat{\chi}_{AB}$ is then obtained by solving~\eqref{knbhuikawefkf2348920} with the ``initial conditition'' that $\Omega^{-1}\hat{\chi}_{AB}|_{v=\underline{v}} = 0$. Using the theory developed in Section~\ref{pdetheoryforformal} we will see that this solution converges as $v\to 0$ to the unique solution to~\eqref{knbhuikawefkf2}. Thus, we see that this choice of data effectively drives down $\Omega^{-1}\hat{\chi}_{AB}$, as $v$ increases from $0$, from a value of size $\epsilon^{-1}$ to $0$ in a fashion which, to leading order, is consistent with $\kappa$-self-similar bounds. In fact, we will see that $\Omega^{-1}\hat{\chi}_{AB} \sim \epsilon$ for $v \in \left[C\exp\left(-\left(C\epsilon\right)^{-1}\right),\underline{v}\right]$ for a suitable constant $C$, which is independent of $\epsilon$.

The fundamental downside of this procedure is that $\Omega^{-1}\hat{\chi}_{AB}$ will only be H\"{o}lder continuous as $v\to 0$.\footnote{We note that this loss of regularity appears to be a genuine feature associated to the construction of $\kappa$-self-similar solutions which become small away from the cone $\{\hat{v} = 0\}$. Thus, to construct more regular naked singularities, one must either deviate from solutions based on $\kappa$-self-similarity or give up on obtaining smallness away from $\{\hat{v}=0\}$.} This restricts the regularity of the solutions we construct. On a more technical level, this has the effect that $\left\vert\left\vert \Omega^{-2}\alpha \right\vert\right\vert_{L^2\left(\mathbb{S}^2_{-1,v}\right)}  \sim v^{-1+C\epsilon^2}$ as $v\to 0$ which, among other things, means that $\alpha_{AB}$ is not in $L^2_v$ and cannot be controlled directly by energy estimates. However, the most singular part of $\alpha_{AB}$ will behave in an approximately $\kappa$-self-similar manner, and we will be able to effectively subtract it off. 

Having determined our initial data, we may apply Theorem~\ref{localexistencecharbetter} to obtain the existence of a solution in a region as follows:

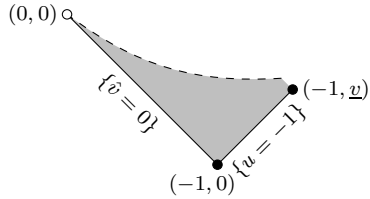
\begin{center}
\begin{tikzpicture}[scale = 1]
\fill[lightgray] (0,0)--(2,-2)--(3,-1)--(2.85,-.85)  to [bend left = 20] (0,0);
\draw [dashed] (2.85,-.85)to [bend left= 20]  (0,0) ;
\draw (0,0) -- (2,-2) node[sloped,below,midway]{\footnotesize $\{\hat{v} = 0\}$};
\draw (2,-2) -- (3,-1) node[sloped,below,midway]{\footnotesize $\{u = -1\}$};
\path [draw=black,fill=black] (2,-2) circle (1/16); 
\draw (1.8,-2) node[below]{\footnotesize $(-1,0)$};
\path [draw=black,fill=white] (0,0) circle (1/16); 
\draw (0,0) node[left]{\footnotesize $(0,0)$};
\path [draw=black,fill=black] (3,-1) circle (1/16); 
\draw (3,-1) node[right]{\footnotesize $(-1,\underline{v})$};
\end{tikzpicture}
\end{center}
Note that the region of local existence obtained by Theorem~\ref{localexistencecharbetter} degenerates as we approach the point $(u,\hat{v}) = (0,0)$ because the size of our incoming initial data is diverging as we approach $u = 0$ along $\{\hat{v} = 0\}$. In particular, we do not have any quantitative control of the curve represented by the dashed line as it approaches $(0,0)$.

\subsection{Bootstrap Argument for Region I}
In Section~\ref{bootforregone} we will carry out a bootstrap argument and eventually show that the solution constructed in the previous section may be extended to the region $\{0 \leq \frac{v}{-u} \leq \underline{v}\} \cap \{u \in (0,-1]\}$:
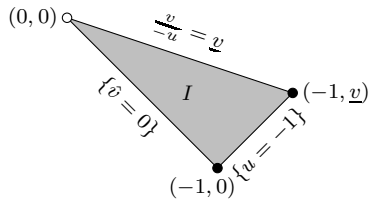
\begin{center}
\begin{tikzpicture}[scale = 1]
\fill[lightgray] (0,0)--(2,-2)--(3,-1)-- (0,0);
\draw (3,-1)-- (0,0) node[sloped,above,midway]{\footnotesize $\frac{v}{-u} = \underline{v}$};
\draw (0,0) -- (2,-2) node[sloped,below,midway]{\footnotesize $\{\hat{v} = 0\}$};
\draw (2,-2) -- (3,-1) node[sloped,below,midway]{\footnotesize $\{u = -1\}$};
\path [draw=black,fill=black] (2,-2) circle (1/16); 
\draw (1.8,-2) node[below]{\footnotesize $(-1,0)$};
\path [draw=black,fill=white] (0,0) circle (1/16); 
\draw (0,0) node[left]{\footnotesize $(0,0)$};
\path [draw=black,fill=black] (3,-1) circle (1/16); 
\draw (3,-1) node[right]{\footnotesize $(-1,\underline{v})$};
\draw (1.6,-1) node {\footnotesize $I$};
\end{tikzpicture}
\end{center}

Recalling that $0 < \epsilon \ll \underline{v} \ll 1$ and that $\Omega^{-1}\hat{\chi}_{AB}|_{v=0} \sim \epsilon^{-1}$, we see that existence up to the hypersurface $\frac{v}{-u} = \underline{v}$ may be a considered a ``semi-global'' existence result. The scheme we shall use to control the solution in this region is very close in spirit to the scheme used in the work~\cite{scaleinvariant}, and we refer the reader to Section 2 of~\cite{scaleinvariant} for an overview of how one carries out scale-invariant estimates for the curvature components and Ricci coefficients. Here we will simply note the key differences of this paper with the work~\cite{scaleinvariant}:
\begin{enumerate}
	\item As in~\cite{scaleinvariant}, we will not carry out energy estimates directly with null curvature components, but instead use certain renormalizations of them. Keeping Lemma~\ref{dkwowsss322} in mind, we will define 
	\[\widetilde{\underline{\alpha}}_{AB} \doteq \Omega^2\underline{\alpha}_{AB} - \lim_{v\to 0}\Omega^2\underline{\alpha}_{AB},\ \ \ \ \ \widetilde{\underline{\beta}}_A \doteq \Omega\underline{\beta}_A - \lim_{v\to 0}\Omega\underline{\beta}_A,\ \ \ \ \ \widetilde{\rho} \doteq \rho - \lim_{v\to 0}\rho,\ \ \ \ \ \widetilde{\sigma} \doteq \sigma - \lim_{v\to 0}\sigma.\]
	Here the limits are taken in a Lie-propagated frame. Other than the weighting by the lapse $\Omega$, this is analogous to the scheme from~\cite{scaleinvariant}. Since $0 < \epsilon \ll \underline{v}$, unlike in~\cite{scaleinvariant}, we cannot compensate for the largeness of $\hat{\chi}_{AB}$ (and hence $\beta_A$) with the smallness of $\frac{v}{-u}$. Instead we will explicitly subtract off a leading order self-similar ansatz for $\hat{\chi}_{AB}$ (and hence $\beta_A$).\footnote{We note, however, that an analogy may be drawn to the renormalizations of $\alpha$ carried out in~\cite{scaleinvariant} in the $n > 2$ case.} More specifically, we define $\overset{\triangleright}{\Omega^{-1}\hat{\chi}}_{AB}$ by self-similarly extending $\Omega^{-1}\hat{\chi}_{AB}|_{u=-1}$ to the whole spacetime, and then define $\overset{\triangleright}{\Omega^{-2}\alpha}_{AB}$ and $\overset{\triangleright}{\Omega^{-1}\beta}_A$ to be the parts of $\alpha_{AB}$ and $\beta_A$ which are sourced  by  $\overset{\triangleright}{\Omega^{-1}\hat{\chi}}_{AB}$ in~\eqref{4hatchi} and~\eqref{tcod1} respectively, that is,
	\[\overset{\triangleright}{\Omega^{-2}\alpha}_{AB} \doteq  -\left(\frac{v}{-u}\right)^{2\kappa}\left(\overline{\left(\frac{v}{-u}\right)^{-2\kappa}\Omega^{-2}}\right)\mathcal{L}_v\left(\overset{\triangleright}{\Omega^{-1}\hat{\chi}}\right),\]
	\[\overset{\triangleright}{\Omega^{-1}\beta}_A \doteq -\overline{\slashed{\rm div}}\left(\overset{\triangleright}{\Omega^{-1}\hat{\chi}}\right)_A - \overline{\eta}_B\left(\overset{\triangleright}{\Omega^{-1}\hat{\chi}}\right)_{AC}\slashed{g}^{BC},\]
	 Then we set
	\[\widetilde{\alpha}_{AB} \doteq \Omega^{-2}\alpha_{AB} - \overset{\triangleright}{\Omega^{-2}\alpha}_{AB},\qquad \widetilde{\beta}_A \doteq \Omega^{-1}\beta_A - \overset{\triangleright}{\Omega^{-1}\beta}_A- \lim_{v\to 0}\left[\frac{1}{2}\slashed{\nabla}\left(\Omega^{-1}{\rm tr}\chi\right)_A-\frac{1}{2} \eta_A\left(\Omega^{-1}{\rm tr}\chi\right)\right].\]
	We will find that $\widetilde{\alpha}_{AB}$ and $\widetilde{\beta}_A$ are both vanishing as $v\to 0$ and have size$~\epsilon$.
	\item The renormalization procedure for $\widetilde{\alpha}_{AB}$ and $\widetilde{\beta}_A$ produces inhomogeneous terms in the Bianchi equations which are large in $L^{\infty}$. However, when carrying out the energy estimates, these inhomogeneous terms are always integrated in $v$ and these integrals will be sufficiently small.
	\item The Ricci coefficient $\omega$ will satisfy, in general, $\Omega\omega \sim \epsilon^2 v^{-1}$. Because $v^{-1}$ is not integrable, we cannot treat such terms perturbatively. Thus, we will always multiply through with an appropriate power of the lapse $\Omega$ so as to remove $\omega$ from the equations. This elimination of $\omega$ in fact removes the need to have any estimates for $\omega$ in our bootstrap norm. The correct power of the lapse turns out be completely determined by signature considerations; that is, you simply multiply any Ricci coefficient or curvature component of signature $s$ by $\Omega^s$. 
	\item As opposed to the situation studied in~\cite{scaleinvariant}, the Ricci coefficients are, in general, non-vanishing as $v\to 0$. This will force us to work with renormalizations of the Ricci coefficients which are analogous to the renormalizations of the curvature components.
\end{enumerate}

\subsection{Bootstrap Argument for Region II}
In Section~\ref{ohmyregiontwo} we will carry out a bootstrap argument and eventually show that the solution constructed in the previous section may be extended to the region $\{\underline{v} \leq \frac{v}{-u} \leq \underline{v}^{-1}\} \cap \{u \in (0,-1]\}\cap \{v \leq \underline{v}\}$:
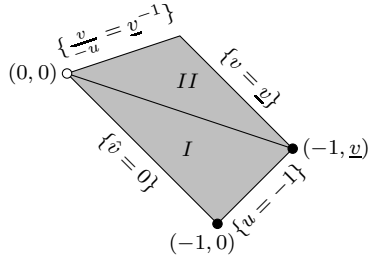
\begin{center}
\begin{tikzpicture}[scale = 1]
\fill[lightgray] (0,0)--(2,-2)--(3,-1)-- (3/2,1/2)-- (0,0);
\draw (3,-1)-- (0,0);
\draw (0,0) -- (3/2,1/2)node[sloped,above,midway]{\footnotesize $\{\frac{v}{-u} = \underline{v}^{-1}\}$} -- (3,-1) node[sloped,above,midway]{\footnotesize $\{v = \underline{v}\}$};
\draw (0,0) -- (2,-2) node[sloped,below,midway]{\footnotesize $\{\hat{v} = 0\}$};
\draw (2,-2) -- (3,-1) node[sloped,below,midway]{\footnotesize $\{u = -1\}$};
\path [draw=black,fill=black] (2,-2) circle (1/16); 
\draw (1.8,-2) node[below]{\footnotesize $(-1,0)$};
\path [draw=black,fill=white] (0,0) circle (1/16); 
\draw (0,0) node[left]{\footnotesize $(0,0)$};
\path [draw=black,fill=black] (3,-1) circle (1/16); 
\draw (3,-1) node[right]{\footnotesize $(-1,\underline{v})$};
\draw (1.6,-1) node {\footnotesize $I$};
\draw (1.6,-.1) node {\footnotesize $II$};

\end{tikzpicture}
\end{center}

In this region we will have that $u$ and $v$ are comparable. Furthermore, we will have shown in the previous section that the largeness of $\hat{\chi}_{AB}$ has dissipated by the time the hypersurface $\frac{v}{-u} = \underline{v}$ is reached. These two facts makes the analysis considerably simpler than that of region I. In particular, since the natural ``time'' variable is $\frac{v}{-u}$, this region may be considered to be a finite-in-time local existence result. The key technique is to use conjugation of various equations by $\exp\left(D\frac{v}{u}\right)$ where $1 \ll D \ll \epsilon^{-1}$ to generate lower order terms of good signs.
\subsection{Shifting the Shift  and Gluing in an Asymptotically Flat Cone}
In Section~\ref{secshift} we will carry out two preliminary changes to the solution constructed in Section~\ref{ohmyregiontwo}. 

First we will equip the portion of the spacetime covered by $\{1 \leq \frac{v}{-u} \leq \underline{v}^{-1}\}$ with a new double-null foliation where the shift vector is in the $e_4$-direction. (See Remark~\ref{letusshifthteshift}.) We briefly explain the reason for this: In the original double-null foliation, the shift vector $b^A$ satisfies the propagation equation
\begin{equation}\label{thispropskfowshift}
\mathcal{L}_{\partial_v}b^A = -4\Omega^2\zeta^A.
\end{equation}
From the analysis of Section~\ref{ohmyregiontwo} we will have that $\left|b|_{\frac{v}{-u} = \underline{v}^{-1}}\right| \sim \epsilon$ which implies that in the Lie-propagated coordinate frame, $b^A|_{\frac{v}{-u} = \underline{v}^{-1}} \sim \epsilon v^{-1}$. Now let $\tilde u < 0$ and consider a point $(\tilde u,\tilde v)$ to the future of $\frac{v}{-u} = \underline{v}^{-1}$. If we try to integrate~\eqref{thispropskfowshift} from $\frac{v}{-u} = \underline{v}^{-1}$ to control $b^A$, the best estimate we could possibly obtain is
\[\left|b\right| \lesssim \epsilon \frac{\tilde v}{-\tilde u},\]
which severely blows-up as $\tilde u \to 0$. (In principle, there could be cancellations with the right hand side of~\eqref{thispropskfowshift}, but we will not be able to exploit this.) However, once we have shifted the double-null foliation to put the shift in the $e_4$ direction, then instead of~\eqref{thispropskfowshift} we will have 
\begin{equation}\label{thispropskfowshift2}
\mathcal{L}_{\partial_u}b^A = 4\Omega^2\zeta^A.
\end{equation}
This equation will allow us to obtain the desired estimates for $b^A$ in a straightforward fashion. 

The second important change will be to glue in an asymptotically flat cone $\mathcal{H}$ along $\{u = -\underline{v}^2\}$ (which will thus be completely contained in the region $\{\frac{v}{-u} \geq \underline{v}^{-1}/2\}$), and apply a local existence result to extend our solution to the following region:
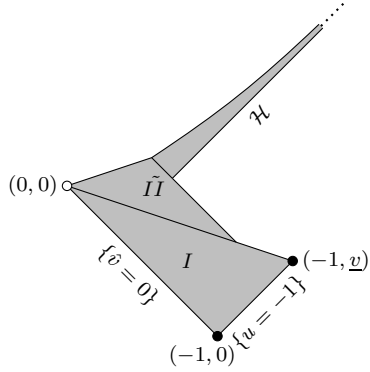
\begin{center}
\begin{tikzpicture}[scale = 1]

\fill[lightgray] (0,0)--(2,-2)--(3,-1)-- (2.25,-.75) -- (1.125,.375) -- (0,0);
\fill[lightgray] (1.125,.375)to [bend right= 5]  (3.35,2.15) -- (3.4,2.1) -- (1.4,.1);
\draw (3,-1)-- (0,0);
\draw (2.25,-.75) -- (1.4,.1);
\draw (0,0) -- (1.125,.375) -- (1.4,.1);
\draw (0,0) -- (2,-2) node[sloped,below,midway]{\footnotesize $\{\hat{v} = 0\}$};
\draw (2,-2) -- (3,-1) node[sloped,below,midway]{\footnotesize $\{u = -1\}$};
\path [draw=black,fill=black] (2,-2) circle (1/16); 
\draw (1.8,-2) node[below]{\footnotesize $(-1,0)$};
\path [draw=black,fill=white] (0,0) circle (1/16); 
\draw (0,0) node[left]{\footnotesize $(0,0)$};
\path [draw=black,fill=black] (3,-1) circle (1/16); 
\draw (3,-1) node[right]{\footnotesize $(-1,\underline{v})$};
\draw (1.6,-1) node {\footnotesize $I$};
\draw (1.15,0) node {\footnotesize $\tilde{II}$};
\draw (1.4,.1) -- (3.4,2.1) node[sloped,below,midway]{\footnotesize $\mathcal{H}$};
\draw [dotted,thick] (3.375,2.125) -- (3.675,2.425);
\draw  (1.125,.375)to [bend right= 5]  (3.35,2.15) ;

\end{tikzpicture}
\end{center}
We will pose data along $\mathcal{H}$ so that by a domain of dependence argument the solution in the region $I$ and $\tilde{II}$ agrees with the corresponding subset of the previous solution produced from Section~\ref{ohmyregiontwo}. We emphasize that at this step of the argument, we will not have quantitative control of the region of existence to the future of $\mathcal{H}$. 

\subsection{The Bootstrap Argument for Region III}
In Section~\ref{bottIIIIIIII} we will carry out a bootstrap argument and eventually show that the solution constructed in the previous section may be extended up to $\{u < 0\}$:
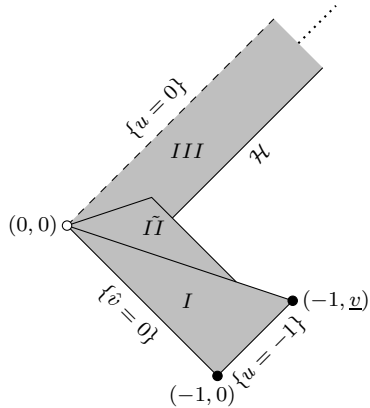
\begin{center}
\begin{tikzpicture}[scale = 1]
\fill[lightgray] (0,0)--(2,-2)--(3,-1)-- (2.25,-.75) -- (1.125,.375) -- (0,0);
\fill[lightgray] (0,0) -- (2.75,2.75) -- (3.4,2.1) -- (1.4,.1);
\draw (2.25,-.75) -- (1.4,.1);
\draw [dashed] (0,0) -- (2.75,2.75) node[sloped,above,midway]{\footnotesize $\{ u = 0\}$};
\draw (3,-1)-- (0,0);
\draw (0,0) -- (1.125,.375) -- (1.4,.1);
\draw (0,0) -- (2,-2) node[sloped,below,midway]{\footnotesize $\{\hat{v} = 0\}$};
\draw (2,-2) -- (3,-1) node[sloped,below,midway]{\footnotesize $\{u = -1\}$};
\path [draw=black,fill=black] (2,-2) circle (1/16); 
\draw (1.8,-2) node[below]{\footnotesize $(-1,0)$};
\path [draw=black,fill=white] (0,0) circle (1/16); 
\draw (0,0) node[left]{\footnotesize $(0,0)$};
\path [draw=black,fill=black] (3,-1) circle (1/16); 
\draw (3,-1) node[right]{\footnotesize $(-1,\underline{v})$};
\draw (1.6,-1) node {\footnotesize $I$};
\draw (1.15,0) node {\footnotesize $\tilde{II}$};
\draw (1.6,1) node {\footnotesize $III$};
\draw (1.4,.1) -- (3.4,2.1) node[sloped,below,midway]{\footnotesize $\mathcal{H}$};
\draw [dotted,thick] (3.075,2.425) -- (3.575,2.925);
\end{tikzpicture}
\end{center}

We now explain some of the key ideas for this bootstrap argument:
\subsubsection{Expected Bounds}
The basic expectation as we approach $u = 0$ for $\{v \lesssim 1\}$ is that the solution may be modeled by a $\kappa$-self-similar solution where the role of $\{v = 0\}$ is replaced by $\{u = 0\}$, and where $\kappa$ may acquire some angular dependence in the sense of Remark~\ref{moregeneral}. However, there is a very important difference with the analysis that we undertook along $\{v = 0\}$ and for $0 \leq \frac{v}{-u} \lesssim 1$: \emph{Since we will not require precise information about the $u\to 0$ limits of all of the double-null unknowns in order to close our bootstrap argument, it will turn out that we will not need to obtain precise information about the lapse $\Omega$, shear $\hat{\underline{\chi}}_{AB}$, or curvature component $\underline{\alpha}_{AB}$.}\footnote{The main reason for the differences in the study of the solution as $u\to 0$ versus $v\to 0$ is the following: Our various estimates will employ $|u|$ and $v$ weights. In order to generate lower order terms of a good sign, in these weights the power of $|u|$ will be generally be non-negative and the power of $v$ will be non-positive. In some of the situations when the power of the $v$-weight is strictly negative then we will need to subtract off the leading behavior as $v\to 0$ of our Ricci coefficient or curvature component so that our corresponding initial flux is finite. No analogous issue occurs in the region $u\to 0$.}

Due to the asymptotically flat cone $\mathcal{H}$, we do not expect the solution to be exactly modeled on a $\kappa$-self-similar solution when $v \gg 1$; nevertheless, we will propagate scale-invariant bounds. In particular, for all Ricci coefficients $\psi \neq \underline{\omega},\eta_A,\hat{\underline{\chi}}_{AB}$ and curvature components $\Psi \neq \underline{\alpha}_{AB}$, we expect the following bounds to hold in region $III$:
\begin{equation}\label{dkowkdpwkfko98762}
\left|\Omega^{-s}\psi^*\right| \lesssim \epsilon v^{-1},\qquad \left|\Omega^{-s}\Psi\right| \lesssim \epsilon v^{-2},
\end{equation}
where we let $s$ denote the signature of the Ricci coefficient $\psi$ or the curvature component $\Psi$, and $\psi^*$ denotes the difference of $\psi$ and its Minkowski value.  Note that in contrast to the situation in region $I$, we will want to eliminate $\underline{\omega}$ from our system, and it is thus natural to weight quantities with $\Omega^{-s}$ as opposed to $\Omega^s$. The lapse $\Omega$ will satisfy the following bound:
\begin{equation}\label{lolop1}
\left|\log\Omega\right| \lesssim \epsilon\left|\log\left(\frac{v}{-u}\right)\right|.
\end{equation}
(In principle, one expects to be able to replace $\epsilon$ with $\epsilon^2$ in~\eqref{lolop1} but we will not need this improvement and thus will not try to establish such a bound.)

In analogy with the behavior of $\alpha_{AB}$ near $\{v = 0\}$, for fixed $v$, we expect that $\underline{\alpha}_{AB}$ will be singular as $u \to 0$. In our analysis of region $I$, we posed explicitly the value of $\hat{\chi}_{AB}$ along $\{u = -1\}$, and we could thus use this value to effectively subtract off the singular behavior of $\alpha_{AB}$. As we approach $\{ u = 0\}$ it is less straightforward to regularize $\underline{\alpha}_{AB}$, and we will instead work with the renormalized Bianchi equations where $\underline{\alpha}_{AB}$ has been removed. (See~\eqref{ren1}-\eqref{ren6}.) Similarly, we will forgo explicitly renormalizing out the most singular self-similar behavior for $\hat{\underline{\chi}}_{AB}$. Instead, for some $0 < p \ll 1$, we will propagate the following bound:
\begin{equation}\label{lolop2}
\left|\Omega^{-1}\hat{\underline{\chi}}\right| \lesssim \epsilon\left(\frac{v}{-u}\right)^{p}v^{-1}.
\end{equation}
Note that this is weaker than the expected bounds
\[\left|\Omega^{-1}\hat{\underline{\chi}}\right| \lesssim {\rm min}\left(\left(1+\left|\log^k\left(\frac{-u}{v}\right)\right|\right)\epsilon,\epsilon^{-1}\right),\]
which are the true analogues of the bounds satisfied by $\hat{\chi}_{AB}$ near $\{\hat{v} = 0\}$. 

For $\underline{\omega}$, in analogy to the scheme used in region $I$, we will multiply through by a power of the lapse $\Omega$ which eliminates $\underline{\omega}$. Thus we will never need to estimate $\underline{\omega}$ in the context of the bootstrap argument. 

In contrast to region $I$, we will propagate estimates which are consistent with $\eta_A$ blowing up as $u \to 0$; more concretely, we will have, for some $0 < p \ll 1$, 
\begin{equation}\label{lolop3}
\left|\eta\right| \lesssim \epsilon\left(\frac{v}{-u}\right)^{p}v^{-1}.
\end{equation}
Conceptually, the reason we must allow for this is because of the potential existence of $\kappa$-self-similar solutions where $\kappa$ has angular dependence. (See Remark~\ref{moregeneral}.) If such a solution existed, then $\eta_A$ would blow-up logarithmically as $u\to 0$. 

\subsubsection{Energy Estimates}
We now explain the basic idea for our energy estimates. Since we are working with the renormalized Bianchi equations of~\cite{impulsive}, our ``first'' Bianchi pair is $(\underline{\beta}_A,\left(\rho,\sigma\right))$. We write the equations schematically as
\begin{align}\label{itmfokwdkod}
\Omega\nabla_4\left(\Omega \underline{\beta}\right)_A + \left(\Omega{\rm tr}\chi\right)\left(\Omega \underline{\beta}\right)_A &= -\slashed{\nabla}_A\check{\rho} + \left({}^*\slashed{\nabla}\right)_A\check{\sigma} + \cdots,
\\ \label{dkowdko} \Omega\nabla_3\check{\rho} &= - \slashed{\rm div}\left(\Omega \underline{\beta}\right) + \cdots,
\\ \label{adkldwko} \Omega\nabla_3\check{\sigma} &= \slashed{\rm div}{}^*\left(\Omega \underline{\beta}\right) + \cdots.
\end{align}
Written in this way, there are no appearences of $\underline{\omega}$ in the equations~\eqref{itmfokwdkod}-\eqref{adkldwko}. Next, we would like to conjugate these equations by a suitable weight function $w$, contract~\eqref{itmfokwdkod} with $w\Omega\underline{\beta}_A$, contract~\eqref{dkowdko} with $w\check{\rho}$, contract~\eqref{adkldwko} with $w\check{\sigma}$, and then integrate by parts. It will be a consequence of our  bootstrap assumptions that within region $III$:
\begin{equation}\label{kodkowadsellojhudw}
\left|\Omega{\rm tr}\chi - \frac{2}{v}\right| \lesssim \frac{\underline{v}}{v}.
\end{equation}
With this in mind, we let $0 < p \ll 1$, and set
\[w \doteq \left(\frac{-u}{v}\right)^pv^{-3/2}.\]
(In reality, to avoid certain logarithmic divergences, we use the weight $\tilde w \doteq \left(\frac{-u}{v}\right)^pv^{-3/2}(-u)^{-\delta}$ and then multiply the final estimate by $(-u)^{\delta}$, but we will suppress this point for the introduction.)

We obtain
\begin{align}\label{itmfokwdkod23}
\Omega\nabla_4\left(w\Omega \underline{\beta}\right)_A + \left(\left(\Omega{\rm tr}\chi\right)-\left(\frac{3}{2} + p\right)v^{-1}\right)\left(w\Omega \underline{\beta}\right)_A &= -\slashed{\nabla}_A\left(w\check{\rho}\right) + \left({}^*\slashed{\nabla}\right)_A\left(w\check{\sigma}\right) + w\left(\cdots\right)_A,
\\ \label{dkowdko23} \Omega\nabla_3\left(w\check{\rho}\right) + \frac{p}{-u}w\check{\rho} &= - \slashed{\rm div}\left(w\Omega \underline{\beta}\right) + w\left(\cdots\right),
\\ \label{adkldwko23} \Omega\nabla_3\left(w\check{\sigma}\right) + \frac{p}{-u}w\check{\sigma} &= \slashed{\rm div}{}^*\left(w\Omega \underline{\beta}\right) + w\left(\cdots\right).\end{align}
Using~\eqref{kodkowadsellojhudw}, we see that all of the lower order terms produced by scheme are positive; furthermore, in the $\check{\rho}$ and $\check{\sigma}$ equations, the lower order is proportional to $u^{-1}$ which is a good weight since we are in a region where $\frac{-u}{v} \ll 1$. 

A similar scheme is used for the rest of the Bianchi pairs. Note that the $(-u)^{-1}$ weights in the lower order terms will produce very good spacetime estimates for all curvature components except for $\underline{\beta}_A$; in particular, the lower order term proportional to ${\rm tr}\chi$ is only important in the analysis of $\underline{\beta}_A$'s equation. 

For the control of the nonlinear terms the key observations are the absence of  nonlinear terms which involve contractions of $\underline{\beta}_A$ and $\hat{\underline{\chi}}_{AB}$ and the absence of any nonlinear term with $\eta_A$ in~\eqref{ren6}. Finally, in order to apply Sobolev inequalities, we will also need to commute with angular derivatives; these commutations will be carried out in a way which avoids the creation of terms containing $\eta_A$ in the $\nabla_4$ equations (see~\eqref{com1}).

\subsubsection{Integrating Transport Equations}
We close this sketch of our scheme for the bootstrap argument with a discussion about integrating transport equations. For every double-null quantity $y$, other than $\Omega$, $\underline{\omega}$, $\hat{\underline{\chi}}_{AB}$, $\underline{\beta}_A$, $\eta_A$, and $\underline{\alpha}_{AB}$ we will have a $\nabla_3$ equation of the form
\[\nabla_3y + a\underline{\omega} y = \cdots.\]
Here $a \in \mathbb{R}$ is a suitable constant, which may be determined by the signature of $y$. Letting $s$ denote the signature of $y$, we may derive from this equation
\[\nabla_u\left(\Omega^{-s}y\right) = \Omega^2\left(\cdots\right),\]
where the terms in the $\cdots$ all are expected to have regular limits as $u\to 0$. Since $\Omega^2$ blows-up slowly as $u\to 0$, we can easily integrate these equations in the $u$-direction to obtain good estimates for $y$. Note that even if the only estimates available for the terms in the $\cdots$ blow-up as $u\to 0$, as long as the rate is integrable, the estimate will still close. This is why it will not be a problem that our energy estimates involve weights which degenerate as $u\to 0$. 

For the estimates of $\Omega$, $\hat{\underline{\chi}}_{AB}$, $\eta_A$, and $\underline{\beta}_A$ we will only have access to $\nabla_4$ equations. (Note that in our bootstrap argument we will not estimate $\underline{\omega}$ or $\underline{\alpha}_{AB}$.) More specifically, we will have that $\left(\Omega,\hat{\underline{\chi}}_{AB},\eta_A\right)$ satisfy equations of the following forms:
\[\Omega\nabla_4\log\Omega = O\left(\epsilon v^{-1}\right),\qquad \Omega\nabla_4\left(\Omega^{-1}\hat{\underline{\chi}}_{AB},\eta_A\right)+ \left(\frac{1+O\left(\epsilon\right)}{v}\right)\left(\Omega^{-1}\hat{\underline{\chi}}_{AB},\eta_A\right) = O\left(\epsilon v^{-2}\right).\]
Integrating these from the hypersurface $\frac{v}{-u} = \underline{v}^{-1}$, it is immediately clear that we cannot expect to show that these quantities are bounded as $u\to 0$ (cf.~the discussion of the equation~\eqref{thispropskfowshift} above). Instead we will close the bootstrap argument with the bounds~\eqref{lolop1},~\eqref{lolop2}, and~\eqref{lolop3}.

\subsection{Incompleteness of Future Null Infinity, the $\left(u,\hat{v},\theta^A\right)$ Coordinates, and the Hawking Mass}  
In Section~\ref{incompsec} we will first truncate our solution to region to $\{-\underline{v}^2 \leq u < 0\}$ so that we obtain a globally hyperbolic region:
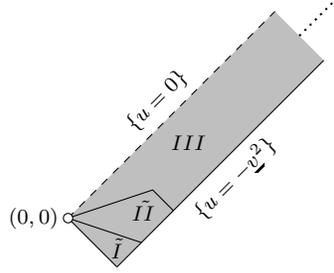
\begin{center}
\begin{tikzpicture}[scale = 1]
\fill[lightgray] (0,0) -- (.65,-.65) -- (3.4,2.1) -- (2.75,2.75)  -- (0,0);
\draw [dashed] (0,0) -- (2.75,2.75) node[sloped,above,midway]{\footnotesize $\{ u = 0\}$};
\draw (.975,-.325)-- (0,0);
\draw (0,0) -- (1.125,.375) -- (1.4,.1);
\draw (0,0) -- (.65,-.65);
\path [draw=black,fill=white] (0,0) circle (1/16); 
\draw (0,0) node[left]{\footnotesize $(0,0)$};
\draw (.65,-.4) node {\footnotesize $\tilde{I}$};
\draw (1,.1) node {\footnotesize $\tilde{II}$};
\draw (1.6,1) node {\footnotesize $III$};
\draw (.65,-.65) -- (3.4,2.1) node[sloped,below,midway]{\footnotesize $\{u = -\underline{v}^2\}$};
\draw [dotted,thick] (3.075,2.425) -- (3.575,2.925);
\end{tikzpicture}
\end{center}

Then we will show the solution obtain has an incomplete future null infinity in the sense of Definition~\ref{nakeddef} (this will be straightforward given the estimates we will have already established), we will define global $\left(u,\hat{v},\theta\right)$ coordinates by setting $\hat{v} = v^{1-2\kappa}$ (see Definition~\ref{kapselfseim}), and finally conclude the proof of Theorem~\ref{themainextresult} by computing the Hawking mass of each sphere $\mathbb{S}^2_{u,0}$ to establish~\eqref{masskfow}.

\section{Degenerate Transport Equations on $\mathbb{S}^2$}\label{pdetheoryforformal}
In this section we will establish existence results and a priori estimates for various classes of linear and nonlinear PDE's which will show up in the context of setting up our characteristic initial data.

We let $\mathcal{T}^{(s,k)}\left(\mathbb{S}^2\right)$ denote the space of $(s,k)$-tensors on $\mathbb{S}^2$ and $\hat{\mathcal{S}}\left(\mathbb{S}^2\right)$ will denote the space of symmetric trace-free $(0,2)$-tensors. Using the covariant derivative for the round sphere, we may define the corresponding Sobolev spaces $\mathring{H}^j$. A metric $\slashed{g}_{AB} \in \mathring{H}^{{\rm max}\left(2,j\right)}$ also allows us to the Sobolev spaces $H^j\left(\mathcal{T}^{(s,k)}\left(\mathbb{S}^2\right)\right)$ for $j = 0,1,\cdots$ by
\[\left\vert\left\vert u\right\vert\right\vert_{H^j} \doteq \sum_{i=0}^j\left\vert\left\vert \slashed{\nabla}^iu\right\vert\right\vert_{L^2},\]
where, in general, we let $\slashed{\nabla}_A$ denote the covariant derivative associated to a metric $\slashed{g}_{AB}$. We will denote the covariant derivative corresponding to the round sphere by $\mathring{\nabla}_A$. We will write integrals with $\slashed{dVol}$ and $\mathring{dVol}$ to denote integration against the volume forms of $\slashed{g}_{AB}$ and $\mathring{\slashed{g}}_{AB}$ respectively, and $\slashed{\rm div}$ and $\mathring{\rm div}$ to denote the divergence operators of $\slashed{g}_{AB}$ and $\mathring{\slashed{g}}_{AB}$ respectively. 

We collect various facts about Sobolev spaces inequalities that we will use later in the following lemmas. We start with two basic Sobolev inequalities on $\mathbb{S}^2$.
\begin{lemma}\label{soblemm}On the round sphere, we have, for any tensor $u$, 
\begin{equation}\label{sobolevonS2}
\left\vert\left\vert u\right\vert\right\vert_{L^p} \lesssim_p \left\vert\left\vert u\right\vert\right\vert_{L^2}^{\frac{2}{p}}\left\vert\left\vert \mathring{\nabla} u\right\vert\right\vert^{1-\frac{2}{p}}_{\mathring{H}^1}\qquad \forall p < \infty,\qquad \left\vert\left\vert u\right\vert\right\vert_{L^{\infty}} \lesssim \left\vert\left\vert u\right\vert\right\vert_{L^2}^{\frac{1}{2}}\left\vert\left\vert u\right\vert\right\vert_{\mathring{H}^2}^{\frac{1}{2}}.
\end{equation}
\end{lemma}
\begin{proof}This simply follows by applying Euclidean Sobolev inequalities in suitable coordinate charts.
\end{proof}

It will be useful to compare Sobolev spaces generated by various metrics $\slashed{g}_{AB}$ and those generated by the round metric $\mathring{\slashed{g}}_{AB}$. 
\begin{lemma}\label{comparethespaces}Let $i \geq 2$, $k \geq 0$, and $\slashed{g}_{AB}$ be a Riemannian metric on $\mathbb{S}^2$ satisfying at least one of the following two assumptions:
\begin{enumerate}
	\item We have $\slashed{g}_{AB} = e^{2\varphi}\mathring{\slashed{g}}_{AB}$ for $\varphi \in \mathring{H}^i\left(\mathbb{S}^2\right)$.
	\item We have $\left\vert\left\vert \slashed{g} - \mathring{\slashed{g}}\right\vert\right\vert_{\mathring{H}^i} \leq C(i,k)$, for a suitably small constant $C(i,k)$ which depends only on $i$ and $k$.
\end{enumerate}
Then we have that for every $(0,k)$-tensor $w$
\[\left\vert\left\vert w\right\vert\right\vert_{H^i} \sim_{i,k} \left\vert\left\vert w\right\vert\right\vert_{\mathring{H}^i}.\]
\end{lemma}
\begin{proof}This is a straightforward consequence of Lemma~\ref{soblemm}.
\end{proof}

The following well-known lemma is also useful.
\begin{lemma}\label{interlemm}Let $w_1$ and $w_2$ be tensors on $\mathbb{S}^2$. Then, for $i \geq 2$, we have
\begin{equation}\label{genformoos}
\left\vert\left\vert w_1w_2\right\vert\right\vert_{\mathring{H}^i} \lesssim \left\vert\left\vert w_1\right\vert\right\vert_{\mathring{H}^i}\left\vert\left\vert w_2\right\vert\right\vert_{\mathring{H}^i}.
\end{equation}
\end{lemma}

Lastly, we record the formula for the commutator between a Lie derivative and a covariant derivative.
\begin{lemma}\label{whenaliederiavtivemeetsacovariantderivativethingscanhappen}Let $\slashed{g}_{AB}$ be a Riemannian metric on $\mathbb{S}^2$, $X^A$ be a vector field on $\mathbb{S}^2$, and  $H_{B_1\cdots B_k}$ be a tensor on $\mathbb{S}^2$. Then we have
\[\left[\slashed{\nabla}_A,\mathcal{L}_X\right]H_{B_1\cdots B_k} = \sum_{i=1}^k{}^{(X)}\Gamma_{B_iAC}H_{B_1\cdots\ \cdots B_k}^{\ \ \ \ C},\]
where
\[{}^{(X)}\Gamma_{ABC} = \frac{1}{2}\left(\slashed{\nabla}_A{}^{(X)}\pi_{BC} + \slashed{\nabla}_B{}^{(X)}\pi_{AC} - \slashed{\nabla}_C{}^{(X)}\pi_{AB}\right),\]
and ${}^{(X)}\pi$ denotes the deformation tensor of $X$. 
\end{lemma}
\begin{proof}This is Lemma 7.1.3 in~\cite{CK}.
\end{proof}

\subsection{First Order Perturbations of the Identity}\label{perturbtheidentityforever}

\begin{definition}\label{KuF}Let $X^A$ be a $C^1$ vector field on $\mathbb{S}^2$ and let $h = \sum_{i=0}^k c_i\left(h^{(i)}\right)_{A_1\dots A_i}^{\ \ \ \ \ \ \ B_1\cdots B_i}$ denote a linear combination of tensors $h^{(i)} \in \mathcal{T}^{(i,i)}$ with each $h^{(i)}$ continuous. Then we define a differential operator $P$ acting on tensors $u_{A_1\cdots A_k} \in C^1\left(\mathcal{T}^{(0,k)}\left(\mathbb{S}^2\right)\right)$ by
\[Pu_{A_1\cdots A_k} \doteq \left(u + \mathcal{L}_Xu + h\cdot u\right)_{A_1\cdots A_k},\]
where we do not specify which indices the contraction $\left(h\cdot u\right)_{A_1\cdots A_k}$ is taken with respect to. (Note, however, that this will always be a $k$-tensor, e.g., $h_{A_1\cdots A_i}^{\ \ \ \ \ \ \ B_1\cdots B_i}u_{C_1\cdots C_{k-i} B_1\cdots B_i} \in \mathcal{T}^{(0,k)}$.) 
\end{definition}

Our goal in this section will be to establish a theory which yields existence and uniqueness results and a priori estimates for solutions $u$ to
\begin{equation}\label{thebasicequation}
Pu_{A_1\cdots A_k} = F_{A_1\cdots A_k},
\end{equation}
whenever $X^A$ and $h$ satisfy suitable regularity and smallness assumptions. Ultimately, because of the smallness assumptions, we will be able to treat the operator $P$ as a perturbation of the identity. Let $M \geq 0$ be a  non-negative integer. Then we define
\[A_M \doteq \sup_{0 \leq j \leq M+1}\left\vert\left\vert \mathring{\nabla}^jX\right\vert\right\vert_{L^2\left(\mathbb{S}^2\right)} + \sup_{0 \leq j \leq M}\left\vert\left\vert \mathring{\nabla}^jh\right\vert\right\vert_{L^2\left(\mathbb{S}^2\right)},\]
where the $L^2$ spaces are defined with respect to the round metric. We will always work with $X$ and $h$ which satisfy $A_2 < \infty$. Note that, by Lemma~\ref{soblemm}, we will thus have $\left\vert\left\vert \mathring{\nabla} X\right\vert\right\vert_{L^{\infty}} + \left\vert\left\vert h\right\vert\right\vert_{L^{\infty}} \lesssim A_2$.

The main result of this section will be the following.
\begin{proposition}\label{qualitativeexistence}
Let $M \geq 2 $ be a positive integer. Then, if $F_{A_1\cdots A_k} \in \mathring{H}^M\left(\mathcal{T}^{(0,k)}\left(\mathbb{S}^2\right)\right)$ and  $A_2$ is suitably small, depending on $(0,k)$ and $M$, and $A_M < \infty$, then there exists $u_{A_1\cdots A_k} \in \mathring{H}^M\left(\mathcal{T}^{(0,k)}\left(\mathbb{S}^2\right)\right)$ solving
\begin{equation}\label{theequationwesolveinthisparticularprop}
Pu_{A_1\cdots A_k} = F_{A_1\cdots A_k},
\end{equation}
and such that moreover $\mathcal{L}_Xu_{A_1\cdots A_k} \in \mathring{H}^M\left(\mathcal{T}^{(0,k)}\left(\mathbb{S}^2\right)\right)$. We also have that $u_{A_1\cdots A_k}$ satisfies the estimate
\[\left\vert\left\vert u\right\vert\right\vert_{\mathring{H}^M} + \left\vert\left\vert \mathcal{L}_Xu\right\vert\right\vert_{\mathring{H}^M} \lesssim_{k,r,M} \left(A_M+1\right)\left\vert\left\vert F\right\vert\right\vert_{\mathring{H}^M}.\]
Finally, if $w_{A_1\cdots A_k}$ is another solution to~\eqref{theequationwesolveinthisparticularprop} with the same right hand side $F_{A_1\cdots A_k}$, $w_{A_1\cdots A_k} \in L^2$, and $\mathcal{L}_Xw_{A_1\cdots A_k} \in L^2$, then we must have that $w_{A_1\cdots A_k} = u_{A_1\cdots A_k}$. 
\end{proposition}

In order to prove Proposition~\ref{qualitativeexistence} we will introduce an elliptic regularization of $P$:
\begin{definition}\label{regularizeK}For every $q > 0$ we define the operator 
\[P^{(q)} \doteq P - q\mathring{\Delta},\]
where $\mathring{\Delta}$ is the Laplace-Beltrami operator associated to the round metric $\mathring{\slashed{g}}_{AB}$.
\end{definition}

Since the operator $P^{(q)}$ is elliptic, it is straightforward to establish existence and uniqueness for solutions to $P^{(q)}u^{(q)}_{A_1\cdots A_k} = F_{A_1\cdots A_k}$.
\begin{proposition}\label{qregexists}Let $M \geq 2$ be a positive integer, $q > 0$, $F_{A_1\cdots A_k}  \in \mathring{H}^M\left(\mathcal{T}^{(0,k)}\left(\mathbb{S}^2\right)\right)$, $A_M < \infty$, and $A_2$ be sufficiently small independently of $q$. Then there exists a unique $u^{(q)}_{A_1\cdots A_k} \in \mathring{H}^{M+2}\left(\mathcal{T}^{(0,k)}\left(\mathbb{S}^2\right)\right)$ which solves
\begin{equation}\label{qregueqn}
P^{(q)}u^{(q)}_{A_1\cdots A_k} = F_{A_1\cdots A_k}.
\end{equation} 
\end{proposition}
\begin{proof}

We start by showing that ${\rm ker}\left(P^{(q)}\right) = 0$. Indeed, suppose that $w_{A_1\cdots A_k} \in \mathring{H}^1\left(\mathcal{T}^{(0,k)}\left(\mathbb{S}^2\right)\right)$ is a weak-solution to 
\begin{equation}\label{kernelKq}
\left(w + \mathcal{L}_Xw + h\cdot w - q\slashed{\Delta}w\right)_{A_1\cdots A_k} = 0.
\end{equation}
Note that the divergence theorem, a straightforward integration by parts, and the Sobolev inequality from Lemma~\ref{soblemm} yields that
\begin{align*}
\int_{\mathbb{S}^2}\left\langle \mathcal{L}_Xw,w\right\rangle_{\mathring{\slashed{g}}} \mathring{dVol} &\geq -C\left\vert\left\vert \mathcal{L}_X\mathring{\slashed{g}}\right\vert\right\vert_{L^{\infty}}\left\vert\left\vert w\right\vert\right\vert^2_{L^2}
\\ \nonumber &\geq -CA_2\left\vert\left\vert w\right\vert\right\vert^2_{L^2},
\end{align*}
for some constant $C$ which just depends on $k$. Thus, taking the $\mathring{\slashed{g}}_{AB}$ inner product of~\eqref{kernelKq} with $w_{A_1\cdots A_k}$, integrating over $\mathbb{S}^2$, and integrating by parts leads to the following identity: 
\begin{equation}\label{basicestwhichshowsnoker}
\int_{\mathbb{S}^2}\left[\left(1-CA_2-\left\vert\left\vert h\right\vert\right\vert_{L^{\infty}}\right)\left|w\right|^2 + q\left|\slashed{\nabla}w\right|^2\right]\mathring{dVol} \leq 0.
\end{equation}
Thus, $A_2$ suitably small (and the Sobolev inequality~\eqref{sobolevonS2}) implies that $w_{A_1\cdots A_k} = 0$. The adjoint of $P^{(q)}$ is clearly of the same essential form as $P^{(q)}$, and thus the same integration by parts identity also implies that ${\rm ker}\left(\left(P^{(q)}\right)^*\right) = 0$. Therefore, by standard $L^2$-elliptic theory, given $F_{A_1\cdots A_k}  \in \mathring{H}^M\left(\mathcal{T}^{(0,k)}\left(\mathbb{S}^2\right)\right)$ there exists a unique $u^{(q)}_{A_1\cdots A_k} \in \mathring{H}^{M+2}\left(\mathcal{T}^{(0,k)}\left(\mathbb{S}^2\right)\right)$ solving~\eqref{qregueqn}. 
\end{proof}

Now we turn to proof of Proposition~\ref{qualitativeexistence}
\begin{proof}Let us fix some $k$ throughout the proof and allow all constants to depend on them. For each $q > 0$,  we may appeal to Proposition~\ref{qregexists} to  produce a solution $u^{(q)}_{A_1\cdots A_k}$ to the equation~\eqref{qregueqn}. Our plan will be to show that there exists $u_{A_1\cdots A_k} = \lim_{q\to 0}u^{(q)}_{A_1\cdots A_k}$ which solves~\eqref{theequationwesolveinthisparticularprop} and satisfies the desired estimates. 

We start by establishing estimates for $u^{(q)}_{A_1\cdots A_k}$ which are uniform as $q \to 0$. Repeating the integration by parts which lead to the identity~\eqref{basicestwhichshowsnoker} now establishes the basic estimate
\begin{equation}\label{basicestforKq}
\int_{\mathbb{S}^2}\left(\left|u^{(q)}\right|^2 + q\left|\mathring{\nabla}u^{(q)}\right|^2\right) \mathring{dVol}\lesssim \int_{\mathbb{S}^2}\left|F\right|^2.
\end{equation}
Next, we observe that an integration by parts and Lemma~\ref{whenaliederiavtivemeetsacovariantderivativethingscanhappen} establishes the following inequality
\begin{align*}
q\int_{\mathbb{S}^2}\left\langle \mathring{\Delta}u^{(q)},\mathcal{L}_Xu^{(q)}\right\rangle_{\mathring{\slashed{g}}} \mathring{dVol} &\geq -Cq\int_{\mathbb{S}^2}\left[\left|\mathcal{L}_X\mathring{\slashed{g}}\right|\left|\mathring{\nabla}u^{(q)}\right|^2 + \left|\mathring{\nabla}\mathcal{L}_X\mathring{\slashed{g}}\right|\left|\mathring{\nabla}u^{(q)}\right|\left|u^{(q)}\right|\right]\mathring{dVol}
\\ \nonumber &\geq -Cq\int_{\mathbb{S}^2}\left[(A_2+A_3)\left|\mathring{\nabla}u^{(q)}\right|^2 + A_3\left|u^{(q)}\right|^2\right]\mathring{dVol}.
\end{align*}
In particular, we can contract~\eqref{qregueqn} with $\mathcal{L}_Xu^{(q)}_{A_1\cdots A_k}$, integrate by parts, and add the result to a suitably large constant times the estimate~\eqref{basicestforKq}, choose $q$ so that $qA_3 \ll 1$, and then establish that 
\begin{equation}\label{basicestforKq2}
\int_{\mathbb{S}^2}\left(\left|u^{(q)}\right|^2 + q\left|\mathring{\nabla}u^{(q)}\right|^2 + \left|\mathcal{L}_Xu^{(q)}\right|^2\right)\mathring{dVol} \lesssim \int_{\mathbb{S}^2}\left|F\right|^2\mathring{dVol}.
\end{equation}
For higher order estimates we will need to differentiate the equation. Commuting through by $\mathring{\nabla}^M$ produces the following equation
\begin{align}\label{KqAftertheCommute}
&\Big(\mathring{\nabla}^Mu^{(q)} + \mathcal{L}_X\left(\mathring{\nabla}^Mu^{(q)}\right) - q\left(\mathring{\Delta}\left(\mathring{\nabla}^Mu^{(q)}\right) + \left[\mathring{\nabla}^M,\mathring{\Delta}\right]u^{(q)}\right) 
\\ \nonumber &\qquad+ \left[\mathring{\nabla}^M,\mathcal{L}_X\right]u^{(q)} + \mathring{\nabla}^M\left(h\cdot u^{(q)}\right)\Big)_{B_1\cdots B_M A_1\cdots A_k} 
 = \mathring{\nabla}^MF_{B_1\cdots B_M A_1\cdots A_k}.
\end{align}
Let's examine more closely the commutator terms. We have, using Lemma~\ref{whenaliederiavtivemeetsacovariantderivativethingscanhappen},
\begin{align}\label{commutateerror1dddd}
\left[\mathring{\nabla}^M,\mathcal{L}_X\right]u^{(q)}_{B_1\cdots B_MA_1\cdots A_k}  &= \sum_{i=0}^{M-1}\left(O\left(\mathring{\nabla}^{M-i}\mathcal{L}_X\mathring{\slashed{g}}\right)\cdot \mathring{\nabla}^iu^{(q)}\right)_{B_1\cdots B_MA_1\cdots A_k}.
\end{align}
From the definition of the curvature tensor, we also have 
\begin{align}\label{commutateerror1dddd2}
\left|\left[\mathring{\nabla}^M,\mathring{\Delta}\right]u^{(q)}\right|  \lesssim_M \sum_{i=0}^M\left|\mathring{\nabla}^iu^{(q)}\right|.
\end{align}

Now we contract~\eqref{KqAftertheCommute} with $\mathring{\nabla}^Mu^{(q)}_{B_1\cdots B_MA_1\cdots A_k}$ and integrate by parts as before. We end up with the estimate (using only that $A_2$ is sufficiently small)
\begin{align}\label{thebasichigherorderestimate}
&\int_{\mathbb{S}^2}\left(\left|\mathring{\nabla}^Mu^{(q)}\right|^2 + q\left|\mathring{\nabla}^{M+1}u^{(q)}\right|^2 + \left|\mathcal{L}_X\left(\mathring{\nabla}^Mu^{(q)}\right)\right|^2\right)\mathring{dVol} \lesssim_M
\\ \nonumber &\qquad  \int_{\mathbb{S}^2}\left[\left|\mathring{\nabla}^MF\right|^2 + q\left|\left[\mathring{\nabla}^M,\mathring{\Delta}\right]u^{(q)}\right|^2 + \left|\left[\mathring{\nabla}^M,\mathcal{L}_X\right]u^{(q)}\right|^2 + \left|\mathring{\nabla}^M\left(h\cdot u^{(q)}\right)\right|^2\right]\mathring{dVol}.
\end{align}
Next we will examine the various terms on the right hand side of~\eqref{thebasichigherorderestimate}. Using~\eqref{commutateerror1dddd2} and  an interpolation inequality, for every $0 < p \ll 1$, we may easily establish that
\begin{align}\label{termthroughlaplcomm}
q\int_{\mathbb{S}^2}\left|\left[\mathring{\nabla}^M,\mathring{\Delta}\right]u^{(q)}\right|^2\mathring{dVol} &\lesssim_M q\left\vert\left\vert u^{(q)}\right\vert\right\vert^2_{\mathring{H}^M}
\\ \nonumber &\lesssim_M q\left[p^{-M-1}\left\vert\left\vert u^{(q)}\right\vert\right\vert^2_{L^2} + p\left\vert\left\vert u^{(q)}\right\vert\right\vert^2_{\mathring{H}^{M+1}}\right].
\end{align}
Taking $p$ sufficiently small we can thus combine~\eqref{termthroughlaplcomm} with~\eqref{basicestforKq2} and~\eqref{thebasichigherorderestimate} to establish
\begin{align}\label{thebasichigherorderestimate2}
&\int_{\mathbb{S}^2}\left(\left|\mathring{\nabla}^Mu^{(q)}\right|^2 + q\left|\mathring{\nabla}^{M+1}u^{(q)}\right|^2 + \left|\mathcal{L}_X\left(\mathring{\nabla}^Mu^{(q)}\right)\right|^2\right) \lesssim_M
\\ \nonumber &\qquad  \int_{\mathbb{S}^2}\left[\left|\mathring{\nabla}^MF\right|^2 +\left|F\right|^2+ \left|\left[\mathring{\nabla}^M,\mathcal{L}_X\right]u^{(q)}\right|^2 + \left|\mathring{\nabla}^M\left(h\cdot u^{(q)}\right)\right|^2\right].
\end{align}
Next we turn to the term $\left|\left[\mathring{\nabla}^M,\mathcal{L}_X\right]u^{(q)}\right|^2$. Using~\eqref{commutateerror1dddd} and interpolation we have
\begin{align}\label{tocontrolthatcommuttermwithslashandlie}
\int_{\mathbb{S}^2}\left|\left[\mathring{\nabla}^M,\mathcal{L}_X\right]u^{(q)}\right|^2 \mathring{dVol}&\lesssim_M \sum_{i=0}^{M-1}\int_{\mathbb{S}^2}\left|\mathring{\nabla}^{M+1-i}X\right|^2\left|\mathring{\nabla}^iu^{(q)}\right|^2
\\ \nonumber &\lesssim_M A_M\int_{\mathbb{S}^2}\left[\left|\mathring{\nabla}^{M-1}u^{(q)}\right|^2 + \left|u^{(q)}\right|^2\right]\mathring{dVol}.
\end{align}
Similarly, one may establish that
\begin{equation}\label{tocontrolthatcommuttermwithslashandlie2}
\int_{\mathbb{S}^2} \left|\mathring{\nabla}^M\left(h\cdot u^{(q)}\right)\right|^2\mathring{dVol} \lesssim A_2\int_{\mathbb{S}^2}\left|\mathring{\nabla}u^{(q)}\right|^2\mathring{dVol} +  C(M)A_M\int_{\mathbb{S}^2}\left[\left|\mathring{\nabla}^{M-1}u^{(q)}\right|^2 + \left|u^{(q)}\right|^2\right]\mathring{dVol}.
\end{equation}
Combining~\eqref{tocontrolthatcommuttermwithslashandlie} and~\eqref{tocontrolthatcommuttermwithslashandlie2} with~\eqref{thebasichigherorderestimate2} and~\eqref{basicestforKq2} and carrying out a straightforward induction argument, we obtain the desired uniform estimate
\begin{equation}\label{thatuniformestinqthatwesodespwanted}
\left\vert\left\vert u^{(q)}\right\vert\right\vert_{\mathring{H}^M} + \left\vert\left\vert \mathcal{L}_Xu^{(q)}\right\vert\right\vert_{\mathring{H}^M} \lesssim_M (A_M+1)\left\vert\left\vert F\right\vert\right\vert_{\mathring{H}^M}.
\end{equation}
(Note that we have simply dropped the higher order term on the left hand side multiplied by $q$.)

Now we turn a study of the limit of the $u^{(q)}_{A_1\cdots A_k}$ as $q\to 0$. Let $0 < q_1 < q_2$. We may easily derive 
\begin{align}\label{differenceoftwoqus}
&\left(\left(u^{(q_2)}-u^{(q_1)}\right) + \mathcal{L}_X\left(u^{(q_2)}-u^{(q_1)}\right) + h\cdot\left(u^{(q_2)}-u^{(q_1)}\right) - q_2\mathring{\Delta}\left(u^{(q_2)}-u^{(q_1)}\right) \right)_{A_1\cdots A_k}
\\ \nonumber &\qquad = \left(q_2-q_1\right)\mathring{\Delta}u^{(q_1)}_{A_1\cdots A_k}.
\end{align}
Contracting~\eqref{differenceoftwoqus} with a suitable linear combination of $\left(u^{(q_2)}-u^{(q_1)}\right)_{A_1\cdots A_k}$ and $\mathcal{L}_X\left(u^{(q_2)}-u^{(q_1)}\right)_{A_1\cdots A_k}$ and integrating by parts as we have done above then yields
\begin{equation}\label{cauchyinl2asq0}
\left\vert\left\vert u^{(q_2)}-u^{(q_1)}\right\vert\right\vert_{L^2} + \left\vert\left\vert \mathcal{L}_X\left(u^{(q_2)}-u^{(q_1)}\right)\right\vert\right\vert_{L^2} \lesssim  \left|q_2-q_1\right|\left\vert\left\vert \mathring{\Delta}u^{(q_1)}\right\vert\right\vert_{L^2} \lesssim \left(1+A_3\right)\left|q_2-q_1\right|\left\vert\left\vert F\right\vert\right\vert_{\mathring{H}^2}.
\end{equation}
In the last inequality we have used~\eqref{thatuniformestinqthatwesodespwanted}. In particular, $u^{(q)}_{A_1\cdots A_k}$ and $\mathcal{L}_Xu^{(q)}_{A_1\cdots A_k}$ form Cauchy sequences as $q\to 0$ in $L^2$. Moreover, by interpolating with~\eqref{thatuniformestinqthatwesodespwanted} we have that both $u^{(q)}_{A_1\cdots A_k}$ and $\mathcal{L}_Xu^{(q)}_{A_1\cdots A_k}$ form Cauchy sequences in $\mathring{H}^s$ for any $s < M$. Let $u_{A_1\cdots A_k} \in \mathring{H}^{M-1}$ denote $\lim_{q\to 0}u^{(q)}_{A_1\cdots A_k}$. We clearly then have that $u_{A_1\cdots A_k}$ solves the desired equation~\eqref{theequationwesolveinthisparticularprop}. Next, using~\eqref{thatuniformestinqthatwesodespwanted}, a standard weak-$*$ compactness argument yields the desired bound 
\begin{equation}\label{thatuniformestinqthatwesodespwantednowusedforthelimit}
\left\vert\left\vert u\right\vert\right\vert_{\mathring{H}^M} + \left\vert\left\vert \mathcal{L}_Xu\right\vert\right\vert_{\mathring{H}^M} \lesssim_M (A_M+1) \left\vert\left\vert F\right\vert\right\vert_{\mathring{H}^M}.
\end{equation}
Lastly, we have to show that $u_{A_1\cdots A_k}$ is unique among all solutions $w_{A_1\cdots A_k}$ to~\eqref{theequationwesolveinthisparticularprop} where $\left\vert\left\vert w\right\vert\right\vert_{L^2} + \left\vert\left\vert \mathcal{L}_Xw\right\vert\right\vert_{L^2} < \infty$. Indeed, let $w_{A_1\cdots A_k}$ be such a solution. Then, we have
\[\left(\left(w-u\right) + \mathcal{L}_X\left(w-u\right) + h\cdot \left(w-u\right)\right)_{A_1\cdots A_k} = 0.\]
Contracting with $\left(w-u + \mathcal{L}_X\left(w-u\right)\right)_{A_1\cdots A_k}$ and integrating by parts as above yields immediately that 
\[\left\vert\left\vert w-u\right\vert\right\vert_{L^2} + \left\vert\left\vert \mathcal{L}_X\left(w-u\right)\right\vert\right\vert_{L^2} = 0.\]

\end{proof}

\subsection{The $\kappa$-Constraint Equation}\label{kconstraychausdf}
In this section we will study the ``$\kappa$-constraint equation.'' We first collect below various constants that we will use and their respective hierarchy of smallness:
\begin{equation}\label{constants}
0 < \epsilon  \ll \gamma  \ll \delta  \ll 1.
\end{equation}
We will assume that
\[\epsilon^{\frac{\delta}{500}}\gamma^{-200} \ll 1.\]

We next fix our conventions for spherical coordinates on $\mathbb{S}^2$. 
\begin{convention}\label{sphcoordconv}Throughout the rest of the paper, we will use $(\theta,\phi)$ to denote spherical coordinates on $\mathbb{S}^2$, where $\phi \in [0,2\pi)$ is the azimuthal angle, and $\theta \in [0,\pi]$ is the polar angle. We also have the corresponding round metric on the $\mathbb{S}^2$, given by the formula $\mathring{\slashed{g}} = d\theta^2 + \sin^2\theta d\phi^2$. 
\end{convention}

We now give a sequence of important definitions.
\begin{definition}We say that a $4$-tuple $\left(\slashed{g}_{AB},b^A,\kappa,\Omega\right)$ consisting of a Riemannian metric $\slashed{g}_{AB}$ on $\mathbb{S}^2$, a vector field $b^A$ on $\mathbb{S}^2$, a positive constant $\kappa > 0$, and a function $\Omega$ on $\mathbb{S}^2$ satisfies the $\kappa$-constraint equation if
\begin{equation}\label{rayconstr}
\slashed{\rm div}b - \mathcal{L}_b\left(\slashed{\rm div}b\right) = \frac{1}{2}\left(\slashed{\rm div}b\right)^2 + \frac{1}{4}\left|\slashed{\nabla}\hat{\otimes}b\right|^2 - 4\kappa + 2\kappa\slashed{\rm div}b -2\Omega^{-1}\left(\mathcal{L}_b\Omega\right)\slashed{\rm div}b + 4\Omega^{-1}\mathcal{L}_b\Omega.
\end{equation}
\end{definition}

The starting point for our construction of $\kappa$-self-similar solutions will be a four tuple $\left(\slashed{g}_{AB},b^A,\kappa,\Omega\right)$ satisfying the $\kappa$-constraint equation, which also satisfy certain regularity requirements. We now give the relevant definition. We start by defining a certain type of ``seed'' data.
\begin{definition}\label{areallynicedefinitionofseeddata} Given $0 < \epsilon \ll \gamma \ll 1$ and $M_0,M_1 \in \mathbb{Z}_{> 0}$ satisfying $M_0,M_1 \gg 1$, ``seed data''  refers to any smooth vector field $\check{b}^A$ on $\mathbb{S}^2$ which satisfies
\[\check{b}^A = \epsilon \tilde b^A+ z^A,\qquad \tilde b|_{(\theta,\phi)} \doteq \left(\left(\int_{\pi/2}^{\theta}\frac{a(\hat{\theta})}{\sin\hat{\theta}}\, d\hat{\theta}\right)+r\right)\partial_{\phi},\]
where $r \in \mathbb{R}$ satisfies $\left|r\right| \lesssim \epsilon$, and we require that $a$ satisfy the following:
\begin{enumerate}
	\item $a(\theta)$ is a smooth function of $\theta$. 
	\item $a(\theta)$ is identically $1$ for $\theta \in [2\gamma,\pi-2\gamma]$.
	\item $a(\theta)$ is identically $0$ for $\theta \in [0,\gamma] \cup [\pi-\gamma,\pi]$.
	\item $\left|\frac{d^ka}{d\theta^k}\right| \lesssim \gamma^{-k}$.
	
\end{enumerate}
And we require that $z$ is a smooth vector field on $\mathbb{S}^2$ with
\begin{equation}\label{theotherrequirementsbutheyareimportant}
 \mathring{\nabla}_Az^A = 0,\qquad \left\vert\left\vert z\right\vert\right\vert_{\mathring{H}^{M_1}} \leq \epsilon^{M_0}.
\end{equation}

\end{definition}
\begin{remark}\label{comp}For
\[\tilde b = \left(\left(\int_{\pi/2}^{\theta}\frac{a(\hat{\theta})}{\sin\hat{\theta}}\, d\hat{\theta}\right)+r\right)\partial_{\phi},\]
a computation yields
\begin{equation}\label{thiswillbeuseful}
\mathring{\nabla}\hat{\otimes} \tilde b=  \frac{a(\theta)}{\sin(\theta)}\left(\partial_{\phi}\hat{\otimes}\partial_{\theta} + \partial_{\theta}\hat{\otimes}\partial_{\phi}\right).
\end{equation}
(Here we are raising and lowering indices with the round metric $\mathring{\slashed{g}}$.)

In particular,
\[\left|\mathring{\nabla}\hat{\otimes}\tilde{b}\right|_{\mathring{\slashed{g}}}^2 = 2a^2(\theta).\]
We also have
\[\mathring{\nabla}_A\tilde b^A = 0.\]
\end{remark}

Now we can define the notion of an $\left(\epsilon,\gamma,\delta,N_0,M_0,M_1\right)$-regular $4$-tuple $\left(\slashed{g}_{AB},b^A,\kappa,\Omega\right)$.
\begin{definition}\label{Mreg}Let  $0 < \epsilon \ll \gamma \ll \delta  \ll 1$, $\left(N_0,M_0,M_1\right) \in \left(\mathbb{Z}_{>0}\right)^3$ satisfy $N_0\gg 1$, $M_0 \gg N_0$, and $M_1 \gg N_0$, and we recall that $\mathring{\slashed{g}}_{AB}$ denotes the fixed choice of a round metric on $\mathbb{S}^2$. We say that a four tuple $\left(\slashed{g}_{AB},b^A,\kappa,\Omega\right)$ of a metric, vector field, constant, and function on $\mathbb{S}^2$, is ``$\left(\epsilon,\gamma,\delta,N_0,M_0,M_1\right)$-regular'' if they solve the $\kappa$-constraint equation, $\slashed{g}_{AB} = e^{2\varphi}\mathring{\slashed{g}}_{AB}$, $b^A = \check{b}^A + \mathring{\nabla}^Af$ (for $\check{b}^A$ as in Definition~\ref{areallynicedefinitionofseeddata}) with $\int_{\mathbb{S}^2}f\mathring{\rm dVol} = 0$, and we have
 \begin{equation}\label{somebound1ssssss}
 \left\vert\left\vert \varphi\right\vert\right\vert_{\mathring{H}^{N_0}} +\left\vert\left\vert \log\Omega\right\vert\right\vert_{\mathring{H}^{N_0+2}} +\left\vert\left\vert f\right\vert\right\vert_{\mathring{H}^{N_0+2}} \lesssim \epsilon^{2-\delta},
 \end{equation}
\begin{equation}\label{somebound3ssssss}
\left\vert\left\vert \mathcal{L}_{\partial_{\phi}}\varphi\right\vert\right\vert_{\mathring{H}^{N_0-1}} +\left\vert\left\vert \mathcal{L}_{\partial_{\phi}}\log\Omega\right\vert\right\vert_{\mathring{H}^{N_0+1}} +\left\vert\left\vert \mathcal{L}_{\partial_{\phi}}f\right\vert\right\vert_{\mathring{H}^{N_0+1}} \lesssim \epsilon^{M_1/2}.
\end{equation}
\end{definition}
\begin{remark}\label{whatisfdoinghere}We briefly explain the role of the function $f$: The $4$-tuple $\left(\slashed{g}_{AB},b^A,\kappa,\Omega\right)$ will determine our metric along the null hypersurface $\{v = 0\}$. Since we will need for this induced data along $\{v = 0\}$ to solve the null constraint equation~\eqref{basicconstr2}, we cannot expect to freely choose each of $\slashed{g}_{AB}$, $b^A$, $\kappa$, and $\Omega$. In the small data regime, we may consider~\eqref{basicconstr2} to be an equation which determines $\mathring{\rm div}b$ and $\kappa$, while $\slashed{g}_{AB}$, $\Omega$, and $\mathring{\rm curl}b$ are free. Since the function $f$ satisfies $\mathring{\Delta}f = \mathring{\rm div}b$, we may consider $f$ and $\kappa$ to be determined in terms of our choice of $\slashed{g}_{AB}$ and $\Omega$. (Note that $\mathring{\rm curl}b = \mathring{\rm curl}\check{b}$ is fixed by our choice for $b$.)
\end{remark}
\begin{remark}The key way in which we will use that $\mathring{\rm curl}b = \mathring{\rm curl}\check{b}$ is that it will allow us to understand well the leading order form of $b$ with respect to $\epsilon$ (see Lemmas~\ref{babablacksheep} and~\ref{eestimates}) and this detailed information will be important in the proof of Propositions~\ref{qregkapexists} and~\ref{kapissingbutthereisasolnatleastevolve}. 
\end{remark}

For most of the results of this section we will assume at the beginning that we have a $\left(\epsilon,\gamma,\delta,N_0,M_0,M_1\right)$-regular $4$-tuples $\left(\slashed{g}_{AB},b^A,\kappa,\Omega\right)$ and establish various additional properties of the $4$-tuple. (In Appendix~\ref{constructthetupleregul} we show that one may construct $\left(\epsilon,\gamma,\delta,N_0,M_0,M_1\right)$-regular $4$-tuples.) For the sake of brevity we will generally refer to a $\left(\epsilon,\gamma,\delta,N_0,M_0,M_1\right)$-regular $4$-tuples as a ``regular $4$-tuple.''

\begin{remark}It is not in fact necessary for our construction to assume that $\slashed{g}_{AB} = e^{2\varphi}\mathring{\slashed{g}}_{AB}$. Instead one may replace $\varphi$ with $\left(\slashed{g} -\mathring{\slashed{g}}\right)_{AB}$ in both~\eqref{somebound1ssssss} and~\eqref{somebound3ssssss}. Then, one may show that there exists a diffeomorphism $\mathscr{F}$ and a function $\varphi$ so that $\mathscr{F}^*\slashed{g}_{AB} = e^{2\varphi}\mathring{\slashed{g}}_{AB}$ and so that~\eqref{somebound1ssssss} and~\eqref{somebound3ssssss} all hold with $\log\Omega$, $b^A$, and $f$ replaced by $\mathscr{F}^*\log\Omega$, $\mathscr{F}^*b^A$, and $\mathscr{F}^*f$. We omit the details as we will not need this more general result. 
\end{remark}

\begin{convention}\label{lowerraisewhoknows}In the remainder of this section we will often be working with two metrics $\slashed{g}_{AB}$ and $\mathring{\slashed{g}}_{AB}$, and thus there may be some ambiguity when raising or lowering indices.  Thus we make the following definitions for any vector field $X^A$:
\[\left(\mathring{\nabla}\hat{\otimes}X\right)^{AB} \doteq \mathring{\slashed{g}}^{AC}\mathring{\nabla}_CX^B + \mathring{\slashed{g}}^{AB}\mathring{\nabla}_BX^C - \mathring{\slashed{g}}^{AB}\mathring{\nabla}_CX^C,\] \[\left(\slashed{\nabla}\hat{\otimes}X\right)^{AB} \doteq \slashed{g}^{AC}\slashed{\nabla}_CX^B + \slashed{g}^{AB}\slashed{\nabla}_BX^C - \slashed{g}^{AB}\slashed{\nabla}_CX^C.\]
\end{convention}

In the next lemma, we recall two formulas for how certain differential operators transform under a conformal change of the metric.
\begin{lemma}\label{conformal}Let $\slashed{g}_{AB} = e^{2\varphi}\mathring{\slashed{g}}_{AB}$. Then, for any vector field $X^A$ we have
\[\slashed{\nabla}_AX^A = \mathring{\nabla}_AX^A + 2\mathcal{L}_X\varphi,\qquad \left(\slashed{\nabla}\hat{\otimes} X\right)^{AB} = e^{-2\varphi}\left(\mathring{\nabla}\hat{\otimes} X\right)^{AB}\]
\end{lemma}
\begin{proof}These follow from the well-known coordinate expressions
\[\slashed{\nabla}_AX^A = \frac{1}{\sqrt{\slashed{g}}}\partial_A\left(\sqrt{\slashed{g}}X^A\right),\qquad \left(\slashed{\nabla}\hat{\otimes}X\right)^{AB} = \partial^AX^B + \partial^BX^A - X^C\partial_C\left(\slashed{g}^{AB}\right) - \slashed{g}^{AB}\frac{1}{\sqrt{\slashed{g}}}\partial_C\left(\sqrt{\slashed{g}}X^C\right).\]

\end{proof}

The next lemma concerns a precise estimate for $\kappa$.
\begin{lemma}\label{kappaisthis}Let $\left(\slashed{g}_{AB},b^A,\kappa,\Omega\right)$ be a regular $4$-tuple in the sense of Definition~\ref{Mreg}. Then we have 
\begin{equation}\label{kappatoproe}
\kappa = \frac{\epsilon^2}{16}\int_0^{\pi}a^2(\theta)\sin(\theta)\, d\theta + O\left(\epsilon^{3-\delta}\right).
\end{equation}
In particular, $\kappa \sim \epsilon^2$. 

\end{lemma}
\begin{proof}Integrating~\eqref{rayconstr} over $\mathbb{S}^2$ leads to
\begin{align}\label{lslslslosj}
&-4\kappa \slashed{\rm Area}\left(\mathbb{S}^2\right) + \frac{1}{4}\int_{\mathbb{S}^2}\left|\slashed{\nabla}\hat{\otimes}b\right|_{\slashed{g}}^2\slashed{\rm dVol}
\\ \nonumber &\qquad + \int_{\mathbb{S}^2}\left[-\frac{1}{2}\left(\slashed{\rm div}b\right)^2 +2\kappa\left(\slashed{\rm div}b\right) -2\Omega^{-1}\left(\mathcal{L}_b\Omega\right)\slashed{\rm div}b -4\left(\slashed{\rm div}b\right)\left(\log\Omega\right)\right]\slashed{\rm dVol}= 0.
\end{align}
From Lemma~\ref{conformal} we have 
\[\slashed{\rm div}b = \mathring{\rm div}b + 2\mathcal{L}_b\varphi = \mathring{\Delta}f + 2\mathcal{L}_b\varphi.\]
Thus, using~\eqref{somebound1ssssss} and Sobolev inequalities, we have
\begin{align}\label{divboundkskkksks}
\left\vert\left\vert \slashed{\rm div}b\right\vert\right\vert_{L^{\infty}} &\lesssim \left\vert\left\vert \mathring{\Delta}f\right\vert\right\vert_{L^{\infty}} + \left\vert\left\vert \mathcal{L}_b\varphi\right\vert\right\vert_{L^{\infty}}
\\ \nonumber &\lesssim  \left\vert\left\vert f\right\vert\right\vert_{\mathring{H}^3}^{1/2} \left\vert\left\vert f\right\vert\right\vert_{\mathring{H}^4}^{1/2}+ \left\vert\left\vert \mathcal{L}_b\varphi\right\vert\right\vert_{\mathring{H}^1}^{1/2}\left\vert\left\vert \mathcal{L}_b\varphi\right\vert\right\vert_{\mathring{H}^2}^{1/2}
\\ \nonumber &\lesssim \epsilon^{2-\delta}.
\end{align}
It follows from~\eqref{lslslslosj} and the bounds~\eqref{divboundkskkksks} and~\eqref{somebound1ssssss} that
\begin{equation}\label{kappaisep2osksksks}
|\kappa| \lesssim \epsilon^2 + \epsilon^{2-\delta}\left|\kappa\right| \Rightarrow \left|\kappa\right| \lesssim \epsilon^2.
\end{equation}
Having established~\eqref{kappaisep2osksksks} and appealing again to~\eqref{divboundkskkksks} and~\eqref{somebound1ssssss}, we now see that
\begin{equation}\label{cococo}
 \left|\int_{\mathbb{S}^2}\left[-\frac{1}{2}\left(\slashed{\rm div}b\right)^2 +2\kappa\left(\slashed{\rm div}b\right) -2\Omega^{-1}\left(\mathcal{L}_b\Omega\right)\slashed{\rm div}b -4\left(\slashed{\rm div}b\right)\left(\log\Omega\right)\right]\slashed{\rm dVol} \right| \lesssim \epsilon^3.
\end{equation}

Next, using Remark~\ref{comp} and Lemma~\ref{conformal}, we note that
\begin{align}\label{lkjhgywwwww}
\left|\slashed{\nabla}\hat{\otimes}b\right|_{\slashed{g}}^2 = \left|\mathring{\nabla}\hat{\otimes}b\right|_{\mathring{\slashed{g}}}^2 = 2\epsilon^2a^2(\theta) + 2\epsilon\mathring{\slashed{g}}\left(\mathring{\nabla}\hat{\otimes}\tilde b,\mathring{\nabla}\hat{\otimes}\left(z + \mathring{\nabla}f\right)\right) + \left|\mathring{\nabla}\hat{\otimes}\left(z+\mathring{\nabla}f\right)\right|^2. 
\end{align}
From~\eqref{lkjhgywwwww}, the bounds from~\eqref{theotherrequirementsbutheyareimportant} and~\eqref{somebound1ssssss}, and Sobolev inequalities, we thus obtain that
\begin{equation}\label{pqkmsihfje}
\left|\left|\slashed{\nabla}\hat{\otimes}b\right|_{\slashed{g}}^2 - 2\epsilon^2a^2(\theta)\right| \lesssim \epsilon^{3-\delta}.
\end{equation}
Since we also have that $\left\vert\left\vert\varphi\right\vert\right\vert_{L^{\infty}} \lesssim \epsilon^{2-\delta}$, we then obtain that
\begin{align}
\frac{1}{\slashed{\rm Area}\left(\mathbb{S}^2\right)}\int_{\mathbb{S}^2}\left|\slashed{\nabla}\hat{\otimes}b\right|_{\slashed{g}}^2\slashed{\rm dVol} &= \frac{1}{4\pi}\int_0^{2\pi}\int_0^{\pi}\left|\slashed{\nabla}\hat{\otimes}b\right|_{\slashed{g}}^2\sin\theta\, d\theta d\phi + O\left(\epsilon^3\right)
\\ \nonumber &= \int_0^{\pi}a^2(\theta)\sin\theta\, d\theta + O\left(\epsilon^{3-\delta}\right).
\end{align}
The formula~\eqref{kappatoproe} then follows from~\eqref{lslslslosj}.
\end{proof}

The next lemma provides a more precise estimate for $\slashed{\rm div}b$.
\begin{lemma}\label{babablacksheep}Let $\left(\slashed{g}_{AB},b^A,\kappa,\Omega\right)$ be a regular $4$-tuple in the sense of Definition~\ref{Mreg}. Then we have 
\begin{equation}\label{0jsodvna}
	\slashed{\rm div}b = \frac{\epsilon^2}{2}a^2(\theta) - \frac{\epsilon^2}{4}\int_0^{\pi}a^2(\theta')\, \sin(\theta')\, d\theta' + O\left(\epsilon^{3-3\delta}\right).
	\end{equation}
	In particular,
	\begin{equation}\label{lkadojagdjagri}
	-\slashed{\rm div}b \gtrsim -\gamma^2 \epsilon^2 - \epsilon^{3-3\delta}.
	\end{equation}

\end{lemma}
\begin{proof}We start by re-writing~\eqref{rayconstr} as 
\begin{align}\label{werewrotedivbeqn}
&\left(\slashed{\rm div}b - \frac{1}{4}\left|\slashed{\nabla}\hat{\otimes}b\right|_{\slashed{g}}^2 + 4\kappa\right) - \mathcal{L}_b\left(\slashed{\rm div}b - \frac{1}{4}\left|\slashed{\nabla}\hat{\otimes}b\right|_{\slashed{g}}^2 + 4\kappa\right)
\\ \nonumber &\qquad = \frac{1}{4}\mathcal{L}_b\left(\left|\slashed{\nabla}\hat{\otimes}b\right|_{\slashed{g}}^2\right) + \frac{1}{2}\left(\slashed{\rm div}b\right)^2 + 2\kappa\slashed{\rm div}b - 2\Omega^{-1}\left(\mathcal{L}_b\Omega\right)\slashed{\rm div}b + 4\Omega^{-1}\mathcal{L}_b\Omega.
\end{align}
We have
\[\left\vert\left\vert b\right\vert\right\vert_{\mathring{H}^3} \lesssim \epsilon \left\vert\left\vert \tilde b\right\vert\right\vert_{\mathring{H}^3} + \left\vert\left\vert z\right\vert\right\vert_{\mathring{H}^3} + \left\vert\left\vert f\right\vert\right\vert_{\mathring{H}^4} \lesssim \epsilon^{1-\delta}.\]

Thus we can apply Proposition~\ref{qualitativeexistence} and Lemma~\ref{conformal} to conclude that
\begin{align}\label{estimatefignsofikdivbstuff}
&\left\vert\left\vert \slashed{\rm div}b - \frac{1}{4}\left|\slashed{\nabla}\hat{\otimes}b\right|_{\slashed{g}}^2 + 4\kappa\right\vert\right\vert_{\mathring{H}^2} 
\\ \nonumber &\qquad \lesssim \left\vert\left\vert \frac{1}{4}\mathcal{L}_b\left(\left|\slashed{\nabla}\hat{\otimes}b\right|_{\slashed{g}}^2\right) + \frac{1}{2}\left(\slashed{\rm div}b\right)^2 + 2\kappa\slashed{\rm div}b - 2\Omega^{-1}\left(\mathcal{L}_b\Omega\right)\slashed{\rm div}b + 4\Omega^{-1}\mathcal{L}_b\Omega\right\vert\right\vert_{\mathring{H}^2}
\\ \nonumber &\qquad \lesssim \epsilon^{3-3\delta}.
\end{align}
Applying a Sobolev inequality thus immediately implies that
\begin{equation}\label{formuforforfdivdiv}
\slashed{\rm div}b = \frac{1}{4}\left|\slashed{\nabla}\hat{\otimes}b\right|_{\slashed{g}}^2 - 4\kappa + O\left(\epsilon^{3-3\delta}\right).
\end{equation}
Now we combine~\eqref{formuforforfdivdiv} with~\eqref{pqkmsihfje} and Lemma~\ref{kappaisthis} to obtain~\eqref{0jsodvna}.

\end{proof}

To leading order in $\epsilon$ the vector field $b^A$ is given by $\epsilon \tilde b^A$. However, it will be important for us to have precise estimates on the $\epsilon^2$ order part of $b^A$. This will be established in the next sequence of results.

\begin{definition}\label{admissibleh}Let $a(\theta)$ be as in Definition~\ref{areallynicedefinitionofseeddata}. We then define a function $h(\theta): [0,\pi] \to \mathbb{R}$ by	
\begin{equation}\label{aneqnforhthatisratherexplicit}
	h(\theta) \doteq \frac{1}{2\sin\theta}\left[\left(\int_0^{\theta}a^2(x)\sin(x)\, dx\right) - \left(\int_0^{\pi}a^2(x)\sin(x)\, dx\right)\left(\frac{1-\cos\theta}{2}\right)\right].
	\end{equation}

\end{definition}

In the next lemma we will establish some useful properties of the function $h(\theta)$.
\begin{lemma}\label{propertiesofhadmissible} Let $h(\theta): [0,\pi] \to \mathbb{R}$ be as in Definition~\ref{admissibleh}. Then the following are true.
\begin{enumerate}
	\item We have 
	\begin{equation}\label{thebasicequaiotnforh}
	\frac{1}{\sin\theta}\frac{d}{d\theta}\left(\sin\theta h(\theta)\right) = \frac{1}{2}a^2(\theta) - \frac{1}{4}\int_0^{\pi}a^2(x)\sin(x)\, dx.
	\end{equation}
	\item The function $h(\theta)$ has exactly three zeros, all simple, which occur at $\theta = 0$, $y_0$, and $\pi$ for $\left|y_0 - \pi/2\right| \lesssim \gamma$.
	\item The function $\sin(\theta)h(\theta)$ has exactly two interior critical points $\theta_1 \in [\gamma,2\gamma]$ and $\theta_2 \in[\pi-2\gamma,\pi-\gamma]$. At $\theta_1$ there is a local minimum and at $\theta_2$ there is a local maximum.
	\item There exists a small constant $c > 0$ (independent of $\gamma \ll 1$) so that
	\begin{equation}\label{lowerbound}
	|h(\theta)| \geq \frac{\gamma^2}{100},\qquad \text{ for }\theta \in [\gamma^2/4,y_0-c] \cup [y_0+c,\pi-\gamma^2/4].
	\end{equation}
	\item Near the poles $\theta = 0$ and $\theta = \pi$ we have
	\begin{equation}\label{poleexpansion1}
	h = -\frac{\theta}{4}\left(1+O\left(\gamma^2\right)\right) + O\left(\theta^2 \right),\qquad \forall \theta \in [0,\gamma],
	\end{equation}
	\begin{equation}\label{poleexpansion2}
h = \frac{(\pi-\theta)}{4}\left(1+O\left(\gamma^2\right)\right) + O\left( \left(\pi-\theta\right)^2 \right) ,\qquad \forall \theta \in [\pi-\gamma,\pi].
	\end{equation}
	\item For $\left|\theta - y_0\right| \leq c$ (see~\eqref{lowerbound}) there exists a constant $d > 0$, independent of $\gamma$, so that
	\begin{equation}\label{zeroexpansion}
	h(\theta) = -d\gamma^2\left(y_0 - \theta\right) + O\left(\gamma^2\left(y_0-\theta\right)^2\right).
	\end{equation}
	\item We have the following global bound on $h$:
	\begin{equation}\label{globalboundonh}
	\left|\frac{d^k}{d\theta^k}\left(\sin(\theta)h(\theta)\right)\right| \lesssim_k \gamma^{2-k},\qquad \forall \theta \in [0,\pi].
	\end{equation}
	\item If we are far enough in the interior of $[0,\pi]$, then we have a much better bound:
	\begin{equation}\label{globalboundonh2}
	\left|\frac{d^k}{d\theta^k}\left(\sin(\theta)h(\theta)\right)\right| \lesssim_k \gamma^2,\qquad \forall \theta \in [3\gamma,\pi-3\gamma].
	\end{equation}
	
	\item We have
	\begin{equation}\label{itvanishesatthepolesprettyfast}
	\left|\slashed{\nabla}\hat{\otimes}\left(h\partial_{\theta}\right)\right|_{\slashed{g}} \lesssim {\rm min}\left(\theta^2,\left(\pi-\theta\right)^2\right)\text{ for }\theta \in \left[0,\frac{\gamma}{2}\right]\cup \left[\pi-\frac{\gamma}{2},\pi\right],
	\end{equation}
	\begin{equation}\label{itvanishesatthepolesprettyfast2}
	\left|\slashed{\nabla}\hat{\otimes}\left(h\partial_{\theta}\right)\right|_{\slashed{g}} \lesssim \gamma^2\text{ for }\theta \in  \left[y_0-2\gamma,y_0+2\gamma\right].
\end{equation}
	\item We have
	\begin{align}\label{ogopgjejogaerjoaefowefkopafe2}
\mathring{\nabla}\hat{\otimes}\left(h\partial_{\theta}\right) &= \left(h(\theta)\frac{\cos\theta}{\sin\theta} + \frac{1}{4}\int_0^{\pi}a^2(x)\sin(x)\, dx -\frac{1}{2}a^2(\theta)\right)\left(\frac{1}{\sin\theta}\partial_{\phi}\right)\hat\otimes\left(\frac{1}{\sin\theta}\partial_{\phi}\right)  
\\ \nonumber &\qquad - \left(h(\theta)\frac{\cos\theta}{\sin\theta} + \frac{1}{4}\int_0^{\pi}a^2(x)\sin(x)\, dx -\frac{1}{2}a^2(\theta)\right)\partial_{\theta}\hat\otimes\partial_{\theta}.
\end{align}

\end{enumerate}

\end{lemma}
\begin{proof}
Given the formula~\eqref{aneqnforhthatisratherexplicit}, the equation~\eqref{thebasicequaiotnforh} is a straightforward calculation. Next, we turn to understanding the graph of $h(\theta)$. We start by considering the function
\begin{equation}\label{worifkdjnsnsns}
r\left(\theta\right) \doteq \frac{1}{2}a^2(\theta) - \frac{1}{4}\int_0^{\pi}a^2(x)\sin(x)\, dx.
\end{equation}
First of all, 
\begin{align*}
\cos(2\gamma) - \cos(\pi-2\gamma) = \int_{2\gamma}^{\pi-2\gamma}\sin(x)\ dx \leq \int_0^{\pi}a^2(x)\sin(x)\, dx \leq \int_{\gamma}^{\pi-\gamma}\sin(x)\, dx =\cos(\gamma)- \cos(\pi-\gamma).
\end{align*}
Thus we conclude that
\begin{equation*}
2- 4\gamma^2 + O\left(\gamma^4\right) \leq \int_0^{\pi}a^2(x)\sin(x)\, dx \leq 2 - \gamma^2 + O\left(\gamma^4\right).
\end{equation*}
From this it is immediate that there exists $\theta_1 \in [\gamma,2\gamma]$ and $\theta_2 \in [\pi-2\gamma,\pi-\gamma]$ such that 
\begin{equation}\label{integralasinest}
r\left(\theta_1\right) = 0,\qquad r(\theta_2) = 0,\qquad r(\theta) \neq 0\qquad \text{ if }\theta \not= \theta_1\text{ or }\theta_2.
\end{equation}
In particular,~\eqref{integralasinest} and~\eqref{thebasicequaiotnforh} imply that $\theta_1$ and $\theta_2$ are the interior critical points of $\sin(\theta)h(\theta)$. Furthermore, one sees immediately that $\sin(\theta)h(\theta)$ is decreasing from $0$ to $\theta_1$, is increasing from $\theta_1$ to $\theta_2$, and is decreasing again from $\theta_2$ to $\pi$. We also see immediately that $\sin(\theta)h(\theta)$ vanishes at $\theta = 0$ and $\theta = \pi$ and it is clear that there will be exactly one other zero. It is straightforward to see that the remaining zero must occur at a $y_0$ which satisfies $\left|y_0-\pi/2\right| \lesssim \gamma^2$. 

Next, we come to the bound~\eqref{poleexpansion1}. For this, we simply examine the formula~\eqref{aneqnforhthatisratherexplicit} when $\theta \in [0,\gamma]$. We find that
\begin{equation*}
h(\theta) = -\frac{1}{4}\left(\int_0^{\pi}a^2(x)\sin(x)\, dx\right)\frac{1-\cos(\theta)}{\sin(\theta)},\qquad \forall \theta\in [0,\gamma] \Rightarrow 
\end{equation*}
\begin{equation*}
h(\theta) = -\frac{1}{4}\left(2+O\left(\gamma^2\right)\right)\frac{\theta^2/2 + O\left(\theta^4\right)}{\theta + O\left(\theta^3\right)},\qquad \forall \theta \in[0,\gamma].
\end{equation*}
This establishes~\eqref{poleexpansion1}. The bound~\eqref{poleexpansion2} is established in an analogous fashion.

Next, we examine the behavior near $\theta = y_0$. Taylor expanding $\sin(\theta)$ yields that
\[\sin(\theta) = 1+O\left(\left(\theta-\pi/2\right)^2\right).\]
Thus, for $\theta- y_0$ sufficiently small (independently of $\gamma$) we may use~\eqref{thebasicequaiotnforh} to see that
\[\left|\frac{d}{d\theta}\left(\sin(\theta)h(\theta)\right)\right| \geq \frac{1}{10}\gamma^2.\]
We also obtain~\eqref{zeroexpansion} since one easily finds that $\frac{d^2}{d\theta^2}\left(\sin(\theta)h(\theta)\right)|_{\theta = y_0} = O\left(\gamma^2\right)$. Finally, using the estimates above, one can also obtain~\eqref{lowerbound}.

Next, we turn to the bound~\eqref{globalboundonh}. We start with $k = 0$. It will be important to recall that $a(x)$ is globally bounded, is identically $1$ for $x \in [2\gamma,\pi-2\gamma]$, and vanishes for $x \in [0,\gamma] \cup [\pi-\gamma,\pi]$.  We have
\begin{align*}
2\left|\sin(\theta)h(\theta)\right| &= \left|\int_0^{\theta}a^2(x)\sin(x)\, dx - \left(\int_0^{\pi}a^2(y)\sin(y)\, dy\right)\frac{1-\cos(\theta)}{2}\right|
\\ \nonumber &= \left|\int_0^{\theta}\left(a^2(x)-1\right)\sin(x)\, dx - \left(\int_0^{\pi}\left(a^2(y)-1\right)\sin(y)\, dy\right)\frac{1-\cos(\theta)}{2}\right|
\\ \nonumber &\lesssim \int_0^{2\gamma}\sin(x)\, dx + \int_{\pi-2\gamma}^{\pi}\sin(x)\, dx
\\ \nonumber &\lesssim \gamma^2.
\end{align*}
This establishes~\eqref{globalboundonh} for $k = 0$. The general case follows similarly. Finally, we note that~\eqref{globalboundonh2} is proven in a straightforward manner.

It remains to establish~\eqref{itvanishesatthepolesprettyfast}. First of all, using Lemma~\ref{conformal}, we have that $\left|\slashed{\nabla}\hat{\otimes}\left(h\partial_{\theta}\right)\right|_{\slashed{g}} = \left|\mathring{\nabla}\hat{\otimes}\left(h\partial_{\theta}\right)\right|_{\mathring{\slashed{g}}}$. In spherical coordinates, the non-zero Christoffel symbols on the round metric are $\Gamma_{\phi\phi}^{\theta} = -\sin\theta\cos\theta$ and $\Gamma_{\phi\theta}^{\phi} = \cot\theta$. Thus, one easily computes that
\[\mathring{\nabla}\left(h\partial_{\theta}\right)= h(\theta)\frac{\cos\theta}{\sin\theta} \left(\frac{1}{\sin\theta}\partial_{\phi}\right)\otimes\left(\frac{1}{\sin\theta}\partial_{\phi}\right) + \frac{dh}{d\theta}\partial_{\theta}\otimes\partial_{\theta}.\]
If we keep in mind the equation~\eqref{thebasicequaiotnforh} for $h$, then we see that, we will have 
\begin{align}\label{ogopgjejogaerjoaefowefkopafe}
\mathring{\nabla}\hat{\otimes}\left(h\partial_{\theta}\right) &= \left(2h(\theta)\frac{\cos\theta}{\sin\theta} + \frac{1}{4}\int_0^{\pi}a^2(x)\sin(x)\, dx -\frac{1}{2}a^2(\theta)\right)\left(\frac{1}{\sin\theta}\partial_{\phi}\right)\hat\otimes\left(\frac{1}{\sin\theta}\partial_{\phi}\right)  
\\ \nonumber &\qquad - \left(2h(\theta) \frac{\cos(\theta)}{\sin(\theta)}+ \frac{1}{4}\int_0^{\pi}a^2(x)\sin(x)\, dx  -\frac{1}{2}a^2(\theta)\right)\partial_{\theta}\hat\otimes\partial_{\theta}.
\end{align}
This implies~\eqref{ogopgjejogaerjoaefowefkopafe2}. The estimate~\eqref{itvanishesatthepolesprettyfast2} is a straightforward consequence of the formula~\eqref{ogopgjejogaerjoaefowefkopafe2}. Now, using~\eqref{aneqnforhthatisratherexplicit}, we see that when $\theta < \gamma$,
\[h\left(\theta\right) = -\left(\int_0^{\pi}a^2(x)\sin(x)\, dx\right)\left(\frac{\theta}{8} + O\left(\theta^3\right)\right).\]
Combining this with~\eqref{ogopgjejogaerjoaefowefkopafe} thus leads to 
\[\left|\slashed{\nabla}\hat{\otimes}h\right|_{\slashed{g}} \lesssim \theta^2.\]
A similar argument works when $\left|\pi-\theta\right| \ll 1$.

\end{proof}

\begin{lemma}\label{eestimates}Let $\left(\slashed{g}_{AB},b^A,\kappa,\Omega\right)$ be a regular $4$-tuple in the sense of Definition~\ref{Mreg}.  Then we can write 
\[b^A\slashed{\nabla}_A = \check{b}^A\slashed{\nabla}_A + \epsilon^2h(\theta)\partial_{\theta} + e^A\slashed{\nabla}_A,\] 
where $e$ satisfies 	
\[\left\vert\left\vert e\right\vert\right\vert_{\mathring{H}^3} \lesssim \epsilon^{3-2\delta}.\]
\end{lemma}
\begin{proof}

Using~\eqref{thebasicequaiotnforh}, Lemma~\ref{conformal}, and Remark~\ref{comp} we may derive the following equation for $e$:
\begin{equation}\label{fortheerror}
\mathring{{\rm div}}e = \slashed{\rm div}b - \frac{\epsilon^2}{2}a^2(\theta) + \frac{\epsilon^2}{4}\int_0^{\pi}a^2(\theta)\sin(\theta)\, d\theta - 2\mathcal{L}_b\varphi,\qquad \mathring{{\rm curl}}e = 0.
\end{equation}
Then, using the equation~\eqref{rayconstr},~\eqref{kappatoproe}, and ~\eqref{thiswillbeuseful} we may derive the following equation for $\tilde D \doteq \slashed{\rm div}b - \frac{\epsilon^2}{2}a^2(\theta) + \frac{\epsilon^2}{4}\int_0^{\pi}a^2(\theta)\sin(\theta)\, d\theta$:
\begin{align}\label{fortildeD}
\tilde D - \mathcal{L}_b\tilde D &= \frac{1}{2}\left(\slashed{\rm div}b\right)^2 + 2\kappa\slashed{\rm div}b + \mathcal{L}_b\left[\frac{\epsilon^2}{2}a^2(\theta) - \frac{\epsilon^2}{4}\int_0^{\pi}a^2(\theta)\sin(\theta)\, d\theta\right]
\\ \nonumber &\qquad + \left(\frac{1}{4}\left|\slashed{\nabla}\hat{\otimes}b\right|_{\slashed{g}}^2 - 4\kappa - \frac{\epsilon^2}{2}a^2(\theta) + \frac{\epsilon^2}{4}\int_0^{\pi}a^2(\theta)\sin(\theta)\, d\theta\right)-2\Omega^{-1}\left(\mathcal{L}_b\Omega\right)\slashed{\rm div}b + 4\Omega^{-1}\mathcal{L}_b\Omega
\\ \nonumber &= \frac{1}{2}\left(\slashed{\rm div}b\right)^2 + 2\kappa\slashed{\rm div}b + \mathcal{L}_b\left[\frac{\epsilon^2}{2}a^2(\theta) - \frac{\epsilon^2}{4}\int_0^{\pi}a^2(\theta)\sin(\theta)\, d\theta\right]
\\ \nonumber &\qquad + \frac{\epsilon}{2}\mathring{\slashed{g}}\left(\mathring{\nabla}\hat{\otimes}\tilde b,\mathring{\nabla}\hat{\otimes}\left(\epsilon z + \epsilon^2 h(\theta)\partial_{\theta} + e\right)\right) + \left|\mathring{\nabla}\hat{\otimes}\left(\epsilon z + \epsilon^2 h(\theta)\partial_{\theta} + e\right)\right|_{\mathring{\slashed{g}}}^2 
\\ \nonumber &\qquad + \check{\kappa}  -2\Omega^{-1}\left(\mathcal{L}_b\Omega\right)\slashed{\rm div}b + 4\Omega^{-1}\mathcal{L}_b\Omega,
\end{align}
where, by Lemma~\ref{kappaisthis}, $\check{\kappa}$ is a constant satisfying $|\check{\kappa}| \lesssim \epsilon^{3-\delta}$.

We start with
\begin{align}\label{boundsforbbbvarphigigif}
\left\vert\left\vert \mathcal{L}_b\varphi\right\vert\right\vert_{\mathring{H}^2} &\lesssim \epsilon^{1-\delta}\left\vert\left\vert \mathcal{L}_{\partial_{\phi}}\varphi\right\vert\right\vert_{\mathring{H}^2} + \epsilon \left\vert\left\vert z\right\vert\right\vert_{\mathring{H}^2}\left\vert\left\vert \varphi\right\vert\right\vert_{\mathring{H}^3} + \left\vert\left\vert \nabla f\right\vert\right\vert_{\mathring{H}^3}\left\vert\left\vert\varphi\right\vert\right\vert_{\mathring{H}^3}
\\ \nonumber &\lesssim \epsilon^{4-2\delta}.
\end{align}
Next, from~\eqref{fortheerror} and elliptic estimates, we obtain
\begin{equation}\label{cocobojo}
\left\vert\left\vert e\right\vert\right\vert_{\mathring{H}^3} \lesssim \left\vert\left\vert \tilde D\right\vert\right\vert_{\mathring{H}^2} + \left\vert\left\vert \mathcal{L}_b\varphi\right\vert\right\vert_{\mathring{H}^2} \lesssim \left\vert\left\vert \tilde D\right\vert\right\vert_{\mathring{H}^2} + \epsilon^{4-2\delta}.
\end{equation}
Next we observe that~\eqref{rayconstr} and Proposition~\ref{qualitativeexistence} easily lead to the bound
\[\left\vert\left\vert \slashed{\rm div}b\right\vert\right\vert_{\mathring{H}^2} \lesssim \epsilon^{2-\delta}.\]
Then we apply Proposition~\ref{qualitativeexistence} to~\eqref{fortildeD} and obtain that
\begin{align}\label{uuejeju}
\left\vert\left\vert \tilde D\right\vert\right\vert_{\mathring{H}^2} &\lesssim \epsilon^{3-\delta} + \epsilon^{1-\delta}\left\vert\left\vert e\right\vert\right\vert_{\mathring{H}^3} + \left\vert\left\vert e\right\vert\right\vert_{\mathring{H}^3}^2.
\end{align}
Combining~\eqref{cocobojo} and~\eqref{uuejeju} leads to
\[\left\vert\left\vert e\right\vert\right\vert_{\mathring{H}^3} \lesssim \epsilon^{3-2\delta},\qquad \left\vert\left\vert \tilde D\right\vert\right\vert_{\mathring{H}^3} \lesssim \epsilon^{3-\delta}.\]

\end{proof}
\begin{remark}Lemma~\ref{eestimates} shows that $f$  (from Definition~\ref{Mreg}) is forced, in view of the requirement that~\eqref{basicconstr2} holds,  to take the form 
\[\mathring{\nabla}f = \epsilon^2 h(\theta)\partial_{\theta} + e.\]
See also Remark~\ref{whatisfdoinghere}.
\end{remark}

\subsection{The $\kappa$-Singular Equation}\label{kapsingsec}

In this section we will study the ``$\kappa$-singular equation'' which will later play a fundamental role in setting up our characteristic data. (See Section~\ref{previewofchardata}.)  This equation will be similar to the equation $Pu_{A_1\cdots A_k} = F_{A_1\cdots A_k}$ from Definition~\ref{KuF} except that $X^A$ and $h$ will \emph{not} have a smallness condition. Instead, the study of the $\kappa$-singular equation will be tractable only because of a certain anti-symmetric structure. We now turn to the relevant definitions.

\begin{definition}\label{defofkapsing} Let $\left(\slashed{g}_{AB},b^A,\kappa,\Omega\right)$ be a regular $4$-tuple in the sense of Definition~\ref{Mreg}. Then we say that $t \in \hat{\mathcal{S}}\left(\mathbb{S}^2\right)$ satisfies the corresponding $\kappa$-singular equation with right hand side $H  \in \hat{\mathcal{S}}\left(\mathbb{S}^2\right)$ if
\begin{equation}\label{thisisthekappasingularequation}
\mathscr{L}t_{AB} -2\kappa t_{AB}\doteq \left(\mathcal{L}_bt_{AB}  - \left(\slashed{\nabla}\hat{\otimes}b\right)^C_{\ \ (A}t_{B)C} - \frac{1}{2}\slashed{\rm div}b t_{AB}\right) -2\kappa t_{AB}= H_{AB}.
\end{equation}
We recall that $\hat{\mathcal{S}}\left(\mathbb{S}^2\right)$ denotes the space of trace-free symmetric tensors in $\mathcal{T}^{(0,2)}\left(\mathbb{S}^2\right)$.

It will also be convenient to extend the action of $\mathscr{L}$ to tensors $w_{C_1\cdots C_kAB}$ by having $\mathscr{L}$ just act on the last two indices.
\end{definition}
\begin{remark}We quickly recall for the reader where the need to study equations of this form arises (see the detailed discussion in Section~\ref{previewofchardata}). Namely, we will need to pose outgoing characteristic data for our solution. We will desire to pose this outgoing data in such a way that the prescribed value of $\Omega^{-1}\hat{\chi}$ is consistent with the solution behaving in an approximately self-similar fashion as $v\to 0$. If our spacetime was exactly $\kappa$-self-similar, then one finds that $\Omega^{-1}\hat{\chi}$ must satisfy the equation~\eqref{knbhuikawefkf2} along any sphere on $\{v = 0\}$. The equation~\eqref{knbhuikawefkf2} is exactly of the form~\eqref{thisisthekappasingularequation} for suitable $H_{AB}$. 
\end{remark}

Just as in Section~\ref{perturbtheidentityforever}, it will be convenient to define an elliptic regularization of the $\kappa$-singular equation.
\begin{definition}\label{kappsinreg}Given a $\kappa$-singular equation as in Definition~\ref{defofkapsing} and $q > 0$, we define the corresponding $q$-regularization by
\begin{equation}\label{thisisthekappasingularequationqReg}
\mathcal{L}_bt_{AB} -2\kappa t_{AB} - \left(\slashed{\nabla}\hat{\otimes}b\right)^C_{\ \ (A}t_{B)C} - \frac{1}{2}\slashed{\rm div}b t_{AB}+q\slashed{\Delta}t_{AB}= H_{AB}.
\end{equation}
\end{definition}

The next lemma identifies an important anti-symmetric structure in~\eqref{thisisthekappasingularequation}.
\begin{lemma}\label{thatscrlisantisymm}Let $\left(\slashed{g}_{AB},b^A,\kappa,\Omega\right)$ be a regular $4$-tuple in the sense of Definition~\ref{Mreg} and $\mathscr{L}$ be defined as in Definition~\ref{defofkapsing}. Then $\mathscr{L}$ is anti-symmetric with respect to the Hilbert space structure induced by $\slashed{g}_{AB}$.
\end{lemma}
\begin{proof}We have
\begin{align*}
\left(\slashed{g}^{-1}\right)^{AC}\left(\slashed{g}^{-1}\right)^{BD}\mathcal{L}_bt_{AB}h_{CD} &= \mathcal{L}_b\left(t\cdot h\right) - \mathcal{L}_b\left(\left(\slashed{g}^{-1}\right)^{AC}\left(\slashed{g}^{-1}\right)^{BD}\right)t_{AB}h_{CD} - \left(\slashed{g}^{-1}\right)^{AC}\left(\slashed{g}^{-1}\right)^{BD}f_{AB}\mathcal{L}_bh_{CD}
\\ \nonumber &= \mathcal{L}_b\left(t\cdot h\right)+ 2\left(\mathcal{L}_b\slashed{g}\right)^{AC}\left(\slashed{g}^{-1}\right)^{BD}t_{AB}h_{CD}- \left(\slashed{g}^{-1}\right)^{AC}\left(\slashed{g}^{-1}\right)^{BD}t_{AB}\mathcal{L}_bh_{CD}
\\ \nonumber &= \mathcal{L}_b\left(t\cdot h\right) + 2{\rm div}b\left(t\cdot h\right)+2\left(\slashed{\nabla}\hat{\otimes}b\right)^{AC}\left(\slashed{g}^{-1}\right)^{BD}t_{AB}h_{CD}
\\ \nonumber &\qquad - \left(\slashed{g}^{-1}\right)^{AC}\left(\slashed{g}^{-1}\right)^{BD}t_{AB}\mathcal{L}_bh_{CD}.
\end{align*}
Thus, after integrating and applying the divergence theorem, we obtain
\begin{align*}
&\int_{\mathbb{S}^2}\left[\left(\slashed{g}^{-1}\right)^{AC}\left(\slashed{g}^{-1}\right)^{BD}\mathscr{L}t_{AB}h_{CD} + \left(\slashed{g}^{-1}\right)^{AC}\left(\slashed{g}^{-1}\right)^{BD}t_{AB}\mathscr{L}h_{CD}\right] \slashed{dVol} 
\\ \nonumber &\qquad =\int_{\mathbb{S}^2}\left(\mathcal{L}_b\left(t\cdot h\right) + \slashed{\rm div}b\left(t\cdot h\right)\right)\slashed{dVol}
\\ \nonumber &\qquad = 0.
\end{align*}
\end{proof}

It will be convenient to have a version of Lemma~\ref{thatscrlisantisymm} which holds for higher order tensors (where we lose the exact anti-symmetry).
\begin{lemma}\label{higherorderalmostantisymm}Let $\left(\slashed{g}_{AB},b^A,\kappa,\Omega\right)$ be a regular $4$-tuple in the sense of Definition~\ref{Mreg}, $k > 0$, $w_{C_1\cdots C_kAB}$ be a $(0,k+2)$ tensor which is symmetric and trace-free in the $AB$ indices, and $y$ be a $C^1$ function on $\mathbb{S}^2$. Then we have 
\begin{align*}
&\int_{\mathbb{S}^2}y\Bigg[\mathcal{L}_bw_{C_1\cdots C_kAB} -\frac{1}{2}\left(\slashed{\nabla}\hat{\otimes}b\right)^D_{\ \ A}w_{C_1\cdots C_kBD}
\\ \nonumber &\qquad  - \frac{1}{2}\left(\slashed{\nabla}\hat{\otimes}b\right)^D_{\ \ B}w_{C_1\cdots C_kAD} - \frac{1}{2}\left(\slashed{\rm div}b\right) w_{C_1\cdots C_kAB}\Bigg]w^{C_1\cdots C_kAB}
\\ \nonumber &=\int_{\mathbb{S}^2}\left[\frac{1}{2}y\sum_{i=1}^k\left(\slashed{\nabla}\hat{\otimes}b\right)^D_{\ \ E}w_{C_1\cdots D\cdots C_kAB}w^{C_1\cdots E\cdots C_kAB} + \frac{k}{2}y\slashed{\rm div}b\left|w\right|^2 - \frac{1}{2}\left(\mathcal{L}_by\right)\left|w\right|^2\right].
\end{align*}
\end{lemma}
\begin{proof}This is proven in a similar fashion as Lemma~\ref{thatscrlisantisymm}.
\end{proof}

The following consequence of Lemma~\ref{comparethespaces} is useful.
\begin{lemma}\label{sobcomparable}Let $\left(\slashed{g}_{AB},b^A,\kappa,\Omega\right)$ be a regular $4$-tuple in the sense of Definition~\ref{Mreg}. From the metric $\slashed{g}_{AB}$ we may define Sobolev spaces $H^i$ on tensor fields. For any tensor field $w_{A_1\cdots A_k} \in \mathcal{T}^{(0,k)}$, we have
\begin{equation}\label{toproveyayainsobcomplemm}
\left\vert\left\vert w\right\vert\right\vert_{H^i} \sim_{s,k} \left\vert\left\vert w\right\vert\right\vert_{\mathring{H}^i},\qquad i = 0,1,\cdots , N_0.
\end{equation}
\end{lemma}

Now we are ready for the analogue of Proposition~\ref{qregexists}.
\begin{proposition}\label{qregkapexists}Let $\left(\slashed{g}_{AB},b^A,\kappa,\Omega\right)$ be a regular $4$-tuple in the sense of Definition~\ref{Mreg}. Let $q > 0$, $H_{AB}  \in \mathring{H}^2\left(\mathcal{S}^{(2)}\left(\mathbb{S}^2\right)\right)$. Then there exists a unique $t_{AB}^{(q)} \in \mathring{H}^4\left(\mathcal{S}^{(2)}\left(\mathbb{S}^2\right)\right)$ which solves
\begin{equation}\label{thisisthekappasingularequationqReg22}
\mathcal{L}_bt^{(q)}_{AB} -2\kappa t^{(q)}_{AB} - \left(\slashed{\nabla}\hat{\otimes}b\right)^C_{\ \ (A}t^{(q)}_{B)C} - \frac{1}{2}\slashed{\rm div}b t^{(q)}_{AB} +q\slashed{\Delta}t^{(q)}_{AB}= H_{AB}.
\end{equation}
\end{proposition}

\begin{proof}Let's write~\eqref{thisisthekappasingularequationqReg22} as
\[L^{(q)}t^{(q)}_{AB} = H_{AB}.\]
By standard $L^2$-elliptic theory, in order to prove the proposition, it suffices to show that ${\rm ker}\left(L^{(q)}\right) = \{0\}$ and ${\rm ker}\left(\left(L^{(q)}\right)^*\right) = \{0\}$. We start with $L^{(q)}$. Suppose that we have a trace-free symmetric tensor $w_{AB}$ solving
\begin{equation}\label{wisinkerofkap}
\mathcal{L}_bw_{AB} -2\kappa w_{AB} - \left(\slashed{\nabla}\hat{\otimes}b\right)^C_{\ \ (A}w_{B)C} - \frac{1}{2}\slashed{\rm div}b w_{AB} +q\slashed{\Delta}w_{AB}= 0.
\end{equation}
Keeping Lemmas~\ref{thatscrlisantisymm} and~\ref{kappaisthis} in mind, we take the inner product of~\eqref{wisinkerofkap} with $-w_{AB}$ and integrate by parts. We end up with 
\begin{equation}\label{nonononokernel}
\int_{\mathbb{S}^2}\left[\left(\epsilon^2-O\left(\epsilon^{3-\delta}\right)\right)\left|w\right|_{\slashed{g}}^2 +q\left|\slashed{\nabla}w\right|_{\slashed{g}}^2\right]\slashed{dVol}= 0.
\end{equation}
Assuming that $\epsilon$ is small enough, we conclude that $w_{AB} = 0$. The proof that ${\rm ker}\left(\left(L^{(q)}\right)^*\right) = \{0\}$ is analogous. 
\end{proof}

In the next lemma we collect various estimates which will be used to control commutators later.
\begin{lemma}\label{bunchofboundsbahckd}Let $\left(\slashed{g}_{AB},b^A,\kappa,\Omega\right)$ be a regular $4$-tuple in the sense of Definition~\ref{Mreg} and let $w_{A_1\cdots A_k}$ be $(0,k)$-tensor. Then
\begin{equation}\label{omsmsjsjksoqjnojq1}
\left\vert\left\vert \left[\mathcal{L}_{\partial_{\phi}},\mathcal{L}_b\right]w\right\vert\right\vert_{\mathring{H}^i} \lesssim_i \epsilon^{M_1/2}\left\vert\left\vert w\right\vert\right\vert_{\mathring{H}^{i+1}},\qquad 0 \leq i \leq \frac{N_0}{2},
\end{equation}
\begin{equation}\label{omsmsjsjksoqjnojq2}
\left\vert\left\vert \mathring{\nabla}\hat{\otimes}\left(\mathcal{L}_{\partial_\phi}b\right)w\right\vert\right\vert_{\mathring{H}^i} +\left\vert\left\vert \left(\mathcal{L}_{\partial_\phi}\slashed{\rm div}b\right)w\right\vert\right\vert_{\mathring{H}^i}\lesssim_i \epsilon^{M_1/2}\left\vert\left\vert w\right\vert\right\vert_{\mathring{H}^i},\qquad 0 \leq i \leq \frac{N_0}{2},
\end{equation}
\begin{equation}\label{omsmsjsjksoqjnojq3}
\left\vert\left\vert\left( \slashed{\nabla}\left(\slashed{\nabla}\hat{\otimes}b\right)\right)w\right\vert\right\vert_{\mathring{H}^i} \lesssim_i \epsilon^{1-\delta}\left\vert\left\vert w\right\vert\right\vert_{\mathring{H}^i},\qquad 0 \leq i \leq \frac{N_0}{2},
\end{equation}
\begin{equation}\label{omsmsjsjksoqjnojq4}
\left\vert\left\vert\left( \slashed{\nabla}\slashed{\rm div}b\right)w\right\vert\right\vert_{\mathring{H}^i} \lesssim_i \epsilon^{1-\delta}\left\vert\left\vert w\right\vert\right\vert_{\mathring{H}^i},\qquad 0 \leq i \leq \frac{N_0}{2},
\end{equation}
\begin{equation}\label{omsmsjsjksoqjnojq5}
\left\vert\left\vert\left[\slashed{\nabla},\slashed{\Delta}\right]w\right\vert\right\vert_{\mathring{H}^i} \lesssim_i \left\vert\left\vert w\right\vert\right\vert_{\mathring{H}^{i+1}},\qquad 0 \leq i \leq \frac{N_0}{2}, 
\end{equation}
\begin{equation}\label{omsmsjsjksoqjnojq6}
\left\vert\left\vert\left[\slashed{\nabla},\mathscr{L}\right]w\right\vert\right\vert_{\mathring{H}^i} \lesssim_i \epsilon^{1-\delta}\left\vert\left\vert w\right\vert\right\vert_{\mathring{H}^i},\qquad 0 \leq i \leq \frac{N_0}{2},
\end{equation}
\begin{equation}\label{omsmsjsjksoqjnojq6234}
\left\vert\left\vert\left[\slashed{\Delta},\mathscr{L}\right]w\right\vert\right\vert_{\mathring{H}^i} \lesssim_i \epsilon^{1-\delta} \left\vert\left\vert w\right\vert\right\vert_{\mathring{H}^{i+2}},\qquad 0 \leq i \leq \frac{N_0}{2},
\end{equation}
\begin{equation}\label{omsmsjsjksoqjnojq7}
\left\vert\left\vert\left[\mathcal{L}_{\partial_\phi},\slashed{\Delta}\right]w\right\vert\right\vert_{\mathring{H}^i} \lesssim_i \epsilon^{M_1/2}\left\vert\left\vert w\right\vert\right\vert_{\mathring{H}^{i+2}},\qquad 0 \leq i \leq \frac{N_0}{2}.
\end{equation}
\end{lemma}
\begin{proof}These all follow in a straightforward fashion from Sobolev inequalities.
\end{proof}

In the next proposition, we will show that we can take the limit as $q \to 0$ for the solutions produced by Proposition~\ref{qregkapexists}; however, the estimates for higher derivatives of $t_{AB}$ that we get will have a bad dependence on $\epsilon$. 
\begin{proposition}\label{kapissingbutthereisasolnatleast}Let $\left(\slashed{g}_{AB},b^A,\kappa,\Omega\right)$ be a regular $4$-tuple in the sense of Definition~\ref{Mreg}, and 
$H_{AB} \in \mathring{H}^j\left(\hat{\mathcal{S}}\left(\mathbb{S}^2\right)\right)$ for $1 \leq j \leq \frac{N_0}{2}$. Then then there exists a unique $t_{AB} \in \mathring{H}^j\left(\hat{\mathcal{S}}\left(\mathbb{S}^2\right)\right)$ solving
\begin{equation}\label{kappamustbeasolntothisoritwouldbebad}
\mathcal{L}_bt_{AB} -2\kappa t_{AB} - \left(\slashed{\nabla}\hat{\otimes}b\right)^C_{\ \ (A}t_{B)C} - \frac{1}{2}\slashed{\rm div}b t_{AB} +2\left(\mathcal{L}_b\log\Omega\right) t_{AB}= H_{AB}.
\end{equation}
Moreover,  we have the following estimates for $t$:
\begin{equation}\label{odkodwkodwkodwponq}
\epsilon^4\sum_{k=0}^j\int_{\mathbb{S}^2}\left|\slashed{\nabla}^kt\right|_{\slashed{g}}^2\slashed{dVol} \lesssim_j \epsilon^{-(2+\delta)j}\sum_{k=0}^j\sum_{i=0}^{j-k}\epsilon^{-(2+2\delta)k}\int_{\mathbb{S}^2}\left|\slashed{\nabla}^i\mathcal{L}_{\partial_{\phi}}^kH\right|_{\slashed{g}}^2\slashed{dVol}.
\end{equation}

If we take a $\mathcal{L}_{\partial_{\phi}}$ derivative, then we have the alternative estimate:
\begin{equation}\label{odkodwkodwkodw230019122}
\epsilon^4\sum_{k=0}^{j-1}\int_{\mathbb{S}^2}\left|\slashed{\nabla}^k\mathcal{L}_{\partial_{\phi}}t\right|_{\slashed{g}}^2\slashed{dVol} \lesssim_j \sum_{k=0}^{j-1}\sum_{i=0}^{j-k}\epsilon^{-(2+2\delta)k}\int_{\mathbb{S}^2}\left|\slashed{\nabla}^i\mathcal{L}_{\partial_{\phi}}^k\mathcal{L}_{\partial_{\phi}}H\right|_{\slashed{g}}^2\slashed{dVol} + \epsilon^{M_0/10}\sum_{k=0}^j\int_{\mathbb{S}^2}\left|\slashed{\nabla}^kH\right|_{\slashed{g}}^2\slashed{dVol}.
\end{equation}

\end{proposition}
\begin{proof}Before entering into the details we give a sketch of the proof. First of all, multiplying the equation through by $\Omega^2$ leads to the equation 
\begin{equation*}
\mathcal{L}_b\left(\Omega^2t\right)_{AB} -2\kappa \Omega^2 t_{AB} - \left(\slashed{\nabla}\hat{\otimes}b\right)^C_{\ \ (A}\left(\Omega^2t\right)_{B)C} - \frac{1}{2}\slashed{\rm div}b \left(\Omega^2t\right)_{AB} = \Omega^2H_{AB}.
\end{equation*}
In view of the bounds~\eqref{somebound1ssssss} it is clear that we can assume, without loss of generality, that $\Omega^2 = 1$. Now we describe the strategy for the estimates. For each $q > 0$,  we may appeal to Proposition~\ref{qregkapexists} to  produce a solution $t^{(q)}_{AB}$ to the equation~\eqref{thisisthekappasingularequationqReg22}. Our plan will be to show that there exists $t_{AB} = \lim_{q\to 0}t^{(q)}_{AB}$ which solves~\eqref{kappamustbeasolntothisoritwouldbebad} and satisfies the desired estimates. In what follows, we assume that $q \ll \epsilon^{100}$. As usual in such an argument, the fundamental challenge is to prove estimates for $t^{(q)}_{AB}$ which are independent of $q$. In particular, we cannot exploit the ellipticity coming from the $q\slashed{\Delta}$ term. 

The  basic estimate at our disposal is the one obtained by contracting the equation~\eqref{thisisthekappasingularequationqReg22} with $t^{(q)}_{AB}$ and using the anti-symmetric structure. (This is what was exploited to obtain the estimate~\eqref{nonononokernel} in Lemma~\ref{qregkapexists}.) This leads to an estimate for $\epsilon^2\left\vert\left\vert t^{(q)}\right\vert\right\vert_{L^2}$. 

In order to obtain higher order estimates for $t_{AB}^{(q)}$ we will need to commute the equation. However, naive commutation produces an equation where the anti-symmetric structure from~\eqref{kappamustbeasolntothisoritwouldbebad} is completely destroyed. Thus, we will have to design a careful scheme for commuting. We begin by exploiting the fact that the coefficients in our equation are almost axisymmetric in that the commutators with $\mathcal{L}_{\partial_{\phi}}$ can be controlled by large powers of $\epsilon$. (See Lemma~\ref{bunchofboundsbahckd}.) Thus, we may commute with $\mathcal{L}_{\partial_{\phi}}$ and repeat the basic $L^2$-estimate to control $\epsilon^2\left\vert\left\vert \mathcal{L}_{\partial_{\phi}}t^{(q)}\right\vert\right\vert_{L^2}$ in terms of quantities we already control and a large power of $\epsilon$ times $\left\vert\left\vert \mathring{\nabla}t^{(q)}\right\vert\right\vert_{L^2}$. Next, it is also natural to commute with $\mathscr{L}$. Noting that we have, schematically, 
\begin{equation}\label{schematicmathscrL}
\mathscr{L}t_{AB}^{(q)} \sim \epsilon^2h(\theta)\mathcal{L}_{\partial_{\theta}}t^{(q)}_{AB} + O\left(\epsilon\right)\mathcal{L}_{\partial_{\phi}}t_{AB}^{(q)} + O\left(\epsilon\right)t_{AB}^{(q)} + O\left(\epsilon^3\left|\mathring{\nabla}t\right|\right),
\end{equation}
we may use the control of $\mathscr{L}t_{AB}^{(q)}$, $\mathcal{L}_{\partial_{\phi}}t_{AB}^{(q)}$, and $t_{AB}^{(q)}$ together to control $|h(\theta)|\mathcal{L}_{\partial_{\theta}}t_{AB}^{(q)}$. Away from small neighborhoods of $\theta = 0$, $y_0$, and $\pi$, this suffices to control $\mathring{\nabla}_At_{BC}^{(q)}$ (with a loss of $\epsilon^{-\delta}$). In order to cure this degeneration we commute the equation~\eqref{kappamustbeasolntothisoritwouldbebad} with $\slashed{\nabla}_A$. The resulting equation does not have the exact anti-symmetric structure that we used before, but when we carry out localized estimates near $\theta = 0$, $y_0$, and $\pi$ the additional terms which show up may be analyzed and turn out to be controllable in terms of previously controlled quantities. Putting this all together leads to an estimate for $\mathring{\nabla}_At_{BC}^{(q)}$. A similar strategy then allows for additional commutations. Given this control, it is straightforward to prove that $t_{AB}^{(q)}$ converges to a unique limit $t_{AB}$ as $q\to 0$. One unavoidable downside of the technique we will use is that the identity~\eqref{schematicmathscrL} only allows for us to control $\epsilon^2\mathcal{L}_{\partial_{\theta}}t_{AB}^{(q)}$ in terms of terms $O\left(\epsilon\right)t_{AB}^{(q)}$ and $O\left(\epsilon\right)\mathcal{L}_{\partial_{\phi}}t_{AB}^{(q)}$. Similarly, when we commute with $\slashed{\nabla}_A$ there are lower order terms produced of the form $O\left(\epsilon^{1-\delta}\right)t_{AB}^{(q)}$. These terms lead to the fact that the estimate~\eqref{odkodwkodwkodwponq} loses additional powers of $\epsilon$ every time we consider an extra derivative.

We now turn to the details. First of all, let us agree to the convention that throughout this proof \emph{unless noted otherwise all norms are computed with respect to $\slashed{g}$ and all volume forms are computed with respect to $\slashed{g}_{AB}$}. Now let $q \ll \epsilon^{100}$ and, using Proposition~\ref{qregkapexists}, let $t_{AB}^{(q)}$ solve~\eqref{thisisthekappasingularequationqReg22}. Contracting~\eqref{thisisthekappasingularequationqReg22} with $-\epsilon^2t_{AB}^{(q)}$, integrating by parts, and using Lemma~\ref{thatscrlisantisymm} leads to 
\begin{equation}\label{thefirstenestforfq}
\int_{\mathbb{S}^2}\left[\epsilon^4\left|t^{(q)}\right|^2 + q\epsilon^2\left|\slashed{\nabla}t^{(q)}\right|^2\right] \lesssim \epsilon^2\int_{\mathbb{S}^2}H\cdot t^{(q)} \Rightarrow 
\end{equation}
\begin{equation}\label{thefirstenestforfq2}
\int_{\mathbb{S}^2}\left[\epsilon^4\left|t^{(q)}\right|^2 + q\epsilon^2\left|\slashed{\nabla}t^{(q)}\right|^2\right] \lesssim \int_{\mathbb{S}^2}\left|H\right|^2.
\end{equation}

Next, we want to commute $\mathcal{L}_{\partial_{\phi}}^i$ (for $i \leq \frac{N_0}{2}$) through our equation for $t^{(q)}$. Recalling that we may write our equation as
\begin{equation}\label{thbsoiacsissk}
\mathscr{L}t^{(q)}_{AB} - 2\kappa t^{(q)}_{AB} + q\slashed{\Delta} t^{(q)}_{AB} = H_{AB}^{(q)},
\end{equation}
we obtain
\begin{align}\label{partphi1ia}
&\mathscr{L}\left(\mathcal{L}_{\partial_{\phi}}^it^{(q)}\right)_{AB} - 2\kappa \left(\mathcal{L}_{\partial_{\phi}}^it^{(q)}\right)_{AB} +q\slashed{\Delta}\mathcal{L}_{\partial_{\phi}}^it^{(q)}_{AB}
\\ \nonumber &\qquad \qquad + \sum_{j+k = i-1}\mathcal{L}_{\partial_{\phi}}^j\left(\left[\mathcal{L}_{\partial_{\phi}},\mathscr{L}+q\slashed{\Delta}\right]\mathcal{L}_{\partial_{\phi}}^kt^{(q)}\right)_{AB} = \mathcal{L}_{\partial_{\phi}}^iH_{AB}.
\end{align}
The same integration by parts which leads to~\eqref{thefirstenestforfq2} yields
\begin{align}\label{purephiestimateassss}
\int_{\mathbb{S}^2}\left[\epsilon^4\left|\mathcal{L}_{\partial_{\phi}}^it^{(q)}\right|^2 + q\epsilon^2\left|\slashed{\nabla}\mathcal{L}_{\partial_{\phi}}^it^{(q)}\right|^2\right] \lesssim \int_{\mathbb{S}^2}\left[\left|\mathcal{L}_{\partial_{\phi}}^iH\right|^2 +  \sum_{j+k = i-1}\left|\mathcal{L}_{\partial_{\phi}}^j\left(\left[\mathcal{L}_{\partial_{\phi}},\mathscr{L}+q\slashed{\Delta}\right]\mathcal{L}_{\partial_{\phi}}^kt^{(q)}\right) \right|^2\right].
\end{align}
Let's now examine the second term of the right hand side of~\eqref{purephiestimateassss} a little more closely. We start by observing that Lemma~\ref{conformal} implies 
\begin{equation}\label{whenyoudotheconfmchangetonablbterm2}
\slashed{\nabla}\hat{\otimes}b^C_{\ \ (A}t^{(q)}_{B)C} = \mathring{\slashed{g}}^{CD}\left(\mathring{\nabla}\hat{\otimes}b\right)_{C(A}t^{(q)}_{B)D}.
\end{equation}
In particular
\begin{equation}\label{0oknqqoqoqo}
\left[\mathcal{L}_{\partial_{\phi}},\mathscr{L}\right]w_{AB} = \left[\mathcal{L}_{\partial_{\phi}},\mathcal{L}_b\right]w_{AB} -\mathring{\slashed{g}}^{CD}\mathring{\nabla}\hat{\otimes}\left(\mathcal{L}_{\partial_{\phi}}b\right)_{C(A}w_{B)D}-\frac{1}{2}\left(\mathcal{L}_{\partial_{\phi}}\slashed{\rm div}b\right)w_{AB}.
\end{equation}
Thus, for $0 \leq i \leq \frac{N_0}{2}$, Lemma~\ref{bunchofboundsbahckd} and~\eqref{purephiestimateassss} lead to 
\begin{align}\label{purephiestimateassss29387}
&\int_{\mathbb{S}^2}\left[\epsilon^4\left|\mathcal{L}_{\partial_{\phi}}^it^{(q)}\right|^2 + q\epsilon^2\left|\slashed{\nabla}\mathcal{L}_{\partial_{\phi}}^it^{(q)}\right|^2\right] \lesssim_i 
\\ \nonumber &\qquad \int_{\mathbb{S}^2}\left[\left|\mathcal{L}_{\partial_{\phi}}^iH\right|^2 +\sum_{j=0}^i\left(\epsilon^{M_1/2}\left|\slashed{\nabla}^jt^{(q)}\right|^2+q\epsilon^{M_1/2}\left|\slashed{\nabla}^{j+1}t^{(q)}\right|^2\right)\right].
\end{align}
Next we commute~\eqref{partphi1ia} with $\slashed{\nabla}^l_{C_1\cdots C_l}$. We obtain 
\begin{align}\label{partphi1ia2b3c4d}
&\mathscr{L}\left(\slashed{\nabla}_{C_1\cdots C_l}^l\mathcal{L}_{\partial_{\phi}}^it_{AB}^{(q)}\right) - 2\kappa \left(\slashed{\nabla}_{C_1\cdots C_l}^l\mathcal{L}_{\partial_{\phi}}^it_{AB}^{(q)}\right) +q\slashed{\Delta}\left(\slashed{\nabla}_{C_1\cdots C_l}^l\mathcal{L}_{\partial_{\phi}}^it_{AB}^{(q)}\right) = 
\\ \nonumber &\qquad  - \sum_{p+r = l-1}\slashed{\nabla}_{C_1\cdots C_p}^p\left(\left[\slashed{\nabla}_{C_{p+1}},\mathscr{L}+q\slashed{\Delta}\right]\slashed{\nabla}_{C_{p+2}\cdots C_l}^r\mathcal{L}_{\partial_{\phi}}^it_{AB}^{(q)}\right)
\\ \nonumber &\qquad - \sum_{j+k = i-1}\slashed{\nabla}_{C_1\cdots C_l}^l\mathcal{L}_{\partial_{\phi}}^j\left(\left[\mathcal{L}_{\partial_{\phi}},\mathscr{L}+q\slashed{\Delta}\right]\mathcal{L}_{\partial_{\phi}}^it_{AB}^{(q)}\right) + \slashed{\nabla}_{C_1\cdots C_l}^l\mathcal{L}_{\partial_{\phi}}^iH_{AB}.
\end{align}

Let $\chi_0\left(\theta\right)$ be a non-negative cut-off function which is identically $1$ when $\theta \in \{[0,\gamma/4] \cup [\pi-\gamma/4,\pi]\}$, is identically $0$ for $\theta \in [\gamma/2,\pi-\gamma/2]$, and satisfies $\left|\frac{d\chi_1}{d\theta}\right| \lesssim \gamma^{-1}$. Now, for $0 \leq i + l \leq \frac{N_0}{2}$, we contract with $-\chi_0(\theta)\epsilon^2\slashed{\nabla}_{C_1\cdots C_l}^l\mathcal{L}_{\partial_{\phi}}^it_{AB}^{(q)}$ and integrate by parts. We start the analysis by using~\eqref{lkadojagdjagri} and Lemma~\ref{higherorderalmostantisymm} to see that
\begin{align}\label{omqnqjkwww}
\epsilon^2\int_{\mathbb{S}^2}\chi_0\mathscr{L}\left(\slashed{\nabla}^l\mathcal{L}_{\partial_{\phi}}^it^{(q)}\right)\cdot \slashed{\nabla}^l\mathcal{L}_{\partial_{\phi}}^it^{(q)} & \lesssim \epsilon^2\int_{\mathbb{S}^2}\left(l\left|\chi_0\right|\gamma^2\epsilon^2+\left|\mathcal{L}_b\chi_0\right|\right)\left|\slashed{\nabla}^l\mathcal{L}_{\partial_{\phi}}^it^{(q)}\right|^2
\\ \nonumber &\lesssim_l \epsilon^4\gamma^2\int_{\theta \in [0,\gamma/2] \cup [\pi-\gamma/2,\pi]}\left|\slashed{\nabla}^l\mathcal{L}_{\partial_{\phi}}^it^{(q)}\right|^2
\\ \nonumber &\qquad  + \epsilon^{4-\delta}\int_{\theta \in [\gamma/4,\gamma/2] \cup [\pi-\gamma/2,\pi-\gamma/4]}\left|\slashed{\nabla}^l\mathcal{L}_{\partial_{\phi}}^it^{(q)}\right|^2.
\end{align}
In the last line we have used~\eqref{itvanishesatthepolesprettyfast}, the fact that $\tilde b^A$ vanishes for $\theta \leq \gamma$, Sobolev inequalities, Lemma~\ref{eestimates}, and Lemma~\ref{bunchofboundsbahckd}. With~\eqref{omqnqjkwww} established, we now contract~\eqref{partphi1ia2b3c4d} with $-\chi_0(\theta)\epsilon^2\slashed{\nabla}_{C_1\cdots C_l}^l\mathcal{L}_{\partial_{\phi}}^it_{AB}^{(q)}$ and integrate by parts to eventually obtain
\begin{align}\label{owkdowkdokwjw}
&\int_{\mathbb{S}^2}\chi_0\left(\epsilon^4\left|\slashed{\nabla}^l\mathcal{L}_{\partial_{\phi}}^it^{(q)}\right|^2 + q\epsilon^2\left|\slashed{\nabla}^{l+1}\mathcal{L}_{\partial_{\phi}}^it^{(q)}\right|^2\right) \lesssim_{l+i}
\\ \nonumber &\qquad \epsilon^2\left|\int_{\mathbb{S}^2}\chi_0\mathscr{L}\left(\slashed{\nabla}^l\mathcal{L}_{\partial_{\phi}}^it^{(q)}\right)\cdot \slashed{\nabla}^l\mathcal{L}_{\partial_{\phi}}^it^{(q)}\right| + q\epsilon^2\int_{\mathbb{S}^2}\left|\slashed{\nabla}^2\chi_0\right|\left|\slashed{\nabla}^l\mathcal{L}_{\partial_{\phi}}^it^{(q)}\right|^2 +\epsilon^2 \int_{\mathbb{S}^2}\chi_0\left|\slashed{\nabla}^l\mathcal{L}_{\partial_{\phi}}^iH\right|\left|\slashed{\nabla}^l\mathcal{L}_{\partial_{\phi}}^it^{(q)}\right|
\\ \nonumber &\qquad +\epsilon^2\int_{\mathbb{S}^2}\chi_0\Bigg|\sum_{p+r = l-1}\slashed{\nabla}^p\left(\left[\slashed{\nabla},\mathscr{L}+q\slashed{\Delta}\right]\slashed{\nabla}^r\mathcal{L}_{\partial_{\phi}}^it^{(q)}\right)+
\\ \nonumber &\qquad \qquad \qquad + \sum_{j+k = i-1}\slashed{\nabla}^l\mathcal{L}_{\partial_{\phi}}^j\left(\left[\mathcal{L}_{\partial_{\phi}},\mathscr{L}+q\slashed{\Delta}\right]\mathcal{L}_{\partial_{\phi}}^kt^{(q)}\right)\Bigg|\left|\slashed{\nabla}^l\mathcal{L}_{\partial_{\phi}}^it^{(q)}\right|.
\end{align}
Using Lemma~\ref{bunchofboundsbahckd} to estimate the commutator term $\slashed{\nabla}_{C_1\cdots C_p}^p\left(\left[\slashed{\nabla}_{C_{p+1}},\mathscr{L}+q\slashed{\Delta}\right]\slashed{\nabla}_{C_{p+2}\cdots C_l}^r\mathcal{L}_{\partial_{\phi}}^it_{AB}^{(q)}\right)$ and arguing like we did above to control the other commutator term, the estimate~\eqref{owkdowkdokwjw} and~\eqref{omqnqjkwww} eventually imply the following:
\begin{align}\label{onqmqhqoq}
&\int_{\theta \in [0,\gamma/4]\cup [\pi-\gamma/4,\pi]}\left(\epsilon^4\left|\slashed{\nabla}^l\mathcal{L}_{\partial_{\phi}}^it^{(q)}\right|^2 + q\epsilon^2\left|\slashed{\nabla}^{l+1}\mathcal{L}_{\partial_{\phi}}^it^{(q)}\right|^2\right) \lesssim_{l+i} 
\\ \nonumber &\qquad \int_{\theta \in [\gamma/4,\gamma/2] \cup [\pi-\gamma/2,\pi-\gamma/4]}\left(\epsilon^{4-\delta}\left|\slashed{\nabla}^l\mathcal{L}_{\partial_{\phi}}^it^{(q)}\right|^2 + q\epsilon^{2-\delta}\left|\slashed{\nabla}^l\mathcal{L}_{\partial_{\phi}}^it^{(q)}\right|^2\right)
\\ \nonumber &\qquad +\int_{\theta \in [0,\gamma/2] \cup [\pi-\gamma/2,\pi]}\left(\sum_{0 \leq j \leq l-1}\left(\epsilon^{2-\delta}\left|\slashed{\nabla}^j\mathcal{L}_{\partial_{\phi}}^it^{(q)}\right|^2 + q^2\epsilon^{-\delta}\left|\slashed{\nabla}^{j+1}\mathcal{L}_{\partial_{\phi}}^it^{(q)}\right|^2\right)  +\left|\slashed{\nabla}^l\mathcal{L}_{\partial_{\phi}}^iH\right|^2\right)
\\ \nonumber &\qquad +\sum_{j=0}^{l+i}\int_{\mathbb{S}^2}\left(\epsilon^{M_1/2}\left|\slashed{\nabla}^jt^{(q)}\right|^2+q\epsilon^{M_1/2}\left|\slashed{\nabla}^{j+1}t^{(q)}\right|^2\right)
\end{align}
This estimate will be used to control $\slashed{\nabla}_{C_1\cdots C_l}^l\mathcal{L}_{\partial_{\phi}}^it_{AB}^{(q)}$ near $\theta = 0$ and $\theta = \pi$ in terms of  $\slashed{\nabla}_{C_1\cdots C_l}^l\mathcal{L}_{\partial_{\phi}}^it_{AB}^{(q)}$ for $\theta$ near $[\gamma/4,\gamma/2]$ and $[\pi-\gamma/2,\pi-\gamma/4]$ and in terms of $\slashed{\nabla}_{C_1\cdots C_{l-1}}^{l-1}\mathcal{L}_{\partial_{\phi}}^it_{AB}^{(q)}$.

Now we will discuss another localized higher order estimate. This time we will localize near $\theta = y_0$.  Let $\chi_1(\theta)$ be a non-negative function which is identically $1$ for $|\theta-y_0| \leq \gamma$, identically $0$ for $|\theta-y_0| \geq 2\gamma$, and satisfies $\left|\partial_{\theta}\chi_1\right| \lesssim \gamma^{-1}$. We now turn to an estimate which will be the analogue of~\eqref{omqnqjkwww}. First of all, for $0 \leq i + l \leq \frac{N_0}{2}$, we have
\begin{align}\label{omnbqqwiwiwiw}
&\left|\int_{\mathbb{S}^2}\sum_{i=1}^l\chi_1\left(\slashed{\nabla}\hat{\otimes}b\right)^D_{\ \ E}\slashed{\nabla}_{C_1}\cdots\slashed{\nabla}_D\cdots\slashed{\nabla}_{C_l}\mathcal{L}_{\partial_{\phi}}^it^{(q)}_{AB}\slashed{\nabla}^{C_1}\cdots\slashed{\nabla}^E\cdots\slashed{\nabla}^{C_l}\left(\mathcal{L}_{\partial_{\phi}}^it^{(q)}\right)^{AB}\right| \lesssim_l 
\\ \nonumber &\qquad \epsilon\int_{\mathbb{S}^2}\chi_1\left|\slashed{\nabla}_{\partial_{\phi}}\slashed{\nabla}^{l-1}\mathcal{L}_{\partial_{\phi}}^it^{(q)}\right|\left|\slashed{\nabla}_{\partial_{\theta}}\slashed{\nabla}^{l-1}\mathcal{L}_{\partial_{\phi}}^it^{(q)}\right|
\\ \nonumber &\qquad +\left|\int_{\mathbb{S}^2}\sum_{i=1}^l\chi_1\left(\slashed{\nabla}\hat{\otimes}\left(b-\tilde b\right)\right)^D_{\ \ E}\slashed{\nabla}_{C_1}\cdots\slashed{\nabla}_D\cdots\slashed{\nabla}_{C_l}\mathcal{L}_{\partial_{\phi}}^it^{(q)}_{AB}\slashed{\nabla}^{C_1}\cdots\slashed{\nabla}^E\cdots\slashed{\nabla}^{C_l}\left(\mathcal{L}_{\partial_{\phi}}^it^{(q)}\right)^{AB}\right| \lesssim
\\ \nonumber &\int_{\mathbb{S}^2}\chi_1\left(\epsilon\left|\slashed{\nabla}_{\partial_{\phi}}\slashed{\nabla}^{l-1}\mathcal{L}_{\partial_{\phi}}^it^{(q)}\right|\left|\slashed{\nabla}_{\partial_{\theta}}\slashed{\nabla}^{l-1}\mathcal{L}_{\partial_{\phi}}^it^{(q)}\right|+\gamma^2\epsilon^2\left|\slashed{\nabla}^l\mathcal{L}_{\partial_{\phi}}^it^{(q)}\right|^2\right).
\end{align}
In the last line we have used~\eqref{itvanishesatthepolesprettyfast2}, Sobolev inequalities, and Lemma~\ref{eestimates}. We may combine~\eqref{omnbqqwiwiwiw} with Lemma~\ref{higherorderalmostantisymm} to obtain the following analogue of~\eqref{omqnqjkwww}:
\begin{align}\label{omqnqjkwww2}
&\epsilon^2\left|\int_{\mathbb{S}^2}\chi_1\mathscr{L}\left(\slashed{\nabla}^l\mathcal{L}_{\partial_{\phi}}^it^{(q)}\right)\cdot \slashed{\nabla}^l\mathcal{L}_{\partial_{\phi}}^it^{(q)}\right|  \lesssim_{l+i} 
\\ \nonumber &\qquad \int_{y_0-\gamma \leq \theta \leq y_0+\gamma}\left(\epsilon^{2-\delta}\left|\slashed{\nabla}_{\partial_{\phi}}\slashed{\nabla}^{l-1}\mathcal{L}_{\partial_{\phi}}^it^{(q)}\right|^2+\gamma^2\epsilon^4\left|\slashed{\nabla}^l\mathcal{L}_{\partial_{\phi}}^it^{(q)}\right|^2\right)+
\\ \nonumber &\qquad \int_{\theta \in [y_0-2\gamma,y_0-\gamma] \cup [y_0+\gamma,y_0+2\gamma]}\epsilon^{4-\delta}\left|\slashed{\nabla}^l\mathcal{L}_{\partial_{\phi}}^it^{(q)}\right|^2.
\end{align}
Now we establish an analogue of~\eqref{onqmqhqoq} by contracting~\eqref{partphi1ia2b3c4d} with $-\chi_1(\theta)\epsilon^2\slashed{\nabla}^l_{C_1\cdots C_l}\mathcal{L}_{\partial_{\phi}}^it_{AB}^{(q)}$ and integrating by parts. For $0 \leq i + l \leq \frac{N_0}{2}$ we have
\begin{align}\label{onqmqhqoq2}
&\int_{y_0-\gamma \leq \theta \leq y_0 + \gamma}\left(\epsilon^4\left|\slashed{\nabla}^l\mathcal{L}_{\partial_{\phi}}^it^{(q)}\right|^2 + q\epsilon^2\left|\slashed{\nabla}^{l+1}\mathcal{L}_{\partial_{\phi}}^it^{(q)}\right|^2\right) \lesssim_{i+l}
\\ \nonumber &\qquad  \int_{y_0-\gamma \leq \theta \leq y_0+\gamma}\epsilon^{2-\delta}\left|\slashed{\nabla}_{\partial_{\phi}}\slashed{\nabla}^{l-1}\mathcal{L}_{\partial_{\phi}}^it^{(q)}\right|^2
\\ \nonumber &\qquad + \int_{\theta \in [y_0-2\gamma,y_0-\gamma] \cup [y_0+\gamma,y_0+2\gamma]}\left(\epsilon^{4-\delta}\left|\slashed{\nabla}^l\mathcal{L}_{\partial_{\phi}}^it^{(q)}\right|^2+q\epsilon^{2-\delta}\left|\slashed{\nabla}^l\mathcal{L}_{\partial_{\phi}}^it^{(q)}\right|^2\right)
\\ \nonumber &\qquad +\int_{y_0-2\gamma \leq \theta \leq y_0 + 2\gamma}\left(\sum_{0 \leq j \leq l-1}\left(\epsilon^{2-\delta}\left|\slashed{\nabla}^j\mathcal{L}_{\partial_{\phi}}^it^{(q)}\right|^2 + q^2\epsilon^{-\delta}\left|\slashed{\nabla}^{j+1}\mathcal{L}_{\partial_{\phi}}^it^{(q)}\right|^2\right)  +\left|\slashed{\nabla}^l\mathcal{L}_{\partial_{\phi}}^iH\right|^2\right)
\\ \nonumber &\qquad +\sum_{j=0}^{l+i}\int_{\mathbb{S}^2}\left(\epsilon^{M_1/2}\left|\slashed{\nabla}^jt^{(q)}\right|^2+q\epsilon^{M_1/2}\left|\slashed{\nabla}^{j+1}t^{(q)}\right|^2\right).
\end{align}

Before carrying out the next estimate, we observe the following consequence of Lemma~\ref{eestimates} and a Sobolev inequality. For any $(0,k)$-tensor $w_{A_1\cdots A_k}$, we have
\begin{align}\label{mastscrlwraellycontlthis}
\int_{\mathbb{S}^2}\left|\mathscr{L}w\right|^2 \gtrsim \epsilon^4\int_{\mathbb{S^2}}|h(\theta)|^2\left|\mathcal{L}_{\partial_{\theta}}w\right|^2 - \epsilon^{2-\delta}\int_{\mathbb{S}^2}\left[\left|\mathcal{L}_{\partial_{\phi}}w\right|^2+\left|w\right|^2\right] - \epsilon^{6-4\delta}\int_{\mathbb{S}^2}\left|\mathring{\nabla}w\right|^2.
\end{align}
Next, for $0 \leq i+l \leq \frac{N_0}{2}$, we commute~\eqref{partphi1ia} with $\mathscr{L}^l$, then contract with $\left(\epsilon^{-2l}\mathscr{L}^l\mathcal{L}_{\partial_{\phi}}^it_{AB}^{(q)}\right)$, and integrate by parts. Keeping in mind the trivial fact that $\mathscr{L}$ commutes with itself, arguing as above and using Lemma~\ref{bunchofboundsbahckd} leads to the estimate
\begin{align}\label{odkodkodko}
&\int_{\mathbb{S}^2}\left[\epsilon^{4-4l}\left|\mathscr{L}^l\mathcal{L}_{\partial_{\phi}}^it^{(q)}\right|^2 + q\epsilon^{2-4l}\left|\slashed{\nabla}\mathscr{L}^l\mathcal{L}_{\partial_{\phi}}^it^{(q)}\right|^2\right] \lesssim_{l+i} \\ \nonumber &\qquad \int_{\mathbb{S}^2}\left[\epsilon^{-4l}\left|\mathscr{L}^l\mathcal{L}_{\partial_{\phi}}^iH\right|^2 +\sum_{j=0}^{i+l}\left(\epsilon^{M_1/3}\left|\slashed{\nabla}^jt^{(q)}\right|^2+q\epsilon^{M_1/3}\left|\slashed{\nabla}^{j+1}t^{(q)}\right|^2\right)\right].
\end{align}

Combining with~\eqref{mastscrlwraellycontlthis} leads to 
\begin{align}\label{odkodkodko2}
&\int_{\mathbb{S}^2}\left[\epsilon^4\left|h(\theta)\right|^{2l}\left|\mathcal{L}_{\partial_{\theta}}^l\mathcal{L}_{\partial_{\phi}}^it^{(q)}\right|^2 + q\epsilon^2\left|h(\theta)\right|^{2l}\left|\slashed{\nabla}\mathcal{L}_{\partial_{\theta}}^l\mathcal{L}_{\partial_{\phi}}^it^{(q)}\right|^2\right] \lesssim_{l+i} 
\\ \nonumber &\qquad \int_{\mathbb{S}^2}\left[\epsilon^{6-4\delta}\left|\slashed{\nabla}^l\mathcal{L}_{\partial_{\phi}}^it^{(q)}\right|^2 + \sum_{\overset{k+m+n=l}{0 \leq k < l}}\epsilon^{4-2(m+n)}\left|\slashed{\nabla}^k\mathcal{L}_{\partial_{\phi}}^m\mathcal{L}^i_{\partial_{\phi}}t^{(q)}\right|^2\right]+
 \\ \nonumber &\qquad \int_{\mathbb{S}^2}\left[\epsilon^{-4l}\left|\mathscr{L}^l\mathcal{L}_{\partial_{\phi}}^iH\right|^2 +\sum_{j=0}^{i+l}\left(\epsilon^{M_1/3}\left|\slashed{\nabla}^jt^{(q)}\right|^2+q\epsilon^{M_1/3}\left|\slashed{\nabla}^{j+1}t^{(q)}\right|^2\right)\right].
\end{align}

We now have all of the ingredients to prove
\begin{equation}\label{odkodwkodwkodwponq22232}
\sum_{k=0}^j\int_{\mathbb{S}^2}\left[\epsilon^4\left|\slashed{\nabla}^kt^{(q)}\right|^2+q\epsilon^2\left|\slashed{\nabla}^{k+1}t^{(q)}\right|^2\right]\lesssim_j \epsilon^{-(2+\delta)j}\sum_{k=0}^j\sum_{i=0}^{j-k}\epsilon^{-(2+2\delta)k}\int_{\mathbb{S}^2}\left|\slashed{\nabla}^i\mathcal{L}_{\partial_{\phi}}^kH\right|^2.
\end{equation}
In order to do this, we  let $j = 1$, $2$,$\cdots$, or $\lfloor\frac{N_0}{2}\rfloor$, $0 \leq i \leq j$, $0 \leq k \leq i$, and then define 
\[\mathcal{Y}\left(i,j,k\right) \doteq \int_{\mathbb{S}^2}\left[\epsilon^4\left|\slashed{\nabla}^k\mathcal{L}_{\partial_{\phi}}^{j-i}t^{(q)}\right|^2 +q\epsilon^2\left|\slashed{\nabla}^{k+1}\mathcal{L}_{\partial_{\phi}}^{j-i}t^{(q)}\right|^2\right],\]
\[ \tilde{\mathcal{Y}}\left(i,j,k\right) \doteq \int_{\mathbb{S}^2}\left[\epsilon^4\left|h(\theta)\right|^{2k}\left|\mathcal{L}_{\partial_{\theta}}^k\mathcal{L}_{\partial_{\phi}}^{j-i}t^{(q)}\right|^2 +q\epsilon^2\left|h(\theta)\right|^{2k}\left|\mathcal{L}_{\partial_{\theta}}^{k+1}\mathcal{L}_{\partial_{\phi}}^{j-i}t^{(q)}\right|^2\right].\]

From~\eqref{odkodkodko2}, we obtain
\begin{align}\label{qpvnvniekd98937}
\tilde{\mathcal{Y}}\left(i,j,k\right) &\lesssim \epsilon^{2-4\delta}\mathcal{Y}\left(i,j,k\right) + \sum_{\overset{p+m+n=k}{0 \leq p < k}}\epsilon^{-2(m+n)}\mathcal{Y}\left(i-m,j,p\right) + \epsilon^{-2k}\sum_{\tilde{k} = 0}^k\int_{\mathbb{S}^2}\left|\slashed{\nabla}^{\tilde k}\mathcal{L}_{\partial_{\phi}}^{j-i}H\right|^2
\\ \nonumber &\qquad +\sum_{r=0}^{k+j-i}\int_{\mathbb{S}^2}\left(\epsilon^{M_1/3}\left|\slashed{\nabla}^rt^{(q)}\right|^2+q\epsilon^{M_1/3}\left|\slashed{\nabla}^{r+1}t^{(q)}\right|^2\right).
\end{align}

From~\eqref{onqmqhqoq}, ~\eqref{onqmqhqoq2}, and~\eqref{qpvnvniekd98937} we obtain (with $m,n\ge 0$)
\begin{align}\label{qpvnvniekd98937234}
\mathcal{Y}\left(i,j,k\right) &\lesssim \epsilon^{-2-\delta}\mathcal{Y}\left(i-1,j,k-1\right)  + \sum_{\tilde{k}=1}^k\epsilon^{-2-\delta}\mathcal{Y}\left(i,j,k-1\right)+\sum_{\overset{p+m+n=k}{0 \leq p < k}}\epsilon^{-2(m+n+\delta/2)}\mathcal{Y}\left(i-m,j,p\right)
\\ \nonumber &\qquad + \epsilon^{-2k-\delta}\sum_{\tilde{k} = 0}^k\int_{\mathbb{S}^2}\left|\slashed{\nabla}^{\tilde k}\mathcal{L}_{\partial_{\phi}}^{j-i}H\right|^2 +\epsilon^{-\delta}\sum_{r=0}^{k+j-i}\int_{\mathbb{S}^2}\left(\epsilon^{M_1/3}\left|\slashed{\nabla}^rt^{(q)}\right|^2+q\epsilon^{M_1/3}\left|\slashed{\nabla}^{r+1}t^{(q)}\right|^2\right).
\end{align}

  A straightforward induction argument in $i$ and $k$ (using the estimate~\eqref{purephiestimateassss29387} for the base case) then leads to~\eqref{odkodwkodwkodwponq22232}. (The estimate~\eqref{odkodwkodwkodwponq22232} is not sharp, but that will not matter for us.) Of course, the estimate~\eqref{odkodwkodwkodwponq22232} implies uniform bounds of $t^{(q)}$:
 \begin{equation}\label{okkookkokokoqdlplpq}
\int_{\mathbb{S}^2}\left[\left|\slashed{\nabla}^jt^{(q)}\right|^2+q\left|\slashed{\nabla}^{j+1}t^{(q)}\right|^2\right] \lesssim_{\epsilon,j} \sum_{k=0}^j\sum_{i=0}^{j-k}\int_{\mathbb{S}^2}\left|\slashed{\nabla}^i\mathcal{L}_{\partial_{\phi}}^kH\right|^2.
\end{equation}

Having now established uniform estimates for $t_{AB}^{(q)}$ as $q\to 0$, we turn to showing that $t_{AB}^{(q)}$ converges to a unique $t_{AB}$ as $q\to 0$. We now let $0 < q_2 < q_1$ and derive the following:
\begin{equation}\label{thisisthekappasingularequationqReg22222}
\mathscr{L}\left(t^{(q_1)}-t^{(q_2)}\right)_{AB}-2\kappa\left( t^{(q_1)}-t^{(q_2)}\right)_{AB}  +q_1\mathring{\Delta}\left(t^{(q_1)}-t^{(q_2)}\right)_{AB}= \left(q_2-q_1\right)\mathring{\Delta}t^{(q_2)}_{AB}.
\end{equation}
Contracting with $\left(t^{(q_1)}-t^{(q_2)}\right)^{AB}$, carrying out the usual integration by parts, and using the bound~\eqref{okkookkokokoqdlplpq} leads to
\begin{align*}
\int_{\mathbb{S}^2}\left[\left|t^{(q_1)}-t^{(q_2)}\right|^2 +q_1\left|\mathring{\nabla}\left(t^{(q_1)}-t^{(q_2)}\right)\right|^2 \right]&\lesssim_{\epsilon,H} \left|q_1-q_2\right|\left|\int_{\mathbb{S}^2}\mathring{\Delta}t^{(q_2)}_{AB}\left(t^{(q_1)}-t^{(q_2)}\right)^{AB}\right|.
\\ \nonumber &\lesssim \left|q_1-q_2\right|\int_{\mathbb{S}^2}\left|\mathring{\nabla}t^{(q_2)}\right|\left(\left|t^{(q_1)}-t^{(q_2)}\right| +\left|\mathring{\nabla}\left(t^{(q_1)}-t^{(q_2)}\right)\right|\right)
\end{align*}
This leads to 
\begin{align*}
\int_{\mathbb{S}^2}\left[\left|t^{(q_1)}-t^{(q_2)}\right|^2 +q_1\left|\mathring{\nabla}\left(t^{(q_1)}-t^{(q_2)}\right)\right|^2 \right]&\lesssim_{\epsilon,H} \left|q_1-q_2\right|\int_{\mathbb{S}^2}\left|\mathring{\nabla}t^{(q_2)}\right|^2.
\end{align*}
In particular, we immediately see that $\{t_{AB}^{(q)}\}_{q > 0}$ is Cauchy and we have that $t_{AB}^{(q)} \to_{q\to 0} t_{AB} \in L^2$. Furthermore, by a standard compactness argument, we can take the limit as $q\to 0$ in all of the bounds for $t_{AB}^{(q)}$ that we have established and obtain that $t_{AB} \in \mathring{H}^{\lfloor\frac{N_0}{2}\rfloor}$ and satisfies~\eqref{kappamustbeasolntothisoritwouldbebad}.

Finally, taking the $q\to 0$ limit in the bounds~\eqref{odkodwkodwkodwponq22232} yields~\eqref{odkodwkodwkodwponq}. Finally, with~\eqref{odkodwkodwkodwponq} it is straightforward to revisit the above estimates and establish~\eqref{odkodwkodwkodw230019122}.
\end{proof}

In the remainder of the section we will study the following evolutionary analogue of the $\kappa$-singular equation.
\begin{definition}\label{defofkapsingevolve} Let $\left(\slashed{g}_{AB},b^A,\kappa,\Omega\right)$ be a regular $4$-tuple in the sense of Definition~\ref{Mreg}. Let $\{t_{AB}(v,\theta)\}_{v \in (0,\underline{v}]}$ be a $1$-parameter family of $C^1$ symmetric trace-free tensors on $\mathbb{S}^2$ so that $t_{AB}(v,\theta)$ is $C^1$ in $v \in (0,\underline{v}]$. Then we say that $t_{AB}\left(v,\theta\right)$ satisfies the corresponding $\kappa$-singular evolution equation with right hand side $H_{AB}\left(v,\cdot\right) \in L^{\infty}_v\hat{\mathcal{S}}\left(\mathbb{S}^2\right)$ if
\begin{equation}\label{thisisthekappasingularequationplplplplp}
v\mathcal{L}_{\partial_v} t_{AB} + \mathscr{L}t_{AB} -2\kappa t_{AB} + 2\mathcal{L}_b\log\Omega t_{AB} = H_{AB}.
\end{equation}

\end{definition}

We start with an analogue of Proposition~\ref{kapissingbutthereisasolnatleast}.
\begin{proposition}\label{kapissingbutthereisasolnatleastevolve}Let $\left(\slashed{g}_{AB},b^A,\kappa,\Omega\right)$ be a regular $4$-tuple in the sense of Definition~\ref{Mreg}, and let
$\{H_{AB}(v)\}_{v \in (0,\underline{v}]}$ be a $1$-parameter family of symmetric trace-free tensors in  $\mathring{H}^j\left(\hat{\mathcal{S}}\left(\mathbb{S}^2\right)\right)$ for $1 \leq j \leq \frac{N_0}{2}$. Then then there exists a $1$-parameter family unique $\{t_{AB}(v)\}_{v \in (0,\underline{v}]} \in \mathring{H}^j\left(\hat{\mathcal{S}}\left(\mathbb{S}^2\right)\right)$ solving~\eqref{thisisthekappasingularequationplplplplp}, and $t_{AB}$ is uniquely determined by $t_{AB}(\underline{v})$.

Next we let $\tilde H \in \mathring{H}^j\left(\hat{\mathcal{S}}\right)$ and use Proposition~\ref{kapissingbutthereisasolnatleast} to define $\tilde t$ by solving
\begin{equation}\label{atv0thishiswhathappens}
 \mathscr{L}\tilde t_{AB} -2\kappa \tilde t_{AB} + 2\mathcal{L}_b\log\Omega \tilde{t}_{AB} = \tilde H_{AB}.
\end{equation}

Then, for any  $\tilde v  > 0$, we have the following estimates for $\hat{t}_{AB} \doteq t_{AB} - \tilde t_{AB}$ :

\begin{align}\label{odkodwkodwkodwponq22222plqijdji}
&\sum_{k=0}^j\left[\sup_{v \in [\tilde v,\underline{v}]}\left\vert\left\vert \left(\frac{v}{\underline{v}}\right)^{-\frac{\kappa}{10}}\slashed{\nabla}^j\hat{t}\right\vert\right\vert_{L^2}^2 + \epsilon^2\int_{\tilde v}^{\underline{v}}\left\vert\left\vert  \left(\frac{v}{\underline{v}}\right)^{-\frac{\kappa}{10}}\slashed{\nabla}^j\hat{t}\right\vert\right\vert_{L^2}^2\frac{dv}{v} \right]\lesssim_j
\\ \nonumber &\qquad  \epsilon^{-2j-2j\delta}\left(\sum_{k=0}^j\sum_{i=0}^{j-k}\epsilon^{-2k-2k\delta} \left(\int_{\tilde v}^{\underline{v}}\left\vert\left\vert \left(\frac{v}{\underline{v}}\right)^{-\frac{\kappa}{10}}\slashed{\nabla}^i\mathcal{L}_{\partial_{\phi}}^k\hat{H}(v)\right\vert\right\vert_{L^2}\, \frac{dv}{v}\right)^2\right)
\\ \nonumber &\qquad + \epsilon^{-2j-2j\delta}\sum_{k= 0}^j\left(\epsilon^{-2k-2k\delta}\sum_{i=0}^{j-k}\left\vert\left\vert \slashed{\nabla}^i\mathcal{L}_{\partial_{\phi}}^k\hat{t}(\underline{v})\right\vert\right\vert_{L^2}^2\right),
\end{align}
where 
\[\hat{H}_{AB} \doteq H_{AB} - \tilde H_{AB}.\]

\end{proposition}
\begin{proof}The proof of this proposition is very similar to the proof of Proposition~\ref{kapissingbutthereisasolnatleast}, and thus we will only provide a sketch of the proof. First of all, as with the proof of Proposition~\ref{kapissingbutthereisasolnatleast} we may, without loss of generality, take $\Omega = 1$.

We may write our equation as
\begin{equation*}
v\mathcal{L}_{\partial_v}\hat{t}_{AB} + \mathscr{L}\hat{t}_{AB} -2\kappa \hat{t}_{AB}  = \hat{H}_{AB}.
\end{equation*}
Next, we carry out a change of variables $s \doteq -\log\left(\frac{v}{\underline{v}}\right)$. We then obtain the equation
\begin{equation}\label{kappamustbeasolntothisoritwouldbebad222}
-\mathcal{L}_{\partial_s}\hat{t}_{AB} + \mathscr{L}\hat{t}_{AB} -2\kappa \hat{t}_{AB} = \hat{H}_{AB}.
\end{equation}
In the $s$-variable the bound we need to show is
\begin{align}\label{odkodwkodwkodwponq22222}
&\sum_{k=0}^j\left[\sup_{s \in [0,s_0]}\left\vert\left\vert e^{\frac{\kappa s}{10}}\slashed{\nabla}^k\hat{t}(s)\right\vert\right\vert_{L^2}^2 + \epsilon^2\int_0^{s_0}\left\vert\left\vert e^{\frac{\kappa s}{10}}\slashed{\nabla}^k\hat{t}(s)\right\vert\right\vert_{L^2}^2 \right]\lesssim_j
\\ \nonumber &\qquad  \epsilon^{-2j-2j\delta}\sum_{k=0}^j\left(\sum_{i=0}^{j-k}\epsilon^{-2k-2k\delta}\left[\left(\int_0^{s_0}\left\vert\left\vert e^{\frac{\kappa s}{10}}\slashed{\nabla}^i\mathcal{L}_{\partial_{\phi}}^k\hat{H}\right\vert\right\vert_{L^2}\, ds\right)^2+\left\vert\left\vert \slashed{\nabla}^i\mathcal{L}_{\partial_{\phi}}^k\hat{t}(0)\right\vert\right\vert_{L^2}\right]\right).
\end{align}

We now turn to establishing~\eqref{odkodwkodwkodwponq22222}. First of all, by the theory of characteristics, whenever \\ $\sup_{s\in[0,s_0]}\left\vert\left\vert \hat{H}(s)\right\vert\right\vert_{H^i} < \infty$, we immediately obtain the existence of $f_{AB}(s,\theta)$ with $\sup_{s \in [0,s_0]}\left\vert\left\vert f(s)\right\vert\right\vert_{H^i} \leq C\left(s_0,H\right)$ solving~\eqref{kappamustbeasolntothisoritwouldbebad222}.

We start with the case $j = 0$. Commuting~\eqref{kappamustbeasolntothisoritwouldbebad222} with $e^{\frac{\kappa s}{10}}$ leads to 
\begin{equation}\label{kappamustbeasolntothisoritwouldbebad2223}
-\mathcal{L}_{\partial_s}\left(e^{\frac{\kappa s}{10}}\hat{t}\right)_{AB} + \mathscr{L}\left(e^{\frac{\kappa s}{10}}\hat{t}\right)_{AB} -\left(2-\frac{1}{10}\right)\kappa \left(e^{\frac{\kappa s}{10}}\hat{t}\right)_{AB} = e^{\frac{\kappa s}{10}}\hat{H}_{AB}.
\end{equation}
We now contract~\eqref{kappamustbeasolntothisoritwouldbebad2223} with $-e^{\frac{\kappa s}{10}}\hat{t}^{AB}$, integrate over $[0,s_0] \times \mathbb{S}^2$ with respect to the volume form $ds\slashed{\rm dVol}$ (which will be the implied volume form throughout this proof) and then integrate by parts using Lemma~\ref{thatscrlisantisymm}. We obtain
\begin{equation*}
\sup_{\tilde s \in [0,s_0]}\int_{\mathbb{S}^2}\left|e^{\frac{\kappa s}{10}}\hat{t}\right|^2|_{s=\tilde s} + \epsilon^2\int_0^{s_0}\int_{\mathbb{S}^2}\left|e^{\frac{\kappa s}{10}}\hat{t}\right|^2 \lesssim \int_0^{s_0}\int_{\mathbb{S}^2}\left|e^{\frac{\kappa s}{10}}\hat{H}\cdot e^{\frac{\kappa s}{10}}\hat{t}\right| + \int_{\mathbb{S}^2}\left|\hat{t}\right|^2|_{s=0} \Rightarrow 
\end{equation*}
\begin{equation}\label{firstestiamteevolvevvev}
\sup_{\tilde s \in [0,s_0]}\int_{\mathbb{S}^2}\left|e^{\frac{\kappa s}{10}}\hat{t}\right|^2|_{s=\tilde s} + \epsilon^2\int_0^{s_0}\int_{\mathbb{S}^2}\left|e^{\frac{\kappa s}{10}}\hat{t}\right|^2 \lesssim \left(\int_0^{s_0}\left(\int_{\mathbb{S}^2}\left|e^{\frac{\kappa s}{10}}\hat{H}\right|^2\right)^{1/2}\right)^2 + \int_{\mathbb{S}^2}\left|f\right|^2|_{s=0}.
\end{equation}
As with the proof of Proposition~\ref{kapissingbutthereisasolnatleast} we introduce the convention that all norms are computed with respect to $\slashed{g}$ unless said otherwise. This estimate will serve as the analogue of~\eqref{thefirstenestforfq2}. 

We next explain how to establish~\eqref{odkodwkodwkodwponq22222} in the case $j = 1$. Now, just as in the proof of Proposition~\ref{kapissingbutthereisasolnatleast} we commute~\eqref{kappamustbeasolntothisoritwouldbebad2223} with $\mathcal{L}_{\partial_{\phi}}$ and obtain the following analogue of~\eqref{partphi1ia}
\begin{align}\label{partphi1ia2}
&-\mathcal{L}_{\partial_s}\left(e^{\frac{\kappa s}{10}}\mathcal{L}_{\partial_{\phi}}\hat{t}\right)_{AB} + \mathscr{L}\left(e^{\frac{\kappa s}{10}}\left(\mathcal{L}_{\partial_{\phi}}\hat{t}\right)\right)_{AB} - \left(2-\frac{1}{10}\right)\kappa \left(e^{\frac{\kappa s}{10}}\mathcal{L}_{\partial_{\phi}}\hat{t}\right)_{AB} 
\\ \nonumber &\qquad + e^{\frac{\kappa s}{10}}\left[\mathcal{L}_{\partial_{\phi}},\mathscr{L}\right]\hat{t}_{AB} = e^{\frac{\kappa s}{10}}\mathcal{L}_{\partial_{\phi}}\hat{H}_{AB}.
\end{align}
Now contracting this with $e^{\frac{\kappa s}{11}}\mathcal{L}_{\partial_{\phi}}f^{AB}$, integrating by parts, and arguing like in the derivation of~\eqref{purephiestimateassss29387} leads to the following estimate
\begin{align*}
&\sup_{\tilde s \in [0,s_0]}\int_{\mathbb{S}^2}\left|e^{\frac{\kappa s}{10}}\mathcal{L}_{\partial_{\phi}}\hat{t}\right|^2|_{s= \tilde s} + \epsilon^2 \int_0^{s_0}\int_{\mathbb{S}^2}\left|e^{\frac{\kappa s}{10}}\mathcal{L}_{\partial_{\phi}}\hat{t}\right|^2  \lesssim 
\\ \nonumber &\qquad \int_0^{s_0}\int_{\mathbb{S}^2}\left(\left|e^{\frac{\kappa s}{10}}\hat{H}\right| + e^{\frac{\kappa s}{10}}\left|\left[\mathcal{L}_{\partial_{\phi}},\mathscr{L}\right]\hat{t}\right|\right)\left|e^{\frac{\kappa s}{10}}\mathcal{L}_{\partial_{\phi}}\hat{t}\right| + \int_{\mathbb{S}^2}\left|\mathcal{L}_{\partial_{\phi}}\hat{t}\right|^2|_{s = 0} \Rightarrow 
\end{align*}
\begin{align}\label{purephiestimateassss29387222}
&\sup_{\tilde s \in [0,s_0]}\int_{\mathbb{S}^2}\left|e^{\frac{\kappa s}{10}}\mathcal{L}_{\partial_{\phi}}\hat{t}\right|^2|_{s= \tilde s} + \epsilon^2 \int_0^{s_0}\int_{\mathbb{S}^2}\left|e^{\frac{\kappa s}{10}}\mathcal{L}_{\partial_{\phi}}\hat{t}\right|^2  \lesssim 
\\ \nonumber &\qquad \left(\int_0^{s_0}\left(\int_{\mathbb{S}^2}\left|e^{\frac{\kappa s}{10}}\hat{H}\right|^2\right)^{1/2}\right)^2 +\epsilon^{M_1/2}\int_0^{s_0}\int_{\mathbb{S}^2}\left[\left|e^{\frac{\kappa s}{10}}\mathring{\nabla}\hat{t}\right|^2+\left|e^{\frac{\kappa s}{10}}\hat{t}\right|^2\right] + \int_{\mathbb{S}^2}\left|\mathcal{L}_{\partial_{\phi}}\hat{t}\right|^2|_{s = 0}.
\end{align}
This estimate is an analogue of~\eqref{purephiestimateassss29387} (with $i=1$).

Next, we can, of course, commute our equation with $\mathscr{L}$ to obtain 
\begin{equation}\label{partphi1ia23456}
-\mathcal{L}_{\partial_s}\left(e^{\frac{\kappa s}{10}}\mathscr{L}\hat{t}\right)_{AB} + \mathscr{L}\left(e^{\frac{\kappa s}{10}}\left(\mathscr{L}f\right)\right)_{AB} - \left(2-\frac{1}{10}\right)\kappa \left(e^{\frac{\kappa s}{10}}\mathscr{L}f\right)_{AB}   = e^{\frac{\kappa s}{10}}\mathscr{L}\hat{H}_{AB}.
\end{equation}
Arguing as in the derivation of~\eqref{firstestiamteevolvevvev} leads to
\begin{equation}\label{firstestiamteevolvevvev2}
\sup_{\tilde s \in [0,s_0]}\int_{\mathbb{S}^2}\left|e^{\frac{\kappa s}{10}}\mathscr{L}\hat{t}\right|^2|_{s=\tilde s} + \epsilon^2\int_0^{s_0}\int_{\mathbb{S}^2}\left|e^{\frac{\kappa s}{10}}\mathscr{L}\hat{t}\right|^2 \lesssim \left(\int_0^{s_0}\left(\int_{\mathbb{S}^2}\left|e^{\frac{\kappa s}{10}}\mathscr{L}\hat{H}\right|^2\right)^{1/2}\right)^2 + \int_{\mathbb{S}^2}\left|\mathscr{L}\hat{t}\right|^2|_{s=0}.
\end{equation}
Now, using~\eqref{firstestiamteevolvevvev2} and arguing as in the proof of Proposition~\ref{kapissingbutthereisasolnatleast} with~\eqref{mastscrlwraellycontlthis} we obtain the following:
\begin{align}\label{firstestiamteevolvevvev23}
&\sup_{\tilde s \in [0,s_0]}\int_{\mathbb{S}^2}\left|h(\theta)e^{\frac{\kappa s}{10}}\mathcal{L}_{\partial_{\theta}}\hat{t}\right|^2|_{s=\tilde s} + \epsilon^2\int_0^{s_0}\int_{\mathbb{S}^2}\left|h(\theta)e^{\frac{\kappa s}{10}}\mathcal{L}_{\partial_{\theta}}\hat{t}\right|^2 \lesssim
\\ \nonumber &\qquad \sup_{\tilde s \in [0,s_0]}\int_{\mathbb{S}^2}e^{\frac{\kappa s}{5}}\left[\epsilon^{-2-\delta}\left|\mathcal{L}_{\partial_{\phi}}f\right|^2 + \epsilon^{-2-\delta}\left|\hat{t}\right|^2 +\epsilon^{2-2\delta}\left|\mathring{\nabla}f\right|^2\right]|_{s=\tilde s} +
\\ \nonumber &\qquad  \epsilon^2\int_0^{s_0}\int_{\mathbb{S}^2}e^{\frac{\kappa s}{5}}\left[\epsilon^{-2-\delta}\left|\mathcal{L}_{\partial_{\phi}}f\right|^2 + \epsilon^{-2-\delta}\left|\hat{t}\right|^2 +\epsilon^{2-2\delta}\left|\mathring{\nabla}f\right|^2\right]
 \\ \nonumber &\qquad \epsilon^{-4}\left(\int_0^{s_0}\left(\int_{\mathbb{S}^2}\left|e^{\frac{\kappa s}{10}}\mathscr{L}\hat{H}\right|^2\right)^{1/2}\right)^2 + \epsilon^{-4}\int_{\mathbb{S}^2}\left|\mathscr{L}\hat{t}\right|^2|_{s=0}.
\end{align}
Now we combine this with~\eqref{firstestiamteevolvevvev} and~\eqref{purephiestimateassss29387222} to obtain
\begin{align}\label{firstestiamteevolvevvev234}
&\sup_{\tilde s \in [0,s_0]}\int_{\mathbb{S}^2}\left|h(\theta)e^{\frac{\kappa s}{10}}\mathcal{L}_{\partial_{\theta}}\hat{t}\right|^2|_{s=\tilde s} + \epsilon^2\int_0^{s_0}\int_{\mathbb{S}^2}\left|h(\theta)e^{\frac{\kappa s}{10}}\mathcal{L}_{\partial_{\theta}}\hat{t}\right|^2 \lesssim
 \\ \nonumber &\qquad \left(\int_0^{s_0}\left(\int_{\mathbb{S}^2}\left|e^{\frac{\kappa s}{10}}\mathring{\nabla}\hat{H}\right|^2\right)^{1/2}\right)^2 +\epsilon^{-2-\delta}\left(\int_0^{s_0}\left(\int_{\mathbb{S}^2}\left|e^{\frac{\kappa s}{10}}\mathcal{L}_{\partial_{\phi}}\hat{H}\right|^2\right)^{1/2}+\left(\int_{\mathbb{S}^2}\left|e^{\frac{\kappa s}{10}}\hat{H}\right|^2\right)^{1/2}\right)^2+
 \\ \nonumber &\qquad \sup_{\tilde s \in [0,s_0]}\int_{\mathbb{S}^2}e^{\frac{\kappa s}{5}}\epsilon^{2-2\delta}\left|\mathring{\nabla}f\right|^2|_{s=\tilde s} +  \int_0^{s_0}\int_{\mathbb{S}^2}e^{\frac{\kappa s}{5}}\epsilon^{4-2\delta}\left|\mathring{\nabla}f\right|^2+
 \\ \nonumber &\qquad \int_{\mathbb{S}^2}\left[\left|\mathring{\nabla}\hat{t}\right|^2+ \epsilon^{-2-\delta}\left|\mathcal{L}_{\partial_{\phi}}\hat{t}\right|^2 + \epsilon^{-2-\delta}\left|\hat{t}\right|^2\right]|_{s=0}.
\end{align}
The combination of the left hand side of the estimates~\eqref{firstestiamteevolvevvev234} and~\eqref{purephiestimateassss29387222} controls $\mathring{\nabla}_Af$ for all $\theta$ outside of small neighborhoods of $\{0,y_0,\pi\}$. 

As in the proof of Proposition~\ref{kapissingbutthereisasolnatleast} we will now commute with $\slashed{\nabla}_A$ and carry out localized estimates $\theta = 0$, $y_0$, and $\pi$. We have
\begin{align}\label{kappamustbeasolntothisoritwouldbebad22234}
&-\mathcal{L}_{\partial_s}\left(e^{\frac{\kappa s}{10}}\slashed{\nabla}_C\hat{t}\right)_{AB} + \mathscr{L}\left(e^{\frac{\kappa s}{10}}\slashed{\nabla}_C\hat{t}\right)_{AB} -\left(2-\frac{1}{10}\right)\kappa \left(e^{\frac{\kappa s}{10}}\slashed{\nabla}_C\hat{t}\right)_{AB} 
\\ \nonumber&\qquad = - \left[\mathscr{L},\slashed{\nabla}_C\right]e^{\frac{\kappa s}{10}}\hat{t} + e^{\frac{\kappa s}{10}}\slashed{\nabla}_C\hat{H}_{AB}.
\end{align}
Now we simply follow the derivation~\eqref{onqmqhqoq}, that is, we contract~\eqref{kappamustbeasolntothisoritwouldbebad22234} with $-\chi_0(\theta)\slashed{\nabla}^C\hat{t}^{AB}$ where $\chi(\theta)$ is a suitable cut-off function which is identically $1$ for $\theta \in [0,\gamma/4] \cup [\pi-\gamma/4,\pi]$, identically $0$ for $\theta \in [\gamma/2,\pi-\gamma/2]$, and satisfies $\left|\chi'_0\right| \lesssim \gamma^{-1}$. Arguing as in the derivation of~\eqref{onqmqhqoq} then leads to
\begin{align}\label{onqmqhqoq2939203}
&\sup_{\tilde s \in [0,s_0]}\int_{\{\theta \in [0,\gamma/4] \cup [\pi-\gamma/4,\pi]\}}\left|e^{\frac{\kappa s}{10}}\slashed{\nabla}\hat{t}\right|^2 + \epsilon^2\int_0^{s_0}\int_{\{\theta \in [0,\gamma/4] \cup [\pi-\gamma/4,\pi]\}}\left|e^{\frac{\kappa s}{10}}\slashed{\nabla}\hat{t}\right|^2 \lesssim
\\ \nonumber &\qquad \epsilon^{2-\delta}\int_0^{s_0}\int_{\{\theta \in [\gamma/4,\gamma/2] \cup [\pi-\gamma/2,\pi-\gamma/4]\}}\left[\left|e^{\frac{\kappa s}{10}}\slashed{\nabla}\hat{t}\right|^2 + \left|e^{\frac{\kappa s}{10}}\hat{t}\right|^2\right]+
\\ \nonumber &\qquad \left( \int_0^{s_0}\left(\int_{\mathbb{S}^2}\left|e^{\frac{\kappa s}{10}}\hat{H}\right|^2\right)^{1/2} \right)^2+\int_{\mathbb{S}^2}\left|\slashed{\nabla}\hat{t}\right|^2|_{s = 0}.
\end{align}
Similarly, one may establish the following analogue of~\eqref{onqmqhqoq2}:
\begin{align}\label{onqmqhqoq29392030938297772}
&\sup_{\tilde s \in [0,s_0]}\int_{\{\theta \in [y_0-\gamma,y_0+\gamma]\}}\left|e^{\frac{\kappa s}{10}}\slashed{\nabla}\hat{t}\right|^2 + \epsilon^2\int_0^{s_0}\int_{\{\theta \in [y_0-\gamma,y_0+\gamma]\}}\left|e^{\frac{\kappa s}{10}}\slashed{\nabla}\hat{t}\right|^2 \lesssim
\\ \nonumber &\qquad \epsilon^{2-\delta}\int_0^{s_0}\int_{\{\theta \in [y_0-2\gamma,y_0-\gamma] \cup [y_0+\gamma,y_0+2\gamma]\}}\left|e^{\frac{\kappa s}{10}}\slashed{\nabla}\hat{t}\right|^2 + \epsilon^{-\delta}\int_0^{s_0}\int_{\mathbb{S}^2}\left[\left|e^{\frac{\kappa s}{10}}\mathcal{L}_{\partial_{\phi}}\hat{t}\right|^2 +\left|e^{\frac{\kappa s}{10}}\hat{t}\right|^2\right]+
\\ \nonumber &\qquad \left( \int_0^{s_0}\left(\int_{\mathbb{S}^2}\left|e^{\frac{\kappa s}{10}}\hat{H}\right|^2\right)^{1/2} \right)^2+\int_{\mathbb{S}^2}\left|\slashed{\nabla}\hat{t}\right|^2|_{s = 0}.
\end{align}

Combining~\eqref{purephiestimateassss29387222},~\eqref{firstestiamteevolvevvev234},~\eqref{onqmqhqoq2939203}, and~\eqref{onqmqhqoq29392030938297772} leads to the establishment of~\eqref{odkodwkodwkodwponq22222} with $j = 1$.

The proof then concludes with an induction argument completely analogous to Proposition~\ref{kapissingbutthereisasolnatleast}.

\end{proof}
\section{Setting up the Characteristic Initial Data}\label{setitupayya}


In this section we will set-up the characteristic initial data which will form the starting point of our construction. We start with a definition/lemma for a quantity $\overset{\triangle}{\eta}_A$.

\begin{lemma}\label{uminus1data}Let $\left(\slashed{g}_{AB},b^A,\kappa,\tilde\Omega\right)$ be a regular $4$-tuple. Then, using Proposition~\ref{qualitativeexistence}, we define a $1$-form on $\mathbb{S}^2$, $\overset{\triangle}{\eta}_A$, to be the unique solution to 
\begin{equation}\label{etatriangle}
\left(2-\slashed{\rm div}b\right)\overset{\triangle}{\eta}_A  - \mathcal{L}_b\left(\overset{\triangle}{\eta}_A\right) = -2\slashed{\nabla}_A\mathcal{L}_b\log\tilde\Omega + \slashed{\nabla}^B\left(\slashed{\nabla}\hat{\otimes}b\right)_{BA} - \frac{1}{2}\slashed{\nabla}_A\slashed{\rm div}b.
\end{equation}
Then we have that
\[\left\vert\left\vert \overset{\triangle}{\eta}\right\vert\right\vert_{\mathring{H}^{N_0-1}} \lesssim \epsilon^{1-\delta},\qquad \left\vert\left\vert \mathcal{L}_{\partial_{\phi}}\overset{\triangle}{\eta}\right\vert\right\vert_{\mathring{H}^{N_0-2}} \lesssim \epsilon^{M_1/2}.\]

\end{lemma}
\begin{proof} The bound on $\left\vert\left\vert \overset{\triangle}{\eta}\right\vert\right\vert_{\mathring{H}^{N_0-1}}$ follows from Proposition~\ref{qualitativeexistence}, Sobolev inequalities, Lemma~\ref{conformal}, and the bounds~\eqref{somebound1ssssss}.

To estimate $\mathcal{L}_{\partial_{\phi}}\overset{\triangle}{\eta}_A$ we commute~\eqref{etatriangle} with $\mathcal{L}_{\partial_{\phi}}$ to obtain
\begin{align}\label{etatrianglephi}
\left(2-\slashed{\rm div}b\right)\mathcal{L}_{\partial_{\phi}}\overset{\triangle}{\eta}_A  - \mathcal{L}_b\mathcal{L}_{\partial_{\phi}}\left(\overset{\triangle}{\eta}_A\right) &= \left(\mathcal{L}_{\partial_{\phi}}\slashed{\rm div}b\right)\overset{\triangle}{\eta}_A+\left[\mathcal{L}_{\partial_{\phi}},\mathcal{L}_b\right]\overset{\triangle}{\eta}_A +
\\ \nonumber &\qquad \mathcal{L}_{\partial_{\phi}}\left(-2\slashed{\nabla}_A\mathcal{L}_b\log\tilde\Omega + \slashed{\nabla}^B\left(\slashed{\nabla}\hat{\otimes}b\right)_{BA} - \frac{1}{2}\slashed{\nabla}_A\slashed{\rm div}b\right).
\end{align}
Then we use Proposition~\ref{qualitativeexistence} again as well as the bounds~\eqref{somebound3ssssss} and~\eqref{omsmsjsjksoqjnojq1}.
\end{proof}

Next we have a definition/lemma for a quantity $\overset{\triangle}{\Omega^{-1}{\rm tr}\chi}$.
\begin{lemma}\label{uminus2data}Let $\left(\slashed{g}_{AB},b^A,\kappa,\tilde\Omega\right)$ be a regular $4$-tuple. Then, using Proposition~\ref{qualitativeexistence} we define a function on $\mathbb{S}^2$, $\overset{\triangle}{\Omega^{-1}{\rm tr}\chi}$, to be the unique solution to
\begin{equation}\label{trchitriangle}
\mathcal{L}_b\left(\overset{\triangle}{\Omega^{-1}{\rm tr}\chi}\right) + \left(\overset{\triangle}{\Omega^{-1}{\rm tr}\chi}\right)\left(-1+{\rm div}b-2\kappa + 2\mathcal{L}_b\log\tilde\Omega\right) = -2K  + 2\slashed{\rm div}\overset{\triangle}{\eta} + 2\left|\overset{\triangle}{\eta}\right|^2,
\end{equation}
where $K$ denotes the Gaussian curvature of $\slashed{g}_{AB}$. Then we have the following bounds for $\overset{\triangle}{\Omega^{-1}{\rm tr}\chi}$:
\[\left\vert\left\vert \overset{\triangle}{\Omega^{-1}{\rm tr}\chi} - 2\right\vert\right\vert_{\mathring{H}^{N_0-2}} \lesssim \epsilon^{1-\delta},\qquad \left\vert\left\vert \mathcal{L}_{\partial_{\phi}}\overset{\triangle}{\Omega^{-1}{\rm tr}\chi}\right\vert\right\vert_{\mathring{H}^{N_0-3}} \lesssim \epsilon^{M_1/2}.\]
\end{lemma}
\begin{proof}We can re-write~\eqref{trchitriangle} as  
\begin{align}\label{trchitriangle2}
&\mathcal{L}_b\left(\overset{\triangle}{\Omega^{-1}{\rm tr}\chi}-2\right) + \left(\left(\overset{\triangle}{\Omega^{-1}{\rm tr}\chi}\right)-2\right)\left(-1+{\rm div}b-2\kappa + 2\mathcal{L}_b\log\tilde\Omega\right) = 
\\ \nonumber &-2(K-1) -2\left({\rm div}b-2\kappa + 2\mathcal{L}_b\log\tilde\Omega\right) + 2\slashed{\rm div}\overset{\triangle}{\eta} + 2\left|\overset{\triangle}{\eta}\right|^2,
\end{align}
and then the  bound on $\left\vert\left\vert \overset{\triangle}{\Omega^{-1}{\rm tr}\chi}-2\right\vert\right\vert_{\mathring{H}^{N_0-2}}$ follows from Proposition~\ref{qualitativeexistence}, Lemma~\ref{uminus1data}, and the bounds~\eqref{somebound1ssssss}.

In order to estimate $\mathcal{L}_{\partial_{\phi}}\overset{\triangle}{\Omega^{-1}{\rm tr}\chi}$ we commute~\eqref{trchitriangle} with $\mathcal{L}_{\partial_{\phi}}$ to obtain
\begin{align}\label{trchitrianglephi}
&\mathcal{L}_b\mathcal{L}_{\partial_{\phi}}\left(\overset{\triangle}{\Omega^{-1}{\rm tr}\chi}\right) + \left(\mathcal{L}_{\partial_{\phi}}\overset{\triangle}{\Omega^{-1}{\rm tr}\chi}\right)\left(-1+{\rm div}b-2\kappa + 2\mathcal{L}_b\log\tilde\Omega\right) =
\\ \nonumber &\qquad  \mathcal{L}_{\partial_{\phi}}\left(-2K  + 2\slashed{\rm div}\overset{\triangle}{\eta} + 2\left|\overset{\triangle}{\eta}\right|^2\right)
\\ \nonumber &\qquad -\left[\mathcal{L}_{\partial_{\phi}},\mathcal{L}_b\right]\left(\overset{\triangle}{\Omega^{-1}{\rm tr}\chi}\right) - \left(\overset{\triangle}{\Omega^{-1}{\rm tr}\chi}\right)\mathcal{L}_{\partial_{\phi}}\left(-1+{\rm div}b-2\kappa + 2\mathcal{L}_b\log\tilde\Omega\right).
\end{align}
Then we use Proposition~\ref{qualitativeexistence} again as well as the bounds~\eqref{somebound3ssssss} and~\eqref{omsmsjsjksoqjnojq1}. 
\end{proof}

Lastly, we have a definition/lemma concerning the quantity $\overset{\triangleright}{\Omega^{-1}\hat{\chi}}_{AB}$.
\begin{lemma}\label{uminus3data}Let $\left(\slashed{g}_{AB},b^A,\kappa,\tilde\Omega\right)$ be a regular $4$-tuple and $\underline{v} > 0$. Then, using Proposition~\ref{kapissingbutthereisasolnatleastevolve}, we define a $1$-parameter family of symmetric trace-free tensors on $\mathbb{S}^2$, $\left\{\overset{\triangleright}{\Omega^{-1}\hat{\chi}}_{AB}(v)\right\}_{v \in (0,\underline{v}]}$, in the coordinate frame by solving
\begin{align}\label{hatchitraingle}
&v\mathcal{L}_v\overset{\triangleright}{\Omega^{-1}\hat{\chi}}_{AB} + \mathscr{L}\overset{\triangleright}{\Omega^{-1}\hat{\chi}}_{AB} - 2\kappa \overset{\triangleright}{\Omega^{-1}\hat{\chi}}_{AB} + 2\tilde\Omega^{-1}\left(\mathcal{L}_b\tilde\Omega\right)\overset{\triangleright}{\Omega^{-1}\hat{\chi}}_{AB}= 
\\ \nonumber &\qquad \left(\slashed{\nabla}\hat{\otimes}\overset{\triangle}{\eta}\right)_{AB} + \left(\overset{\triangle}{\eta}\hat{\otimes}\overset{\triangle}{\eta}\right)_{AB} - \frac{1}{2}\left(\overset{\triangle}{\Omega^{-1}{\rm tr}\chi}\right)\left(\slashed{\nabla}\hat{\otimes}b\right)_{AB},
\end{align}
\[\overset{\triangleright}{\Omega^{-1}\hat{\chi}}_{AB}|_{v=\underline{v}} = 0.\]
Here we have used the natural extension of tensors defined on $\mathbb{S}^2$ to tensors defined on $\mathbb{S}^2\times [0,\underline{v}]$ by simply extending the tensors to be independent of $v$.

We have the following bounds for $\overset{\triangleright}{\Omega^{-1}\hat{\chi}}_{AB}$, for any constant $1 \ll N_1 \ll N_0$:
\begin{equation}\label{precisebadepdep}
\sup_{v \in [0,\underline{v}]}\left\vert\left\vert \overset{\triangleright}{\Omega^{-1}\hat{\chi}}\right\vert\right\vert_{\mathring{H}^{\lfloor N_1/2\rfloor }} \lesssim \epsilon^{-2N_1},
\end{equation}
\begin{equation}\label{yayayepsislsons}
\sup_{v\in [0,\underline{v}]}\left(\left|\log^{-\lfloor\frac{N_1}{10}\rfloor}\left(\frac{v}{\underline{v}}\right)\right|+1\right)\left[\left\vert\left\vert \overset{\triangleright}{\Omega^{-1}\hat{\chi}}\right\vert\right\vert_{\mathring{H}^{\lfloor N_1/100\rfloor }} +\left\vert\left\vert v\mathcal{L}_{\partial_v}\overset{\triangleright}{\Omega^{-1}\hat{\chi}}\right\vert\right\vert_{\mathring{H}^{\lfloor N_1/100\rfloor }}\right]\lesssim \epsilon^{1-\delta}.
\end{equation}

Furthermore, we have that $\lim_{v \to 0}\overset{\triangleright}{\Omega^{-1}\hat{\chi}}_{AB}$ exists (in a Lie-propagated frame), and we also have that
\begin{equation}\label{thisgivesacontstateformchihat}
\lim_{v\to 0}v^{-\frac{\kappa}{10}}\left\vert\left\vert \overset{\triangleright}{\Omega^{-1}\hat{\chi}} - \left(\lim_{v\to 0 }\overset{\triangleright}{\Omega^{-1}\hat{\chi}}\right)\right\vert\right\vert_{\mathring{H}^{\lfloor N_1/2 \rfloor}} = 0.
\end{equation} 
Finally, we note that we will have
\begin{equation}\label{0901jnuu3ii1341}
\slashed{g}^{AB}\overset{\triangleright}{\Omega^{-1}\hat{\chi}}_{AB} = 0.
\end{equation}
\end{lemma}

\begin{remark}
Though it will be sufficient for our purposes, the bound~\eqref{precisebadepdep} is far from sharp. However, it is possible to show that in fact
\[\left\vert\left\vert \overset{\triangleright}{\Omega^{-1}\hat{\chi}}|_{v=0}\right\vert\right\vert_{L^2} \sim \epsilon^{-1}.\]
Thus, while~\eqref{precisebadepdep} could be improved, it is in fact necessary for any estimate of $\overset{\triangleright}{\Omega^{-1}\hat{\chi}}_{AB}$ which is uniform in $v$ to degenerate as $\epsilon \to 0$. 
\end{remark}
\begin{proof}The bound~\eqref{precisebadepdep} and limit~\eqref{thisgivesacontstateformchihat} are immediate consequences of Proposition~\ref{kapissingbutthereisasolnatleastevolve} (with $\tilde H_{AB} = \left(\slashed{\nabla}\hat{\otimes}\overset{\triangle}{\eta}\right)_{AB} + \left(\overset{\triangle}{\eta}\hat{\otimes}\overset{\triangle}{\eta}\right)_{AB} - \frac{1}{2}\left(\overset{\triangle}{\Omega^{-1}{\rm tr}\chi}\right)\left(\slashed{\nabla}\hat{\otimes}b\right)_{AB}$), Proposition~\ref{kapissingbutthereisasolnatleast}, Lemmas~\ref{uminus1data} and~\ref{uminus2data}, and~\eqref{somebound1ssssss}. (The bound of $\epsilon^{-2N_1}$ may of course be improved, but this estimate will suffice for this paper.)

For the bound~\eqref{yayayepsislsons} it is clear that we cannot directly appeal to Proposition~\ref{kapissingbutthereisasolnatleastevolve}. Instead we start by proving that for all $j \geq 1$ and $k \geq 1$ with $0 \leq j + k \leq N_1-1$, the following bound for $\overset{\triangleright}{\Omega^{-1}\hat{\chi}}_{AB}$:
\begin{equation}\label{atthebegginginignginwehadsog}
\left\vert\left\vert \left(v\mathcal{L}_{\partial_v}\right)^j\left(\overset{\triangleright}{\Omega^{-1}\hat{\chi}}\right)|_{v=\underline{v}}\right\vert\right\vert_{\mathring{H}^k} \lesssim \epsilon^{j-\delta},\qquad \left\vert\left\vert \left(v\mathcal{L}_{\partial_v}\right)^j\mathcal{L}_{\partial_{\phi}}\left(\overset{\triangleright}{\Omega^{-1}\hat{\chi}}\right)|_{v=\underline{v}}\right\vert\right\vert_{\mathring{H}^{k-1}} \lesssim \epsilon^{M_1/2}.
\end{equation}
To see why~\eqref{atthebegginginignginwehadsog} holds, we note that a direct consequence of~\eqref{hatchitraingle} is that
\begin{align}\label{dwkokodwaaaa}
&\left(v\mathcal{L}_v\right)^j\overset{\triangleright}{\Omega^{-1}\hat{\chi}}_{AB}|_{v=\underline{v}} = 
\\ \nonumber &\qquad \left(-1\right)^{j-1}\left(\mathscr{L}-2\kappa + 2\mathcal{L}_b\log\tilde\Omega\right)^{j-1}\left(\slashed{\nabla}\hat{\otimes}\overset{\triangle}{\eta} + \overset{\triangle}{\eta}\hat{\otimes}\overset{\triangle}{\eta} - \frac{1}{2}\left(\overset{\triangle}{\Omega^{-1}{\rm tr}\chi}\right)\left(\slashed{\nabla}\hat{\otimes}b\right)\right)_{AB}|_{v=\underline{v}}.
\end{align}
Combining~\eqref{dwkokodwaaaa} with~\eqref{somebound1ssssss},~\eqref{somebound3ssssss}, Lemma~\ref{uminus1data}, and Lemma~\ref{uminus2data} then yields~\eqref{atthebegginginignginwehadsog}.

Now we define
\[\mathcal{P}_{AB} \doteq \left(v\mathcal{L}_{\partial_v}\right)^{\lfloor \frac{N_1}{10}\rfloor }\left(\overset{\triangleright}{\Omega^{-1}\hat{\chi}}\right)_{AB}.\]
Commuting~\eqref{hatchitraingle} with $\left(v\mathcal{L}_{\partial_v}\right)^{\lfloor \frac{N_1}{10}\rfloor }$ leads to the following equation:
\begin{align}\label{hatchitraingle2}
&v\mathcal{L}_v\mathcal{P}_{AB} + \mathscr{L}\mathcal{P}_{AB}- 2\kappa\mathcal{P}_{AB}+ 2\tilde\Omega^{-1}\left(\mathcal{L}_b\tilde\Omega\right)\mathcal{P}_{AB}= 0.
\end{align}
Now we apply Proposition~\ref{kapissingbutthereisasolnatleastevolve} (with $\tilde H = 0$) and obtain the following estimate for $\mathcal{P}$:
\begin{align}\label{estimateforppppp}
\sup_{v \in [0,\underline{v}]}\left\vert\left\vert \mathcal{P}\right\vert\right\vert_{\mathring{H}^{\lfloor \frac{N_1}{100}\rfloor }}^2 &\lesssim \epsilon^{-\lfloor\frac{N_1}{40}\rfloor }\left\vert\left\vert \mathcal{P}|_{v=\underline{v}}\right\vert\right\vert_{\mathring{H}^{\lfloor \frac{N_1}{100}\rfloor }}^2 + \epsilon^{-\lfloor\frac{N_1}{40}\rfloor }\left\vert\left\vert \mathcal{L}_{\partial_{\phi}}\mathcal{P}|_{v=\underline{v}}\right\vert\right\vert_{\mathring{H}^{\lfloor \frac{N_1}{100}\rfloor }}^2
\\ \nonumber &\lesssim \epsilon^{\lfloor \frac{N_1}{500}\rfloor },
\end{align}
where in the final inequality we used~\eqref{atthebegginginignginwehadsog}.

Next, we observe that after changing variables to $s \doteq -\log\left(\frac{v}{\underline{v}}\right)$, we have
\[\partial_s^{\lfloor \frac{N_1}{10}\rfloor}\overset{\triangleright}{\Omega^{-1}\hat{\chi}}_{AB} = \mathcal{P}.\]
In particular, repeatedly integrating from $s = 0$, using the bounds~\eqref{estimateforppppp} and~\eqref{atthebegginginignginwehadsog}, and then switching back to the $v$-variable leads to the bound~\eqref{yayayepsislsons}.

It only remains to establish~\eqref{0901jnuu3ii1341}. Contracting~\eqref{hatchitraingle} with $\slashed{g}^{AB}$ leads to the following equation for $\Theta \doteq \slashed{g}^{AB}\overset{\triangleright}{\Omega^{-1}\hat{\chi}}_{AB}$:
\[v\mathcal{L}_{\partial_v}\Theta+ \mathcal{L}_b\Theta + \frac{1}{2}\slashed{\rm div}b \Theta - 2\kappa \Theta + 2\tilde{\Omega}^{-1}\left(\mathcal{L}_b\tilde{\Omega}\right)\Theta = 0.\]
This is a transport equation for $\Theta$, and since $\Theta|_{v = \underline{v}} = 0$, we conclude that $\Theta$ vanishes everywhere.
\end{proof}

Now we are ready for the main result of the section.
\begin{proposition}\label{itexistsbutforalittlewhile}Let $\left(\left(\slashed{g}_0\right)_{AB},\left(b_0\right)^A,\kappa,\tilde\Omega\right)$ be a regular $4$-tuple, $N > 0$ be a sufficiently large integer satisfying $1 \ll N \ll N_0$, and $\underline{v} > 0$. Then, for $\epsilon$ sufficiently small,  there exists an open set $\mathcal{U} \subset \mathbb{R}^2$ with $\{v = 0\} \cup \{u=-1\} \subset \mathcal{U}$ and a  spacetime $\left(\mathcal{M},g_{\mu\nu}\right)$ in the double-null form~\eqref{metricform} for $(u,v) \in \mathcal{U} \cap \{-1 \leq u < 0\} \cap \{0 \leq v <\underline{v}\}$  satisfying the following properties:
\begin{enumerate}
	\item Within the region $\mathcal{U}$, we have the following regularity for various double-null unknowns:
	\begin{equation}\label{02020202020922332}
	\left(v^{\kappa}\Omega, v^{-\kappa}\Omega^{-1}\right) \in L^{\infty}_{\rm loc},
	\end{equation}
	\begin{equation}\label{bound11212232}
	 \left(\Omega^{-1}\beta_A,\rho,\sigma,\Omega\underline{\beta}_A\right) \in L^{\infty}_{u,{\rm loc}}L^2_{v,{\rm loc}}\mathring{H}^N_{\rm loc}, \qquad  \left(\rho,\sigma,\Omega\underline{\beta}_A,\Omega^2\underline{\alpha}_{AB}\right) \in L^{\infty}_{v,{\rm loc}}L^2_{u,{\rm loc}}\mathring{H}^N_{\rm loc}, 
	\end{equation}
	\begin{equation}\label{11212222222}
 	\left(b_A,\Omega\underline{\chi}_{AB},\Omega^{-1}\chi_{AB},v\Omega\omega,\Omega\underline{\omega},\slashed{\nabla}_A\log\Omega\right) \in L^{\infty}_{u,{\rm loc}}L^{\infty}_{v,{\rm loc}}\mathring{H}^N_{\rm loc},
	\end{equation}
	\begin{equation}\label{11212332123123123}
	\left({\rm det}\left(\slashed{g}\right)\right)^{-1}\in L^{\infty}_{\rm loc},\qquad  K \in L^{\infty}_{u,{\rm loc}}L^{\infty}_{v,{\rm loc}}\mathring{H}^N_{\rm loc},
	\end{equation}
	and for every $v_0 > 0$, we have 
	\begin{equation}\label{k0kokoeko2}
	\alpha_{AB} \in L^{\infty}_{u,{\rm loc}}L^2_{v > v_0,{\rm loc}}\mathring{H}^N_{\rm loc} < \infty.
	\end{equation}
	\item All of the quantities listed in~\eqref{02020202020922332}-\eqref{11212332123123123} have limits in $\mathring{H}^N$ on $\{v = 0\} \cup \{u=-1\}$. 
		\item The initial data along $\{v = 0\}$ is obtained:
	\[\lim_{v\to 0}\slashed{g}_{AB} = u^2\left(\slashed{g}_0\right)_{AB},\qquad \lim_{v \to 0}\left(b_A\right) = -u\left(b_0\right)_A,\qquad \lim_{v\to 0} \left(\frac{v}{-u}\right)^{\kappa}\Omega = \tilde\Omega,\]
	where the limits are taken in coordinate frames.
	\item We have 
	\begin{equation}\label{okmdqwdcs}
	\eta_A|_{\{(u,v) = (-1,0)\}} = \overset{\triangle}{\eta}_A,\qquad \Omega^{-1}{\rm tr}\chi|_{\{(u,v) = (-1,0)\}} = \overset{\triangle}{\Omega^{-1}{\rm tr}\chi},
	\end{equation}
	\begin{equation}\label{oioi1981892}
	\Omega^{-1}\hat{\chi}_{AB}|_{\{u=-1\}} = \overset{\triangleright}{\Omega^{-1}\hat{\chi}}_{C(A}\left(\slashed{g}_0^{-1}\right)^{CD}\slashed{g}_{B)D},
	\end{equation}
	where $\overset{\triangle}{\eta}_A$, $\overset{\triangle}{\Omega^{-1}{\rm tr}\chi}$, and $\overset{\triangleright}{\Omega^{-1}\hat{\chi}}_{AB}$ are as in Lemmas~\ref{uminus1data}, \ref{uminus2data}, and~\ref{uminus3data} respectively, and we define, in a coordinate frame along $\{u = -1\}$, $\left(\slashed{g}_0\right)_{AB}\left(v,\theta^C\right) \doteq \left(\slashed{g}_0\right)_{AB}\left(\theta^C\right)$. 
\item We have that $\mathcal{L}_{\partial_v}\left(v^{\kappa}\Omega\right)|_{u=-1} = 0$.
\end{enumerate}

\end{proposition}
\begin{proof}The idea of the proof is to change variables to $\left(u,\hat{v},\theta^A\right)$ defined by $\hat{v} \doteq \left(1-2\kappa\right)^{-1}v^{1-2\kappa}$ (cf.~Definition~\ref{kapselfseim}) and then apply Theorem~\ref{localexistencecharbetter}.

We need to determine the correct choice of incoming and outgoing characteristic data sets (in the sense of Definition~\ref{indatasets}) so that after undoing the coordinate change $\hat{v} \doteq \left(1-2\kappa\right)^{-1}v^{1-2\kappa}$, we end up with the desired spacetime. Let us start with the incoming data and $\zeta_A|_{(-1,0)}$. Keeping in mind that $dv = v^{2\kappa}d\hat{v}$, we set
\[\Omega^{(\rm in)}\left(u,\theta^A\right) \doteq \left(-u\right)^{\kappa}\tilde\Omega\left(\theta^A\right),\qquad \slashed{g}_{AB}^{(\rm in)}\left(u,\theta^A\right) \doteq u^2\left(\slashed{g}_0\right)_{AB}(\theta^A),\qquad \left(b^A\right)^{(\rm in)}\left(u,\theta^A\right) = u^{-1}\left(b^A\right)_0(\theta^A),\]
\[\zeta_A|_{(-1,0)} \doteq \overset{\triangle}{\eta}_A -\slashed{\nabla}_A\tilde{\Omega}.\]

Next we turn to the outgoing characteristic data. Here we will set $\Omega^{(\rm out)}\left(\hat{v},\theta^A\right) \doteq \tilde\Omega\left(\theta^A\right)$. It will also be convenient to define $\Theta_A^{ \ \ B} \doteq \overset{\triangleright}{\Omega^{-1}\hat{\chi}}_{AC}\left(\slashed{g}_0^{-1}\right)^{CB}$, which will satisfy $\Theta_A^{\ \ A} = 0$. Keeping in mind that $\partial_{\hat v} = v^{2\kappa}\partial_v$, we see that we will need to define $\slashed{g}^{(\rm out)}_{AB}\left(\hat{v},\theta^A\right)$ so that the following all hold
\begin{equation}\label{thezerothingtohold}
\slashed{g}^{(\rm out)}_{AB}|_{\hat{v} = 0} = \slashed{g}_{AB},
\end{equation}
\begin{equation}\label{thefirstthingtohold}
\frac{1}{2}\tilde\Omega^{-2}\left(\slashed{g}^{(\rm out)}\right)^{AB}\partial_{\hat{v}}\left(\slashed{g}^{(\rm out)}\right)_{AB}|_{\hat{v} = 0} = \overset{\triangle}{\Omega^{-1}{\rm tr}\chi},
\end{equation}
\begin{equation}\label{thesecondthingtohold}
\frac{1}{2}\tilde\Omega^{-2}{\rm tf}\left(\partial_{\hat{v}}\slashed{g}^{(\rm out)}\right)_{AB} = \Theta_{(A}^{\ \ \ C}\slashed{g}_{B)C},
\end{equation}
\begin{align}\label{thethirdthingtohold}
&\tilde\Omega^{-1}\partial_v\left(\frac{1}{2}\tilde\Omega^{-1}\left(\slashed{g}^{(\rm out)}\right)^{AB}\partial_{\hat{v}}\left(\slashed{g}^{(\rm out)}\right)_{AB}\right) + \frac{1}{2}\left(\frac{1}{2}\tilde\Omega^{-1}\left(\slashed{g}^{(\rm out)}\right)^{AB}\partial_{\hat{v}}\left(\slashed{g}^{(\rm out)}\right)_{AB}\right)^2 =
\\ \nonumber &\qquad  -\frac{1}{4}\left(\slashed{g}^{(\rm out)}\right)^{AC}\left(\slashed{g}^{(\rm out)}\right)^{BD}\tilde\Omega^{-2}{\rm tf}\left(\partial_{\hat{v}}\slashed{g}^{(\rm out)}\right)_{AB} {\rm tf}\left(\partial_{\hat{v}}\slashed{g}^{(\rm out)}\right)_{CD}. 
\end{align}
As usual, the tf denotes the trace-free part, and the $AB$ denote a Lie-propagated coordinate frame. The condition~\eqref{thezerothingtohold} is necessary so that $\slashed{g}_{AB}^{(\rm in)}|_{u=-1} = \slashed{g}_{AB}^{(\rm out)}|_{v=0}$, the conditions~\eqref{thefirstthingtohold} and~\eqref{thesecondthingtohold} are necessary so that the last two equalities of~\eqref{okmdqwdcs} will hold, and~\eqref{thethirdthingtohold} is necessary so that~\eqref{ray2cons} will hold. We now follow the well-known procedure for finding such a $\slashed{g}^{(\rm out)}_{AB}$. We will look for $\slashed{g}^{(\rm out)}_{AB}$ in the form
\begin{equation}\label{thisishatg}
\slashed{g}^{(\rm out)}_{AB} = e^{2\varphi^{(\rm out)}}\hat{\slashed{g}}_{AB}^{(\rm out)},
\end{equation}
where $\hat{\slashed{g}}_{AB}^{(\rm out)}$ is defined both by~\eqref{thisishatg} and the requirement that it has the same volume form as $\slashed{g}_{AB}$. The constancy of the volume form implies that
\[\hat{\slashed{g}}^{AB}\partial_{\hat{v}}\hat{\slashed{g}}_{AB} = 0.\]
Thus,~\eqref{thesecondthingtohold} is seen to become
\begin{align}
& \partial_{\hat{v}}\hat{\slashed{g}}^{(\rm out)}_{AB} = 2\tilde\Omega^2 e^{-2\varphi^{(\rm out)}}\Theta_{(A}^{\ \ \ C}\slashed{g}_{B)C}\Rightarrow 
\\ \label{yayawegothatslashg} &\qquad \hat{\slashed{g}}^{(\rm out)}_{AB}(\hat{v},\theta) = \slashed{g}_{AB}(\theta) + \int_0^{\hat{v}}\left(2\tilde{\Omega}^2\Theta_{(A}^{\ \ \ C}\slashed{g}_{B)C}e^{-2\varphi^{(\rm out)}}\right)\, d\hat{v}.
\end{align}
Next, plugging~\eqref{thisishatg} into~\eqref{thezerothingtohold}, \eqref{thethirdthingtohold}, and~\eqref{thefirstthingtohold} eventually yields the following second order o.d.e. for $\varphi^{(\rm out)}$:
\begin{equation}\label{theodeforvarphiddd}
\partial^2_{\hat{v}}\varphi^{(\rm out)} + \left(\partial_{\hat{v}}\varphi^{(\rm out)}\right)^2 = -\frac{1}{2}\tilde\Omega^2 \left(\hat{\slashed{g}}^{(\rm out)}\right)^{AC}\left(\hat{\slashed{g}}^{(\rm out)}\right)^{BD}\left(\overset{\triangleright}{\Omega^{-1}\hat{\chi}}\right)_{AB}\left(\hat{\slashed{g}}^{(\rm out)}\right)^{BD}\left(\overset{\triangleright}{\Omega^{-1}\hat{\chi}}\right)_{CD},
\end{equation}
\begin{equation}\label{initialdkfowdw}
\varphi^{(\rm out)}|_{\hat{v} = 0} = 0,\qquad \partial_{\hat{v}}\varphi^{(\rm out)}|_{\hat{v} = 0} = \frac{1}{2}\tilde\Omega^2\left(\overset{\triangle}{\Omega^{-1}{\rm tr}\chi}\right).
\end{equation}
We thus see that~\eqref{yayawegothatslashg},~\eqref{theodeforvarphiddd}, and~\eqref{initialdkfowdw} determine an integral-differential system for $\varphi^{(\rm out)}$ and $\hat{\slashed{g}}_{AB}^{(\rm out)}$. 

Now, standard arguments and  the bound~\eqref{yayayepsislsons} show that this integral-differential system for $\varphi^{(\rm out)}$ and $\hat{\slashed{g}}_{AB}^{(\rm out)}$ has a solution for $\hat{v} \in [0,(1-2\kappa)^{-1}\underline{v}^{1-2\kappa}]$ satisfying the bounds
\begin{equation}\label{boundsforvaroutoutout}
\left\vert\left\vert \hat{\slashed{g}}^{(\rm out)}-\slashed{g}\right\vert\right\vert_{\mathring{H}^{\tilde N}} \leq A\epsilon^{1-\delta},\qquad \left\vert\left\vert \partial_v\varphi^{(\rm out)}\right\vert\right\vert_{\mathring{H}^{\tilde N}} + \left\vert\left\vert \varphi^{(\rm out)}\right\vert\right\vert_{\mathring{H}^{\tilde N}} \leq A\epsilon^{1-\delta},
\end{equation}
where $A \gg 1$ is a suitable constant independent of $\epsilon$ and $\tilde N$ is a suitable integer satisfying $N \ll \tilde N \ll N_0$.

Having constructed the outgoing and ingoing data, we may appeal to Theorem~\ref{localexistencecharbetter} in the $\left(u,\hat{v},\theta^A\right)$ variables. (The necessary regularity statements for the outgoing and ingoing data follow from Definition~\ref{Mreg}, Lemma~\ref{uminus1data}, Lemma~\ref{uminus3data}, and the bound~\eqref{boundsforvaroutoutout}.) Finally, we obtain the desired spacetime in the $\left(u,v,\theta^A\right)$ variables by setting 
\[v \doteq \left((1-2\kappa)\hat{v}\right)^{\frac{1}{1-2\kappa}}.\]
The regularity statements follow from noting that the lapse $\Omega$ of the spacetime in the $\left(u,v,\theta\right)$ coordinates will satisfy $\Omega \sim v^{\kappa}$, from the definitions of the various metric components, Ricci coefficients, and curvature components, and the fact that
\[\frac{\partial}{\partial v} = v^{-2\kappa}\frac{\partial}{\partial \hat{v}}.\]

\end{proof}

\begin{remark}\label{siglapweight}For any Ricci coefficient $\psi$ not equal to $\omega$ and any null curvature component $\Psi$ not equal to $\alpha$, we will have that $s\left(\psi\right)$ or $s\left(\Psi\right)$  gives the power of $\Omega$ which shows up in~\eqref{bound11212232} and~\eqref{11212332123123123}.
\end{remark}

\section{The Bootstrap Argument for Region I}\label{bootforregone}
The main result of this section will be the following:
\begin{theorem}\label{itisreg1}Let $(\mathcal{M},g_{\mu\nu})$ be a spacetime produced by Proposition~\ref{itexistsbutforalittlewhile}, and let $\underline{v} > 0$ be sufficiently small. Then we can pick $\epsilon$ sufficiently small so that $\left(\mathcal{M},g_{\mu\nu}\right)$ exists in the region $\left\{(u,v) : -1 \leq u < 0 \text{ and } 0 \leq \frac{v}{-u} \leq \underline{v} \right\}$, and in this region the spacetime satisfies the regularity bounds~\eqref{02020202020922332}-\eqref{11212332123123123} and the estimates~\eqref{betterboot1} and~\eqref{betterboot2}.
\end{theorem}

We will prove this theorem with a bootstrap argument. In Section~\ref{normsregion1} we will define the relevant norms. Then in Section~\ref{estimatesregion1} we will carry out the bootstrap argument. 

\subsection{Norms}\label{normsregion1}
Often, instead of working directly with metric quantities $\vartheta$, Ricci coefficients $\psi$, or null curvature components $\Psi$ we will define quantities  $\widetilde{\vartheta}$, $\widetilde{\Omega^s\psi}$ and $\widetilde{\Omega^s\Psi}$  where we have subtracted off terms reflecting the leading order self-simiar behavior as $v\to 0$. We turn now to the relevant definitions and conventions.
\begin{convention}Throughout this section, unless said otherwise, all norms of tensorial quantities are computed with respect to $\slashed{g}_{AB}$, and we will always use the round metric induced volume form when we integrate on each $\mathbb{S}^2_{u,v}$. 
\end{convention}

\begin{definition}We introduce the definition that, unless said otherwise, for any tensor or function $\mu$,
\[\overline{\mu}\left(u,v,\theta^A\right) \doteq \lim_{v\to 0}\mu\left(u,v,\theta^A\right),\]
where the limit is taken with respect to a \underline{Lie-propagated coordinate frame} if $\mu$ is tensorial. Similarly, $\overline{\slashed{\nabla}}_A$ will denote the connection with respect to $\overline{\slashed{g}}_{AB}$, and we will overline differential operators such as $\overline{\Omega\nabla_3}$ or $\overline{\slashed{\rm div}}$ to indicate that they should be computed from the $v\to 0$ limits of the relevant quantities. For example, 
\[\overline{\slashed{\rm div}}\vartheta = \overline{\slashed{g}}^{AB}\overline{\slashed{\nabla}}_A\vartheta_B = \overline{\slashed{g}}^{AB}\left(\mathcal{L}_{\theta^A}\vartheta_B + \overline{\slashed{\Gamma}}_{AB}^{\ \ C}\vartheta_C\right),\qquad \overline{\Omega\nabla}_3\vartheta_A \doteq \mathcal{L}_{\partial_u}\vartheta_A +\mathcal{L}_{\overline{b}}\vartheta_A - \overline{\slashed{g}}^{BC}\overline{\underline{\chi}}_{AB}\vartheta_C,\]
where $\slashed{\Gamma}_{AB}^{\ \ C}$ denotes the Christoffel symbols of $\slashed{g}_{AB}$ (and hence $\overline{\slashed{\Gamma}}_{AB}^{\ \ C}$ denote the Christoffel symbols of $\overline{\slashed{g}}_{AB}$).
\end{definition}

We start with metric quantities.
\begin{definition}\label{initsuboffdefmetric}Let $(\mathcal{M},g)$ be a spacetime produced by Proposition~\ref{itexistsbutforalittlewhile}. We define, in the coordinate frame,
\[ \widetilde{b^A} \doteq b^A - \overline{b^A},\qquad \widetilde{\Omega} \doteq \left(\frac{v}{-u}\right)^{\kappa}\Omega - \overline{\left(\frac{v}{-u}\right)^{\kappa}\Omega}, \qquad \widetilde{\slashed{\nabla}}_A \doteq \slashed{\nabla}_A - \overline{\slashed{\nabla}}_A,\qquad \widetilde{\slashed{g}}_{AB} \doteq \slashed{g}_{AB}- \overline{\slashed{g}}_{AB}.\]
Note that it is a consequence of Proposition~\ref{itexistsbutforalittlewhile} that these limits exist.

\end{definition}

Next we turn to Ricci coefficients except for $\omega$ and $\hat{\chi}_{AB}$.
\begin{definition}\label{initsuboffdefricci}
Let $(\mathcal{M},g_{\mu\nu})$ be a spacetime produced by Proposition~\ref{itexistsbutforalittlewhile}, let $\psi$ denote any Ricci coefficient other than $\hat{\chi}_{AB}$ or $\omega$, and let $s$ be the signature of $\psi$. Then we define
\[\widetilde{\psi} \doteq \Omega^s\psi - \overline{\Omega^s\psi}.\]
 Note that it is a consequence of Proposition~\ref{itexistsbutforalittlewhile} that this limit exists. 
\end{definition}

For $\hat{\chi}_{AB}$, $\alpha_{AB}$, $\beta_A$, $\rho$, and $\sigma$, we will need to consider a more involved renormalization scheme.
\begin{definition}\label{leadingself}We may uniquely extend $\overset{\triangleright}{\Omega^{-1}\hat{\chi}}_{AB}$ self-similarly to the whole spacetime by setting,
\\ \underline{in the coordinate frame},
\[\overset{\triangleright}{\Omega^{-1}\hat{\chi}}_{AB}\left(u,v,\theta^C\right) \doteq (- u)\overset{\triangleright}{\Omega^{-1}\hat{\chi}}_{AB}\left(\frac{v}{u},\theta^C\right).\]
(We recall that $\overset{\triangleright}{\Omega^{-1}\hat{\chi}}_{AB}$ is defined in the course of Proposition~\ref{itexistsbutforalittlewhile} (using Lemma~\ref{uminus3data}). )

Using $\overset{\triangleright}{\Omega^{-1}\hat{\chi}}_{AB}$, we now define quantities $\overset{\triangleright}{\Omega^{-2}\alpha}_{AB}$, $\overset{\triangleright}{\Omega^{-1}\beta}_A$, $\overset{\triangleright}{\rho}$, and $\overset{\triangleright}{\sigma}$.
\[\overset{\triangleright}{\Omega^{-2}\alpha} \doteq -\left(\frac{v}{-u}\right)^{2\kappa}\left(\overline{\left(\frac{v}{-u}\right)^{-2\kappa}\Omega^{-2}}\right)\mathcal{L}_v\left(\overset{\triangleright}{\Omega^{-1}\hat{\chi}}\right),\]

\[\overset{\triangleright}{\Omega^{-1}\beta}_A \doteq -\overline{\slashed{\rm div}}\left(\overset{\triangleright}{\Omega^{-1}\hat{\chi}}\right)_A - \overline{\eta}_B\left(\overset{\triangleright}{\Omega^{-1}\hat{\chi}}\right)_{AC}\slashed{g}^{BC},\]

\[\overset{\triangleright}{\rho} \doteq \overline{\rho - \frac{1}{2}\hat{\underline{\chi}}\cdot\hat{\chi}} + \frac{1}{2}\overline{\Omega\hat{\underline{\chi}}}_{AB}\left(\overset{\triangleright}{\Omega^{-1}\hat{\chi}}\right)_{CD}\slashed{g}^{AC}\slashed{g}^{BD},\qquad \overset{\triangleright}{\sigma} \doteq \overline{\sigma - \frac{1}{2}\hat{\chi}\wedge\hat{\underline{\chi}}} + \frac{1}{2}\left(\overset{\triangleright}{\Omega^{-1}\hat{\chi}}\right)_{AB}\overline{\left(\Omega\hat{\underline{\chi}}\right)}_{CD}\slashed{\epsilon}^{AC}\slashed{g}^{BD}.\]

\end{definition}
\begin{remark}Note that both $\rho$ and $\hat{\chi}\cdot\hat{\underline{\chi}}$ have regular limits as $v\to 0$. However, $\rho - \frac{1}{2}\hat{\chi}\cdot\hat{\underline{\chi}}$ exhibits cancellation in the sense that we will have
\[u^2\left|\left(\rho - \frac{1}{2}\hat{\chi}\cdot\hat{\underline{\chi}}\right)|_{v=0}\right| \lesssim \epsilon,\] 
even though individually for $\rho$ and $\hat{\chi}\cdot\hat{\underline{\chi}}$, the best estimate we have is
\[u^2\left|\rho|_{v=0}\right| \lesssim 1,\qquad u^2\left|\hat{\chi}\cdot\hat{\underline{\chi}}|_{v=0}\right| \lesssim 1.\]
An analogous remark holds for $\sigma$ and $\hat{\chi}\wedge\hat{\underline{\chi}}$. 
\end{remark}

Now we can define the renormalized $\hat{\chi}_{AB}$.
\begin{definition}Let $(\mathcal{M},g_{\mu\nu})$ be a spacetime produced by Proposition~\ref{itexistsbutforalittlewhile}. Then we define
\[\widetilde{\hat{\chi}}_{AB} \doteq \Omega^{-1}\hat{\chi}_{AB} - \overset{\triangleright}{\Omega^{-1}\hat{\chi}}_{AB}.\]
\end{definition}

Next we come to the renormalized curvature components $\check{\Psi}$ except for $\beta_A$.
\begin{definition}\label{initsuboffdefricci2}
Let $(\mathcal{M},g_{\mu\nu})$ be a spacetime produced by Proposition~\ref{itexistsbutforalittlewhile}. Then we define
\[\widetilde{\alpha}_{AB} \doteq \Omega^{-2}\alpha_{AB} - \overset{\triangleright}{\Omega^{-2}\alpha}_{AB},\qquad \widetilde{\beta}_A \doteq \Omega^{-1}\beta_A - \overset{\triangleright}{\Omega^{-1}\beta}_A- \frac{1}{2}\overline{\slashed{\nabla}_A\left(\Omega^{-1}{\rm tr}\chi\right)}-\frac{1}{2} \overline{\eta_A\left(\Omega^{-1}{\rm tr}\chi\right)}
,\qquad \widetilde{\rho} \doteq \rho - \overset{\triangleright}{\rho},\]
\[\widetilde{\sigma} \doteq \sigma - \overset{\triangleright}{\sigma},\qquad \widetilde{\underline{\beta}}_A \doteq \underline{\beta}_A - \overline{\underline{\beta}}_A,\qquad \widetilde{\underline{\alpha}}_{AB} \doteq \underline{\alpha}_{AB} - \overline{\underline{\alpha}}_{AB}.\]
\end{definition}

For any $\left(\tilde u,\tilde v\right)$ satisfying $-1 \leq \tilde u < 0$ and $0 < \frac{\tilde v}{-\tilde u} \leq \underline{v}$, it will be convenient to define
\[\mathcal{R}_{\tilde u,\tilde v} = \{\left(u,v\right) : u \in [-1,\tilde u]\text{ and }v \in [0,\tilde v]\}.\]
We emphasize that $\mathring{\rm dVol}$ refers to the volume form on the round sphere of radius $1$, and  (by our conventions) we calculate $\left|\cdot\right|$ with respect to $\slashed{g}_{AB}$.

We are now ready to define the energy norms for the null curvature components.
\begin{definition}\label{energynormcurv}Let $0 < q < 1/2$, $-1 \leq \tilde u   < 0$, and $0 < \frac{\tilde v}{-\tilde u} \leq \underline{v}$. Then we define
\begin{align*}
&\left\vert\left\vert \alpha\right\vert\right\vert^2_{\mathscr{A}_{q,\tilde u,\tilde v}} \doteq
\\ \nonumber &\qquad  \sum_{i=0}^2\sup_{\left(u,v\right) \in \mathcal{R}_{\tilde u,\tilde v}}\left[\int_0^v\int_{\mathbb{S}^2}\Omega^2\frac{(-u)^{4-2q+2i}}{\dot{v}^{1-2q}}\left|\slashed{\nabla}^i\widetilde{\alpha}\right|^2\, \mathring{\rm dVol}\, d\dot{v} + \int_{-1}^u\int_0^v\int_{\mathbb{S}^2}\frac{(-\dot{u})^{3-2q+2i}}{\dot{v}^{1-2q}}\Omega^2\left|\slashed{\nabla}^i\widetilde{\alpha}\right|^2\mathring{\rm dVol}\, d\dot{u}\, d\dot{v}\right],
\end{align*}
\begin{align*}
&\left\vert\left\vert \underline{\alpha}\right\vert\right\vert^2_{\mathscr{A}_{q,\tilde u,\tilde v}} \doteq
\\ \nonumber &\qquad   \sum_{i=0}^2\sup_{\left(u,v\right) \in \mathcal{R}_{\tilde u,\tilde v}}\left[\int_{-1}^u\int_{\mathbb{S}^2}\frac{(-\dot{u})^{4-\frac{q}{100}+2i}}{v^{1-\frac{q}{100}}}\left|\slashed{\nabla}^i\widetilde{\underline{\alpha}}\right|^2\, \mathring{\rm dVol}\, d\dot{u} + \int_{-1}^u\int_0^v\int_{\mathbb{S}^2}\frac{(-\dot{u})^{4-\frac{q}{100}+2i}}{\dot{v}^{2-\frac{q}{100}}}\left|\slashed{\nabla}^i\widetilde{\underline{\alpha}}\right|^2\mathring{\rm dVol}\, d\dot{u}\, d\dot{v}\right],
\end{align*}
and, for any curvature component $\Psi \in \{\rho,\sigma,\underline{\beta}_A\}$ we have
\begin{align*}
&\left\vert\left\vert \Psi\right\vert\right\vert^2_{\mathscr{A}_{q,\tilde u,\tilde v}} \doteq
\\ \nonumber &\qquad   \sum_{i=0}^2\sup_{\left(u,v\right) \in \mathcal{R}_{\tilde u,\tilde v}}\Bigg[\int_{-1}^u\int_{\mathbb{S}^2}\frac{(-\dot{u})^{4-\frac{q}{100}+2i}}{v^{1-\frac{q}{100}}}\left|\slashed{\nabla}^i\widetilde{\Psi}\right|^2\, \mathring{\rm dVol}\, d\dot{u} + \int_0^v\int_{\mathbb{S}^2}\Omega^2\frac{(-u)^{4-\frac{q}{100}+2i}}{\dot{v}^{1-\frac{q}{100}}}\left|\slashed{\nabla}^i\widetilde{\Psi}\right|^2\, \mathring{\rm dVol}\, d\dot{v}+ 
\\ \nonumber &\qquad \qquad \qquad \qquad \qquad\ \int_{-1}^u\int_0^v\int_{\mathbb{S}^2}\frac{(-\dot{u})^{4-\frac{q}{100}+2i}}{\dot{v}^{2-\frac{q}{100}}}\left|\slashed{\nabla}^i\widetilde{\Psi}\right|^2\mathring{\rm dVol}\, d\dot{u}\, d\dot{v}\Bigg],
\end{align*}
and for $\beta_A$ we have 
\begin{align*}
&\left\vert\left\vert \beta\right\vert\right\vert^2_{\mathscr{A}_{q,\tilde u,\tilde v}} \doteq
\\ \nonumber &\qquad   \sum_{i=0}^2\sup_{\left(u,v\right) \in \mathcal{R}_{\tilde u,\tilde v}}\Bigg[\int_{-1}^u\int_{\mathbb{S}^2}\frac{(-\dot{u})^{4-2q+2i}}{v^{1-2q}}\left|\slashed{\nabla}^i\widetilde{\beta}\right|^2\, \mathring{\rm dVol}\, d\dot{u} + \int_0^v\int_{\mathbb{S}^2}\Omega^2\frac{(-u)^{4-\frac{q}{100}+2i}}{\dot{v}^{1-\frac{q}{100}}}\left|\slashed{\nabla}^i\widetilde{\beta}\right|^2\, \mathring{\rm dVol}\, d\dot{v}+ 
\\ \nonumber &\qquad \qquad \qquad \qquad \qquad\ \int_{-1}^u\int_0^v\int_{\mathbb{S}^2}\frac{(-\dot{u})^{4-2q+2i}}{\dot{v}^{2-2q}}\left|\slashed{\nabla}^i\widetilde{\beta}\right|^2\mathring{\rm dVol}\, d\dot{u}\, d\dot{v}\Bigg],
\end{align*}
We also introduce the notation 
\[\mathfrak{A}_{q,\tilde u,\tilde v} \doteq \sum_{\Psi \in \{\alpha,\beta,\rho,\sigma,\underline{\beta},\underline{\alpha}\}}\left\vert\left\vert \Psi\right\vert\right\vert_{\mathscr{A}_{q,\tilde u,\tilde v}}.\]
Finally, when it will not cause confusion, we will often suppress a subset of the $\left(q,\tilde u,\tilde v\right)$ indices from the $\mathscr{A}$ or $\mathfrak{A}$ subscript. 
\end{definition}
\begin{remark}As we have mentioned in the introduction, the rationale behind these weights is similar to the rationale behind the weights in the work~\cite{scaleinvariant}. For the convenience of the reader, we now quickly recapitulate the main points:
\begin{enumerate}
	\item We expect the solution to be ``asymptotically self-similar'' as we approach $(u,v) = (0,0)$. Thus we choose our norms to be invariant under the rescaling diffeomorphism $\left(u,v,\theta^A\right) \mapsto \left(\lambda u,\lambda v,\theta^A\right)$ for $\lambda > 0$. In particular, our energy norms along constant $u$ or $v$ hypersurfaces should have a total $u$ and $v$ weight which adds up to $3$ plus $2$ times the number of angular dervative commutations. For a spacetime energy norm, the total weight should add up to $2$ plus $2$ times the number of angular dervative commutations.
	\item When we carry out the energy estimates for $\left(\tilde{\alpha},\tilde{\beta}\right)$, we will need to conjugate the $\nabla_3$ equation for $\alpha$ by a suitable $u$-weight. This produces lower order terms proportional to $|u|^{-1}$; the norms have to be chosen so that these additional lower order terms in combination with the already present terms from ${\rm tr}\underline{\chi}\alpha$ have a good sign in our energy estimate. Every other curvature component satisfies a $\nabla_4$ equation, and the conjugation by a negative $v$-weight will produce a good spacetime term.
	\item We have to make sure the norms are chosen so that we have a finite contribution to our estimates  from initial data and from the inhomogeneities produced by the various curvature renormalizations.
\end{enumerate}
\end{remark}

Next, we have the norms for the Ricci coefficients (except for the lower regularity norm for ${\rm tr}\chi$ and ${\rm tr}\underline{\chi}$).
\begin{definition}\label{lowordernormricci}Let $0 < q < 1/2$, $-1 \leq \tilde u   < 0$, and $0 < \frac{\tilde v}{-\tilde u} \leq \underline{v}$. Then we define for any Ricci coefficient $\psi \in \{\underline{\omega},\eta_A,\hat{\underline{\chi}}_{AB},\underline{\eta}_A\}$:\begin{align}\label{riccivanish}
&\left\vert\left\vert \psi\right\vert\right\vert^2_{\mathscr{B}_{q,\tilde u,\tilde v}} \doteq \sum_{i=0}^2\sup_{\left(u,v\right) \in \mathcal{R}_{\tilde u,\tilde v}}\frac{(-u)^{4-\frac{q}{100}+2i}}{v^{2-\frac{q}{100}}}\int_{\mathbb{S}^2}\left|\slashed{\nabla}^i\widetilde{\psi}\right|^2\mathring{\rm dVol},
\end{align}
for $\hat{\chi}_{AB}$ we have
\begin{align}\label{riccivanish12345}
&\left\vert\left\vert \hat{\chi}\right\vert\right\vert^2_{\mathscr{B}_{q,\tilde u,\tilde v}} \doteq \sum_{i=0}^2\sup_{\left(u,v\right) \in \mathcal{R}_{\tilde u,\tilde v}}\frac{(-u)^{4-2q+2i}}{v^{2-2q}}\int_{\mathbb{S}^2}\left|\slashed{\nabla}^i\widetilde{\hat{\chi}}\right|^2\mathring{\rm dVol},
\end{align}
for $\psi \in \{{\rm tr}\chi,{\rm tr}\underline{\chi}\}$ we set 
\begin{align}\label{riccivanish2}
&\left\vert\left\vert \psi\right\vert\right\vert^2_{\mathscr{B}_{q,\tilde u,\tilde v}} \doteq \sum_{i=1}^2\sup_{\left(u,v\right) \in \mathcal{R}_{\tilde u,\tilde v}}\frac{(-u)^{4-\frac{q}{100}+2i}}{v^{2-\frac{q}{100}}}\int_{\mathbb{S}^2}\left|\slashed{\nabla}^i\widetilde{\psi}\right|^2\mathring{\rm dVol}.
\end{align}
(Note that the sum starts with $i = 1$.)

We also introduce the notation 
\[\mathfrak{B}_{q,\tilde u,\tilde v} \doteq \sum_{\psi \not\in \{\omega\}}\left\vert\left\vert \psi\right\vert\right\vert_{\mathscr{B}_{q,\tilde u,\tilde v}}.\]
Finally, when it will not cause confusion, we will often suppress a subset of the $\left(q,\tilde u,\tilde v\right)$ indices from the $\mathscr{B}$ or $\mathfrak{B}$ subscript. 
\end{definition}
Next we have a norm for the $L^2$ norm of ${\rm tr}\chi$ and ${\rm tr}\underline{\chi}$
\begin{definition}Let  $0 < q < 1/2$, $-1 \leq \tilde u   < 0$, and $0 < \frac{\tilde v}{-\tilde u} \leq \underline{v}$. Then we define for $\psi \in \{{\rm tr}\underline{\chi},{\rm tr}\chi\}$:
\begin{equation}\label{riccivanish3}
\left\vert\left\vert \psi\right\vert\right\vert_{\mathscr{C}_{q,\tilde u,\tilde v}} \doteq \sup_{\left(u,v\right) \in \mathcal{R}_{\tilde u,\tilde v}}\frac{(-u)^{4-\frac{q}{100}}}{v^{2-\frac{q}{100}}}\int_{\mathbb{S}^2}\left|\widetilde{\psi}\right|^2\mathring{\rm dVol}.
\end{equation}
We also introduce the notation 
\[\mathfrak{C}_{q,\tilde u,\tilde v} \doteq \sum_{\psi \in \{{\rm tr}\chi,{\rm tr}\underline{\chi}\}}\left\vert\left\vert \psi\right\vert\right\vert_{\mathscr{C}_{q,\tilde u,\tilde v}}.\]
Finally, when it will not cause confusion, we will often suppress a subset of the $\left(p,\tilde u,\tilde v\right)$ indices from the $\mathscr{C}$ or $\mathfrak{C}$ subscript. 
\end{definition}

Lastly, we have the norms for the metric coefficients.
\begin{definition}Let $0 < q < 1/2$, $-1 \leq \tilde u   < 0$, and $0 < \frac{\tilde v}{-\tilde u} \leq \underline{v}$. Then we define
\begin{align}
&\left\vert\left\vert \slashed{g}\right\vert\right\vert_{\mathscr{D}_{q,\tilde u,\tilde v} }\doteq \sum_{i=0}^2\sup_{\left(u,v\right) \in \mathcal{R}_{\tilde u,\tilde v}}\frac{(-u)^{2+2i-\frac{q}{100}}}{v^{2-\frac{q}{100}}}\int_{\mathbb{S}^2}\left|\mathring{\nabla}^i\widetilde{\slashed{g}}\right|^2\mathring{\rm dVol},
\\ \label{444} &\qquad \left\vert\left\vert b\right\vert\right\vert_{\mathscr{D},q,i,\tilde u,\tilde v} \doteq \sum_{i=0}^2\sup_{\left(u,v\right) \in \mathcal{R}_{\tilde u,\tilde v}}\frac{(-u)^{2+2i-\frac{q}{100}}}{v^{2-\frac{q}{100}}}\int_{\mathbb{S}^2}\left|\slashed{\nabla}^i\widetilde{b}\right|^2\mathring{\rm dVol},
\end{align}
\begin{align*}
&\left\vert\left\vert \log\Omega \right\vert\right\vert_{\mathscr{D}_{p,\tilde u,\tilde v}} \doteq \sum_{i=0}^2\sup_{\left(u,v\right) \in \mathcal{R}_{\tilde u,\tilde v}}\frac{(-u)^{2+2i-\frac{q}{100}}}{v^{2-\frac{q}{100}}}\int_{\mathbb{S}^2}\left|\slashed{\nabla}^i\log\left(\widetilde{\Omega}\right)\right|^2\mathring{\rm dVol},
\end{align*}
where we recall that $\mathring{\nabla}$ is the covariant derivative with respect to a fixed round metric $\mathring{\slashed{g}}_{AB}$ on the sphere.
We also introduce the notation 
\[\mathfrak{D}_{q,\tilde u,\tilde v} \doteq \sum_{\vartheta \neq \slashed{g}}\left\vert\left\vert \vartheta \right\vert\right\vert_{\mathscr{D}_{q,\tilde u,\tilde v}}.\]
Finally, when it will not cause confusion, we will often suppress a subset of the $\left(q,\tilde u,\tilde v\right)$ indices from the $\mathscr{D}$ or $\mathfrak{D}$ subscript.

\end{definition}

Our last definition concerns weighted Sobolev spaces on $\mathbb{S}^2$:
\begin{definition}\label{idkdtildedefdef}
For any $(0,k)$-tensor $w$, $(u,v)$, and $i \in \{0,1,2\}$ we define
\[\left\vert\left\vert w\right\vert\right\vert_{\tilde H^i\left(\mathbb{S}^2_{u,v}\right)} \doteq  \sum_{j = 0}^i(v-u)^j\left(\int_{\mathbb{S}^2_{u,v}}\slashed{g}\left(\slashed{\nabla}^jw,\slashed{\nabla}^jw\right)\, \mathring{\rm dVol}\right)^{1/2},\]
\[\left\vert\left\vert w\right\vert\right\vert_{\tilde L^p\left(\mathbb{S}^2_{u,v}\right)} \doteq  \left(\int_{\mathbb{S}_{u,v}^2}\slashed{g}\left(w,w\right)^{p/2}\, \mathring{\rm dVol}\right)^{1/p},\]
where $\mathring{\rm dVol}$ denotes the volume form of the round sphere. 
\end{definition}

We close with a useful lemma.
\begin{lemma}\label{boundfromdatasub}
There exists $R > 0$ (independent of $\epsilon$) so that we have the following bounds:
\begin{equation}\label{boundsfortimesgmie}
\left[\left\vert\left\vert \overset{\triangleright}{\Omega^{-1}\hat{\chi}}\right\vert\right\vert^2_{\tilde H^4\left(\mathbb{S}^2_{u,v}\right)}+\left\vert\left\vert v\mathcal{L}_{\partial_v}\overset{\triangleright}{\Omega^{-1}\hat{\chi}}\right\vert\right\vert^2_{\tilde H^4\left(\mathbb{S}^2_{u,v}\right)}\right] \lesssim \left|\log^R\left(\frac{v}{-u}\right)\right|\epsilon^{2-2\delta}(-u)^{-2},
\end{equation}
\begin{equation}\label{dwodw2}
\left\vert\left\vert \overset{\triangleright}{\Omega^{-1}\beta}\right\vert\right\vert^2_{\tilde H^4\left(\mathbb{S}^2_{u,v}\right)} \lesssim \left|\log^R\left(\frac{v}{-u}\right)\right|\epsilon^{2-2\delta}(-u)^{-4},
\end{equation}
if $\psi$ is a Ricci coefficient, $\psi \not\in\{\hat{\chi}_{AB},{\rm tr}\chi,{\rm tr}\underline{\chi}\}$, $s$ denotes the signature of $\psi$, and $\Theta \in \left\{\overline{\rho-\frac{1}{2}\hat{\underline{\chi}}\cdot\hat{\chi}}, \overline{\sigma-\frac{1}{2}\hat{\chi}\wedge\hat{\underline{\chi}}},\overline{\underline{\beta}},\overline{\underline{\alpha}}\right\}$, then
\begin{equation}\label{boundsfortimesgmie2}
\left\vert\left\vert \overline{\Omega^s\psi}\right\vert\right\vert^2_{\tilde H^4\left(\mathbb{S}^2_{u,0}\right)} \lesssim \epsilon^{2-2\delta}(-u)^{-2},\qquad \left\vert\left\vert \Theta\right\vert\right\vert^2_{\tilde H^4\left(\mathbb{S}^2_{u,0}\right)} \lesssim \epsilon^{2-2\delta}(-u)^{-4},
\end{equation}
and lastly, 
\begin{equation}\label{boundsfortimesgmie3}
\left\vert\left\vert \Omega^{-1}{\rm tr}\chi\right\vert\right\vert^2_{\tilde H^4\left(\mathbb{S}^2_{u,0}\right)} \lesssim (-u)^{-2},\qquad \left\vert\left\vert \Omega{\rm tr}\underline{\chi}\right\vert\right\vert^2_{\tilde H^4\left(\mathbb{S}^2_{u,0}\right)} \lesssim (-u)^{-2},
\end{equation}
\begin{equation}\label{boundsfortimesgmie4}
\left\vert\left\vert \slashed{\nabla}\left(\Omega^{-1}{\rm tr}\chi\right)\right\vert\right\vert^2_{\tilde H^3\left(\mathbb{S}^2_{u,0}\right)} \lesssim \epsilon^{2-2\delta}(-u)^{-2},\qquad \left\vert\left\vert\slashed{\nabla}\left(\Omega {\rm tr}\underline{\chi}\right)\right\vert\right\vert^2_{\tilde H^3\left(\mathbb{S}^2_{u,0}\right)} \lesssim \epsilon^{2-2\delta}(-u)^{-2}.
\end{equation}
\end{lemma}
\begin{proof}
These bounds follow immediately from Lemmas~\ref{uminus1data}-\ref{uminus3data}, the construction of $\left(\mathcal{M},g_{\mu\nu}\right)$ in Proposition~\ref{itexistsbutforalittlewhile}, and the null structure equations which relate curvature components to Ricci coefficients. For example, let us consider $\eta_A$. We have (see footnote~\ref{agoodthingtoputinafootnoteyayaya}) that in the coordinate frame $\eta_A\left(u,\theta^B\right) = \overset{\triangle}{\eta}_A\left(\theta^B\right)$, and thus the estimate~\eqref{boundsfortimesgmie3} for $\eta_A$ follows from the corresponding estimates for $\overset{\triangle}{\eta}_A$ which were established in Lemma~\ref{uminus1data}.
\end{proof}

\subsection{The Estimates}\label{estimatesregion1}
A standard argument using Proposition~\ref{itexistsbutforalittlewhile} shows that Theorem~\ref{itisreg1} will follow from the following proposition.
\begin{proposition}\label{letsbootit}Let $q \in (0,1/2)$, $0 < p \ll 1$, and $(\mathcal{M},g_{\mu\nu})$ be a spacetime produced by Proposition~\ref{itexistsbutforalittlewhile} which exists in a rectangle $\mathcal{R}_{\tilde u,\tilde v}$ for some $\tilde u \in (-1,0)$ and $\tilde v \in (0,\underline{v}]$ satisfying 
\[ 0 < \frac{\tilde v}{-\tilde u} \leq \underline{v},\]
and which satisfies the ``bootstrap assumption''
\begin{equation}\label{boot1}
\mathfrak{A}_{q,\tilde u,\tilde v} + \mathfrak{B}_{q,\tilde u,\tilde v} + \mathfrak{D}_{q,\tilde u,\tilde v} \leq 2A\epsilon^{1-\delta},
\end{equation}
\begin{equation}\label{boot2}
\left\vert\left\vert \slashed{g}\right\vert\right\vert_{\mathscr{D}_{q,\tilde u,\tilde v}} + \sup_{(u,v) \in \mathcal{R}_{\tilde u,\tilde v}}\sum_{i=0}^1(-u)^{2+i}\left\vert\left\vert \slashed{\nabla}^i\widetilde{K}\right\vert\right\vert_{L^2\left(\mathbb{S}^2_{u,v}\right)} + \mathfrak{C}_{q,\tilde u,\tilde v} \leq 2A \underline{v}^{\frac{p}{10}},
\end{equation}
where $A \gg 1$ is a suitable constant.

Then, if $\epsilon$ and $\underline{v}$ are suitably small, depending only on $q$ and $p$ and $A$, we have the following estimate which improves on the bootstrap assumption:
\begin{equation}\label{betterboot1}
\mathfrak{A}_{q,\tilde u,\tilde v} + \mathfrak{B}_{q,p,\tilde u,\tilde v}  + \mathfrak{D}_{q,\tilde u,\tilde v} \leq A\epsilon^{1-\delta},
\end{equation}
\begin{equation}\label{betterboot2}
\left\vert\left\vert \slashed{g}\right\vert\right\vert_{\mathscr{D}_{q,\tilde u,\tilde v}}+ \sup_{(u,v) \in \mathcal{R}_{\tilde u,\tilde v}}\sum_{i=0}^1(-u)^{2+i}\left\vert\left\vert \slashed{\nabla}^i\widetilde{K}\right\vert\right\vert_{L^2\left(\mathbb{S}^2_{u,v}\right)} +\mathfrak{C}_{q,\tilde u,\tilde v} \leq A \underline{v}^{\frac{p}{10}}.
\end{equation}
\end{proposition}

The proof of Proposition~\ref{letsbootit} will be broken into stages, primarily based on what norm we are estimating.
\subsubsection{Sobolev spaces, Sobolev inequalities, and elliptic estimates on $\mathbb{S}^2$}
In this section we will analyze the Sobolev spaces $\tilde{H}^i$ and establish some standard elliptic estimates. We start by observing the Sobolev spaces generated by $\slashed{g}_{AB}$ and $\mathring{\slashed{g}}_{AB}$ are comparable.
\begin{lemma}\label{sosimilararethepscaes}Let $\left(\mathcal{M},g_{\mu\nu}\right)$ satisfy the hypothesis of Proposition~\ref{letsbootit} and $w_{A_1\cdots A_k}$ be a $(0,k)$-tensor. Then we have that
\[\left\vert\left\vert w\right\vert\right\vert_{\tilde H^i\left(\mathbb{S}^2_{u,v}\right)} \sim_k (-u)^{-k}\left\vert\left\vert w\right\vert\right\vert_{\mathring{H}^i\left(\mathbb{S}^2_{u,v}\right)}\text{ for }i \in \{0,1,2\},\]
\[\left\vert\left\vert w\right\vert\right\vert_{\tilde L^p\left(\mathbb{S}^2_{u,v}\right)} \sim_k (-u)^{-k}\left\vert\left\vert w\right\vert\right\vert_{\mathring{L}^p\left(\mathbb{S}^2_{u,v}\right)},\]
where we recall that $\mathring{H}^i$ and $\mathring{L}^p$ denote the Sobolev and $L^p$ spaces generated by the round metric $\mathring{\slashed{g}}_{AB}$. 
\end{lemma}
\begin{proof}This is an immediate consequence of Lemma~\ref{comparethespaces}, the bootstrap hypothesis, and the smallness of $\underline{v}$ and $\epsilon$.
\end{proof}

Now we observe that the standard Sobolev inequalities hold for the spaces $\tilde{H}^i$.
\begin{lemma}\label{sosososob}Let $\left(\mathcal{M},g_{\mu\nu}\right)$ satisfy the hypothesis of Proposition~\ref{letsbootit}. Then, for any $(0,k)$-tensor $w_{A_1\cdots A_k}$ and $p \in [1,\infty)$, we have that
\[\left\vert\left\vert w\right\vert\right\vert_{\tilde{L}^p\left(\mathbb{S}^2_{u,v}\right)} \lesssim_{p,k} \left\vert\left\vert w\right\vert\right\vert_{\tilde H^1\left(\mathbb{S}^2_{u,v}\right)},\qquad \sup_{\mathbb{S}^2_{u,v}}\left|w\right|\lesssim_k \left\vert\left\vert w\right\vert\right\vert_{\tilde{H}^2\left(\mathbb{S}^2_{u,v}\right)},\]
and for any $(0,k)$-tensor $w_{A_1\cdots A_k}$ and $(0,k')$-tensor $v_{A_1\cdots A_{k'}}$ we have
\[\left\vert\left\vert w\cdot v\right\vert\right\vert_{\tilde H^2\left(\mathbb{S}^2_{u,v}\right)} \lesssim_{k,k'} \left\vert\left\vert w\right\vert\right\vert_{\tilde H^2\left(\mathbb{S}^2_{u,v}\right)} \left\vert\left\vert v\right\vert\right\vert_{\tilde H^2\left(\mathbb{S}^2_{u,v}\right)}.\]

\end{lemma}
\begin{proof}This is an immediate consequence of Lemma~\ref{sosimilararethepscaes} and Lemma~\ref{soblemm}.
\end{proof}
These Sobolev inequalities will be used repeatedly in our estimates of nonlinear terms and we will often do so without explicit comment. 

We close the section with some standard elliptic estimates.
\begin{lemma}\label{someelleststs}Let $\left(\mathcal{M},g_{\mu\nu}\right)$ satisfy the hypothesis of Proposition~\ref{letsbootit}. Then for any function $f$, $1$-form $\theta_A$, and symmetric trace-free $2$-tensor $\nu_{AB}$ we have
\begin{equation}\label{pokdwimnfe3}
\left\vert\left\vert f\right\vert\right\vert_{\tilde H^2\left(\mathbb{S}^2_{u,v}\right)} \lesssim (-u)^2\left\vert\left\vert \slashed{\Delta}f\right\vert\right\vert_{L^2\left(\mathbb{S}^2_{u,v}\right)} + \left\vert\left\vert f\right\vert\right\vert_{L^2\left(\mathbb{S}^2_{u,v}\right)},
\end{equation}
\begin{equation}\label{pokdwimnfe}
\left\vert\left\vert\theta \right\vert\right\vert_{\tilde H^i\left(\mathbb{S}^2_{u,v}\right)} \lesssim \left(-u\right)\left[\left\vert\left\vert \slashed{\rm div}\theta\right\vert\right\vert_{\tilde H^{i-1}\left(\mathbb{S}^2_{u,v}\right)} +\left\vert\left\vert \slashed{\rm curl}\theta\right\vert\right\vert_{\tilde H^{i-1}\left(\mathbb{S}^2_{u,v}\right)}\right],\text{ for }i\in \{1,2\},
\end{equation}
\begin{equation}\label{pokdwimnfe2}
\left\vert\left\vert \nu\right\vert\right\vert_{\tilde H^i\left(\mathbb{S}^2_{u,v}\right)} \lesssim (-u)\left\vert\left\vert \slashed{\rm div}\nu\right\vert\right\vert_{\tilde H^{i-1}\left(\mathbb{S}^2_{u,v}\right)},\text{ for }i \in \{1,2\}.
\end{equation}
\end{lemma}
\begin{proof}
The estimates~\eqref{pokdwimnfe} and~\eqref{pokdwimnfe2} are straightforward consequences of the following well-known identities and the bootstrap assumptions.
\[\int_{\mathbb{S}^2}\left[\left|\slashed{\nabla}\theta\right|^2 + K\left|\theta\right|^2\right]\slashed{dVol} = \int_{\mathbb{S}^2}\left[\left|\slashed{\rm div}\theta\right|^2 + \left|\slashed{\rm curl}\theta\right|^2\right]\slashed{dVol},\]
\[\int_{\mathbb{S}^2}\left[\left|\slashed{\nabla}\nu\right|^2 + 2K\left|\nu\right|^2\right]\slashed{dVol} = 2\int_{\mathbb{S}^2}\left|\slashed{\rm div}\nu\right|^2\slashed{dVol}.\]

Finally, to obtain~\eqref{pokdwimnfe3}, we simply write 
\[\slashed{\rm div}\left(\slashed{\nabla}f\right) = \slashed{\Delta}f,\qquad \slashed{\rm curl}\left(\slashed{\nabla}f\right) = 0,\]
and apply~\eqref{pokdwimnfe}.
\end{proof}
\subsubsection{Estimates for the curvature component norm $\mathscr{A}$}
The following lemma will play an important role in the energy estimates for $\widetilde{\alpha}$.\begin{lemma}\label{renormalalphaeqn}We have 
\begin{equation}\label{renormalaplha}
\overline{\Omega\nabla}_3\left(\overset{\triangleright}{\Omega^{-2}\alpha}\right)_{AB} + \frac{1}{2}\overline{\Omega{\rm tr}\underline{\chi}}\left(\overset{\triangleright}{\Omega^{-2}\alpha}\right)_{AB} - 8\overline{\Omega\underline{\omega}}\left(\overset{\triangleright}{\Omega^{-2}\alpha}\right)_{AB} = 0.
\end{equation}

\end{lemma}
\begin{proof}We start by noting that a straightforward calculation shows that~\eqref{hatchitraingle} is equivalent to
\begin{equation}\label{3hatjcjow}
\overline{\Omega\nabla}_3\left(\overset{\triangleright}{\Omega^{-1}\hat{\chi}}\right)_{AB} + \frac{1}{2}\overline{\Omega{\rm tr}\underline{\chi}}\left(\overset{\triangleright}{\Omega^{-1}\hat{\chi}}\right)_{AB} - 4\overline{\Omega\underline{\omega}}\left(\overset{\triangleright}{\Omega^{-1}\hat{\chi}}\right)_{AB} = \left(\overline{\slashed{\nabla}\hat{\otimes}\eta + \eta\hat{\otimes}\eta - \frac{1}{2}\left(\Omega^{-1}{\rm tr}\chi\right)\Omega\hat{\underline{\chi}}}\right)_{AB}.
\end{equation}
Next we note that $\left[\overline{\Omega\nabla}_3,\mathcal{L}_v\right] = 0$ and that 
\[\overline{\Omega\underline{\omega}} = -\frac{\kappa}{2u} - \frac{1}{2}\mathcal{L}_{\overline{b}}\log\left(\overline{\left(\frac{v}{-u}\right)^{\kappa}\Omega}\right).\]
Thus we immediately obtain~\eqref{renormalaplha}.
\end{proof}

Next, we analyze the ``initial data'' terms which will come up in our energy estimates.
\begin{lemma}\label{initenergyisok}Let $\left(\mathcal{M},g_{\mu\nu}\right)$ satisfy the hypothesis of Proposition~\ref{letsbootit}. Then, for every $v \in (0,\underline{v}]$, $i \in \{0,1,2\}$ and  curvature component $\Psi \neq \underline{\alpha}_{AB}$ we have 
\begin{align}\label{somefikindoksokfosss}
\int_0^v\int_{\mathbb{S}^2}\Omega^2v^{-1+\frac{q}{100}}\left|\slashed{\nabla}^i\widetilde{\Psi}\right|^2|_{(-1,\dot{v},\theta^A)}\, d\dot{v} \mathring{\rm dVol} \lesssim \epsilon^{2-2\delta}.
\end{align}
\end{lemma}
\begin{proof}

Let us start with $\widetilde{\beta}_A$. It follows from the definition of $\widetilde{\beta}_A$ and~\eqref{tcod1} that we have
\begin{align}\label{defbetabeta}
\widetilde{\beta}_A &=  -\left(\slashed{\rm div} - \overline{\slashed{\rm div}}\right)\left(\overset{\triangleright}{\Omega^{-1}\hat{\chi}}\right)_A + \left(\frac{1}{2}\slashed{\nabla}_A\left(\Omega^{-1}{\rm tr}\chi\right) - \frac{1}{2}\overline{\slashed{\nabla}\left(\Omega^{-1}{\rm tr}\chi\right)}_A\right)
\\ \nonumber &\qquad  - \widetilde{\eta_B}\slashed{g}^{BC}\left(\overset{\triangleright}{\Omega^{-1}\hat{\chi}}_{CA}\right) + \frac{1}{2}\left(\left(\eta\left(\Omega^{-1}{\rm tr}\chi\right)\right)_A - \overline{\eta\left(\Omega^{-1}{\rm tr}\chi\right)}_A\right).
\end{align} 
Before we estimate this, we need to establish some bounds for $\widetilde{\slashed{g}}_{AB}$, $\widetilde{\eta}_A$,  and $\widetilde{{\rm tr}\chi}$. Since along $\{u = -1\}$ we have that $\Omega = v^{-\kappa}$ and $\Omega^{-1}\hat{\chi} = \overset{\triangleright}{\Omega^{-1}\hat{\chi}}_{C(A}\left(\slashed{g}_0^{-1}\right)^{CD}\slashed{g}_{B)D}$, we have the following equations along $\{u=-1\}$:
\begin{equation}\label{somestuffuminu1thiswilldothetrickok}
\mathcal{L}_{\partial_v}\slashed{g}_{AB} = 2v^{-2\kappa}\left(\overset{\triangleright}{\Omega^{-1}\hat{\chi}}_{C(A}\left(\slashed{g}_0^{-1}\right)^{CD}\slashed{g}_{B)D} + \frac{1}{2}\left(\Omega^{-1}{\rm tr}\chi\right)\slashed{g}_{AB}\right),
\end{equation}
\begin{equation}\label{ijoi3io923}
 \mathcal{L}_{\partial_v}\left(\Omega^{-1}{\rm tr}\chi\right) =  v^{-2\kappa}\left(-\frac{1}{2}\left(\Omega^{-1}{\rm tr}\chi\right)^2 - \left|\Omega^{-1}\hat{\chi}\right|^2\right),
\end{equation} 
\begin{equation}\label{asdfefefefe}
\Omega\nabla_4\eta_A = v^{-2\kappa}\left(2\Omega^{-1}\hat{\chi}\cdot\eta +\Omega^{-1}{\rm tr}\chi\eta+ \Omega^{-1}\beta\right)_A.
\end{equation}
It is now straightforward to use the bound~\eqref{yayayepsislsons} for $\overset{\triangleright}{\Omega^{-1}\hat{\chi}}_{AB}$ as well as Gr\"{o}nwall's inequality to conclude from~\eqref{defbetabeta}-\eqref{asdfefefefe} that
\begin{equation}\label{deaybetatata}
\sup_{v \in (0,\underline{v}]}v^{-1+\frac{q}{100}}\left\vert\left\vert \widetilde{\beta}\right\vert\right\vert_{\mathring{H}^6} \lesssim \epsilon^{1-\delta}.
\end{equation}
Of course, from~\eqref{deaybetatata} we obtain that~\eqref{somefikindoksokfosss} holds for $\widetilde{\Psi} = \widetilde{\beta}_A$. For $\widetilde{\alpha}_{AB}$, the desired bound is immediate from~\eqref{boundsfortimesgmie}, the argument above, and the definition of $\widetilde{\alpha}_{AB}$. 

Next we come to $\widetilde{\sigma}$. Using~\eqref{4sigma}, the following equation is easily derived
\begin{align}
\label{4sigmabetter} \Omega\nabla_4\widetilde{\sigma} &= \Omega^2\left[-\frac{3}{2}\left(\Omega^{-1}{\rm tr}\chi\right) \left(\widetilde{\sigma}+\frac{1}{2}\overline{\hat{\underline{\chi}}\wedge}\overset{\triangleright}{\Omega^{-1}\hat{\chi}}\right) -\slashed{\rm div}{}^*\left(\Omega^{-1}\beta\right) + \frac{1}{2}\widetilde{\hat{\underline{\chi}}\wedge} \alpha + \frac{1}{2}\overline{\hat{\underline{\chi}}\wedge}\widetilde{\alpha}+\underline{\eta}\wedge \left(\Omega^{-1}\beta\right)\right].
\end{align}
Integrating the equation~\eqref{4sigmabetter}, using also the $\nabla_4$ equation for $\hat{\underline{\chi}}_{AB}$, and arguing as above easily lead to the desired bound for $\widetilde{\sigma}$. The bounds for $\widetilde{\rho}$ and $\widetilde{\underline{\beta}}_A$ are obtained in a similar fashion.
\end{proof}

Now we are ready to begin the energy estimates.
\begin{proposition}\label{betarhosigenergyest}Let $\left(\mathcal{M},g_{\mu\nu}\right)$ satisfy the hypothesis of Proposition~\ref{letsbootit}. Then we have 
\[\left\vert\left\vert \left(\alpha,\beta\right)\right\vert\right\vert_{\mathscr{A}_{q,\tilde u,\tilde v}} \lesssim \epsilon^{2-2\delta}.\]

\end{proposition}
\begin{proof}From~\eqref{3alpha}  we may easily derive the following:
\begin{align}\label{3alpha2}
&\Omega\nabla_3\left(\Omega^{-2}\alpha\right)_{AB} +\frac{1}{2}{\rm tr}\underline{\chi} \left(\Omega^{-2}\alpha\right)_{AB} -8\left(\Omega\underline{\omega}\right)\left(\Omega^{-2}\alpha\right)_{AB} = 
\\ \nonumber &\qquad \left(\slashed{\nabla}\hat{\otimes}\left(\Omega^{-1}\beta\right)  -3\left(\left(\Omega^{-1}\hat{\chi}\right)\rho + {}^*\left(\Omega^{-1}\hat{\chi}\right)\sigma\right) + 5\eta\hat{\otimes}\left(\Omega^{-1}\beta\right)\right)_{AB},
\end{align}
Then, from Lemma~\ref{renormalalphaeqn} and~\eqref{3alpha2} we obtain
\begin{align}\label{3alpha3}
&\left(\Omega\nabla_3\widetilde{\alpha} +\frac{1}{2}\Omega{\rm tr}\underline{\chi} \widetilde{\alpha}-8\left(\Omega\underline{\omega}\right)\widetilde{\alpha}\right)_{AB} = 
\\ \nonumber &\qquad \underbrace{\left(\widetilde{\Omega\nabla}_3\left(\overset{\triangleright}{\Omega^{-2}\alpha}\right) +\frac{1}{2}\widetilde{{\rm tr}\underline{\chi}}\left( \overset{\triangleright}{\Omega^{-2}\alpha}\right)-6\widetilde{\underline{\omega}}\left(\overset{\triangleright}{\Omega^{-2}\alpha} \right)+\slashed{\nabla}\hat{\otimes}\left(\overset{\triangleright}{\Omega^{-1}\beta}\right)\right)_{AB}}_{\mathcal{E}_1}+ 
\\ \nonumber &\qquad \left(\slashed{\nabla}\hat{\otimes}\widetilde{\beta}\right)_{AB} +\underbrace{\left( -3\left(\left(\Omega^{-1}\hat{\chi}\right)\rho + {}^*\left(\Omega^{-1}\hat{\chi}\right)\sigma\right) + 5\eta\hat{\otimes}\left(\Omega^{-1}\beta\right)\right)_{AB}}_{\mathcal{F}_1}.
\end{align}

Next, we use~\eqref{4beta} to derive
\begin{align}\label{4beta2} &\Omega\nabla_4\left(\Omega^{-1}\beta\right)_A + 2\Omega^2\left(\Omega^{-1}{\rm tr}\chi\right)\left(\Omega^{-1}\beta\right)_A = \slashed{\rm div}\alpha_A + \left(\eta\cdot\alpha\right)_A.
\end{align}
Our next goal is to write this in terms of $\widetilde{\beta}_A$ and $\widetilde{\alpha}_{AB}$. We start with
\begin{align}\label{kodwkodwkdwok}
\slashed{\rm div} \alpha_A &= \slashed{\rm div}\left(\Omega^2\Omega^{-2}\alpha\right)_A
\\ \nonumber &= \slashed{\rm div}\left(\Omega^2\widetilde{\Omega^{-2}\alpha}\right)_A + \widetilde{\slashed{\rm div}\Big(\Omega^2}\left(\overset{\triangleright}{\Omega^{-2}\alpha}\right)\Big)_A + \overline{\slashed{\rm div}}\left[\left(\overline{\left(\frac{v}{-u}\right)^{\kappa}\Omega}\right)^2\left(\frac{v}{-u}\right)^{-2\kappa}\left(\overset{\triangleright}{\Omega^{-2}\alpha}\right)\right]_A
\\ \nonumber &= \slashed{\rm div}\left(\Omega^2\widetilde{\alpha}\right)_A + \left(\frac{v}{-u}\right)^{-2\kappa}\widetilde{\slashed{\rm div}\Big(\left(\left(\frac{v}{-u}\right)^{\kappa}\Omega\right)^2}\left(\overset{\triangleright}{\Omega^{-2}\alpha}\right)\Big)_A + \mathcal{L}_{\partial_v}\left(\overline{\slashed{\rm div}}\left(\overset{\triangleright}{\Omega^{-1}\hat{\chi}}\right)\right)_A.
\end{align}
Similarly
\begin{align}\label{kodwkodwkdwok2}
\left(\eta\cdot\alpha\right)_A &= \left(\eta\cdot\Omega^2\widetilde{\alpha}\right)_A + \left(\frac{v}{-u}\right)^{-2\kappa}\left(\widetilde{\eta\cdot\left(\left(\frac{v}{-u}\right)^{\kappa}\Omega\right)^2}\overset{\triangleright}{\Omega^{-2}\alpha}\right)_A + \mathcal{L}_{\partial_v}\left(\overline{\eta}\cdot\left(\overset{\triangleright}{\Omega^{-1}\hat{\chi}}\right)\right)_A.
\end{align}
Combining~\eqref{4beta2} with~\eqref{kodwkodwkdwok} and~\eqref{kodwkodwkdwok2} leads to
\begin{align}\label{4beta3}
&\Omega\nabla_4\widetilde{\beta}_A + 2\Omega^2\left(\Omega^{-1}{\rm tr}\chi\right)\widetilde{\beta}_A = \Omega^2\slashed{\rm div}\widetilde{\alpha}_A  +\Omega^2 \left(\left(2\eta+\underline{\eta}\right)\cdot \widetilde{\alpha}\right)_A
\\ \nonumber &+\Omega^2\Bigg[\Omega^{-1}\chi\cdot\left(\overset{\triangleright}{\Omega^{-1}\beta}\right) - 2\left(\Omega^{-1}{\rm tr}\chi\right)\overset{\triangleright}{\Omega^{-1}\beta}+ \left(\frac{v}{-u}\right)^{-2\kappa}\widetilde{\slashed{\rm div}\Big(\left(\left(\frac{v}{-u}\right)^{\kappa} \Omega\right)^2}\left(\overset{\triangleright}{\Omega^{-2}\alpha}\right)\Big)
\\ \nonumber &+ \left(\frac{v}{-u}\right)^{-2\kappa}\widetilde{\eta\left(\left(\frac{v}{-u}\right)^{\kappa}\Omega\right)^2}\overset{\triangleright}{\Omega^{-2}\alpha}\Bigg]_A 
\\ \nonumber &\doteq \Omega^2\slashed{\rm div}\widetilde{\alpha}_A  +\Omega^2 \left(\left(2\eta+\underline{\eta}\right)\widetilde{\alpha}\right)_A + \mathcal{E}_2.
\end{align}

Finally, for each $i \in \{0,1,2\}$, we commute~\eqref{3alpha3} with $\Omega\slashed{\nabla}^i$ and~\eqref{4beta3} with $\slashed{\nabla}^i$. We end up with (suppressing the indices on the covariant derivative in the rest of the proof for typographical purposes)
\begin{align}\label{finallyforalphawenergy}
&\Omega\nabla_3\left(\Omega\slashed{\nabla}^i\widetilde{\alpha}\right)_{AB} + \frac{1+i}{2}\Omega{\rm tr}\underline{\chi}\left(\Omega\slashed{\nabla}^i\widetilde{\alpha}\right)_{AB} - 6\left(\Omega\underline{\omega}\right)\Omega\slashed{\nabla}^i\widetilde{\alpha}_{AB} =
\\ \nonumber &\Omega\left(\slashed{\nabla}\hat{\otimes}\slashed{\nabla}^i\widetilde{\beta}\right)_{AB} + \Omega\slashed{\nabla}^i\mathcal{E}_1+ \Omega\slashed{\nabla}^i\mathcal{F}_1
\\ \nonumber &+\underbrace{\Omega\left[\Omega\nabla_3,\slashed{\nabla}^i\right]\widetilde{\alpha}_{AB} +\frac{i}{2}\Omega{\rm tr}\underline{\chi}\widetilde{\alpha}_{AB}+ \Omega\left[\slashed{\nabla}^i,\slashed{\nabla}\hat{\otimes}\right]\widetilde{\beta}_{AB} + \Omega\slashed{\nabla}^i\Big(\frac{1}{2}\Omega{\rm tr}\underline{\chi}\left(\Omega\widetilde{\alpha}\right)_{AB}}_{\mathcal{G}_1}
\\ \nonumber &\underbrace{ - 6\left(\Omega\underline{\omega}\right)\widetilde{\alpha}_{AB} \Big)-\frac{1}{2}\Omega{\rm tr}\underline{\chi}\left(\Omega\slashed{\nabla}^i\widetilde{\alpha}\right)_{AB} - 6\left(\Omega\underline{\omega}\right)\Omega\slashed{\nabla}^i\widetilde{\alpha}_{AB}}_{\mathcal{G}_1},
\end{align}
\begin{align}\label{finallyforbetawenergy}
&\Omega\nabla_4\slashed{\nabla}^i\widetilde{\beta}_A =
\\ \nonumber &\qquad \left(\Omega^2\slashed{\rm div}\left(\slashed{\nabla}^i\widetilde{\alpha}\right) + \slashed{\nabla}^i\left[\Omega^2\left(2\eta+\underline{\eta}\right)\widetilde{\alpha} - 2\Omega^2\left(\Omega^{-1}{\rm tr}\chi\right)\widetilde{\beta}\right]\right)_A + \slashed{\nabla}^i\mathcal{E}_2 + \underbrace{\left(\left[\Omega\nabla_4,\slashed{\nabla}^i\right]\widetilde{\beta} + \left[\slashed{\nabla}^i,\Omega^2\slashed{\rm div}\right]\widetilde{\alpha}\right)_A}_{\mathcal{G}_2}.
\end{align}
Finally, we conjugate with the weight $w \doteq \frac{(-u)^{2-q+i}}{v^{1/2-q}}$ and obtain
\begin{align}\label{finallyforalphawenergy2}
&\Omega\nabla_3\left(w\Omega\slashed{\nabla}^i\widetilde{\alpha}\right)_{AB} + \left(\frac{2-q+i}{-u}+\frac{1+i}{2}\Omega{\rm tr}\underline{\chi}\right)\left(w\Omega\slashed{\nabla}^i\widetilde{\alpha}\right)_{AB} - 6\left(\Omega\underline{\omega}\right)w\Omega\slashed{\nabla}^i\widetilde{\alpha}_{AB} = \\ \nonumber &\qquad w\Omega\left(\slashed{\nabla}\hat{\otimes}\slashed{\nabla}^i\widetilde{\beta}\right)_{AB} + w\Omega\slashed{\nabla}^i\mathcal{E}_1+ w\Omega\slashed{\nabla}^i\mathcal{F}_1 + w\mathcal{G}_1,
\end{align}
\begin{align}\label{finallyforbetawenergy2}
&\Omega\nabla_4\left(w\slashed{\nabla}^i\widetilde{\beta}\right)_A + \frac{1/2-q}{v}w\slashed{\nabla}^i\widetilde{\beta}_A =
\\ \nonumber &\qquad  w\Omega^2\slashed{\rm div}\left(\slashed{\nabla}^i\widetilde{\alpha}\right)_A + w\slashed{\nabla}^i\left[\Omega^2\left(\left(2\eta+\underline{\eta}\right)\cdot \widetilde{\alpha}\right)_A - 2\Omega^2\left(\Omega^{-1}{\rm tr}\chi\right)\widetilde{\beta}_A\right] + w\slashed{\nabla}^i\mathcal{E}_2 + w\mathcal{G}_2.
\end{align}
We now note a key fact that by the bootstrap assumptions~\eqref{boot1} and~\eqref{boot2} we have that 
\[ \left(\frac{2-q+i}{-u}+\frac{1+i}{2}\Omega{\rm tr}\underline{\chi}\right) \gtrsim (-u)^{-1} > 0.\]
Of course we also have that
\[\left(\frac{1}{2}-q\right)v^{-1} > 0.\]
Thus, after contracting~\eqref{finallyforalphawenergy2} with $w\Omega\slashed{\nabla}^i\widetilde{\alpha}_{AB}$,~\eqref{finallyforbetawenergy2} with $w\slashed{\nabla}^i\widetilde{\beta}_A$, adding the resulting equations together, integrating by parts, and using the bootstrap assumptions, we end up with 
\begin{align}
&\sup_{\left(u,v\right) \in \mathcal{R}_{\tilde u,\tilde v}}\left[\int_0^v\int_{\mathbb{S}^2}\Omega^2\frac{(-u)^{4-2q+2i}}{\dot{v}^{1-2q}}\left|\slashed{\nabla}^i\widetilde{\alpha}\right|^2\, \mathring{\rm dVol}\, d\dot{v} + \int_{-1}^u\int_0^v\int_{\mathbb{S}^2}\frac{(-\dot{u})^{3-2q+2i}}{\dot{v}^{1-2q}}\Omega^2\left|\slashed{\nabla}^i\widetilde{\alpha}\right|^2\mathring{\rm dVol}\, d\dot{u}\, d\dot{v}\right]
\\ \nonumber &\qquad + \sup_{\left(u,v\right) \in \mathcal{R}_{\tilde u,\tilde v}}\Bigg[\int_{-1}^u\int_{\mathbb{S}^2}\frac{(-\dot{u})^{4-2q+2i}}{v^{1-2q}}\left|\slashed{\nabla}^i\widetilde{\beta}\right|^2\, \mathring{\rm dVol}\, d\dot{u} + \int_{-1}^u\int_0^v\int_{\mathbb{S}^2}\frac{(-\dot{u})^{4-2q+2i}}{\dot{v}^{2-2q}}\left|\slashed{\nabla}^i\widetilde{\beta}\right|^2\mathring{\rm dVol}\, d\dot{u}\, d\dot{v}\Bigg] \lesssim 
\\ \nonumber &\qquad \int_{-1}^{\tilde u}\int_0^{\tilde v}\int_{\mathbb{S}^2}\frac{(-\dot{u})^{5-2q+2i}}{\dot{v}^{1-2q}}\left[ \left|\Omega\slashed{\nabla}^i\mathcal{E}_1\right|^2+ \left|\Omega\slashed{\nabla}^i\mathcal{F}_1\right|^2 + \left|\mathcal{G}_1\right|^2\right]\mathring{\rm dVol}\, d\dot{u}\, d\dot{v}
\\ \nonumber &\qquad +  \int_{-1}^{\tilde u}\int_0^{\tilde v}\int_{\mathbb{S}^2}\frac{(-\dot{u})^{4-2q+2i}}{\dot{v}^{-2q}}\left[\left|\slashed{\nabla}^i\left[\Omega^2\left(2\eta+\underline{\eta}\right)\widetilde{\alpha} - 2\Omega^2\left(\Omega^{-1}{\rm tr}\chi\right)\widetilde{\beta}\right]\right|^2+ \left|\slashed{\nabla}^i\mathcal{E}_2\right|^2 + \left|\mathcal{G}_2\right|^2\right]\mathring{\rm dVol}\, d\dot{u}\, d\dot{v}
\\ \label{fkowddwdwdwdwdw} &\qquad \doteq \int_{-1}^{\tilde u}\int_0^{\tilde v}\int_{\mathbb{S}^2}\frac{(-\dot{u})^{5-2q+2i}}{\dot{v}^{1-2q}}\left|I\right| \mathring{\rm dVol}\, d\dot{u}\, d\dot{v} +  \int_{-1}^{\tilde u}\int_0^{\tilde v}\int_{\mathbb{S}^2}\frac{(-\dot{u})^{4-2q+2i}}{\dot{v}^{-2q}} \left| II \right|\mathring{\rm dVol}\, d\dot{u}\, d\dot{v}.
\end{align}

Let us start with the analysis of the terms in $I$. We will group these terms into three categories:
\begin{enumerate}
	\item We have the ``linear/data'' terms contained in $\Omega\slashed{\nabla}^i\mathcal{E}_1$. Using the bootstrap hypothesis, Sobolev inequalities on $\mathbb{S}^2$,  Lemma~\ref{boundfromdatasub}, and the crucial fact that in the norms~\eqref{riccivanish}, ~\eqref{riccivanish2},~\eqref{riccivanish3}, and~\eqref{444} we see $\frac{q}{100}$ instead of $2q$, leads to the bound (where we suppress the volume form for typographical reasons)
	\begin{align}\label{admowd23244353}
	\int_{-1}^{\tilde u}\int_0^{\tilde v}\int_{\mathbb{S}^2}\frac{(-\dot{u})^{5-2q+2i}}{\dot{v}^{1-2q}}\left|\Omega\slashed{\nabla}^i\mathcal{E}_1\right|^2 &\lesssim \epsilon^{2-2\delta}\int_{-1}^{\tilde u}\int_0^{\tilde v}(-\dot{u})^{-1-\frac{q}{10}}\dot{v}^{-1+\frac{q}{10}}\log^{24}\left(\frac{\dot{v}}{-\dot{u}}\right)
	\\ \nonumber &\lesssim \epsilon^{2-2\delta}.
	\end{align}
	\item Next, we have the nonlinear terms generated by $\Omega\slashed{\nabla}^i\mathcal{F}_1$ and $\mathcal{G}_1$ except for the terms generated by the commutator $\left[\slashed{\nabla}^i,\slashed{\nabla}\hat{\otimes}\right]\widetilde{\beta}_{AB}$. The contributions of these terms to $I$ are easily seen to all schematically be one of the following forms: 
	\[\Omega^2\left|\slashed{\nabla}^i\left(\Omega^{s_1}\psi_1\cdot\Omega^{s_2}\Psi_{s_2}\right)\right|^2,\qquad \Omega^2\left|\slashed{\nabla}^i\left(\Omega^{s_1}\psi_1\cdot\widetilde{\alpha}\right)\right|^2,\qquad \Omega^2\left|\slashed{\nabla}^{i-1}\left(\Omega^{s_1}\psi_1\cdot\Omega^{s_2}\psi_2\cdot\widetilde{\alpha}\right)\right|^2,\]
	where 
	\begin{enumerate}
	\item Each $\psi_i$ denotes a Ricci coefficient of signature $s_i$ which is not equal to $\omega$ or ${\rm tr}\chi$.
	\item If $\psi_i$ multiplies $\widetilde{\alpha}_{AB}$, then we must have $\psi_i \neq \hat{\chi}_{AB}$. 
	\item If ${\rm tr}\underline{\chi}$ shows up in one of the quadratic terms, it must be acted on by $\slashed{\nabla}$, and ${\rm tr}\underline{\chi}$ cannot be both of the Ricci coefficients in the cubic term. 
	\item Each $\Psi_i$ denotes a null curvature component of signature $s_i$ which is not equal to $\alpha_{AB}$. 
	\end{enumerate}
	Using the bootstrap hypothesis, Sobolev inequalities on $\mathbb{S}^2$, and Lemma~\ref{boundfromdatasub} lead to the bound (where we suppress the volume form for typographical reasons)
	\begin{align*}
	&\int_{-1}^{\tilde u}\int_0^{\tilde v}\int_{\mathbb{S}^2}\frac{(-\dot{u})^{5-2q+2i}}{\dot{v}^{1-2q}}\Omega^2\left[\left|\slashed{\nabla}^i\left(\Omega^{s_1}\psi_1\cdot\Omega^{s_2}\Psi_{s_2}\right)\right|^2+ \left|\slashed{\nabla}^i\left(\Omega^{s_1}\psi_1\cdot\widetilde{\alpha}\right)\right|^2+\left|\slashed{\nabla}^{i-1}\left(\Omega^{s_1}\psi_1\cdot\Omega^{s_2}\psi_2\cdot\widetilde{\alpha}\right)\right|^2\right]
	\\ \label{kodkowdkowd} &\qquad \lesssim \epsilon^{2-2\delta}.
	\end{align*}
\item Lastly, we have the terms generated by the commutator $\left[\slashed{\nabla}^i,\slashed{\nabla}\hat{\otimes}\right]\widetilde{\beta}_{AB}$. These occur only if $i \in \{1,2\}$ and are of the following schematic form
\[\Omega^2\left|\slashed{\nabla}^{i-1}\left(\left(K,1\right)\cdot\widetilde{\beta}\right)\right|^2.\]
We have the following immediate consequences of Sobolev inequalities on $\mathbb{S}^2$:
\[\left\vert\left\vert K \widetilde{\beta}\right\vert\right\vert_{\tilde{L}^2} \lesssim \left\vert\left\vert \widetilde{\beta}\right\vert\right\vert_{\tilde{L}^2} + \left\vert\left\vert \widetilde{K}\right\vert\right\vert_{\tilde{L}^4}\left\vert\left\vert \widetilde{\beta}\right\vert\right\vert_{\tilde{L}^4} \lesssim  \left\vert\left\vert \widetilde{\beta}\right\vert\right\vert_{\tilde{L}^2}+\left\vert\left\vert \widetilde{K}\right\vert\right\vert_{\tilde{H}^1}\left\vert\left\vert \widetilde{\beta}\right\vert\right\vert_{\tilde{H}^1},\]
\[\left\vert\left\vert \left(\slashed{\nabla}K \right)\widetilde{\beta}\right\vert\right\vert_{\tilde{L}^2} \lesssim \left\vert\left\vert \slashed{\nabla}\widetilde{K}\right\vert\right\vert_{\tilde{L}^2}\left\vert\left\vert \widetilde{\beta}\right\vert\right\vert_{\tilde{L}^{\infty}} \lesssim \left\vert\left\vert \slashed{\nabla}K\right\vert\right\vert_{\tilde{L}^2}\left\vert\left\vert \widetilde{\beta}\right\vert\right\vert_{\tilde{H}^2},\]
\[\left\vert\left\vert K\cdot\slashed{\nabla}\widetilde{\beta}\right\vert\right\vert_{\tilde{L}^2} \lesssim \left\vert\left\vert \slashed{\nabla}\widetilde{\beta}\right\vert\right\vert_{\tilde{L}^2} + \left\vert\left\vert \widetilde{K}\right\vert\right\vert_{\tilde{L}^4} \left\vert\left\vert \slashed{\nabla}\widetilde{\beta}\right\vert\right\vert_{\tilde{L}^4} \lesssim \left\vert\left\vert \slashed{\nabla}\widetilde{\beta}\right\vert\right\vert_{\tilde{L}^2} + \left\vert\left\vert \widetilde{K}\right\vert\right\vert_{\tilde{H}^1} \left\vert\left\vert \widetilde{\beta}\right\vert\right\vert_{\tilde{H}^2},\]
where all of the spaces are defined over $\mathbb{S}^2_{u,v}$. Thus, using the bootstrap assumption and the smallness of $\underline{v}$, we may easily establish that (suppressing the volume forms)
\begin{align}\label{okdkowdqqqpkbnwif}
\int_{-1}^{\tilde u}\int_0^v\int_{\mathbb{S}^2}\frac{(-\dot{u})^{5-2q+2i}}{\dot{v}^{1-2q}}\Omega^2\left|\slashed{\nabla}^{i-1}\left(\left(K,1\right)\cdot\widetilde{\beta}\right)\right|^2 &\lesssim  \sum_{j=0}^2\int_{-1}^{\tilde u}\int_0^v\int_{\mathbb{S}^2}\frac{(-\dot{u})^{5-2q+2j}}{\dot{v}^{1-2q}}\Omega^2\left|\slashed{\nabla}^j\widetilde{\beta}\right|^2  
\\ \nonumber &\lesssim \epsilon^{2-2\delta}.
\end{align}

	\end{enumerate}
	
Now we discuss the terms contained in $II$ from the expression~\eqref{fkowddwdwdwdwdw}. We will again group these into three different categories which we treat in a similar fashion to the terms in $I$. 
\begin{enumerate}
	\item We have the ``linear/data'' terms contained in $\Omega\slashed{\nabla}^i\mathcal{E}_2$. Using the bootstrap hypothesis, Sobolev inequalities on $\mathbb{S}^2$, and Lemma~\ref{boundfromdatasub} leads to the bound (where we suppress the volume form for typographical reasons)
	\begin{align}\label{aaaaoskowm3332}
	\int_{-1}^{\tilde u}\int_0^{\tilde v}\int_{\mathbb{S}^2}\frac{(-\dot{u})^{4-2q+2i}}{\dot{v}^{-2q}}\left|\Omega\slashed{\nabla}^i\mathcal{E}_2\right|^2 &\lesssim \epsilon^{2-2\delta}\int_{-1}^{\tilde u}\int_0^{\tilde v}(-\dot{u})^{-1-\frac{q}{10}}\dot{v}^{-1+\frac{q}{10}}\log^{24}\left(\frac{\dot{v}}{-\dot{u}}\right)
	\\ \nonumber &\lesssim \epsilon^{2-2\delta}.
	\end{align}
	Note that, as opposed to the bound~\eqref{admowd23244353}, we do not need to exploit fully the vanishing of any Ricci coefficients $\widetilde{\psi}$.
	\item Next we have all of the terms except for $\Omega\slashed{\nabla}^i\mathcal{E}_2$ and those generated by the commutator  $\left[\slashed{\nabla}^i,\Omega^2\slashed{\rm div}\right]\widetilde{\alpha}_A$. These are all of the following schematic form:
	\[\left|\slashed{\nabla}^i\left(\Omega^2\cdot\left(\Omega^{s_1}\psi_1\right)\cdot\left(\widetilde{\alpha},\widetilde{\beta}\right)\right)\right|^2,\qquad \left|\slashed{\nabla}^i\left(\Omega^2\cdot\left(\Omega^{s_1}\psi_1\right)\cdot\left(\Omega^{s_2}\psi_2\right)\cdot\widetilde{\beta}\right)\right|^2,\]
	where $\psi_i$ denotes a Ricci coefficient of signature $s_i$ which is not equal to $\omega$. Using the bootstrap hypothesis, Sobolev inequalities on $\mathbb{S}^2$, and Lemma~\ref{boundfromdatasub} leads to the bound (where we suppress the volume form for typographical reasons) 
	\begin{align}\label{aaaaoskowm23i9kmd2}
	&\int_{-1}^{\tilde u}\int_0^{\tilde v}\int_{\mathbb{S}^2}\frac{(-\dot{u})^{4-2q+2i}}{\dot{v}^{-2q}}\left[\left|\slashed{\nabla}^i\left(\Omega^2\cdot\left(\Omega^{s_1}\psi_1\right)\cdot\left(\widetilde{\alpha},\widetilde{\beta}\right)\right)\right|^2+ \left|\slashed{\nabla}^i\left(\Omega^2\cdot\left(\Omega^{s_1}\psi_1\right)\cdot\left(\Omega^{s_2}\psi_2\right)\cdot\widetilde{\beta}\right)\right|^2\right] \lesssim
	\\ \nonumber &\qquad  \epsilon^{2-2\delta}.
	\end{align}
	\item Finally, we have the terms generated by the commutator $\left[\slashed{\nabla}^i,\Omega^2\slashed{\rm div}\right]\widetilde{\alpha}_A$. These will be of the schematic form
	\[\Omega^2\left|\slashed{\nabla}^{i-1}\left(\left(K,1\right)\cdot\widetilde{\alpha}\right)\right|^2.\]
	The same argument which established~\eqref{okdkowdqqqpkbnwif} leads to 
	\begin{align}\label{okdkowdqqqpkbnwif2}
\int_{-1}^{\tilde u}\int_0^v\int_{\mathbb{S}^2}\frac{(-\dot{u})^{4-2q+2i}}{\dot{v}^{-2q}}\Omega^2\left|\slashed{\nabla}^{i-1}\left(\left(K,1\right)\cdot\widetilde{\alpha}\right)\right|^2 &\lesssim  \sum_{j=0}^2\int_{-1}^{\tilde u}\int_0^v\int_{\mathbb{S}^2}\frac{(-\dot{u})^{4-2q+2j}}{\dot{v}^{-2q}}\Omega^2\left|\slashed{\nabla}^j\widetilde{\alpha}\right|^2  
\\ \nonumber &\lesssim \epsilon^{2-2\delta}.
\end{align}
This concludes the proof that the terms in~\eqref{fkowddwdwdwdwdw} are bounded by $\epsilon^{2-2\delta}$ and hence finishes the proof. 
\end{enumerate}
\end{proof}

In the next proposition we carry out the analogous energy estimates for the Bianchi pairs $\left(\beta_A,\rho,\sigma\right)$, $\left(\rho,\sigma,\underline{\beta}_A\right)$, and $\left(\underline{\beta}_A,\underline{\alpha}_{AB}\right)$. 
\begin{proposition}\label{mathfrakaisdone}Let $\left(\mathcal{M},g\right)$ satisfy the hypothesis of Proposition~\ref{letsbootit}. Then we have that 
\[\mathfrak{A} \lesssim \epsilon^{1-\delta}.\]
\end{proposition}
\begin{proof}We may re-write the Bianchi equations~\eqref{3beta}-\eqref{4undalpha} as follows:
\begin{align} \label{3beta2} \Omega\nabla_3\left(\Omega^{-1}\beta\right)_A + \left(\Omega{\rm tr}\underline{\chi}\right)\Omega^{-1}\beta_A &= \slashed{\nabla}_A\widetilde{\rho} + 4\left(\Omega\underline{\omega}\right)\Omega^{-1}\beta_A + {}^*\slashed{\nabla}_A\widetilde{\sigma} + 2\left(\Omega^{-1}\hat{\chi}\right)\cdot\left(\Omega\underline{\beta}\right)_A
\\ \nonumber &\qquad + 3\left(\eta_A\rho + {}^*\eta_A\sigma\right) + \slashed{\nabla}_A\left(\rho - \widetilde{\rho}\right) + {}^*\slashed{\nabla}_A\left(\sigma-\widetilde{\sigma}\right),
\\ \label{4sigma2} \Omega\nabla_4\widetilde{\sigma} + \frac{3}{2}\Omega^2\left(\Omega^{-1}{\rm tr}\chi\right) \sigma &=\Omega^2\left[ -\slashed{\rm div}{}^*\left(\Omega^{-1}\beta\right)  -\left(\eta+2\underline{\eta}\right)\wedge \left(\Omega^{-1}\beta\right) + \frac{1}{2}\Omega\hat{\underline{\chi}}\wedge \widetilde{\alpha}\right]
\\ \nonumber &\qquad + \frac{1}{2}\left(\frac{v}{-u}\right)^{-2\kappa}\widetilde{\left(\left(\frac{v}{-u}\right)^{\kappa}\Omega\right)^2\left(\Omega\hat{\underline{\chi}}\right)}\wedge \overset{\triangleright}{\Omega^{-2}\alpha},
\\ \label{3sigma2} \Omega\nabla_3\sigma + \frac{3}{2}\left(\Omega{\rm tr}\underline{\chi}\right) \sigma &= -\slashed{\rm div}{}^*\widetilde{\underline{\beta}} - \frac{1}{2}\left(\Omega^{-1}\hat{\chi}\right)\wedge \left(\Omega^2\underline{\alpha}\right) -\eta\wedge \left(\Omega\underline{\beta}\right)-\slashed{\rm div}{}^*\left(\overline{\Omega\underline{\beta}}\right),
\\ \label{4rho2} \Omega\nabla_4\widetilde{\rho} + \frac{3}{2}\Omega^2\left(\Omega^{-1}{\rm tr}\chi\right) \rho &= \Omega^2\left[\slashed{\rm div}\left(\Omega^{-1}\beta\right) -\frac{1}{2}\left(\Omega\hat{\underline{\chi}}\right)\cdot\widetilde{\alpha} + \left(\eta+2\underline{\eta}\right)\cdot\left(\Omega^{-1}\beta\right)\right]
\\ \nonumber &\qquad - \frac{1}{2}\left(\frac{v}{-u}\right)^{-2\kappa}\widetilde{\left(\left(\frac{v}{-u}\right)^{\kappa}\Omega\right)^2\left(\Omega\hat{\underline{\chi}}\right)}\cdot\overset{\triangleright}{\Omega^{-2}\alpha},
\\ \label{3rho2} \Omega\nabla_3\rho + \frac{3}{2}\left(\Omega{\rm tr}\underline{\chi} \right)\rho &=  -\slashed{\rm div}\widetilde{\underline{\beta}} -\frac{1}{2}\left(\Omega^{-1}\hat{\chi}\right)\cdot\left(\Omega^2\underline{\alpha}\right) - \eta\cdot\left(\Omega\underline{\beta}\right)-\slashed{\rm div}\left(\overline{\Omega\underline{\beta}}\right),
\\ \label{4undbeta2} \Omega\nabla_4\widetilde{\underline\beta}_A + \Omega^2\left(\Omega^{-1}{\rm tr}\chi\right)\left(\Omega\underline{\beta}\right)_A &= \Omega^2\left[-\Omega\chi\overline{\underline{\beta}}-\slashed{\nabla}\rho + {}^*\slashed{\nabla}\sigma + 2\left(\Omega\hat{\underline{\chi}}\right)\cdot\left(\Omega^{-1}\beta\right)+ 3\left(-\underline{\eta}\rho + {}^*\underline{\eta}\sigma\right)\right]_A,
\\ \label{3undbeta2} \Omega\nabla_3\left(\Omega\underline{\beta}\right)_A + 2\left(\Omega{\rm tr}\underline{\chi}\right)\left(\Omega \underline{\beta}\right)_A &=- \slashed{\rm div}\widetilde{\underline{\alpha}}_A - 4\left(\Omega\underline{\omega}\right)\Omega\underline{\beta}_A + 2 \eta\cdot\left(\Omega^2\underline{\alpha}\right)_A - \slashed{\rm div}\left(\overline{\Omega^2\underline{\alpha}}\right)_A,
\\ \label{4undalpha2} \Omega\nabla_4\widetilde{\underline{\alpha}}_{AB} +\frac{1}{2}\Omega^2\left(\Omega^{-1}{\rm tr}\chi\right) \Omega^2\underline{\alpha}_{AB} &= \Omega^2\Big[-2\Omega\chi\cdot\overline{\underline{\alpha}}-\slashed{\nabla}\hat{\otimes}\left(\Omega\underline{\beta}\right)
\\ \nonumber &\qquad  -3\left(\left(\Omega\hat{\underline{\chi}}\right)\rho - {}^*\left(\Omega\hat{\underline{\chi}}\right)\sigma\right) + \left(\eta- 4\underline{\eta}\right)\hat{\otimes}\left(\Omega\underline{\beta}\right)\Big]_{AB}.
\end{align}
(It may be useful for the reader to draw an analogy with this form of the Bianchi equations and the renormalized Bianchi equations~\eqref{ren1}-\eqref{ren6}.)

Now we treat each of the Bianchi pairs $\left(\eqref{3beta2},\eqref{4sigma2},\eqref{4rho2}\right)$, $\left(\eqref{3sigma2},\eqref{3rho2},\eqref{4undbeta2}\right)$, and $\left(\eqref{3undbeta2},\eqref{4undalpha2}\right)$ just as we treated the $\left(\alpha_{AB},\beta_A\right)$ pair in the proof of Proposition~\ref{betarhosigenergyest}. That is, for each $i \in \{0,1,2\}$, we commute each $\nabla_4$ equation with $\slashed{\nabla}^i_{A_1\cdots A_i}$, commute each $\nabla_3$ equation with $\Omega\slashed{\nabla}^i_{A_1\cdots A_i}$, conjugate each equation with the weight $\frac{(-u)^{2-\frac{q}{200}+i}}{v^{1/2-\frac{q}{200}}}$ (note that the weight has changed from Proposition~\ref{betarhosigenergyest}), and finally carry out the energy estimate. Note the key point that other than $\alpha_{AB}$, every null curvature component satisfies a $\nabla_4$ equation, and thus, in analogy to $\beta_A$ in the proof of Proposition~\ref{betarhosigenergyest}, there will be a good spacetime term for $\rho$, $\sigma$, $\underline{\beta}_A$, and $\underline{\alpha}_{AB}$ with a $v$-weight $v^{-2+\frac{q}{100}}$. Thus, (using that we already have estimated a spacetime term for $\beta_A$ in Proposition~\ref{betarhosigenergyest}) we end up with 
\begin{equation}\label{kdkwokdow}
\mathfrak{A}^2 \lesssim \epsilon^{2-2\delta} +\sum_{i=0}^2 \int_{-1}^{\tilde u}\int_0^{\tilde v}\int_{\mathbb{S}^2}\frac{(-\dot{u})^{4-\frac{q}{100}+2i}}{\dot{v}^{-\frac{q}{100}}} \left| \mathcal{N}_i\right|\mathring{\rm dVol}\, d\dot{u}\, d\dot{v},
\end{equation}
where, just as in the proof of Proposition~\ref{betarhosigenergyest}, the terms making up $\mathcal{N}_i$ may be sorted into three categories:
\begin{enumerate}
	\item We have ``linear/data'' terms involving $\overset{\triangleright}{\Omega^{-2}\alpha}_{AB}$, $\rho-\widetilde{\rho}$, $\overline{\Omega\underline{\beta}}_A$, and $\overline{\Omega^2\underline{\alpha}}_{AB}$. We collect all of these terms below:
	\[\left|\Omega\slashed{\nabla}^i\left(\slashed{\nabla}\left(\rho - \widetilde{\rho}\right)\right)\right|^2,\qquad  \left|\Omega\slashed{\nabla}^i\left({}^*\slashed{\nabla}\left(\sigma-\widetilde{\sigma}\right)\right)\right|^2,\qquad \left|\slashed{\nabla}^i\left(\left(\frac{v}{-u}\right)^{-2\kappa}\widetilde{\left(\left(\frac{v}{-u}\right)^{\kappa}\Omega\right)^2\left(\Omega\hat{\underline{\chi}}\right)}\wedge \overset{\triangleright}{\Omega^{-2}\alpha}\right)\right|^2,\]
	\[\left|\Omega\slashed{\nabla}^i\left(\slashed{\rm div}{}^*\left(\overline{\Omega\underline{\beta}}\right)\right)\right|^2,\qquad \left|\slashed{\nabla}^i\left(\left(\frac{v}{-u}\right)^{-2\kappa}\widetilde{\left(\left(\frac{v}{-u}\right)^{\kappa}\Omega\right)^2\left(\Omega\hat{\underline{\chi}}\right)}\cdot\overset{\triangleright}{\Omega^{-2}\alpha}\right)\right|^2,\qquad \left|\Omega\slashed{\nabla}^i\left(\slashed{\rm div}\left(\overline{\Omega\underline{\beta}}\right)\right)\right|^2,\]
	\[ \left|\Omega\slashed{\nabla}^i\left(\slashed{\rm div}\left(\overline{\Omega^2\underline{\alpha}}\right)\right)\right|^2.\]
	Let $\mathcal{F}_i$ denote the sum of all of these terms. Using the bootstrap hypothesis, Sobolev inequalities on $\mathbb{S}^2$, and Lemma~\ref{boundfromdatasub} leads to the bound (where we suppress the volume form for typographical reasons)
	\begin{align}\label{aaaaoskowm333}
	\int_{-1}^{\tilde u}\int_0^{\tilde v}\int_{\mathbb{S}^2}\frac{(-\dot{u})^{4-\frac{q}{100}+2i}}{\dot{v}^{-\frac{q}{100}}}\left|\mathcal{F}_i\right|^2 &\lesssim \epsilon^{2-2\delta}\int_{-1}^{\tilde u}\int_0^{\tilde v}(-\dot{u})^{-1-\frac{q}{1000}}\dot{v}^{-1+\frac{q}{1000}}\log^{24}\left(\frac{\dot{v}}{-\dot{u}}\right)
	\\ \nonumber &\lesssim \epsilon^{2-2\delta}.
	\end{align}

	\item Next, we have all of the remaining terms except those generated by commutators of angular operators with $\slashed{\nabla}^i_{A_1\cdots A_i}$. These terms are all of the following schematic form:
	\[\left(\Omega^2,1\right)\left|\slashed{\nabla}^i\left(\Omega^{s_1}\psi_1\cdot\Omega^{s_2}\Psi_{s_2}\right)\right|^2,\qquad \left(\Omega^2,1\right)\left|\slashed{\nabla}^i\left(\Omega^{s_1}\psi_1\cdot\left(\Omega^{s_2}\Psi_2,\widetilde{\Psi}\right)\right)\right|^2\]
	\[\left(\Omega^2,1\right)\left|\slashed{\nabla}^{i-1}\left(\Omega^{s_1}\psi_1\cdot\Omega^{s_2}\psi_2\cdot\left(\Omega^{s_3}\Psi_3,\widetilde{\Psi}\right)\right)\right|^2,\]
	where 
	\begin{enumerate}
	\item Each $\psi_i$ denotes a Ricci coefficient of signature $s_i$ which is not equal to $\omega$.	
	\item Each $\Psi_i$  denotes an arbitrary null curvature component of signature $s_i$ which is not equal to $\alpha_{AB}$. 
	\item Each $\widetilde{\Psi}$ denotes an arbitrary null curvature component. 
	\end{enumerate}
	Using the bootstrap hypothesis, Sobolev inequalities on $\mathbb{S}^2$, and Lemma~\ref{boundfromdatasub} leads to the bound (where we suppress the volume form for typographical reasons)
	\begin{align}
	&\int_{-1}^{\tilde u}\int_0^{\tilde v}\int_{\mathbb{S}^2}\frac{(-\dot{u})^{4-\frac{q}{100}+2i}}{\dot{v}^{-\frac{q}{100}}}\Bigg[\left(\Omega^2,1\right)\left|\slashed{\nabla}^i\left(\Omega^{s_1}\psi_1\cdot\Omega^{s_2}\Psi_{s_2}\right)\right|^2+\left(\Omega^2,1\right)\left|\slashed{\nabla}^i\left(\Omega^{s_1}\psi_1\cdot\left(\Omega^{s_2}\Psi_2,\widetilde{\Psi}\right)\right)\right|^2
	\\ \nonumber &\qquad \qquad \qquad \qquad +\left(\Omega^2,1\right)\left|\slashed{\nabla}^{i-1}\left(\Omega^{s_1}\psi_1\cdot\Omega^{s_2}\psi_2\cdot\left(\Omega^{s_3}\Psi_3,\widetilde{\Psi}\right)\right)\right|^2\Bigg]\lesssim  \epsilon^{2-2\delta}.
	\end{align}

	\item Finally, we have the terms generated by the commutator of $\slashed{\nabla}^i_{A_1\cdots A_i}$ with angular operators. These are all of the following schematic form:
	\[\left(\Omega^2,1\right)\left|\slashed{\nabla}^{i-1}\left(\left(K,1\right)\cdot\left(\Omega^{s_1}\Psi_1,\widetilde{\Psi}\right)\right)\right|^2,\]
	where $\Psi_1$ denotes a null curvature component of signature $s_1$ not equal to $\alpha_{AB}$, $\widetilde{\Psi}$ also denotes a nulll curvature component not equal to $\alpha_{AB}$. Arguing as we did in the proof of Proposition~\ref{betarhosigenergyest} leads to the bound
	\[\int_{-1}^{\tilde u}\int_0^{\tilde v}\int_{\mathbb{S}^2}\frac{(-\dot{u})^{4-\frac{q}{100}+2i}}{\dot{v}^{-\frac{q}{100}}}\left(\Omega^2,1\right)\left|\slashed{\nabla}^{i-1}\left(\left(K,1\right)\cdot\left(\Omega^{s_1}\Psi_1,\widetilde{\Psi}\right)\right)\right|^2 \lesssim \epsilon^{2-2\delta}.\]
\end{enumerate}
We have thus show that all of the terms on the right hand side of~\eqref{kdkwokdow} are bounded by $\epsilon^{2-2\delta}$ and this completes the proof. 

\end{proof}
\subsubsection{Estimates for Ricci coefficients other than ${\rm tr}\underline{\chi}$ and $\hat{\underline{\chi}}$}

In this section we will carry out the estimates for all of the Ricci coefficients other than ${\rm tr}\underline{\chi}$ and $\hat{\underline{\chi}}_{AB}$. We start with $\underline{\omega}$.
\begin{lemma}\label{Wedididwofwkfowkfomega}Let $\left(\mathcal{M},g_{\mu\nu}\right)$ satisfy the hypothesis of Proposition~\ref{letsbootit}. Then
\begin{equation}\label{dawdecvrbvew}
\left\vert\left\vert \underline{\omega}\right\vert\right\vert_{\mathscr{B}} \lesssim \epsilon^{2-2\delta}.
\end{equation}

Furthermore, if we define
\begin{align*}
\widetilde{\underline{\omega}}^{(0)}\left(u,v,\theta\right) &\doteq \left(1-2\kappa\right)^{-1}\left(\frac{v}{-u}\right)^{1-2\kappa}\left(\overline{\left(\frac{v}{-u}\right)^{\kappa}\Omega}\right)^2\left[\overline{\frac{1}{2}\rho - \frac{1}{4}\left(\Omega\hat{\underline{\chi}}\right)\cdot\left(\Omega^{-1}\hat{\chi}\right) + \frac{1}{2}\left|\eta\right|^2 - \eta\cdot\underline{\eta}} \right]
\\ \nonumber &\qquad + \frac{1}{4}\left(\overline{\left(\frac{v}{-u}\right)^{\kappa}\Omega}\right)^2\overline{\Omega\hat{\underline{\chi}}_{AB}}\slashed{g}^{AB}\int_0^v\left(\frac{\dot{v}}{-u}\right)^{-2\kappa}\overset{\triangleright}{\Omega^{-1}\hat{\chi}}_{AB}\, d\dot{v},
\end{align*}
\[\widetilde{\underline{\omega}}^{(1)} \doteq \underline{\widetilde{\omega}} - \underline{\widetilde{\omega}}^{(0)},\]
Then
\begin{align}\label{betterwhen33}
\sup_{(u,v)\in\mathcal{R}_{\tilde u,\tilde v}}\sum_{i=0}^2\int_{-1}^u\int_{\mathbb{S}^2}\left|\slashed{\nabla}^i\widetilde{\underline{\omega}}^{(1)}\right|^2\frac{(-\dot{u})^{5+2i-\frac{q}{100}}}{v^{3-\frac{q}{100}}}\, d\dot{u}\, \mathring{\rm dVol} \lesssim \epsilon^{2-2\delta}.
\end{align}

\end{lemma}
\begin{proof}We start by multiplying~\eqref{4uomega} through by $\Omega^2$ so as to remove the $\omega\underline{\omega}$ term:
\begin{equation}\label{newundomega4}
\Omega\nabla_4\widetilde{\underline{\omega}} = \Omega^2\left(\frac{1}{2}\rho +\frac{1}{2}\left|\eta\right|^2 - \eta\cdot\underline{\eta}\right) \Rightarrow 
\end{equation}
\begin{align}\label{newundomega4pol}
\Omega\nabla_4\widetilde{\underline{\omega}}^{(1)} &= \left(\frac{v}{-u}\right)^{-2\kappa}\widetilde{\Omega^2}\left(\frac{1}{2}\rho +\frac{1}{2}\left|\eta\right|^2 - \eta\cdot\underline{\eta}\right)
\\ \nonumber &\qquad +\left(\frac{v}{-u}\right)^{-2\kappa}\left(\overline{\left(\frac{v}{-u}\right)^{\kappa}\Omega}\right)^2\left(\frac{1}{2}\widetilde{\rho} +\frac{1}{2}\left|\eta\right|^2 - \eta\cdot\underline{\eta} - \left(\overline{\frac{1}{2}\left|\eta\right|^2 - \eta\cdot\underline{\eta}}\right)\right)
\\ \nonumber &\qquad + \frac{1}{4}\left(\overline{\left(\frac{v}{-u}\right)^{\kappa}\Omega}\right)^2\left(\left(\frac{v}{-u}\right)^{-2\kappa}\left(\Omega\hat{\underline{\chi}}\right)\cdot \left(\Omega^{-1}\hat{\chi}\right) - \overline{\Omega\hat{\underline{\chi}}}_{AB}\mathcal{L}_{\partial_v}\left(\slashed{g}^{AB}\int_0^v\left(\frac{\dot{v}}{-u}\right)^{-2\kappa}\overset{\triangleright}{\Omega^{-1}\hat{\chi}}_{AB}\, d\dot{v}\right)\right)
\\ \nonumber &\doteq \mathcal{F}.
\end{align}
\begin{align*}
\Omega\nabla_4\widetilde{\underline{\omega}}^{(1)} &= \left(\frac{v}{-u}\right)^{-2\kappa}\widetilde{\Omega^2}\left(\frac{1}{2}\widetilde{\rho} +\frac{1}{2}\left(\overline{\rho -\frac{1}{2}\left(\Omega\hat{\underline{\chi}}\right)\cdot\left(\Omega^{-1}\hat{\chi}\right)}\right) + \frac{1}{2}\overline{\Omega\hat{\underline{\chi}}}\cdot\left(\overset{\triangleright}{\Omega^{-1}\hat{\chi}}\right)+\frac{1}{2}\left|\eta\right|^2 - \eta\cdot\underline{\eta}\right) 
\\ \nonumber &\qquad +\left(\frac{v}{-u}\right)^{-2\kappa}\left(\overline{\left(\frac{v}{-u}\right)^{\kappa}\Omega}\right)^2\left(\frac{1}{2}\widetilde{\rho} +\frac{1}{2}\left|\eta\right|^2 - \eta\cdot\underline{\eta} - \left(\overline{\frac{1}{2}\left|\eta\right|^2 - \eta\cdot\underline{\eta}}\right)\right)
\\ \nonumber &\doteq \mathcal{F}.
\end{align*}
Now, for $i \in \{0,1,2\}$ we can commute~\eqref{newundomega4pol} with $\slashed{\nabla}^i_{A_1\cdots A_i}$ and use~\eqref{com1} to obtain
\begin{equation}\label{newundomega42}
\Omega\nabla_4\left(\slashed{\nabla}_{A_1\cdots A_i}^i\widetilde{\underline{\omega}}^{(1)}\right) = \slashed{\nabla}^i\mathcal{F} + \mathcal{E}^{(i)},
\end{equation}
where $\mathcal{E}^{(i)}$ is controlled by terms of the schematic form
\[\slashed{\nabla}^i\left(\Omega^2\left(\Omega^{s_1}\psi_1\right)\left(\Omega\underline{\omega}\right)\right),\qquad \slashed{\nabla}^{i-1}\left(\Omega^2\left(\Omega^{s_1}\psi_1\right)\left(\Omega^{s_2}\psi_{s_2}\right)\left(\Omega\underline{\omega}\right)\right),\]
where $\psi_i$ denotes a Ricci coefficient not equal to $\omega$ of signature $s_i$. 

Now we use Lemma~\ref{boundfromdatasub}, Proposition~\ref{mathfrakaisdone}, the bootstrap assumptions~\eqref{boot1} and~\eqref{boot2}, smallness of $\epsilon$ and $\underline{v}$, Sobolev inequalities, Cauchy-Schwarz, and Gr\"{o}nwall's inequality to integrate~\eqref{newundomega42} and obtain
\begin{align}\label{thishihihssfdsfs}
\left\vert\left\vert \widetilde{\underline{\omega}}^{(1)}\right\vert\right\vert_{\tilde{H}^2\left(\mathbb{S}^2_{u,v}\right) }&\lesssim \int_0^v\left\vert\left\vert \mathcal{F}\right\vert\right\vert_{\tilde{H}^2\left(\mathbb{S}^2_{u,\dot{v}}\right)}\, d\dot{v} \lesssim \epsilon^{1-\delta}\left(\frac{v}{-u}\right)^{1-\frac{q}{200}}(-u)^{-1}.
\end{align}
Since  Lemma~\ref{boundfromdatasub} is easily seen to imply that 
\begin{align}\label{thishihihssfdsfs2}
\left\vert\left\vert \widetilde{\underline{\omega}}^{(0)}\right\vert\right\vert_{\tilde{H}^2\left(\mathbb{S}^2_{u,v}\right) } \lesssim \epsilon^{1-\delta}\left(\frac{v}{-u}\right)^{1-\frac{q}{200}}(-u)^{-1},
\end{align}
we have proven~\eqref{dawdecvrbvew}.

In order to establish~\eqref{betterwhen33} we will need to obtain a better $v$-weight than we saw in the estimate~\eqref{thishihihssfdsfs}. The reason that in~\eqref{thishihihssfdsfs} we are only able to obtain a maximal $v$-weight of $v^{1-\frac{q}{100}}$ is because of the need to control $\slashed{\nabla}^2_{AB}\widetilde{\rho}$ in the $\mathscr{A}$-norm. However, if we also integrate in $u$, then the $\mathscr{A}$-norm for $\widetilde{\rho}$ comes with a more negative $v$-weight. Thus, from~\eqref{newundomega42} we derive
\begin{align}
&\frac{1}{2}\Omega\nabla_4\left(\int_{-1}^u\int_{\mathbb{S}^2}\left|\slashed{\nabla}^i\widetilde{\underline{\omega}}^{(1)}\right|^2\frac{(-\dot{u})^{4+2i-\frac{q}{100}}}{v^{3-\frac{q}{100}}}\, d\dot{u}\, \mathring{\rm dVol} \right)
 + \frac{3/2-\frac{q}{100}}{v}\int_{-1}^u\int_{\mathbb{S}^2}\left|\slashed{\nabla}^i\widetilde{\underline{\omega}}^{(1)}\right|^2\frac{(-\dot{u})^{4+2i-\frac{q}{100}}}{v^{3-\frac{q}{100}}}\, d\dot{u}\, \mathring{\rm dVol} = 
\\ \nonumber &\qquad \int_{-1}^u\int_{\mathbb{S}^2}\left(\slashed{\nabla}^i\mathcal{F} + \mathcal{E}^{(i)} \right)\cdot\slashed{\nabla}^i\widetilde{\underline{\omega}}^{(1)}\frac{(-\dot{u})^{4+2i-\frac{q}{100}}}{v^{3-\frac{q}{100}}}\, d\dot{u}\, \mathring{\rm dVol} \Rightarrow
\end{align}
\begin{align}
\sup_{(u,v)\in\mathcal{R}_{\tilde u,\tilde v}}\int_{-1}^u\int_{\mathbb{S}^2}\left|\slashed{\nabla}^i\widetilde{\underline{\omega}}^{(1)}\right|^2\frac{(-\dot{u})^{4+2i-\frac{q}{100}}}{v^{3-\frac{q}{100}}}\, d\dot{u}\, \mathring{\rm dVol}  \lesssim \int_{-1}^u\int_0^v\left|\slashed{\nabla}^i\mathcal{F}\right|^2\frac{(-\dot{u})^{4+2i-\frac{q}{100}}}{v^{2-\frac{q}{100}}}\, d\dot{u}\, \mathring{\rm dVol} \lesssim \epsilon^{2-2\delta}.
\end{align}

\end{proof}

Next we provide the estimates for $\eta_A$.
\begin{lemma}\label{estimamifwcetateta}Let $\left(\mathcal{M},g_{\mu\nu}\right)$ satisfy the hypothesis of Proposition~\ref{letsbootit}. Then we have that
\[\left\vert\left\vert \eta \right\vert\right\vert_{\mathscr{B}} \lesssim \epsilon^{1-\delta}.\]
\end{lemma}
\begin{proof}Given that we have the following consequence of~\eqref{4eta}:
\begin{equation}\label{akokdowdketaeqn}
\Omega\nabla_4\widetilde{\eta}_A = -\Omega^2\left[\left(\Omega^{-1}\chi\right)\cdot\left(\eta-\underline{\eta}\right) - \Omega^{-1}\beta\right]_A,
\end{equation}
the proof of this lemma is carried out in an analogous manner to the proof of Lemma~\ref{Wedididwofwkfowkfomega} and we thus omit the details. (Of course it is strictly easier since we do not need to establish an analogue of~\eqref{betterwhen33}.)
\end{proof}

Next we will treat $\underline{\eta}_A$.
\begin{lemma}\label{fkowfomwfometa}Let $\left(\mathcal{M},g_{\mu\nu}\right)$ satisfy the hypothesis of Proposition~\ref{letsbootit}. Then we have that
\[\left\vert\left\vert \underline{\eta}\right\vert\right\vert_{\mathscr{B}} \lesssim \epsilon^{1-\delta}.\]
\end{lemma}
\begin{proof}We start with the following consequence of~\eqref{3ueta}:
\begin{equation}\label{etagravervseacW}
\Omega\nabla_3\underline{\eta}_A + \left(\frac{1}{2}\Omega{\rm tr}\underline{\chi} + \Omega\hat{\underline{\chi}}\right)\underline{\eta}_A = \left(\frac{1}{2}\Omega{\rm tr}\underline{\chi} + \Omega\hat{\underline{\chi}}\right)\eta_A + \Omega\underline{\beta}_A.
\end{equation}
Restricting to $\{v = 0\}$ yields 
\begin{equation}\label{etagravervseacW2}
\overline{\left(\Omega\nabla_3\underline{\eta} + \left(\frac{1}{2}\Omega{\rm tr}\underline{\chi} + \Omega\hat{\underline{\chi}}\right)\underline{\eta}\right)_A} = \overline{\left(\left(\frac{1}{2}\Omega{\rm tr}\underline{\chi} + \Omega\hat{\underline{\chi}}\right)\eta + \Omega\underline{\beta}\right)_A}.
\end{equation}
Taking the differences of~\eqref{etagravervseacW} and~\eqref{etagravervseacW2} leads to
\begin{align}\label{etagravervseacW3}
\Omega\nabla_3\widetilde{\underline{\eta}}_A + \left(\frac{1}{2}\Omega{\rm tr}\underline{\chi} + \Omega\hat{\underline{\chi}}\right)\widetilde{\underline{\eta}}_A &= \widetilde{\Omega\nabla_3}\overline{\underline{\eta}}_A + \left(\frac{1}{2}\widetilde{{\rm tr}\underline{\chi}} + \widetilde{\hat{\underline{\chi}}\cdot}\right)\overline{\underline{\eta}}_A  
\\ \nonumber &\qquad +\left(\frac{1}{2}\Omega{\rm tr}\underline{\chi} + \Omega\hat{\underline{\chi}}\right)\eta_A + \Omega\underline{\beta}_A - \overline{\left(\left(\frac{1}{2}\Omega{\rm tr}\underline{\chi} + \Omega\hat{\underline{\chi}}\right)\eta + \Omega\underline{\beta}\right)_A} 
\\ \nonumber &\doteq \mathcal{E} + \mathcal{F}.
\end{align}
Next, for $i \in \{0,1,2\}$, we commute with $\slashed{\nabla}^i$ (suppressing indices on $\slashed{\nabla}$ in the rest of the proof for typographical reasons) and then conjugate by the weight $\frac{(-u)^{i+2-\frac{q}{200}}}{v^{1-\frac{q}{200}}}$ to obtain
\begin{align}\label{etagravervseacW4}
&\Omega\nabla_3\left(\frac{(-u)^{i+2-\frac{q}{200}}}{v^{1-\frac{q}{200}}}\slashed{\nabla}^i\widetilde{\underline{\eta}}\right)_A + \left(\frac{2-\frac{q}{200}+i}{-u}+\frac{1+i}{2}\Omega{\rm tr}\underline{\chi} + \Omega\hat{\underline{\chi}}\right)\left(\frac{(-u)^{i+2-\frac{q}{200}}}{v^{1-\frac{q}{200}}}\slashed{\nabla}^i\widetilde{\underline{\eta}}\right)_A = 
\\ \nonumber &\qquad + \frac{(-u)^{i+2-\frac{q}{200}}}{v^{1-\frac{q}{200}}}\underbrace{\left[\left[\Omega\nabla_3,\slashed{\nabla}^i\right]\widetilde{\underline{\eta}} +\frac{i}{2}\Omega{\rm tr}\underline{\chi}\slashed{\nabla}^i\widetilde{\underline{\eta}}+\sum_{\overset{i_1+i_2 = i}{i_1\neq 0}}\slashed{\nabla}^{i_1}\left(\frac{1}{2}\Omega{\rm tr}\underline{\chi} + \Omega\hat{\underline{\chi}}\right)\slashed{\nabla}^{i_2}\overline{\underline{\eta}} \right]_A}_{\mathcal{G}_i}
\\ \nonumber &\qquad + \frac{(-u)^{i+2-\frac{q}{200}}}{v^{1-\frac{q}{100}}}\slashed{\nabla}^i\left(\mathcal{E} + \mathcal{F}\right).
\end{align}
Note that it is a consequence of the bootstrap assumption that 
\[\frac{2-\frac{q}{100}+i}{-u}+\frac{1+i}{2}\Omega{\rm tr}\underline{\chi} - \left|\Omega\hat{\underline{\chi}}\right| \gtrsim (-u)^{-1}.\]
We also have that $\mathcal{G}_i$ is a sum of terms which are schematically of the form
\[\slashed{\nabla}^i\left(\left(\Omega^{s_1}\psi_1\right)\cdot\widetilde{\underline{\eta}}\right),\qquad \slashed{\nabla}^{i-1}\left(\left(\Omega^{s_1}\psi_1\right)\cdot\left(\Omega^{s_2}\psi_2\right)\cdot\widetilde{\underline{\eta}}\right),\]
where 
\begin{enumerate}
\item $\psi_i$ denotes a Ricci coefficient of signature $s$ which is one of $\eta_A$, $\underline{\omega}$, $\hat{\underline{\chi}}_{AB}$, or ${\rm tr}\underline{\chi}$.
\item If $\psi_i = {\rm tr}\underline{\chi}$ in the first term, then there must be at least one angular derivative applied to it.
\item We cannot have that both $\psi_i$'s in the second term are equal to ${\rm tr}\underline{\chi}$. 
\end{enumerate}
Thus, if we contract~\eqref{etagravervseacW4} with $\frac{(-u)^{i+2-\frac{q}{200}}}{v^{1-\frac{q}{200}}}\slashed{\nabla}^i\widetilde{\underline{\eta}}_A$ and use Lemma~\ref{estimamifwcetateta}, Lemma~\ref{boundfromdatasub}, Proposition~\ref{mathfrakaisdone}, the bootstrap assumptions~\eqref{boot1} and~\eqref{boot2}, smallness of $\epsilon$ and $\underline{v}$, and Sobolev inequalities we end up with
\begin{align}
\left\vert\left\vert \underline{\eta}\right\vert\right\vert^2_{\mathscr{B}} &\lesssim \sup_{(u,v) \in \mathcal{R}_{\tilde u,\tilde v}}\int_{-1}^u\int_{\mathbb{S}^2}\frac{(-\dot{u})^{5+2i-\frac{q}{100}}}{v^{2-\frac{q}{100}}}\left[\left|\slashed{\nabla}^i\mathcal{E}\right|^2 + \left|\slashed{\nabla}^i\mathcal{F}\right|^2\right]\, d\dot{u}\, \mathring{\rm dVol} 
\\ \nonumber &\qquad + \sup_{(u,v) \in -1\times [0,\underline{v}]}\int_{\mathbb{S}^2}v^{-2+\frac{q}{100}}\left|\slashed{\nabla}^i\widetilde{\underline{\eta}}\right|^2\mathring{\rm dVol}
\\ \nonumber &\lesssim \epsilon^{2-2\delta}.
\end{align}
\end{proof}

Next we provide the estimates for ${\rm tr}\chi$ and $\hat{\chi}_{AB}$.
\begin{lemma}Let $\left(\mathcal{M},g_{\mu\nu}\right)$ satisfy the hypothesis of Proposition~\ref{letsbootit}. Then we have that
\[\left\vert\left\vert {\rm tr}\chi\right\vert\right\vert_{\mathscr{C}} \lesssim \underline{v}^{\frac{p}{10}},\qquad \left\vert\left\vert \left(\hat{\chi},{\rm tr}\chi\right)\right\vert\right\vert_{\mathscr{B}} \lesssim \epsilon^{1-\delta}.\]
\end{lemma}
\begin{proof}We start with the estimate for ${\rm tr}\chi$. From~\eqref{4trchi} we may easily derive, for any $i \in \{0,1,2\}$ (suppressing indices on $\slashed{\nabla}^i$):
\begin{equation}\label{normmdwo4trchi}
\Omega\nabla_4\left(\slashed{\nabla}^i\widetilde{{\rm tr}\chi}\right) = -\slashed{\nabla}^i\left(\Omega^2\left(\frac{1}{2}\left(\Omega^{-1}{\rm tr}\chi\right)^2 + \left|\Omega^{-1}\hat{\chi}\right|^2\right)\right) + \left[\slashed{\nabla}^i,\Omega\nabla_4\right]\widetilde{{\rm tr}\chi}.
\end{equation}
Now it is straightforward to integrate in the $v$-direction, use Lemma~\ref{boundfromdatasub}, the bootstrap assumptions~\eqref{boot1} and~\eqref{boot2}, smallness of $\epsilon$ and $\underline{v}$, Sobolev inequalities, Cauchy-Schwarz, and Gr\"{o}nwall's inequality to integrate~\eqref{newundomega42} and obtain in a similar fashion to~\eqref{thishihihssfdsfs} that
\begin{equation}\label{wedidforoateatrchi}
\left\vert\left\vert {\rm tr}\chi\right\vert\right\vert_{\mathscr{B}}  \lesssim \epsilon^{1-\delta}, \qquad \left\vert\left\vert {\rm tr}\chi\right\vert\right\vert_{\mathscr{C}}\lesssim \underline{v}^{\frac{p}{10}}.
\end{equation}

Now we come to $\hat{\chi}_{AB}$. From~\eqref{4hatchi} we may derive the following 
\begin{align}\label{Aafawedaehatchi}
\Omega\nabla_4\left(\Omega^{-1}\hat{\chi}\right)_{AB} + \Omega^2\left(\Omega^{-1}{\rm tr}\chi\right)\left(\Omega^{-1}\hat{\chi}\right)_{AB} &= -\Omega^2 \alpha_{AB}
\\ \nonumber &= -\Omega^2\widetilde{\alpha}_{AB} - \left(\frac{v}{-u}\right)^{-2\kappa}\widetilde{\Omega^2}\left(\overset{\triangleright}{\Omega^{-2}\alpha}\right)_{AB} + \mathcal{L}_v\left(\overset{\triangleright}{\Omega^{-1}\hat{\chi}}\right)_{AB}.
\end{align}
We thus obtain 
\begin{align}\label{Aafawedaehatchi2}
&\Omega\nabla_4\widetilde{\hat{\chi}}_{AB} 
\\ \nonumber &\qquad = \Omega^2\left[- \Omega^2\left(\Omega^{-1}{\rm tr}\chi\right)\left(\Omega^{-1}\hat{\chi}\right) - 2 \left(\Omega^{-1}\chi\right)^C_{\ \ (A}\left(\overset{\triangleright}{\Omega^{-1}\hat{\chi}}\right)_{B)C} -\widetilde{\alpha} - \left(\frac{v}{-u}\right)^{-2\kappa}\widetilde{\Omega^2}\left(\overset{\triangleright}{\Omega^{-2}\alpha}\right)  \right]_{AB} 
\\ \nonumber &\qquad \doteq \mathcal{E}.
\end{align}
Commuting with $\slashed{\nabla}^i$ for $i \in \{0,1,2\}$ leads to (suppressing the indices on $\slashed{\nabla}^i$)
\[\Omega\nabla_4\slashed{\nabla}^i\widetilde{\hat{\chi}}_{AB} = \slashed{\nabla}^i\mathcal{E} + \left[\slashed{\nabla}^i,\Omega\nabla_4\right]\widetilde{\hat{\underline{\chi}}}_{AB}. \]
Now we can treat this equation like~\eqref{normmdwo4trchi}, and we end up with
\[\left\vert\left\vert \hat{\chi}\right\vert\right\vert_{\mathscr{B}} \lesssim \epsilon^{1-\delta}.\]
\end{proof}

\subsubsection{Estimating ${\rm tr}\underline{\chi}$ and $\hat{\underline{\chi}}_{AB}$}
The final Ricci coefficients that we need to estimate are ${\rm tr}\underline{\chi}$ and $\hat{\underline{\chi}}_{AB}$. We will need a preliminary definition and lemmas.
\begin{definition}\label{Defofharmw}Let $\left(\mathcal{M},g_{\mu\nu}\right)$ satisfy the hypothesis of Proposition~\ref{letsbootit}. Then we define,\\ \underline{in a frame which is Lie-propagated from $\{v=0\}$}:
\begin{align}\label{hatchirenroma}
\widetilde{\hat{\underline{\chi}}}^{(0)}_{AB} &\doteq \left(1-2\kappa\right)^{-1}v\left(\frac{v}{-u}\right)^{-2\kappa}\left(\overline{\left(\frac{v}{-u}\right)^{\kappa}\Omega}\right)^2\left(\overline{\frac{1}{2}\left(\Omega^{-1}{\rm tr}\chi\right)\left(\Omega\hat{\underline{\chi}}_{AB}\right) + \left(\slashed{\nabla}\hat{\otimes}\underline{\eta}\right)_{AB} + \left(\underline{\eta}\hat{\otimes}\underline{\eta}\right)_{AB}}\right) 
\\ \nonumber &\qquad -\frac{1}{2} \left(\overline{\left(\frac{v}{-u}\right)^{\kappa}\Omega}\right)^2\overline{\Omega{\rm tr}\underline{\chi}}\int_0^v\left(\frac{\dot{v}}{-u}\right)^{-2\kappa}\overset{\triangleright}{\Omega^{-1}\hat{\chi}}_{AB}\, d\dot{v} +
\\ \nonumber &\qquad + 2 \left(\overline{\left(\frac{v}{-u}\right)^{\kappa}\Omega}\right)^2\overline{\hat{\underline{\chi}}}^{\ \ C}_{(A}\int_0^v\left(\frac{\dot{v}}{-u}\right)^{-2\kappa}\overset{\triangleright}{\Omega^{-1}\hat{\chi}}_{B)C}\, d\dot{v},
\end{align}
\begin{align}\label{betamowfa}
&\widetilde{\underline{\beta}}^{(0)}_A \doteq 
\\ \nonumber &\left(\overline{\left(\frac{v}{-u}\right)^{\kappa}\Omega}\right)^2\int_0^v\left(\frac{\dot{v}}{-u}\right)^{-2\kappa}\left[-\frac{1}{2}\overline{\left(\Omega^{-1}{\rm tr}\chi\right)\left(\Omega
\underline{\beta}\right)} - \overline{\slashed{\nabla}}\overset{\triangleright}{\rho} + \overline{{}^*\slashed{\nabla}}\overset{\triangleright}{\sigma} + 2\left(\overline{\Omega\hat{\underline{\chi}}}\right)\overset{\triangleright}{\beta} + 3\left(\overline{-\underline{\eta}}\overset{\triangleright}{\rho} + \overline{{}^*\underline{\eta}}\overset{\triangleright}{\sigma}\right)\right]_A\, d\dot{v}
\\ \nonumber &\qquad + \left(\overline{\left(\frac{v}{-u}\right)^{\kappa}\Omega}\right)^2\int_0^v\left(\frac{\dot{v}}{-u}\right)^{-2\kappa}\overset{\triangleright}{\Omega^{-1}\hat{\chi}}_{AB}\underline{\beta}^B\, d\dot{v}.
\end{align}
\begin{align}\label{trchiundaplwd}
&\widetilde{{\rm tr}\underline{\chi}}^{(0)} \doteq 
\\ \nonumber &\left(1-2\kappa\right)^{-1}v\left(\frac{v}{-u}\right)^{-2\kappa}\left(\overline{\left(\frac{v}{-u}\right)^{\kappa}\Omega}\right)^2\left(\overline{-\frac{1}{2}\left(\Omega^{-1}{\rm tr}\chi\right)\left(\Omega{\rm tr}\underline{\chi}\right) + 2\left(\rho - \frac{1}{2}\left(\Omega^{-1}\hat{\chi}\right)\cdot\left(\Omega\hat{\underline{\chi}}\right)\right) + 2\slashed{\rm div}\underline{\eta} + \left|\underline{\eta}\right|^2}\right),
\end{align}
\begin{align}\label{etaffwffww}
&\widetilde{\eta}^{(0)}_A \doteq \left(1-2\kappa\right)^{-1}v\left(\frac{v}{-u}\right)^{-2\kappa}\left(\overline{\left(\frac{v}{-u}\right)^{\kappa}\Omega}\right)^2\left(\overline{-\frac{1}{2}\left(\Omega^{-1}{\rm tr}\chi\right)\left(\eta-\underline{\eta}\right)} + \frac{1}{2}\Omega^{-1}{\rm tr}\chi \eta \right)_A +
\\ \nonumber &\qquad \left(\overline{\left(\frac{v}{-u}\right)^{\kappa}\Omega}\right)^2\int_0^v\left(\frac{\dot{v}}{-u}\right)^{-2\kappa}\left(-\overset{\triangleright}{\Omega^{-1}\hat{\chi}}\cdot\left(\overline{\eta-\underline{\eta}}\right) - \overset{\triangleright}{\Omega^{-1}\beta} + \overset{\triangleright}{\Omega^{-1}\hat{\chi}}\cdot\eta \right)_A\, d\dot{v}.
\end{align}
\begin{align*}
\left(\widetilde{\slashed{g}^{-1}}^{AB}\right)^{(0)} &\doteq \left(1-2\kappa\right)^{-1}\left(\overline{\left(\frac{v}{-u}\right)^{\kappa}\Omega}\right)^2\overline{\Omega^{-1}{\rm tr}\chi} \overline{\slashed{g}^{AB}} \left(\frac{v}{-u}\right)^{1-2\kappa}
\\ \nonumber &\qquad \qquad \qquad -2 \left(\overline{\left(\frac{v}{-u}\right)^{\kappa}\Omega}\right)^2\slashed{g}^{AC}\slashed{g}^{BD}\int_0^v\left(\frac{\dot{v}}{-u}\right)^{-2\kappa}\overset{\triangleright}{\Omega^{-1}\hat{\chi}}_{CD}\, d\dot{v},
\end{align*}
\[\widetilde{b^A}^{(0)} \doteq  -\frac{1}{8}\left(1-2\kappa\right)^{-1}\left(\overline{\eta^A-\underline{\eta}^A}\right)\left(\left(\frac{v}{-u}\right)^{\kappa}\Omega\right)^2\left(\frac{v}{-u}\right)^{1-2\kappa}\]
\[\widetilde{\hat{\underline{\chi}}}^{(1)}_{AB} \doteq \widetilde{\hat{\underline{\chi}}}_{AB} -  \widetilde{\hat{\underline{\chi}}}^{(0)}_{AB},\qquad \widetilde{\underline{\beta}}^{(1)}_A \doteq \widetilde{\underline{\beta}}_A -  \widetilde{\underline{\beta}}^{(0)}_A,\qquad \widetilde{{\rm tr}\underline{\chi}}^{(1)} \doteq \widetilde{{\rm tr}\underline{\chi}} -  \widetilde{{\rm tr}\underline{\chi}}^{(0)},\qquad \widetilde{\eta}^{(1)}_A \doteq \widetilde{\eta}_A -  \widetilde{\eta}^{(0)}_A,\]
\[\left(\widetilde{\slashed{g}^{-1}}^{AB}\right)^{(1)} \doteq \widetilde{\slashed{g}^{-1}}^{AB} - \left(\widetilde{\slashed{g}^{-1}}^{AB}\right)^{(0)},\qquad \widetilde{b^A}^{(1)} \doteq \widetilde{b^A} - \widetilde{b^A}^{(0)}.\]
All contractions here are with respect to $\slashed{g}_{AB}$. 

Finally, we will let $\mathcal{H}^{(j)}$ denote an expression for which we have
\[\left\vert\left\vert \mathcal{H}^{(j)}\right\vert\right\vert_{\tilde H^1\left(\mathbb{S}^2_{u,v}\right)} \lesssim \epsilon^{1-\delta}(-u)^{-j}\left(\frac{v}{-u}\right)^{3/2}.\]

\end{definition}

The quantities $\widetilde{\hat{\underline{\chi}}}_{AB}^{(0)}$ and $\widetilde{\hat{\underline{\beta}}}_A^{(0)}$ represents the leading order (in $\frac{v}{-u}$) parts of $\widetilde{\hat{\underline{\chi}}}_{AB}$ and $\widetilde{\underline{\beta}}_A$. In the next lemma we show that $\widetilde{\beta}^{(1)}_A$ does indeed satisfy an estimate with a larger power of $\frac{v}{-u}$ than $\widetilde{\beta}^{(0)}_A$ does.
\begin{lemma}\label{kpdwdddd}Let $\left(\mathcal{M},g_{\mu\nu}\right)$ satisfy the hypothesis of Proposition~\ref{letsbootit}. Then we have that $\widetilde{\underline{\beta}}^{(1)}_A \in \mathcal{H}^{(2)}$ and $\widetilde{\eta}^{(1)}_A \in \mathcal{H}^{(1)}$.
\end{lemma}
\begin{proof}It follows from~\eqref{4undbeta2} and the definition of $\widetilde{\beta}_A^{(0)}$ that we have
\begin{align} \label{4undbeta3} 
&\mathcal{L}_{\partial_v}\widetilde{\underline\beta}^{(1)}_A = 
\\ \nonumber &\qquad + \left(\frac{v}{-u}\right)^{-2\kappa}\widetilde{\Omega^2}\left[-\frac{1}{2}\left(\Omega^{-1}{\rm tr}\chi\right)\left(\Omega\underline{\beta}\right)-\slashed{\nabla}\rho + {}^*\slashed{\nabla}\sigma + 2\left(\Omega\hat{\underline{\chi}}\right)\cdot\left(\Omega^{-1}\beta\right)+ 3\left(-\underline{\eta}\rho + {}^*\underline{\eta}\sigma\right)+ \Omega^{-1}\hat{\chi}\cdot\left(\Omega\underline{\beta}\right)\right]_A
\\ \nonumber &\qquad +\left(\frac{v}{-u}\right)^{-2\kappa}\left(\overline{\left(\frac{v}{-u}\right)^{\kappa}\Omega}\right)^2\left[-\frac{1}{2}\left(\Omega^{-1}{\rm tr}\chi\right)\widetilde{\underline{\beta}}-\slashed{\nabla}\widetilde{\rho} + {}^*\slashed{\nabla}\widetilde{\sigma} + 2\left(\Omega\hat{\underline{\chi}}\right)\cdot\widetilde{\beta}+ 3\left(-\underline{\eta}\widetilde{\rho} + {}^*\underline{\eta}\widetilde{\sigma}\right) + \widetilde{\hat{\chi}}\cdot \Omega\underline{\beta}\right]_A
\\ \nonumber &\qquad +\left(\frac{v}{-u}\right)^{-2\kappa}\left(\overline{\left(\frac{v}{-u}\right)^{\kappa}\Omega}\right)^2\left[-\frac{1}{2}\widetilde{{\rm tr}\chi}\overline{\Omega\underline{\beta}}-\widetilde{\slashed{\nabla}}\overset{\triangleright}{\rho} + \widetilde{{}^*\slashed{\nabla}}\overset{\triangleright}{\sigma} + 2\widetilde{\hat{\underline{\chi}}}\cdot\overset{\triangleright}{\Omega^{-1}\beta}+ 3\left(-\widetilde{\underline{\eta}}\overset{\triangleright}{\rho} + \widetilde{{}^*\underline{\eta}}\overset{\triangleright}{\sigma}\right) + \overset{\triangleright}{\Omega^{-1}\hat{\chi}}\widetilde{\underline{\beta}}\right]_A.
\end{align}
Commuting with $\slashed{\nabla}$, integrating this in $v$ direction, using the bootstrap assumptions, Sobolev inequalities, as well as Proposition~\ref{mathfrakaisdone} yield that $\widetilde{\underline{\beta}}_A^{(1)} \in \mathcal{H}^{(2)}$. 

Next, we observe that it follows immediately from Proposition~\ref{mathfrakaisdone}, the bootstrap assumptions, and Sobolev inequalities, that if we integrate~\eqref{finallyforbetawenergy} in the $v$-direction for $i \in \{0,1\}$ we obtain that
\begin{equation}\label{dwdwdwdwdwdwqqqads}
\left\vert\left\vert \widetilde{\beta}\right\vert\right\vert_{\tilde H^1\left(\mathbb{S}^2_{u,v}\right)} \lesssim \left(\frac{v}{-u}\right)^{4/5}\left(-u\right)^{-2}.
\end{equation}

For $\eta$ we may derive the following equation:
\begin{align}\label{Adaaaeatetata}
&\mathcal{L}_{\partial_v}\widetilde{\eta}^{(1)}_A =
\\ \nonumber &\qquad  \left(\frac{v}{-u}\right)^{-2\kappa}\widetilde{\Omega^2}\left[-\left(\Omega^{-1}\chi\right)\left(\eta-\underline{\eta}\right) - \Omega^{-1}\beta+\Omega^{-1}\chi\cdot\eta\right]_A - \left(\frac{v}{-u}\right)^{-2\kappa}\left(\overline{\frac{v}{-u}^{\kappa}\Omega}\right)^2\left(\widetilde{\beta} + \Omega^{-1}\chi\widetilde{\eta}\right)_A
\\ \nonumber &\qquad + \left(\frac{v}{-u}\right)^{-2\kappa}\left(\overline{\frac{v}{-u}^{\kappa}\Omega}\right)^2\left[-\frac{1}{2}\widetilde{{\rm tr}\chi}\left(\eta-\underline{\eta}\right)-\frac{1}{2}\left(\overline{\Omega^{-1}{\rm tr}\chi}\right)\widetilde{\left(\eta-\underline{\eta}\right)}- \widetilde{\hat{\chi}}\cdot\left(\eta-\underline{\eta}\right) - \overset{\triangleright}{\Omega\hat{\chi}}\cdot\widetilde{\left(\eta-\underline{\eta}\right)}\right]_A
\\ \nonumber &\qquad + \left(\frac{v}{-u}\right)^{-2\kappa}\left(\overline{\frac{v}{-u}^{\kappa}\Omega}\right)^2\left(\widetilde{\chi}\cdot\overline{\eta}\right)_A.
\end{align}
Then, using~\eqref{dwdwdwdwdwdwqqqads}, we may argue as we did for $\underline{\beta}_A$ to obtain that $\widetilde{\eta}_A^{(1)} \in \mathcal{H}^{(1)}$. 
\end{proof}

The next lemma expresses the Codazzi equation in terms of $\widetilde{\hat{\underline{\chi}}}_{AB}^{(1)}$ and $\widetilde{\underline{\beta}}_A^{(1)}$ and uses Lemma~\ref{kpdwdddd}.
\begin{lemma}\label{codisnice}Let $\left(\mathcal{M},g_{\mu\nu}\right)$ satisfy the hypothesis of Proposition~\ref{letsbootit}. Then we have that
\[\slashed{\rm div}\widetilde{\hat{\underline{\chi}}}_A^{(1)} - \frac{1}{2}\slashed{\nabla}_A\widetilde{{\rm tr}\underline{\chi}}^{(1)} = \mathcal{H}^{(2)}.\]
\end{lemma}
\begin{proof}The proof is given in Appendix~\ref{dwlpdwlp}
\end{proof}

\begin{lemma}\label{Wearekofkofwtrhica}Let $\left(\mathcal{M},g_{\mu\nu}\right)$ satisfy the hypothesis of Proposition~\ref{letsbootit}. Then we have that
\[ \left\vert\left\vert \left({\rm tr}\underline{\chi},\hat{\underline{\chi}}\right)\right\vert\right\vert_{\mathscr{B}}  \lesssim \epsilon^{1-\delta},\qquad \left\vert\left\vert {\rm tr}\underline{\chi}\right\vert\right\vert_{\mathscr{C}} \lesssim \underline{v}^{\frac{p}{10}}.\]
\end{lemma}
\begin{proof}We start with the lower order estimates for ${\rm tr}\underline{\chi}$. From~\eqref{4truchi} one may derives 
\begin{equation}\label{kfk9wfk9fk0}
\Omega\nabla_4\widetilde{{\rm tr}\underline{\chi}} = 2\Omega^2\left(-\frac{1}{4}\left(\Omega^{-1}{\rm tr}\chi\right)\left(\Omega{\rm tr}\underline{\chi}\right) + \rho-\frac{1}{2}\left(\Omega\hat{\underline{\chi}}\right)\cdot\left(\Omega^{-1}\hat{\chi}\right) +\slashed{\rm div}\underline{\eta} +\left|\underline{\eta}\right|^2\right).
\end{equation}
Integrating this in the $v$-direction and using the bootstrap assumptions, Proposition~\ref{mathfrakaisdone}, Lemma~\ref{fkowfomwfometa}, and Sobolev inequalities immediately leads to
\begin{equation}\label{kfk9wfk9fk04}
\left\vert\left\vert {\rm tr}\underline{\chi}\right\vert\right\vert_{\mathscr{C}} \lesssim \underline{v}^{\frac{p}{10}}.
\end{equation}

For the $\mathscr{B}$-norm estimate of ${\rm tr}\underline{\chi}$ we cannot use its $\nabla_4$ equation because we do not have any estimates for $\slashed{\nabla}^3\underline{\eta}_A$.  Instead we will use the $\nabla_3$ Raychaudhuri equation. From~\eqref{3truchi} we may derive the following equation:
\begin{equation}\label{dkokwdqqqqscmcow}
\Omega\nabla_3\left(\Omega{\rm tr}\underline{\chi}\right) + \frac{1}{2}\left(\Omega{\rm tr}\underline{\chi}\right)^2 + 4\left(\Omega\underline{\omega}\right)\left(\Omega{\rm tr}\underline{\chi}\right) +\left|\Omega\hat{\underline{\chi}}\right|^2 = 0.
\end{equation}
Restricting this to $\{v = 0\}$ yields 
\begin{equation}\label{dkokwdqqqqscmcow2}
\overline{\Omega\nabla_3\left(\Omega{\rm tr}\underline{\chi}\right) + \frac{1}{2}\left(\Omega{\rm tr}\underline{\chi}\right)^2 + 4\left(\Omega\underline{\omega}\right)\left(\Omega{\rm tr}\underline{\chi}\right) +\left|\Omega\hat{\underline{\chi}}\right|^2} = 0.
\end{equation}
Taking the difference of~\eqref{dkokwdqqqqscmcow} with~\eqref{dkokwdqqqqscmcow2} leads to
\begin{align}\label{dkokwdqqqqscmcow3}
&\left(\overline{\Omega\nabla_3}+ \widetilde{b^A}\slashed{\nabla}_A\right)\left(\widetilde{{\rm tr}\underline{\chi}}\right) + \left(\overline{\Omega{\rm tr}\underline{\chi}}\right)\widetilde{{\rm tr}\underline{\chi}} + 4\left(\overline{\Omega\underline{\omega}}\right)\widetilde{{\rm tr}\underline{\chi}} + \widetilde{b^A}\slashed{\nabla}_A\overline{\Omega{\rm tr}\underline{\chi}} = \\ \nonumber &\qquad -4\widetilde{\underline{\omega}}\overline{\Omega{\rm tr}\underline{\chi}} -4\widetilde{\underline{\omega}}\widetilde{{\rm tr}\underline{\chi}}- 2\overline{\Omega\hat{\underline{\chi}}}\cdot\widetilde{\hat{\underline{\chi}}} - \widetilde{\hat{\underline{\chi}}}\cdot\widetilde{\hat{\underline{\chi}}} - \frac{1}{2}\left(\widetilde{{\rm tr}\underline{\chi}}\right)^2
\\ \nonumber &\qquad  - 2\widetilde{\left(\slashed{g}^{-1}\right)}^{AC}\slashed{g}^{BD}\overline{\Omega\hat{\underline{\chi}}}_{AB}\overline{\Omega\hat{\underline{\chi}}}_{CD} + \widetilde{\left(\slashed{g}^{-1}\right)}^{AC}\widetilde{\left(\slashed{g}^{-1}\right)}^{BD}\overline{\Omega\hat{\underline{\chi}}}_{AB}\overline{\Omega\hat{\underline{\chi}}}_{CD}.
\end{align}

Now we define
\[\widetilde{\underline{\omega}}^{\dagger} \doteq \widetilde{\underline{\omega}} - \frac{1}{4}\left(\overline{\left(\frac{v}{-u}\right)^{\kappa}\Omega}\right)^2\overline{\Omega\hat{\underline{\chi}}}\int_0^v\left(\frac{\dot{v}}{-u}\right)^{-2\kappa}\overset{\triangleright}{\Omega^{-1}\hat{\chi}}\, d\dot{v},\]
\begin{align*}
\widetilde{\hat{\underline{\chi}}}^{\dagger}_{AB} &\doteq \widetilde{\hat{\underline{\chi}}}_{AB} + \frac{1}{2}\left(\overline{\left(\frac{v}{-u}\right)^{\kappa}\Omega}\right)^2\overline{\Omega{\rm tr}\underline{\chi}}\int_0^v\left(\frac{\dot{v}}{-u}\right)^{-2\kappa}\overset{\triangleright}{\Omega^{-1}\hat{\chi}}_{AB}\, d\dot{v}
\\ \nonumber &\qquad -2 \left(\overline{\left(\frac{v}{-u}\right)^{\kappa}\Omega}\right)^2\overline{\hat{\underline{\chi}}}_{C(A}\slashed{g}^{CD}\int_0^v\left(\frac{\dot{v}}{-u}\right)^{-2\kappa}\overset{\triangleright}{\Omega^{-1}\hat{\chi}}_{B)D}\, d\dot{v},
\end{align*}
\begin{align*}
\left(\widetilde{\slashed{g}^{-1}}^{\dagger}\right)^{AB} \doteq \widetilde{\slashed{g}^{-1}}^{AB} +2 \left(\overline{\left(\frac{v}{-u}\right)^{\kappa}\Omega}\right)^2\slashed{g}^{AC}\slashed{g}^{BD}\int_0^v\left(\frac{\dot{v}}{-u}\right)^{-2\kappa}\overset{\triangleright}{\Omega^{-1}\hat{\chi}}_{CD}\, d\dot{v}. 
\end{align*}
This then allows us to write 
\begin{align}\label{dkokwdqqqqscmcow4}
&\left(\overline{\Omega\nabla_3}+ \widetilde{b^A}\slashed{\nabla}_A\right)\left(\widetilde{{\rm tr}\underline{\chi}}\right) + \left(\overline{\Omega{\rm tr}\underline{\chi}}\right)\widetilde{{\rm tr}\underline{\chi}} + 4\left(\overline{\Omega\underline{\omega}}\right)\widetilde{{\rm tr}\underline{\chi}} + \widetilde{b^A}\slashed{\nabla}_A\overline{\Omega{\rm tr}\underline{\chi}} = \\ \nonumber &\qquad -4\widetilde{\underline{\omega}}^{\dagger}\overline{\Omega{\rm tr}\underline{\chi}} -4\widetilde{\underline{\omega}}\widetilde{{\rm tr}\underline{\chi}}- 2\overline{\Omega\hat{\underline{\chi}}}\cdot\widetilde{\hat{\underline{\chi}}}^{\dagger} - \widetilde{\hat{\underline{\chi}}}\cdot\widetilde{\hat{\underline{\chi}}} - \frac{1}{2}\left(\widetilde{{\rm tr}\underline{\chi}}\right)^2
\\ \nonumber &\qquad  - 2\left(\widetilde{\slashed{g}^{-1}}^{\dagger}\right)^{AC}\slashed{g}^{BD}\overline{\Omega\hat{\underline{\chi}}}_{AB}\overline{\Omega\hat{\underline{\chi}}}_{CD} + \widetilde{\left(\slashed{g}^{-1}\right)}^{AC}\widetilde{\left(\slashed{g}^{-1}\right)}^{BD}\overline{\Omega\hat{\underline{\chi}}}_{AB}\overline{\Omega\hat{\underline{\chi}}}_{CD}.
\end{align}
Next, we note that it follows immediately from their respective definitions that each of the quantities $v^{-1+2\kappa}\widetilde{{\rm tr}\underline{\chi}}^{(0)}$, $v^{-1+2\kappa}\widetilde{\underline{\omega}}^{\dagger}$, $v^{-1+2\kappa}
\left(\widetilde{\slashed{g}^{-1}}^{\dagger}\right)^{AB}$, $v^{-1+2\kappa}\widetilde{b^A}^{(0)}$, and $v^{-1+2\kappa}\widetilde{\hat{\underline{\chi}}}^{\dagger}_{AB}$ all depend only on $u$ and $\theta^A$. Since it is also follows from the bootstrap assumptions and their repsective $\nabla_4$ equations that
\[\lim_{v\to 0}v^{-1+2\kappa}\left( \widetilde{b^A}^{(1)},\widetilde{{\rm tr}\underline{\chi}}^{(1)},\left(\widetilde{\slashed{g}^{-1}}^{AB}\right)^{(1)},\widetilde{\hat{\underline{\chi}}}_{AB}^{(1)},\widetilde{b^A\cdot\slashed{\nabla}_A}\widetilde{{\rm tr}\underline{\chi}},\widetilde{\underline{\omega}}\widetilde{{\rm tr}\underline{\chi}}, \widetilde{\hat{\underline{\chi}}}\cdot\widetilde{\hat{\underline{\chi}}},\widetilde{\left(\slashed{g}^{-1}\right)}^{AC}\widetilde{\left(\slashed{g}^{-1}\right)}^{BD}\right) = 0,\]
we may multiply~\eqref{dkokwdqqqqscmcow4} with $v^{-1+2\kappa}$, take the limit as $v\to 0$, extend the resulting equation to be independent of $v$, and finally multiply by $v^{1-2\kappa}$ and subtract the result from~\eqref{dkokwdqqqqscmcow4} to conclude that
\begin{align}\label{dkokwdqqqqscmcow5}
&\left(\overline{\Omega\nabla_3}+ \widetilde{b^A}\slashed{\nabla}_A\right)\left(\widetilde{{\rm tr}\underline{\chi}}^{(1)}\right)+ \left(\overline{\Omega{\rm tr}\underline{\chi}}\right)\widetilde{{\rm tr}\underline{\chi}}^{(1)} + 4\left(\overline{\Omega\underline{\omega}}\right)\widetilde{{\rm tr}\underline{\chi}}^{(1)}  = 
\\ \nonumber &\qquad \qquad \underbrace{-4\widetilde{\underline{\omega}}^{(1)}\overline{\Omega{\rm tr}\underline{\chi}} - 2\overline{\Omega\hat{\underline{\chi}}}_{AB}\widetilde{\hat{\underline{\chi}}}_{CD}^{(1)}\overline{\slashed{g}^{AC}\slashed{g}^{BD}}- 2\left(\widetilde{\slashed{g}^{-1}}^{AC}\right)^{(1)}\overline{\slashed{g}^{BD}}\overline{\Omega\hat{\underline{\chi}}}_{AB}\overline{\Omega\hat{\underline{\chi}}}_{CD}-\widetilde{b^A}^{(1)}\slashed{\nabla}_A\overline{\Omega{\rm tr}\underline{\chi}}}_{\mathcal{E}}
\\ \nonumber &-\underbrace{4\widetilde{\underline{\omega}}\widetilde{{\rm tr}\underline{\chi}} - \widetilde{\hat{\underline{\chi}}}\cdot\widetilde{\hat{\underline{\chi}}}-\widetilde{b\cdot\slashed{\nabla}}\widetilde{{\rm tr}\underline{\chi}}^{(0)}- 2\overline{\Omega\hat{\underline{\chi}}}_{AB}\widetilde{\hat{\underline{\chi}}}_{CD}^{\dagger}\left(\overline{\slashed{g}^{AC}\slashed{g}^{BD}}-\slashed{g}^{AC}\slashed{g}^{BD}\right)- 2\left(\widetilde{\slashed{g}^{-1}}^{\dagger}\right)^{AC}\widetilde{\slashed{g}^{BD}}\overline{\Omega\hat{\underline{\chi}}}_{AB}\overline{\Omega\hat{\underline{\chi}}}_{CD}- \frac{1}{2}\left(\widetilde{{\rm tr}\underline{\chi}}\right)^2}_{\mathcal{G}}.
\end{align}
Now for $i \in \{1,2\}$, we commute~\eqref{dkokwdqqqqscmcow5} with $\slashed{\nabla}^i$ (suppressing the indices for typographical reasons) and conjugate with the weight $w \doteq \frac{(-u)^{i +9/4} }{v^{5/4}}$ obtain
\begin{align}\label{dkokwdqqqqscmcow6}
&\overline{\Omega\nabla}_3\left(w\slashed{\nabla}^i\widetilde{{\rm tr}\underline{\chi}}^{(1)}\right) + \left(\frac{9/4+i}{-u}+\overline{\Omega{\rm tr}\underline{\chi}} + \frac{i}{2}\Omega{\rm tr}\underline{\chi}+4\overline{\Omega\underline{\omega}}\right)w\slashed{\nabla}^i\widetilde{{\rm tr}\underline{\chi}}^{(1)}  = 
\\ \nonumber &\qquad w\slashed{\nabla}^i\left(\mathcal{E}-\mathcal{G}\right) + w\left[\left[\overline{\Omega\nabla}_3,\slashed{\nabla}^i\right]\widetilde{{\rm tr}\underline{\chi}} + \frac{i}{2}\Omega{\rm tr}\underline{\chi}\slashed{\nabla}^i\widetilde{{\rm tr}\underline{\chi}}\right] + w\sum_{\overset{i_1+i_2 = i}{i_1 \neq 0}}\slashed{\nabla}^{i_1}\left(\overline{\Omega{\rm tr}\underline{\chi}}+4\overline{\Omega\underline{\omega}}\right)\slashed{\nabla}^{i_2}\widetilde{{\rm tr}\underline{\chi}}.
\end{align}
It follows easily from the bootstrap assumptions that 
\[\frac{9/4+i}{-u}+\overline{\Omega{\rm tr}\underline{\chi}} + \frac{i}{2}\Omega{\rm tr}\underline{\chi}+4\overline{\Omega\underline{\omega}} \gtrsim (-u)^{-1}.\]
Thus, we may contract~\eqref{dkokwdqqqqscmcow6} with $w\slashed{\nabla}^i\widetilde{{\rm tr}\underline{\chi}}^{(1)}$, integrate in $u$, apply the bootstrap assumption and Sobolev inequalities and eventually obtain, for any $(u,v) \in \mathcal{R}_{\tilde u,\tilde v}$ and $i \in \{1,2\}$:
\begin{align}\label{theestimatmieqqqqq}
&\int_{\mathbb{S}^2_{u,v}}\left|\slashed{\nabla}^i\widetilde{{\rm tr}\underline{\chi}}^{(1)}\right|^2\frac{(-u)^{2i+9/2}}{v^{5/2}}\mathring{\rm dVol} + \int_{-1}^u\int_{\mathbb{S}^2_{\dot{u},v}}\left|\slashed{\nabla}^i\widetilde{{\rm tr}\underline{\chi}}^{(1)}\right|^2\frac{(-\dot{u})^{2i+7/2}}{v^{5/2}}\mathring{\rm dVol}\, d\dot{u} \lesssim
\\ \nonumber &\int_{-1}^u\int_{\mathbb{S}^2}\frac{(-\dot{u})^{2i+11/2}}{v^{5/2}}\left[\left|\slashed{\nabla}^i\mathcal{E}\right|^2 + \left|\slashed{\nabla}^i\mathcal{G}\right|^2\right]\, \mathring{\rm dVol}\, d\dot{u} + \int_{\mathbb{S}^2_{-1,v}}\left|\slashed{\nabla}^i\widetilde{{\rm tr}\underline{\chi}}^{(1)}\right|^2|_{u = -1}v^{-5/2}\mathring{\rm dVol}.
\end{align}
Since every term in $\mathcal{G}$ involves (implicitly) at least two ``tilded'' quantities, the $v$-weight of $v^{-5/2}$ is not a problem and it is immediate from the bootstrap assumptions that
\begin{equation}\label{ploplop1}
\int_{-1}^u\int_{\mathbb{S}^2}\frac{(-\dot{u})^{2i+11/2}}{v^{5/2}}\left|\slashed{\nabla}^i\mathcal{G}\right|^2\, \mathring{\rm dVol}\, d\dot{u}  \lesssim \epsilon^{2-2\delta}.
\end{equation}
Next, we note that it follows from the definition of $\widetilde{{\rm tr}\underline{\chi}}^{(1)}$, the $\nabla_4$ equation for ${\rm tr}\underline{\chi}$, and a straightforward argument using Proposition~\ref{itexistsbutforalittlewhile} that 
\begin{equation}\label{ploplop2}
\int_{\mathbb{S}^2_{-1,v}}\left|\slashed{\nabla}^i\widetilde{{\rm tr}\underline{\chi}}^{(1)}\right|^2|_{u = -1}v^{-5/2}\mathring{\rm dVol} \lesssim \epsilon^{2-2\delta}.
\end{equation}
This leaves us with the $\mathcal{E}$ term. Using Lemma~\ref{Wedididwofwkfowkfomega} and the bootstrap assumptions we have that
\begin{equation}\label{ploplop3}
\int_{-1}^u\int_{\mathbb{S}^2}\frac{(-\dot{u})^{2i+11/2}}{v^{5/2}}\left|\slashed{\nabla}^i\mathcal{G}\right|^2\, \mathring{\rm dVol}\, d\dot{u}  \lesssim \epsilon^{2-2\delta} + \epsilon \int_{-1}^u\int_{\mathbb{S}^2}\frac{(-\dot{u})^{2i+7/2}}{v^{5/2}}\left|\slashed{\nabla}^i\widetilde{\hat{\underline{\chi}}}^{(1)}\right|^2\, \mathring{\rm dVol}\, d\dot{u} .
\end{equation}
Now we appeal to Lemma~\ref{codisnice} and Lemma~\ref{someelleststs} to obtain that
\begin{equation}\label{ploplop4}
\sum_{i=0}^2\int_{-1}^u\int_{\mathbb{S}^2}\frac{(-\dot{u})^{2i+7/2}}{v^{5/2}}\left|\slashed{\nabla}^i\widetilde{\hat{\underline{\chi}}}^{(1)}\right|^2\, \mathring{\rm dVol}\, d\dot{u} \lesssim \epsilon^{2-2\delta} + \sum_{i=1}^2\int_{-1}^u\int_{\mathbb{S}^2}\frac{(-\dot{u})^{2i+7/2}}{v^{5/2}}\left|\slashed{\nabla}^i\widetilde{{\rm tr}\underline{\chi}}^{(1)}\right|^2\, \mathring{\rm dVol}\, d\dot{u}. 
\end{equation}
Combining~\eqref{theestimatmieqqqqq} and \eqref{ploplop1}-\eqref{ploplop4}, and another application of Lemma~\ref{codisnice} and Lemma~\ref{someelleststs} thus leads to
\begin{equation}\label{ploplop5}
\sum_{i=1}^2\int_{\mathbb{S}^2_{u,v}}\left|\slashed{\nabla}^i\widetilde{{\rm tr}\underline{\chi}}^{(1)}\right|^2\frac{(-u)^{2i+9/2}}{v^{5/2}}\mathring{\rm dVol}  \lesssim \epsilon^{2-2\delta},
\end{equation}
\begin{equation}\label{ploplop6}
\sum_{i=1}^2\int_{\mathbb{S}^2_{u,v}}\left|\slashed{\nabla}^i\widetilde{\hat{\underline{\chi}}}^{(1)}\right|^2\frac{(-u)^{2i+9/2}}{v^{5/2}}\mathring{\rm dVol}  \lesssim \epsilon^{2-2\delta}.
\end{equation}
Recalling that $\widetilde{{\rm tr}\underline{\chi}} = \widetilde{{\rm tr}\underline{\chi}}^{(0)}+\widetilde{{\rm tr}\underline{\chi}}^{(1)}$ and $\widetilde{\hat{\underline{\chi}}}_{AB} = \widetilde{\hat{\underline{\chi}}}_{AB}^{(0)}+\widetilde{\hat{\underline{\chi}}}_{AB}^{(1)}$, we see that the proof is finished.

\end{proof}

\subsubsection{Estimates for the metric coefficient norm $\mathscr{D}$}
Finally we come to the estimates for the metric coefficients. 
\begin{lemma}Let $\left(\mathcal{M},g_{\mu\nu}\right)$ satisfy the hypothesis of Proposition~\ref{letsbootit}. Then we have 
\[\mathfrak{D} \lesssim \epsilon^{1-\delta},\qquad \left\vert\left\vert \slashed{g}\right\vert\right\vert_{\mathscr{D}} \lesssim \underline{v}^{\frac{p}{10}}.\]
\end{lemma}
\begin{proof}These estimates are all straightforward consequences of the previously established estimates and integrating the following equations from $v = 0$ or $u= -1$:
\[\mathcal{L}_{\partial_v}b^A = -4\Omega^2\zeta^A,\qquad \mathcal{L}_{\partial_v}\slashed{g}_{AB} = 2\Omega^2\left(\Omega^{-1}\chi\right)_{AB},\qquad \left(\frac{\partial}{\partial u} + b\cdot\slashed{\nabla}\right)\log\widetilde{\Omega} = -2\widetilde{\underline{\omega}}.\]

\end{proof}

This concludes the proof of Proposition~\ref{letsbootit}, and hence also Theorem~\ref{itisreg1}.

\section{The Bootstrap Argument for Region II}\label{ohmyregiontwo}
The main result of this section will be the following:
\begin{theorem}\label{itisreg2}Let $(\mathcal{M},g_{\mu\nu})$ be a spacetime produced by Proposition~\ref{itexistsbutforalittlewhile}, and let $\underline{v} > 0$ be arbitrarily small. Then we can pick $\epsilon$ sufficiently small so that $\left(\mathcal{M},g_{\mu\nu}\right)$ exists in the region $\left\{(u,v) : -1 \leq u < 0 \text{ and } 0 \leq \frac{v}{-u} \leq \underline{v}^{-1} \right\}$, and in this region the spacetime satisfies the regularity bounds~\eqref{02020202020922332}-\eqref{11212332123123123} and the estimates~\eqref{dkowdovoboto}.
\end{theorem}
We will prove this theorem with a bootstrap argument. In Section~\ref{normsregion2} we will define the relevant norms. In Section~\ref{frombefore}, we establish various useful estimates which follow from Theorem~\ref{itisreg1}. Then in Section~\ref{estimatesregion2} we will carry out the bootstrap argument. 
\subsection{Norms}\label{normsregion2}
In this section we will present the norms around which we will base our estimates. Let $0 < \underline{v} \ll 1$ be a small constant.

 Let's set 
\[\mathcal{Q} \doteq \left\{u \in (0,-1)\right\} \cap \left\{v \in [0,1]\right\} \cap \left\{\underline{v} \leq \frac{v}{|u|} \leq \underline{v}^{-1}\right\},\qquad \mathcal{Q}_{\tilde u,\tilde v} \doteq \mathcal{Q} \cap \{u \leq \tilde u\} \cap \{v \leq \tilde v\}.\]
Finally, we introduce a constant $D \gg 1$ and then assume that $\epsilon$ is picked small enough so that
\[\epsilon\exp\left(\left(\frac{D}{\underline{v}}\right)^{100}\right) \ll 1.\]

\begin{convention}Throughout this section, unless said otherwise, all norms of tensorial quantities are computed with respect to $\slashed{g}_{AB}$, and we will always use the round metric induced volume form $\mathring{dVol}$ on each $\mathbb{S}^2_{u,v}$. 
\end{convention}

\begin{definition} Let $\Psi$ be a null curvature component not equal to $\alpha_{AB}$ or $\underline{\alpha}_{AB}$, and $(\tilde u,\tilde v)$ satisfy $\underline{v} \leq \frac{\tilde v}{|\tilde u|} \leq \underline{v}^{-1}$. Then the energy norm $\mathscr{E}_{\tilde u,\tilde v}$ is defined by
\begin{align*}
\left\vert\left\vert \Psi\right\vert\right\vert_{\mathscr{E}_{\tilde u,\tilde v}}^2 \doteq \sup_{0 \leq j \leq 2}\sup_{(u_0,v_0) \in \mathcal{Q}_{\tilde u,\tilde v}}\Bigg[&\int_{-\underline{v} u_0}^{v_0}\int_{\mathbb{S}^2}\left|\exp\left(D\frac{v}{u}\right)\slashed{\nabla}^j\Psi\right|^2\left(v - u_0\right)^{3+2j}\, dv\, \mathring{\rm dVol},
\\ \nonumber &+\int_{{\rm max}\left(-\underline{v}^{-1}v_0,-1\right)}^{u_0}\int_{\mathbb{S}^2}\left|\exp\left(D\frac{v}{u}\right)\slashed{\nabla}^j\Psi\right|^2\left(v_0 - u\right)^{3+2j}\, du\, \mathring{\rm dVol}
\\ \nonumber &+\int_{-\underline{v}u_0}^{v_0}\int_{{\rm max}\left(-\underline{v}^{-1}v_0,-1\right)}^{u_0}\int_{\mathbb{S}^2}\left|\exp\left(D\frac{v}{u}\right)\slashed{\nabla}^j\Psi\right|^2\left(v - u\right)^{2+2j}\, dv\, du\, \mathring{\rm dVol}\Bigg].
\end{align*}

As usual, for $\alpha_{AB}$ and $\underline{\alpha}_{AB}$ we drop the $u$-flux and $v$-flux respectively:
\begin{align*}
\left\vert\left\vert \alpha\right\vert\right\vert_{\mathscr{E}_{\tilde u,\tilde v}}^2 \doteq \sup_{0 \leq j \leq 2}\sup_{(u_0,v_0) \in \mathcal{Q}_{\tilde u,\tilde v}}\Bigg[&\int_{-\underline{v} u_0}^{v_0}\int_{\mathbb{S}^2}\left|\exp\left(D\frac{v}{u}\right)\slashed{\nabla}^j\alpha\right|^2\left(v - u_0\right)^{3+2j}\, dv\, \mathring{\rm dVol}
\\ \nonumber &+\int_{-\underline{v}u_0}^{v_0}\int_{{\rm max}\left(-\underline{v}^{-1}v_0,-1\right)}^{u_0}\int_{\mathbb{S}^2}\left|\exp\left(D\frac{v}{u}\right)\slashed{\nabla}^j\alpha\right|^2\left(v_0 - u\right)^{2+2j}\, dv\, du\, \mathring{\rm dVol}\Bigg],
\end{align*}

\begin{align*}
\left\vert\left\vert \underline{\alpha}\right\vert\right\vert_{\mathscr{E}_{\tilde u,\tilde v}}^2 \doteq \sup_{0 \leq j \leq 2}\sup_{(u_0,v_0) \in \mathcal{Q}_{\tilde u,\tilde v}}\Bigg[&\int_{{\rm min}\left(-\underline{v}^{-1}v_0,-1\right)}^{u_0}\int_{\mathbb{S}^2}\left|\exp\left(D\frac{v}{u}\right)\slashed{\nabla}^j\underline{\alpha}\right|^2\left(v_0 - u\right)^{3+2j}\, du\, \mathring{\rm dVol}
\\ \nonumber &+\int_{-\underline{v} u_0}^{v_0}\int_{{\rm max}\left(-\underline{v}^{-1}v_0,-1\right)}^{u_0}\int_{\mathbb{S}^2}\left|\exp\left(D\frac{v}{u}\right)\slashed{\nabla}^j\underline{\alpha}\right|^2\left(v_0 - u\right)^{2+2j}\, dv\, du\, \mathring{\rm dVol}\Bigg].
\end{align*}
We also introduce the notation 
\[\mathfrak{E}_{\tilde u,\tilde v} \doteq \sum_{\Psi}\left\vert\left\vert \Psi\right\vert\right\vert_{\mathscr{E}_{\tilde u,\tilde v}}.\]
Finally, when it will not cause confusion, we will often suppress a subset of the $\left(\tilde u,\tilde v\right)$ indices from the $\mathscr{E}$ or $\mathfrak{E}$ subscript. 
\end{definition}

Then we have the corresponding norms for the Ricci coefficients.

\begin{definition}\label{3i8193defpsistartstar}For any Ricci coefficient $\psi$, we let $\psi^*$ denote the difference of $\psi$ and its Minkowski value.\footnote{Equivalently, for $\psi \not\in \{{\rm tr}\chi,{\rm tr}\underline{\chi}\}$ we have $\psi^* = \psi$, and otherwise we have ${\rm tr}\chi^* = {\rm tr}\chi - \frac{2}{v-u}$ and ${\rm tr}\underline{\chi}^* = {\rm tr}\underline{\chi} + \frac{2}{v-u}$.} Let $(\tilde u,\tilde v)$ satisfy $\underline{v} \leq \frac{\tilde v}{|\tilde u|} \leq \underline{v}^{-1}$. Then, for any Ricci coefficient $\psi$, the Ricci coefficient norm $\mathscr{F}$ is defined by
\[\left\vert\left\vert \psi\right\vert\right\vert_{\mathscr{F}_{\tilde u,\tilde v}}^2 \doteq \sup_{0 \leq j \leq 2}\sup_{(u_0,v_0) \in \mathcal{Q}_{\tilde u,\tilde v}}\int_{\mathbb{S}^2}\left|\exp\left(D\frac{v}{u}\right)\slashed{\nabla}^j\psi^*\right|^2\left(v_0-u_0\right)^{2+2j}\mathring{\rm dVol}.\]
We also introduce the notation 
\[\mathfrak{F}_{\tilde u,\tilde v} \doteq \sum_{\psi}\left\vert\left\vert \psi\right\vert\right\vert_{\mathscr{F}_{\tilde u,\tilde v}}.\]
Finally, when it will not cause confusion, we will often suppress a subset of the $\left(\tilde u,\tilde v\right)$ indices from the $\mathscr{F}$ or $\mathfrak{F}$ subscript. 
\end{definition}

Finally, we have the norm for the metric coefficients.
\begin{definition}For any metric coefficient $\phi$, we let $\phi^*$ denote the difference of $\phi$ and its Minkowski value.\footnote{Equivalently, we have $\Omega^* = \Omega - 1$, $(b^A)^* = b^A$, and $\slashed{g}_{AB}^* = \slashed{g}_{AB} - (v-u)^2\mathring{\slashed{g}}_{AB}$.} Let $(\tilde u,\tilde v)$ satisfy $\underline{v} \leq \frac{\tilde v}{|\tilde u|} \leq \underline{v}^{-1}$. Then, for any metric coefficient $\phi$ not equal to $\slashed{g}_{AB}$, the metric coefficient norm $\mathscr{G}$ is defined by
\[\left\vert\left\vert \phi\right\vert\right\vert_{\mathscr{G}_{\tilde u,\tilde v}}^2 \doteq \sup_{0 \leq j \leq 2}\sup_{(u_0,v_0) \in \mathcal{Q}_{\tilde u,\tilde v}}\int_{\mathbb{S}^2}\left|\exp\left(D\frac{v}{u}\right)\slashed{\nabla}^j\phi^*\right|^2\left(v_0-u_0\right)^{2j}\mathring{\rm dVol}.\]
We also introduce the notation 
\[\mathfrak{G}_{\tilde u,\tilde v} \doteq \sum_{\phi}\left\vert\left\vert \phi\right\vert\right\vert_{\mathscr{G}_{\tilde u,\tilde v}}.\]
Finally, when it will not cause confusion, we will often suppress a subset of the $\left(\tilde u,\tilde v\right)$ indices from the $\mathscr{G}$ or $\mathfrak{G}$ subscript. 
\end{definition}

Lastly, we define an ``initial data'' norm.
\begin{definition}We set
\begin{align*}
\mathfrak{H} \doteq &\sup_{\Psi}\sup_{0 \leq j \leq 2}\left[\sup_{s\in [0,1]}\int_{-s}^{-\underline{v}s}\int_{\mathbb{S}^2}\left|\slashed{\nabla}^j\Psi\right|^2|_{\{(u,v) = (s,-s\underline{v})\}}s^{3+2j}\, ds + \int_{\underline{v}}^1\int_{\mathbb{S}^2}\left|\slashed{\nabla}^j\Psi\right|^2|_{\{(u,v) = (-1,s)\}}\, ds\right]
\\ \nonumber &+\sup_{\psi}\sup_{0 \leq j \leq 2}\left[\sup_{s\in [-1,0]}\int_{\mathbb{S}^2_{s,-s\underline{v}}}\left|\slashed{\nabla}^j\psi^*\right|^2s^{2+2j-2\kappa}\, ds + \sup_{s\in [\delta_1,1]}\int_{\mathbb{S}^2_{-1,s}}\left|\slashed{\nabla}^j \psi^*\right|^2\, ds\right]
\\ \nonumber &+\sup_{0\leq j \leq 2}\sup_{s\in [-1,-0]}\int_{\mathbb{S}^2_{s,-s\underline{v}}}|s|^{2j}\left[\left|\slashed{\nabla}^j\left(\Omega^{-1}-1\right)\right|^2 + \left|\slashed{\nabla}^j\slashed{g}^*\right|^2 + \left|\slashed{\nabla}^jb\right|^2\right]
\\ \nonumber &+\sup_{0\leq j \leq 2}\sup_{s\in [\underline{v},1]}\int_{\mathbb{S}^2_{-1,s}}\left[\left|\slashed{\nabla}^j\left(\Omega^{-1}-1\right)\right|^2 + \left|\slashed{\nabla}^j\slashed{g}^*\right|^2+ \left|\slashed{\nabla}^jb\right|^2 \right].
\end{align*} 
\end{definition}
\subsection{Preliminaries Consequences of Theorem~\ref{itisreg1}}\label{frombefore}
\begin{proposition}\label{itstartedoutok}Let $\left(\mathcal{M},g_{\mu\nu}\right)$ satisfy the hypothesis of Proposition~\ref{letsbootit}.  Then we have that 
\[\mathfrak{H} \lesssim \epsilon^{1-\delta}.\]
\end{proposition}
\begin{proof}We start with the curvature components $\Psi$. First of all, given that we have closed the bootstrap argument which proves Proposition~\ref{letsbootit}, by a standard preservation of regularity argument (see the proof of Proposition 7.1 from~\cite{scaleinvariant}), and at the cost of an additional angular derivative of initial data, we have the following estimate 
\begin{equation}\label{strongest}
\sup_{0 \leq j \leq 2}\sup_{(u,v)}\int_{\mathbb{S}^2_{u,v}}\left|\slashed{\nabla}^j\Psi\right|^2u^{4+2j} \leq \epsilon^{2-2\delta},\qquad \sup_{0 \leq j \leq 3}\sup_{(u,v)}\int_{\mathbb{S}^2_{u,v}}\left|\slashed{\nabla}^j\left(\eta,\underline{\eta}\right)\right|^2u^{2+2j} \leq \epsilon^{2-2\delta},
\end{equation}
Integrating~\eqref{strongest} immediately yields 
\[\sup_{\Psi}\sup_{0 \leq j \leq 2}\left[\sup_{s\in [0,1]}\int_{-s}^{-\underline{v}s}\int_{\mathbb{S}^2}\left|\slashed{\nabla}^j\Psi\right|^2|_{\{(u,v) = (s,-s\underline{v})\}}s^{3+2j}\, ds + \int_{\underline{v}}^1\int_{\mathbb{S}^2}\left|\slashed{\nabla}^j\Psi\right|^2|_{\{(u,v) = (-1,s)\}}\, ds\right] \lesssim \epsilon^{2-2\delta}.\]

Next we come to the Ricci coefficients. We first note that for $\hat{\chi}_{AB}$, $\hat{\underline{\chi}}_{AB}$, $\eta_A$, $\underline{\eta}_A$, and $\underline{\omega}$,  the desired bounds manifestly follow from Proposition~\ref{letsbootit}. However, we will need to improve the estimates for ${\rm tr}\chi$ and ${\rm tr}\underline{\chi}$ and produce an estimate for $\omega$. Let's start with ${\rm tr}\chi$. We can write the $\nabla_4$ equation for ${\rm tr}\chi$ in the following form:
\begin{align}\label{newtrchi44}
&\frac{\partial}{\partial v}\left(\Omega^{-1}{\rm tr}\chi - \frac{2}{v-u}\right) + \frac{2\Omega^2}{v-u}\left(\Omega^{-1}{\rm tr}\chi - \frac{2}{v-u}\right) =
\\ \nonumber &\qquad  -\frac{\Omega^2}{2}\left(\Omega^{-1}{\rm tr}\chi - \frac{2}{v-u}\right)^2 - \Omega^2\left|\Omega^{-1}\hat{\chi}\right|^2 + \left(1-\Omega^2\right)\frac{2}{(v-u)^2}.
\end{align}
Note that it also follows from Lemma~\ref{uminus2data} that 
\[\left\vert\left\vert \Omega^{-1}{\rm tr}\chi - \frac{2}{-u}\right\vert\right\vert_{\tilde H^2\left(\mathbb{S}^2_{u,0}\right)} \lesssim \epsilon^{1-\delta}.\]
Thus, it is straightforward to commute to~\eqref{newtrchi44} with $\slashed{\nabla}_{A_1\cdots A_i}^i$ for $i \in \{0,1,2\}$ and use the already established bounds to obtain that 
\begin{equation}\label{okokdw}
\sup_{(u,v)}\left\vert\left\vert \Omega^{-1}{\rm tr}\chi - \frac{2}{v-u}\right\vert\right\vert_{\tilde H^2\left(\mathbb{S}^2_{u,v}\right)} \lesssim \epsilon^{1-\delta}.
\end{equation}
Similarly, from the $\nabla_4$ equation for ${\rm tr}\underline{\chi}$, one may derive
\begin{align}\label{newtrchi33}
&\frac{\partial}{\partial v}\left(\Omega{\rm tr}\underline{\chi} + \frac{2}{v-u}\right) + \frac{1}{2}\Omega^2\left(\Omega^{-1}{\rm tr}\chi\right)\left(\Omega{\rm tr}\underline{\chi} + \frac{2}{v-u}\right) = \\ \nonumber &\qquad \Omega^2\left[2\rho - \left(\Omega^{-1}\hat{\chi}\right)\cdot\left(\Omega\hat{\underline{\chi}}\right) + 2\slashed{\rm div}\underline{\eta} + 2\left|\underline{\eta}\right|^2\right] + \left(\frac{1}{v-u} - \frac{1}{2}\Omega^2\left(\Omega^{-1}{\rm tr}\chi\right)\right)\frac{2}{v-u}.
\end{align}
Using~\eqref{okokdw} and~\eqref{strongest} allows to straightforwardly commute to~\eqref{newtrchi44} with $\slashed{\nabla}_{A_1\cdots A_i}^i$ for $i \in \{0,1,2\}$ and use the already established bounds to obtain that 
\begin{equation}\label{okokdw2}
\sup_{(u,v)}\left\vert\left\vert \Omega{\rm tr}\underline{\chi} + \frac{2}{v-u}\right\vert\right\vert_{\tilde H^2\left(\mathbb{S}^2_{u,v}\right)} \lesssim \epsilon^{1-\delta}.
\end{equation}
From~\eqref{okokdw} and~\eqref{okokdw2} the desired bounds for ${\rm tr}\chi$ and ${\rm tr}\underline{\chi}$ easily follow.

Next we turn to estimating $\omega$. From the $\nabla_3$ equation for $\omega$ we may derive the following:
\begin{equation}\label{kodkodwkodwko}
\Omega\nabla_3\left(\Omega^{-1}\omega-\Omega^{-2}\frac{\kappa}{2v}\right) - 4\left(\Omega\underline{\omega}\right)\left(\Omega^{-1}\omega-\Omega^{-2}\frac{\kappa}{2v}\right) = \frac{1}{2}\rho + \frac{1}{2}\left|\underline{\eta}\right|^2 -\eta\cdot\underline{\eta}.
\end{equation}
Using that along $\{u = -1\}$ we have that $\Omega^{-1}\omega-\Omega^{-2}\frac{\kappa}{2v} = 0$, we may use~\eqref{kodkodwkodwko} in an analogous fashion to the proof of Lemma~\ref{fkowfomwfometa} to establish that
\begin{equation}\label{okokdw3}
\sup_{(u,v)}\left\vert\left\vert \Omega^{-1}\omega-\Omega^{-2}\frac{\kappa}{2v}\right\vert\right\vert_{\tilde H^2\left(\mathbb{S}^2_{u,v}\right)} \lesssim \epsilon^{1-\delta}.
\end{equation}
In turn, this easily implies the desired bound on $\omega$. 

Lastly, we just need to improve the bound on $\slashed{g}_{AB}$, since the desired bounds for the other metric coefficients already follow from Proposition~\ref{letsbootit}. For this we simply note that 
\[\mathcal{L}_{\partial_v}\slashed{g}_{AB}^* = 2\left(\Omega\chi\right)_{AB}^*,\]
and argue as we did for ${\rm tr}\chi$ and ${\rm tr}\underline{\chi}$. 
\end{proof}
\subsection{The Estimates}\label{estimatesregion2}
A standard argument using Proposition~\ref{itexistsbutforalittlewhile} shows that Theorem~\ref{itisreg2} will follow from the following proposition.
\begin{proposition}\label{gronwallikeest}Let $\underline{v} > 0$ and $(\mathcal{M},g_{\mu\nu})$ be a spacetime produced by Proposition~\ref{itexistsbutforalittlewhile} which exists in the region rectangle $\mathcal{R}_{\tilde u,\tilde v}$ for some $\tilde u \in (-1,0)$ and $\tilde v \in (0,1]$ satisfying 
\[ 0 < \frac{\tilde v}{-\tilde u} \leq \underline{v}^{-1},\]
and which satisfies the ``bootstrap assumption''

\begin{equation}\label{bootstrapgronwallikeest}
\mathfrak{E}_{\tilde u,\tilde v} + \mathfrak{F}_{\tilde u,\tilde v}+\mathfrak{G}_{\tilde u,\tilde v} \leq 2A\epsilon^{1-\delta}.
\end{equation}

We then claim that~\eqref{bootstrapgronwallikeest} implies
\begin{equation}\label{dkowdovoboto}
\mathfrak{E}_{\tilde u,\tilde v} + \mathfrak{F}_{\tilde u,\tilde v}+\mathfrak{G}_{\tilde u,\tilde v}\leq A\epsilon^{1-\delta}.
\end{equation}
\end{proposition}

As usual, the proof will be broken up into a few separate estimates. We start with estimates for the curvature components, then prove estimates for the Ricci coefficients, and finish with the estimates for the metric coefficients.

Throughout the proofs in this section we will use without comment that for any point in $\mathcal{Q}$, we have
\[\underline{v} |u| \leq v \leq \underline{v}^{-1}|u|,\qquad \underline{v}v \leq |u| \leq \underline{v}^{-1}v.\]

Unless said otherwise, we will also allow all of our constants to depend on $\underline{v}$ and $\underline{v}^{-1}$.  We start by observing the Sobolev spaces generated by $\slashed{g}$ and $\mathring{\slashed{g}}$ are comparable.
\begin{lemma}\label{sosimilararethepscaes2}Let $\left(\mathcal{M},g_{\mu\nu}\right)$ satisfy the hypothesis of Proposition~\ref{gronwallikeest}. Then we have that
\[\left\vert\left\vert w\right\vert\right\vert_{\tilde H^i\left(\mathbb{S}^2_{u,v}\right)} \sim_k (v-u)^{-k}\left\vert\left\vert w\right\vert\right\vert_{\mathring{H}^i\left(\mathbb{S}^2_{u,v}\right)}\text{ for }i \in \{0,1,2\},\]
\[\left\vert\left\vert w\right\vert\right\vert_{\tilde L^p\left(\mathbb{S}^2_{u,v}\right)} \sim_k (v-u)^{-k}\left\vert\left\vert w\right\vert\right\vert_{\mathring{L}^p\left(\mathbb{S}^2_{u,v}\right)},\]
where we recall that $\mathring{H}^i$ and $\mathring{L}^p$ denote the Sobolev and $L^p$ spaces generated by the round metric $\mathring{\slashed{g}}$, and the spaces $\tilde{H}^j$ are defined as in Definition~\ref{idkdtildedefdef}.
\end{lemma}
\begin{proof}This is an immediate consequence of Lemma~\ref{comparethespaces}, the bootstrap hypothesis, and the smallness of $\epsilon$.
\end{proof}

Now we observe that the standard Sobolev inequalities hold for the spaces $H^i$.
\begin{lemma}\label{sosososob2}Let $\left(\mathcal{M},g_{\mu\nu}\right)$ satisfy the hypothesis of Proposition~\ref{gronwallikeest}. Then, for any $(0,k)$-tensor $w_{A_1\cdots A_k}$, we have that
\[\left\vert\left\vert w\right\vert\right\vert_{\tilde{L}^p\left(\mathbb{S}^2_{u,v}\right)} \lesssim_{p,k} \left\vert\left\vert w\right\vert\right\vert_{\tilde H^1\left(\mathbb{S}^2_{u,v}\right)},\qquad \left\vert\left\vert w\right\vert\right\vert_{\tilde{L}^{\infty}\left(\mathbb{S}^2_{u,v}\right)}\lesssim_k \left\vert\left\vert w\right\vert\right\vert_{\tilde{H}^2\left(\mathbb{S}^2_{u,v}\right)},\]
and for any $(0,k)$-tensor $w_{A_1\cdots A_k}$ and $(0,k')$-tensor $v_{A_1\cdots A_k'}$ we have
\[\left\vert\left\vert w\cdot v\right\vert\right\vert_{\tilde H^2\left(\mathbb{S}^2_{u,v}\right)} \lesssim_{k,k'} \left\vert\left\vert w\right\vert\right\vert_{\tilde H^2\left(\mathbb{S}^2_{u,v}\right)} \left\vert\left\vert v\right\vert\right\vert_{\tilde H^2\left(\mathbb{S}^2_{u,v}\right)}.\]

\end{lemma}
\begin{proof}This is an immediate consequence of Lemma~\ref{sosimilararethepscaes2} and Lemma~\ref{soblemm}.
\end{proof}
These Sobolev inequalities will be used repeatedly in our estimates of nonlinear terms and we will often do so without explicit comment. 
\subsubsection{Energy Estimates for Curvature} 
We start with the energy estimates for curvature.
\begin{proposition}\label{thisisenegege}Let $\left(\mathcal{M},g\right)$ satisfy the hypothesis of  Proposition~\ref{gronwallikeest}. Then  we have
\[\mathfrak{E} \lesssim \epsilon^{1-\delta}.\]
\end{proposition}
\begin{proof}We may write each Bianchi pair $\left(\left(\alpha_{AB},\beta_A\right),\left(\beta_A,\left(\rho,\sigma\right)\right), \left(\left(\rho,\sigma\right),\underline{\beta}_A\right), \left(\underline{\beta}_A,\underline{\alpha}_{AB}\right)\right)$ in the schematic form
\begin{equation}\label{gronbianch}
\nabla_4\Psi^{(1)}  = \mathcal{D}\Psi^{(2)} + \psi\cdot\Psi,\qquad \nabla_3\Psi^{(2)}  = -\mathcal{D}^*\Psi^{(1)} + \psi\cdot \Psi,
\end{equation}
where $\mathcal{D}$ represents an angular derivative operator defined with respect to $\slashed{g}_{AB}$ and $\mathcal{D}^*$ denotes the $L^2\left(\slashed{g}\right)$-adjoint on $\mathbb{S}^2$. As we have written the equations, we note that there are ``linear terms'' hiding in the right hand sides  due to the presence of ${\rm tr}\chi$ and ${\rm tr}\underline{\chi}$. For each $i \in \{0,1,2\}$ we may then commute with $\slashed{\nabla}_{A_1\cdots A_i}^i$ to and use Lemma~\ref{commlemm} to obtain an equation of the schematic form (with indices on $\slashed{\nabla}^i$ suppressed)
\begin{align}\label{gronbianch2}
\nabla_4\slashed{\nabla}^i\Psi^{(1)}  &= \mathcal{D}\slashed{\nabla}^i\Psi^{(2)} + \underbrace{\slashed{\nabla}^i\left(\psi\cdot\Psi\right) +\slashed{\nabla}^{i-1}\left(\psi_1\cdot\psi_2\cdot\Psi\right) + \slashed{\nabla}^{i-1}\left(K\Psi\right)}_{\mathcal{E}_1},
\\ \nonumber \nabla_3\slashed{\nabla}^i\Psi^{(2)} &= -\mathcal{D}^*\slashed{\nabla}^i\Psi^{(1)} + \underbrace{\slashed{\nabla}^i\left(\psi\cdot\Psi\right) + \slashed{\nabla}^{i-1}\left(\psi_1\cdot\psi_2\cdot\Psi\right) +\slashed{\nabla}^{i-1}\left(K\Psi\right)}_{\mathcal{E}_2}.
\end{align}

Before carrying out our energy estimate, we conjugate the equations~\eqref{gronbianch2} by $w\left(u,v\right) \doteq \left(v-u\right)^{3/2}\exp\left(D\frac{v}{u}\right)$ where $D$ is a suitable large positive constant to be determined later, depending only on $\underline{v}$:

\begin{equation}\label{gronbianchconj4}
\nabla_4\left(w\slashed{\nabla}^i\Psi^{(1)}\right) - \left[\frac{(3/2)\Omega^{-1}}{v-u}+\Omega^{-1}\frac{D}{u}\right]\left(w\slashed{\nabla}^i\Psi^{(1)}\right) = \mathcal{D}\left(w\slashed{\nabla}^i\Psi^{(2)}\right) + w\mathcal{E}_1,
\end{equation}
\begin{equation}\label{gronbianchconj3}
\nabla_3\left(w\slashed{\nabla}^i\Psi^{(2)}\right) + \left[\frac{(3/2)\Omega^{-1}}{v-u}+\Omega^{-1}\frac{Dv}{u^2}\right]\left(w\slashed{\nabla}^i\Psi^{(2)}\right) = -\mathcal{D}^*\left(w\slashed{\nabla}^i\Psi^{(1)}\right) +w\mathcal{E}_2.
\end{equation}

The point of conjugating with the exponential weight is that for $D$ sufficiently large, the coefficient of the linear term on the left hand sides of~\eqref{gronbianchconj4} and~\eqref{gronbianchconj3} will be positive and thus generate a good spacetime term in the energy estimate. 

Let $(u,v) \in \mathcal{Q}_{\tilde u,\tilde v}$. Multiplying~\eqref{gronbianchconj4} by $w\slashed{\nabla}^i\Psi^{(1)}$ and~\eqref{gronbianchconj3} by $w\slashed{\nabla}^i\Psi^{(2)}$, carrying out the usual integration by parts, using the bootstrap assumption~\eqref{bootstrapgronwallikeest}, and appealing to Proposition~\ref{itstartedoutok} leads to 
\begin{align}\label{okdwplvijnie}
&\sup_{\Psi}\Bigg[\int_{{\rm max}\left(-1,-\underline{v}^{-1}v\right)}^u\int_{\mathbb{S}^2}w^2\left|\slashed{\nabla}^i\Psi^{(1)}\right|^2|_{(s,v)}\, ds\, \mathring{\rm dVol} + \int_{-\underline{v}u}^v\int_{\mathbb{S}^2}w^2\left|\slashed{\nabla}^i\Psi^{(2)}\right|^2|_{(u,s)}\, ds\, \mathring{\rm dVol}
\\ \nonumber &\qquad + D\int_{\mathcal{Q}_{u,v}}\int_{\mathbb{S}^2}w^2\left(\dot{v}-\dot{u}\right)^{-1}\left[\left|\slashed{\nabla}^i\Psi^{(1)}\right|^2 + \left|\slashed{\nabla}^i\Psi^{(2)}\right|^2\right]\, d\dot{u}\  d\dot{v}\, \mathring{\rm dVol}\Bigg]
\\ \nonumber &\qquad \lesssim D^{-1}\int_{\mathcal{Q}_{u,v}}\int_{\mathbb{S}^2}w^2(-\dot{u})\left[\left|\mathcal{E}_1\right|^2 + \left|\mathcal{E}_2\right|^2\right]\, d\dot{u}\, d\dot{v}\, \mathring{\rm dVol} + \epsilon^{2-2\delta}.
\end{align}
Here we have used that  in the region under consideration, $v$ and $|u|$ are comparable and thus we have $w^2 \sim \left(v-u\right)^3$. Next, it is immediate from the bootstrap assumptions, Sobolev inequalities, and the largeness of $D$ that we have (suppressing the volume forms)
\begin{align}\label{kodwdw222sssxcsa}
&D^{-1}\int_{\mathcal{Q}_{u,v}}\int_{\mathbb{S}^2}w^2(-\dot{u})\left[\left|\mathcal{E}_1\right|^2 + \left|\mathcal{E}_2\right|^2\right] 
\\ \nonumber &\qquad \lesssim D^{-1}\sup_{\Psi}\int_{\mathcal{Q}_{u,v}}\int_{\mathbb{S}^2}w^2\left(\dot{v}-\dot{u}\right)^{-1}\left[\left|\slashed{\nabla}^i\Psi\right|^2+\sum_{i_1+i_2 = i-1}w^2(-\dot{u})\left|\left(\slashed{\nabla}^{i_2}K\right)\left(\slashed{\nabla}^{i_2}\Psi\right)\right|^2\right].
\end{align}
The first term on the right hand side of~\eqref{kodwdw222sssxcsa} may be clearly be absorbed into the left hand side of~\eqref{okdwplvijnie}. Next, we note that it is a straightforward consequence of integrating the $\nabla_4$ equation for $\rho$ and the bootstrap assumption that we have
\[\sup_{(\tilde u,\tilde v) \in \mathcal{Q}_{u,v}}\left\vert\left\vert \exp\left(D\frac{v}{u}\right)\rho\right\vert\right\vert_{\tilde H^1\left(\mathbb{S}^2_{\tilde u,\tilde v}\right)}(v-u)^2 \lesssim \epsilon^{1-\delta} + \mathfrak{E}.\]
It then follows easily from the Gauss curvature equation and the bootstrap assumption that
\begin{equation}\label{agaussmboundbound}
\sup_{(\tilde u,\tilde v) \in \mathcal{Q}_{u,v}}\left\vert\left\vert  \exp\left(D\frac{v}{u}\right)\left(K-1\right)\right\vert\right\vert_{\tilde H^1\left(\mathbb{S}^2_{\tilde u,\tilde v}\right)}(v-u)^2 \lesssim \epsilon^{1-\delta} + \mathfrak{E}.
\end{equation}
Finally, a straightforward induction argument, Proposition~\ref{thisisenegege},~\eqref{okdwplvijnie},~\eqref{kodwdw222sssxcsa}, and~\eqref{agaussmboundbound} leads to 
\begin{align}\label{okdwplvijnie}
&\sup_{0\leq i \leq 2}\sup_{\Psi}\Bigg[\int_{{\rm max}\left(-1,-\underline{v}^{-1}v\right)}^u\int_{\mathbb{S}^2}w^2\left|\slashed{\nabla}^i\Psi^{(1)}\right|^2|_{(s,v)}\, ds\, \mathring{\rm dVol} + \int_{-\underline{v}u}^v\int_{\mathbb{S}^2}w^2\left|\slashed{\nabla}^i\Psi^{(2)}\right|^2|_{(u,s)}\, ds\, \mathring{\rm dVol}
\\ \nonumber &\qquad + D\int_{\mathcal{Q}_{u,v}}\int_{\mathbb{S}^2}w^2\left(\dot{v}-\dot{u}\right)^{-1}\left[\left|\slashed{\nabla}^i\Psi^{(1)}\right|^2 + \left|\slashed{\nabla}^i\Psi^{(2)}\right|^2\right]\, d\dot{u}\  d\dot{v}\, \mathring{\rm dVol}\Bigg]
\\ \nonumber &\qquad \lesssim \mathfrak{E}^2+ \epsilon^{2-2\delta}.
\end{align}
Invoking the bootstrap assumption again finishes the proof.

\end{proof}

\subsubsection{Estimates for the Ricci Coefficients}

Next we turn to the estimates for the Ricci coefficients.
\begin{proposition}\label{gronricciest}Let $\left(\mathcal{M},g\right)$ satisfy the hypothesis of  Proposition~\ref{gronwallikeest}. Then  we have
\[\mathfrak{R} \lesssim \epsilon^{1-\delta}.\]
\end{proposition}
\begin{proof}Every Ricci coefficient $\psi \in \{\hat{\underline{\chi}}_{AB},\hat{\chi}_{AB},\omega,\underline{\omega},\eta_A,\underline{\eta}_A\}$ satisfies $\psi = \psi^*$ (recall that $\psi^*$ is defined in Definition~\ref{3i8193defpsistartstar}) and also satisfies an equation of one of the following schematic forms:
\begin{equation}
\nabla_4\psi^* = \left(\psi_1,\psi_1^*\right)\cdot\psi_2^* + \Psi,
\end{equation}
\begin{equation}
\nabla_3\psi^* = \left(\psi_1,\psi_1^*\right)\cdot\psi_2^* + \Psi,
\end{equation}
where $\psi_i$ stands for a Ricci coefficient and $\Psi$ for a null curvature component. For ${\rm tr}\chi$ and ${\rm tr}\underline{\chi}$ we use the corresponding Raychaudhuri equations which may be written in the following form: 
\begin{align}\label{omdqmoqdq}
\nabla_4{\rm tr}\chi^* &= \left(1-\Omega^{-1}\right)\frac{2}{(v-u)^2} - \frac{2}{v-u}{\rm tr}\chi^* - \frac{4\omega}{v-u}  - 2\omega{\rm tr}\chi^* - \frac{1}{2}\left({\rm tr}\chi^*\right)^2 -\left|\hat{\chi}\right|^2,
\end{align}
\begin{align}\label{omdqmoqdq2}
\nabla_3{\rm tr}\underline{\chi}^* &= \left(\Omega^{-1}-1\right)\frac{2}{(v-u)^2} + \frac{2}{v-u}{\rm tr}\underline{\chi}^* +\frac{4\underline{\omega}}{v-u}  - 2\underline{\omega}{\rm tr}\chi^* - \frac{1}{2}\left({\rm tr}\underline{\chi}^*\right)^2 -\left|\hat{\underline{\chi}}\right|^2.
\end{align}
Thus, all together every Ricci coefficient $\psi$ satisfies an equation of one of the following forms
\[\nabla_4\psi^* = \mathcal{E},\qquad \nabla_3\psi^* = \mathcal{E},\]
where $\mathcal{E}$ is controlled by a sum of terms  of the following possible schematic forms:
\[\psi_1\cdot\psi_2^*,\qquad \psi_1^*\cdot\psi_2^*,\qquad \Psi,\qquad \left(\Omega^{-1}\right)^*(v-u)^{-2},\qquad \psi_1^*\left(v-u\right)^{-1}, \]
with the constraint that $\alpha_{AB}$ can only show up in a $\nabla_4$ equations and that $\underline{\alpha}_{AB}$ can only show up in a $\nabla_3$ equation.

Commuting with $\slashed{\nabla}_{A_1\cdots A_i}^i$ leads to equations of the schematic form (with indices on $\slashed{\nabla}^i$ suppressed):
\[\nabla_4\slashed{\nabla}^i\psi^* = \slashed{\nabla}^i\mathcal{E} + \mathcal{F}_i,\qquad \nabla_3\slashed{\nabla}^i\psi^* = \slashed{\nabla}^i\mathcal{E} + \mathcal{F}_i,\]
where $\mathcal{F}_i$ is controlled by a sum of terms of the schematic form:
\[\slashed{\nabla}^i\left(\psi_1\cdot\psi^*\right),\qquad \slashed{\nabla}^{i-1}\left(\psi_1\cdot\psi_2\cdot\psi^*\right).\]
Now we conjugate by $w \doteq \left(v-u\right)\exp\left(D\frac{v}{u}\right)$ to obtain 
\[\nabla_4\left(w\slashed{\nabla}^i\psi^*\right) + \frac{D}{u}w\slashed{\nabla}^i\psi^* = q\slashed{\nabla}^i\mathcal{E} + q\mathcal{F}_i,\qquad \nabla_3\left(w\slashed{\nabla}^i\psi^*\right) + D\frac{v}{u^2}w\slashed{\nabla}^i\psi^* = w\slashed{\nabla}^i\mathcal{E} + w\mathcal{F}_i.\]
Now we contract with $w\slashed{\nabla}^i\psi^*$ and integrate to obtain for every $(u,v) \in \mathcal{Q}_{\tilde u,\tilde v}$, either
\begin{align}\label{theifristsffwfw}
&\int_{\mathbb{S}^2_{u,v}}w^2\left|\slashed{\nabla}^i\psi^*\right|^2\mathring{\rm dVol} + \int_{-u\underline{v}}^v\int_{\mathbb{S}^2_{u,\dot{v}}}\frac{D}{u}w^2\left|\slashed{\nabla}^i\psi^*\right|^2\, d\dot{v}\, \mathring{\rm dVol} \lesssim
\\ \nonumber &\qquad  D^{-1}\int_{-u\underline{v}}^v\int_{\mathbb{S}^2_{u,\dot{v}}}(\dot{v}-u)w^2\left[\left|\slashed{\nabla}^i\mathcal{E}\right|^2 + \left|\mathcal{F}_i\right|^2\right]\, d\dot{v}\, \mathring{\rm dVol},
\end{align}
or
\begin{align}\label{theifristsffwfw2}
&\int_{\mathbb{S}^2_{u,v}}w^2\left|\slashed{\nabla}^i\psi^*\right|^2\mathring{\rm dVol} + \int_{{\rm max}\left(-v\underline{v}^{-1},1\right)}^u\int_{\mathbb{S}^2_{u,\dot{v}}}\frac{Dv}{\dot{u}^2}w^2\left|\slashed{\nabla}^i\psi^*\right|^2\, d\dot{u}\, \mathring{\rm dVol} \lesssim
\\ \nonumber &\qquad  D^{-1}\int_{{\rm max}\left(-v\underline{v}^{-1},1\right)}^u\int_{\mathbb{S}^2_{u,\dot{v}}}(v-\dot{u})w^2\left[\left|\slashed{\nabla}^i\mathcal{E}\right|^2 + \left|\mathcal{F}_i\right|^2\right]\, d\dot{u}\, \mathring{\rm dVol}.
\end{align}
The proof then concludes from the bootstrap assumptions, the largeness of $D$, and absorbing the terms from the right hand side into the spacetimes terms on the left hand sides of~\eqref{theifristsffwfw} and~\eqref{theifristsffwfw2}.

\end{proof}

\subsubsection{Estimates for the Metric Coefficients}

Lastly, we come to the estimates for the metric coefficients. 
\begin{proposition}\label{whatwewillproveformetricgron}Let $\left(\mathcal{M},g_{\mu\nu}\right)$ satisfy the hypothesis of  Proposition~\ref{gronwallikeest}. Then  we have
\begin{equation*}
\mathfrak{G} \lesssim \epsilon^{1-\delta}.
\end{equation*}
\end{proposition}
\begin{proof}

This follows by  simply by integrating the following transport equations for the metric coefficients,
\begin{equation}\label{metriceqnsforintegrating}
\partial_v\left(\Omega^{-1}\right) = 2\omega, \qquad \mathcal{L}_v\slashed{g}_{AB} = 2\Omega\chi_{AB},\qquad \mathcal{L}_vb^A = -4\Omega^2\zeta^A,
\end{equation}
and controlling the Ricci coefficients on the right hand side with Proposition~\ref{gronricciest}. (See the proof of Proposition~\ref{gronricciest}.)
\end{proof}

This concludes the proof of Proposition~\ref{gronwallikeest}, and hence also Theorem~\ref{itisreg2}.

\section{Shifting the Shift  and Gluing in an Asymptotically Flat Cone}\label{secshift}
It will be convenient to introduce the notation:
\[\mathcal{W} \doteq \left\{(u,v) : -1 \leq u < 0 ,\qquad 1 \leq \frac{v}{-u} < \underline{v}^{-1},\qquad 0 \leq v \leq \underline{v} \right\}.\]

We begin by noting the following consequence of a preservation of regularity argument:
\begin{proposition}\label{itstartedoutok2}Let $0 < \underline{v} \ll 1$, let $\epsilon >0$ be sufficiently small, and let $(\mathcal{M},g_{\mu\nu})$ be a spacetime produced by  Theorem~\ref{itisreg2} so that $\left(\mathcal{M},g_{\mu\nu}\right)$ exists in the region $\mathcal{W}$, and in this region the spacetime satisfies the regularity bounds~\eqref{02020202020922332}-\eqref{11212332123123123} and the estimates~\eqref{dkowdovoboto}. 

Then, for any $1 \ll \tilde N \ll N_0$, we have that,
\begin{equation*}
\sup_{0 \leq j \leq \tilde N}\sup_{(u,v) \in \mathcal{W}}\int_{\mathbb{S}^2_{u,v}}\left|\slashed{\nabla}^j\Psi\right|^2v^{4+2j} \lesssim \epsilon^{2-2\delta},\qquad \sup_{0 \leq j \leq \tilde N}\sup_{(u,v) \in \mathcal{W}}\int_{\mathbb{S}^2_{u,v}}\left|\slashed{\nabla}^j\psi^*\right|^2v^{2+2j} \lesssim \epsilon^{2-2\delta},
\end{equation*}
\[\sup_{0 \leq j \leq \tilde N}\sup_{(u,v) \in \mathcal{W}}\int_{\mathbb{S}^2_{u,v}}\left|\slashed{\nabla}^j\phi^*\right|^2v^{2j} \lesssim \epsilon^{2-2\delta}.\]
\end{proposition}
\begin{proof}The proof follows \emph{mutatis mutandis} as in Proposition~\ref{itstartedoutok}.
\end{proof}

\subsection{Shifting the Shift}

In this section we construct a new coordinate system so that the shift vector is in the $e_4$-direction. (See Remark~\ref{letusshifthteshift}.)

\begin{lemma}\label{shifttheshift}Let $0 < \underline{v} \ll 1$, let $\epsilon >0$ be sufficiently small, and let $(\mathcal{M},g_{\mu\nu})$ be a spacetime produced by  Theorem~\ref{itisreg2} so that $\left(\mathcal{M},g_{\mu\nu}\right)$ exists in the region $\mathcal{W}$, and in this region the spacetime satisfies the regularity bounds~\eqref{02020202020922332}-\eqref{11212332123123123} and the estimates~\eqref{dkowdovoboto}. 

Consider the sphere $\mathbb{S}^2_{-\frac{1}{2}\underline{v},\frac{1}{2}\underline{v}}$ at the intersection of the null hypersurfaces $\{u = -\frac{1}{2}\underline{v}\}$ and $\{v = \frac{1}{2}\underline{v}\}$, and then consider an arbitrary cover of $\mathbb{S}_{-\frac{1}{2}\underline{v},\frac{1}{2}\underline{v}}$ by a set of coordinate charts $U_1,\cdots, U_k$ with corresponding coordinate functions $\{\theta^A_{(i)}\}$ for $i = 1,\cdots, k$. The functions $\{\theta^A_{(i)}\}$, originally defined on $U_i$ may then be extended to $\mathcal{W}\times U_i$ by requiring that $\partial_u\theta^A_{(i)} = \partial_v\theta^A_{(i)} = 0$. (This is possible because $[\partial_u,\partial_v] = 0$.) These coordinates $\left(u,v,\theta^A_{(i)}\right)$ are, of course,  the coordinates which may be used in the double-null expression~\eqref{metricform}.

Given any choice of coordinates $\theta^A_{(i)}$ on $\mathbb{S}^2_{-\frac{1}{2}\underline{v},\frac{1}{2}\underline{v}}$, we will now define a new set of functions $\{\mathring{\theta}^A_{(i)}\}$ on $\mathcal{W}\times\mathbb{S}^2$ by requiring that $\mathring{\theta}^A_{(i)}\left(u,v,\theta^B\right):\mathcal{W} \times U_i \to \mathbb{R}$ satisfies 
\begin{equation}\label{thenewcodw}
\mathring{\theta}^A\left(-v,v,\theta^B\right) \doteq \theta^A\left(-\frac{1}{2}\underline{v},\frac{1}{2}\underline{v},\theta^B\right)\qquad \forall v \in [0,\underline{v}],\qquad \qquad e_3\left(\mathring{\theta}^A\right) = 0.
\end{equation}
Then we claim that $\left(u,v,\mathring{\theta}^A_{(i)}\right)$ form regular coordinates on $\mathcal{W}\times\mathbb{S}^2$. Furthermore, the metric $g_{\mu\nu}$ now takes the following form:
\begin{equation}\label{metricform23}
g = -2\Omega^2\left(du\otimes dv + dv\otimes du\right) + \slashed{g}_{AB}\left(d\mathring{\theta}^A - \mathring{b}^Adv\right)\otimes\left(d\mathring{\theta}^B - \mathring{b}^Bdv\right),
\end{equation}
for a shift vector $\mathring{b}^B$ which is uniquely determined by 
\begin{equation}\label{dqqqq}
\mathcal{L}_{\partial_u}\mathring{b}^A = 4\Omega^2\zeta^A,\qquad \mathring{b}^A|_{\frac{v}{-u} = 1} = -b^A.
\end{equation}

Finally, in this new double-null gauge, for any $1 \ll \tilde N \ll N_0$, we have that,
\begin{equation}\label{dkpwqqqedkcow}
\sup_{0 \leq j \leq \tilde N}\sup_{(u,v) \in \mathcal{W}}\int_{\mathbb{S}^2_{u,v}}\left|\slashed{\nabla}^j\Psi\right|^2v^{4+2j} \lesssim \epsilon^{2-2\delta},\qquad \sup_{0 \leq j \leq \tilde N}\sup_{(u,v) \in \mathcal{W}}\int_{\mathbb{S}^2_{u,v}}\left|\slashed{\nabla}^j\psi^*\right|^2v^{2+2j} \lesssim \epsilon^{2-2\delta},
\end{equation}
\begin{equation}\label{odkowmobwea}
\sup_{0 \leq j \leq \tilde N}\sup_{(u,v) \in \mathcal{W}}\int_{\mathbb{S}^2_{u,v}}\left|\slashed{\nabla}^j\phi^*\right|^2v^{2j} \lesssim \epsilon^{2-2\delta}.
\end{equation}
\end{lemma}
\begin{proof}We can re-write the transport equation~\eqref{thenewcodw} defining the new functions $\mathring{\theta}^A_{(i)}$ as the following equation for $\mathring{\vartheta}^A \doteq \mathring{\theta}^A - \theta^A$:
\begin{equation}\label{newtrans}
\frac{\partial}{\partial u}\mathring{\vartheta}^A + \left(b\cdot\slashed{\nabla}\right)\mathring{\vartheta}^A = -b^A.
\end{equation}
Using Proposition~\ref{itstartedoutok2}, it follows easily from~\eqref{newtrans} that for any $1 \ll \check{N} \ll N_0$ we have that 
\begin{equation}\label{estimadwdwfrovartheat}
\sup_{0 \leq j \leq \check{N}}\sup_{(u,v) \in \mathcal{W}}\int_{\mathbb{S}^2_{u,v}}\left|\slashed{\nabla}^{2j}\mathring{\vartheta}^A\right|^2v^{2j}\mathring{\rm dVol} \lesssim \epsilon^{2-2\delta},
\end{equation}
In particular, it is immediate that for each $(u,v) \in \mathcal{W}$, the functions $\{\mathring{\theta}^A_{(i)}\}$ form a regular set of coordinate functions on $\mathbb{S}^2_{u,v}$ lying in the Sobolev space $H^{\check{N}}$ for any $1 \ll \check{N} \ll N_0$.

Next we argue that the metric takes the desired form~\eqref{metricform23}. First of all, the change of variables formula implies each $\frac{\partial}{\partial \mathring{\theta}^A}$ is tangent to $\mathbb{S}^2_{u,v}$. Thus, we have
\[g\left(e_3,\partial_{\mathring{\theta}^A}\right) = g\left(e_4,\partial_{\mathring{\theta}^A}\right) = 0.\]
Furthermore, the change of variables formula implies that there exists $\tilde b^A$ and $\mathring{b}^A$ so that
\[\Omega e_3 = \frac{\partial}{\partial u} + \tilde b^A\frac{\partial}{\partial \mathring{\theta}^A},\qquad \Omega e_4 = \frac{\partial}{\partial v} + \mathring{b}^A\frac{\partial}{\partial \mathring{\theta}^A}.\]
However,~\eqref{thenewcodw} immediately implies that $\tilde b = 0$, and we furthermore have
\[g\left(\partial_v,\partial_{\mathring{\theta}^A}\right) = g\left(\Omega e_4,\partial_{\mathring{\theta}^A}\right) - g\left(\mathring{b}^B\partial_{\mathring{\theta}^B},\partial_{\mathring{\theta}^A}\right) = \mathring{b}_A,\]
\[g\left(\partial_v,\partial_v\right) = \Omega^2g\left(e_4,e_4\right) - g\left(\mathring{b}^A\partial_A,\partial_v\right) = -\left|\mathring{b}\right|^2.\]
It now follows that the metric takes the form~\eqref{metricform23} for some $\mathring{b}^A$. To see that~\eqref{dqqqq} holds, we first note that
\[\left[e_3,e_4\right] = \frac{\partial \mathring{b}^A}{\partial u}\frac{\partial}{\partial \mathring{\theta}^A}.\]
Then the desired propagation equation~\eqref{dqqqq} follows from the definition of torsion~\eqref{okokokthisistrosion}:
\[\zeta_A = \frac{1}{2}g\left(D_Ae_4,e_3\right) \Rightarrow \zeta_A = \frac{1}{4}g\left([e_3,e_4],e_A\right).\]

Finally, the estimates~\eqref{dkpwqqqedkcow} and~\eqref{odkowmobwea} follow easily from~\eqref{estimadwdwfrovartheat}, the new propagation equation for $\mathring{b}$, and Proposition~\ref{itstartedoutok2}. 

\end{proof}

\subsection{Gluing on an Asymptotically Flat Cone}
Next we give a definition which is similar to Definition~\ref{indatasets}.

\begin{definition}\label{indatasets2}Let $\left(\mathcal{M},g_{\mu\nu}\right)$ be a spacetime produced by Lemma~\ref{shifttheshift}. We then say that a $1$-parameter family $\left(\Omega^{({\rm out})}\left(v,\theta^A\right),\left(b^A\right)^{(\rm out)}\left(v,\theta^B\right),\slashed{g}_{AB}^{(\rm out)}\left(v,\theta^C\right)\right)_{v \geq \frac{1}{2}\underline{v}}$ consisting of a non-zero $C^1$ function $\Omega^{(\rm out)}$, a continuous vector field $\left(b^A\right)^{(\rm out)}$, and a $C^1$ Riemannian metric  $\slashed{g}_{AB}^{(\rm out)}$ on $\mathbb{S}^2$ form ``compatible outgoing gluing data'' if the following hold
\begin{enumerate}
	\item $v \in [\frac{1}{2}\underline{v},\underline{v})$ implies that
	\[\Omega^{({\rm out})}\left(v,\theta\right) = \Omega\left(-\underline{v}^2,v,\theta\right),\qquad \left(b^A\right)^{(\rm out)}\left(v,\theta\right) = b^A\left(-\underline{v}^2,v,\theta\right),\qquad \slashed{g}_{AB}^{(\rm out)}\left(v,\theta^C\right) = \slashed{g}_{AB}\left(-\underline{v}^2,v,\theta\right),\]
	where $\Omega$, $b^A$, and $\slashed{g}_{AB}$ are the metric components of the spacetime $\left(\mathcal{M},g_{\mu\nu}\right)$.

	\item After defining ${\rm tr}\chi$, $\hat{\chi}_{AB}$, and $\omega$ for $v \in [\frac{1}{2}\underline{v},\infty)$ by
	\[\left(\Omega^{(\rm out)}\right)^{-1}\mathcal{L}_{\partial_v+b}\slashed{g}^{(\rm out)}_{AB} \doteq {\rm tr}\chi\slashed{g}^{(\rm out)}_{AB} + 2\hat{\chi}_{AB},\qquad \omega \doteq \left(\Omega^{(\rm out)}\right)^{-1}\left(\partial_v+b^{(\rm out)}\cdot\slashed{\nabla}\right)\log\Omega^{(\rm out)},\]
	for a trace-free $\hat{\chi}_{AB}$,
	we have that the following equation is satisfied:
	\begin{equation}\label{ray2cons2}
	\left(\Omega^{(\rm out)}\right)^{-1}\left(\partial_v+b^{(\rm out)}\cdot\slashed{\nabla}\right){\rm tr}\chi + \frac{1}{2}\left({\rm tr}\chi\right)^2 = -2\omega{\rm tr}\chi - \left|\hat{\chi}\right|^2.
	\end{equation}
	
\end{enumerate}
\end{definition}

Now we have
\begin{proposition}\label{mixitloval}Let  $\left(\Omega^{({\rm out})}\left(v,\theta^A\right),\left(b^A\right)^{(\rm out)}\left(v,\theta^B\right),\slashed{g}_{AB}^{(\rm out)}\left(v,\theta^C\right)\right)_{v \geq \frac{1}{2}\underline{v}}$ form ``compatible outgoing gluing data'' such that for suitable $\tilde N \gg 1$ we have
\[\sup_{v > \frac{1}{2}\underline{v},\ i+j \leq \tilde N} \left[\left\vert\left\vert \mathcal{L}_{\partial_v}^i\mathring{\nabla}^j\left(\Omega^{(\rm out)},\left(b^A\right)^{(\rm out)},\slashed{g}_{AB}^{(\rm out)}\right)\right\vert\right\vert_{L^2\left(\mathbb{S}^2_v\right)}\right] < \infty,\]
where $\mathring{\nabla}_A$ is the covariant derivative relative to a reference, $v$-independent, round metric.

Let 
\[\mathcal{H} \doteq \left\{\left(\left\{\frac{v}{-u} = \frac{1}{2}\underline{v}^{-1}\right\} \cap \left\{v \in (0,\frac{1}{2}\underline{v}]\right\}\right) \cup \left(\left\{v \geq \frac{1}{2}\underline{v}\right\} \cap \left\{u = -\underline{v}^2\right\}\right)\right\},\]
and, for a given curve $\tau(v) : (\frac{1}{2}\underline{v},\infty) \to (-\underline{v}^2,0)$,
\[\mathcal{H}_{\tau(v)} \doteq \left\{\left(\left\{\frac{1}{2}\underline{v}^{-1} \leq \frac{v}{-u} \leq \underline{v}^{-1}\right\} \cap \left\{v \in (0,\frac{1}{2}\underline{v}]\right\}\right) \cup \left(\left\{v \geq \frac{1}{2}\underline{v}\right\} \cap \left\{u \in (-\underline{v}^2,\tau(v))\right\}\right)\right\},\]
Then there exists $\tau(v)$  so that there exists a spacetime $\left(\tilde{\mathcal{M}},g_{\mu\nu}\right)$ defined in a region $(u,v,\theta^A) \in \mathcal{H}_{\tau(v)}\times \mathbb{S}^2$ in the double-null foliation form~\eqref{metricform23} such that
\begin{enumerate}
	\item The regularity bounds~\eqref{bound1121}-\eqref{112123} hold.
	\item $(\tilde{\mathcal{M}},g_{\mu\nu})$ agrees with the solution $(\mathcal{M},g_{\mu\nu})$ in the region $\{\mathcal{H}_{\tau(v)} \cap \{v \leq \frac{1}{2}\underline{v}\}\}$.
	\item $\left(\Omega,b^A,\slashed{g}_{AB}\right)|_{\mathcal{H} \cap \{v \geq \frac{1}{2}\underline{v}\}} = \left(\Omega^{(\rm out)},\left(b^A\right)^{(\rm out)},\slashed{g}_{AB}^{(\rm out)}\right)$.
\end{enumerate}
\end{proposition}
\begin{proof}This may be easily deduced via Theorem~\ref{localexistencechar} and a domain of dependence argument.
\end{proof}

In the next proposition we construct ``compatible outgoing gluing data''  which we will use to construct an asymptotically flat null cone.
\begin{proposition}\label{kdowvappend}Let $\left(\mathcal{M},g_{\mu\nu}\right)$ be a spacetime produced by Lemma~\ref{shifttheshift} and $\mathring{\slashed{g}}_{AB}$ denote the reference metric which is used to define the norms in Proposition~\ref{gronwallikeest} . Then let $\left\{\hat{\slashed{g}}_{AB}^{(\rm out)}(v),\Omega^{(\rm out)}(v),(b^A)^{(\rm out)}\right\}_{v > \frac{1}{2}\underline{v}}$ be any $1$-parameter family of metrics, functions, and vector fields on $\mathbb{S}^2$ which, for some $\tilde N \gg 1$, satisfy the following constraints:
\begin{enumerate}
	\item $\left(\hat{\slashed{g}}^{(\rm out)}_{AB}(v),\Omega^{(\rm out)}(v),(b^A)^{(\rm out)}(v)\right) = \left(\slashed{g}_{AB},\Omega,b^A\right)|_{\left(-\underline{v}^2,v\right)}$ for $v \in [\frac{1}{2}\underline{v},\underline{v}]$, where $\left(\slashed{g}_{AB},\Omega,b^A\right)$ are the values of the metric components for the spacetime $\left(\mathcal{M},g_{\mu\nu}\right)$. 
	\item $\left(\hat{\slashed{g}}^{(\rm out)}_{AB}(v),\Omega^{(\rm out)}(v),(b^A)^{(\rm out)}(v)\right) = \left(\mathring{\slashed{g}}_{AB},1,0\right)$ for $v \in [2,\infty)$.
	\item $\sup_v\sup_{1 \leq j \leq \tilde N}\int_{\mathbb{S}^2}\left|\mathring{\nabla}^j\left(\left(\hat{\slashed{g}}^{(\rm out)}_{AB}(v),\Omega^{(\rm out)}(v),(b^A)^{(\rm out)}(v)\right) - \left(\mathring{\slashed{g}}_{AB},1,0\right)\right)\right|^2\mathring{\rm dVol} \lesssim \epsilon^{2-2\delta}$.
\end{enumerate}
Then there exists $\varphi^{(\rm out)}\left(v,\theta\right) : (\frac{1}{2}\underline{v},\infty)\times \mathbb{S}^2\to \mathbb{R}$ such that \[\left(\Omega^{({\rm out})}\left(v,\theta^A\right),\left(b^A\right)^{(\rm out)}\left(v,\theta^B\right),\left(\varphi^{(\rm out)}\right)^2\hat{\slashed{g}}_{AB}^{(\rm out)}\left(v,\theta^C\right)\right)_{v \geq \frac{1}{2}\underline{v}},\]
 form ``compatible outgoing gluing data.'' Furthermore, we have the following estimates for $\varphi$, ${\rm tr}\chi$, and $\hat{\chi}$: 
\[\sup_{1 \leq j \leq \tilde N} \int_{\mathbb{S}^2}\left|\mathring{\nabla}^j\left(\log\varphi - 2\log(v+\underline{v}),v{\rm tr}\chi - \frac{2v}{v+\underline{v}},v\hat{\chi},v^2\alpha\right)\right|^2v^{2j}\mathring{\rm dVol} \lesssim \epsilon^{2-2\delta}.\]

\end{proposition}
\begin{proof}This follows by using~\eqref{ray2cons2} and arguing as in the proof of Proposition~\ref{itexistsbutforalittlewhile}. We omit the details. 
\end{proof}

In the next proposition we analyze the behavior of all Ricci coefficients and curvature components for the initial data produced by Proposition~\ref{kdowvappend}.
\begin{proposition}\label{thisgivesthetildedmda}Let $\left(\mathcal{M},g_{\mu\nu}\right)$ be a spacetime produced by Lemma~\ref{shifttheshift} and $\mathring{\slashed{g}}_{AB}$ denote the reference metric which is used to define the norms in Proposition~\ref{gronwallikeest}, and 
\[\left(\Omega^{({\rm out})}\left(v,\theta^A\right),\left(b^A\right)^{(\rm out)}\left(v,\theta^B\right),\left(\varphi^{(\rm out)}\right)^2\hat{\slashed{g}}_{AB}^{(\rm out)}\left(v,\theta^C\right)\right)_{v > \frac{1}{2}\underline{v}},\]
be the corresponding ``compatible outgoing gluing data'' produced by Proposition~\ref{kdowvappend}. Let $\left(\tilde{\mathcal{M}},g_{\mu\nu}\right)$ denote the spacetime produced by Proposition~\ref{mixitloval}. Then for any $N$ satisfying $1 \ll N \ll N_0$, we have the following estimates for Ricci coefficients $\psi \neq \hat{\underline{\chi}}_{AB},\underline{\omega}$ and curvature components $\Psi \neq \left(\underline{\alpha}_{AB},\underline{\beta}_A\right)$ along $\mathcal{H}$:
\[\sup_{1 \leq j \leq N}\sup_{(u,v) \in \mathcal{H}}\int_{\mathbb{S}^2_{u,v}}\left|\slashed{\nabla}^j\psi^*\right|^2v^{2j+2}\mathring{\rm dVol} \lesssim \epsilon^{2-2\delta},\qquad \sup_{1 \leq j \leq N}\sup_{(u,v) \in \mathcal{H}}\int_{\mathbb{S}^2_{u,v}}\left|\slashed{\nabla}^j\Psi\right|^2v^{2j+4}\mathring{\rm dVol} \lesssim \epsilon^{2-2\delta}.\]
For any $s > 0$, we have that 
\[\sup_{1 \leq j \leq N}\sup_{(u,v) \in \mathcal{H}}\int_{\mathbb{S}^2_{u,v}}\left|\slashed{\nabla}^j\hat{\underline{\chi}}\right|^2v^{2j+2-2s}\mathring{\rm dVol} \lesssim_s \epsilon^{2-2\delta},\qquad \sup_{1 \leq j \leq N}\sup_{(u,v) \in \mathcal{H}}\int_{\mathbb{S}^2_{u,v}}\left|\slashed{\nabla}^j\underline{\beta}\right|^2v^{2j+4-2s}\mathring{\rm dVol} \lesssim_s \epsilon^{2-2\delta}.\]
\end{proposition}
\begin{proof}Due to Proposition~\ref{itstartedoutok2}, we only need to study the case of $v \gg 1$. Let us start with the Ricci coefficients. The desired bounds for $\omega$  follow immediately from the fact that $\Omega|_{\mathcal{H}}$ is identically $1$ for large $v$  and that $b^A|_{\mathcal{H}}$ vanishes for large $v$. Proposition~\ref{kdowvappend} also already provides the desired bounds for ${\rm tr}\underline{\chi}$ and $\hat{\underline{\chi}}_{AB}$. Another consequence $\Omega|_{\mathcal{H}}$ being identically $1$ for large $v$ is that $\eta_A|_v = -\underline{\eta}_A|_v$ when $v$ is large. Thus, for $v \gg 1$ we have may derive the following equation for $\eta_A = -\underline{\eta}_A$: 
\begin{equation}\label{fromedpqetatetata}
\nabla_v\eta_A + \frac{3}{2}{\rm tr}\chi \eta_A = \slashed{\rm div}\hat{\chi}_A - \frac{1}{2}\slashed{\nabla}_A{\rm tr}\chi - \left(\eta\cdot\hat{\chi}\right)_A.
\end{equation}
The key point is that
\[\frac{3}{2}{\rm tr}\chi \gtrsim \frac{3}{v} > \frac{1}{v}\text{ for }v\gg 1.\]
In particular, from~\eqref{fromedpqetatetata}, we have  
\begin{equation}\label{fromedpqetatetata}
\nabla_v\left(v\eta\right)_A +\left( \frac{3}{2}{\rm tr}\chi -\frac{1}{v}\right)v\eta_A = v\left[\slashed{\rm div}\hat{\chi} - \frac{1}{2}\slashed{\nabla}{\rm tr}\chi - \eta\hat{\chi}\right]_A.
\end{equation}
Contracting with $v\eta_A$ and using the previously established estimates leads to
\[\sup_{(u,v) \in \mathcal{H}}\int_{\mathbb{S}^2}\left|v\eta\right|^2 \lesssim \epsilon^{2-2\delta}.\]
It is straightforward to commute with $\slashed{\nabla}_{A_1\cdots A_j}^j$ and then obtain 
\[\sup_{(u,v) \in \mathcal{H}}\int_{\mathbb{S}^2}\left|v\slashed{\nabla}^j\eta\right|^2v^{2j} \lesssim \epsilon^{2-2\delta}.\]
For ${\rm tr}\underline{\chi}$ one may derive the following equation for $v \gg 1$:
\[\nabla_v\left({\rm tr}\underline{\chi}+\frac{2}{v+\underline{v}}\right) + {\rm tr}\chi\left({\rm tr}\underline{\chi} + \frac{2}{v+\underline{v}}\right) = -2\left(K-\frac{1}{\left(v+\underline{v}\right)^2}\right) +\left({\rm tr}\chi-\frac{2}{v+\underline{v}}\right)\frac{2}{v+\underline{v}} + 2\slashed{\rm div}\eta + 2\left|\eta\right|^2.\]
This may be treated just as $\eta_A$ to produce the desired estimate for ${\rm tr}\underline{\chi}$.

For $\hat{\underline{\chi}}_{AB}$, we have the following equation for $v \gg 1$: 
\[\nabla_v\hat{\underline{\chi}}_{AB} + \frac{1}{2}{\rm tr}\chi\hat{\underline{\chi}}_{AB} = \left(\slashed{\nabla}\hat{\otimes}\underline{\eta} - \frac{1}{2}{\rm tr}\underline{\chi}\hat{\chi} + \underline{\eta}\hat{\otimes}\underline{\eta}\right)_{AB}.\]
The key point is that for any $s > 0$ we will have that $v \gg_s 1$ implies that 
\[\frac{1}{2}{\rm tr}\chi - \frac{1-s}{v} \gtrsim_s v^{-1}.\]
In particular, one can conjugate by $v^{1-s}$ and proceed as we did for $\eta$. 

Finally the desired estimates for $\rho$, $\sigma$, $\beta_A$, and $\underline{\beta}_A$ follow immediately form the equations~\eqref{genGauss}, \eqref{curleta}, \eqref{tcod1}, and \eqref{tcod2}.

\end{proof}
\section{The Bootstrap Argument for Region III}\label{bottIIIIIIII}
The main result of this section will be the following:
\begin{theorem}\label{itisreg3}Let $(\mathcal{M},g_{\mu\nu})$ be a spacetime produced by Proposition~\ref{kdowvappend}. Then, possibly taking $\epsilon$ smaller, we then claim that $g_{\mu\nu}$ in this new coordinate system may be extended to the region $\left\{(u,v) : -\underline{v}^2 \leq u < 0 \text{ and } v > 0 \right\}$, and in this region the spacetime satisfies the regularity bounds~\eqref{02020202020922332}-\eqref{11212332123123123} and the estimates~\eqref{dkowdovoboto3}.

\end{theorem}
We will prove this theorem with a bootstrap argument. In Section~\ref{normsregion3} we will define the relevant norms. Then in Section~\ref{estimatesregion3} we will carry out the bootstrap argument. 
\subsection{Norms}\label{normsregion3}
In this section we will present the norms around which we will base our estimates. We will be interested in regions contained in $\{u \in (0,-\underline{v}^2)\} \cap \{v \in [0,\infty)\}$. We also introduce a reference Lie-propagated round metric to define a round volume form $\mathring{\rm dVol}$. 

 Let $\underline{v} > 0$ be sufficiently small, and set
\[\mathcal{P} \doteq \left\{u \in (-\underline{v}^2,0)\right\} \cap  \left\{\underline{v}^{-1} \leq \frac{v}{|u|} < \infty\right\},\qquad \mathcal{P}_{\tilde u,\tilde v} \doteq \mathcal{P} \cap \{u \leq \tilde u\} \cap \{v \leq \tilde v\},\]
where $(\tilde u,\tilde v) \in \mathcal{P}$.

It will be convenient to avoid working with $\underline{\alpha}_{AB}$ and instead only estimate the renormalized curvature components $\alpha_{AB}$, $\beta_A$, $\check{\rho}$, $\check{\sigma}$, and $\underline{\beta}_A$. We will use the notation $\check{\Psi}$ to refer to one of these renormalized curvature components.

As opposed to how we defined the norms for region $I$, it will be natural to weight Ricci coefficients and curvature components with $\Omega^{-s}$, where $s$ denotes the signature. This is because we will want to eliminate $\underline{\omega}$ in certain equations. In contrast, in region $I$  we weighted with $\Omega^s$ because we wanted to eliminate $\omega$ from various equations. 

\begin{convention}Throughout this section, unless said otherwise, all norms of tensorial quantities are computed with respect to $\slashed{g}_{AB}$, and we will always use the round metric induced volume form $\mathring{dVol}$ on each $\mathbb{S}^2_{u,v}$. 
\end{convention}

We now define the energy norm for the renormalized curvature components:
\begin{definition}\label{energy3}Let $0 < q \ll p \ll 1$. For $\underline{\beta}_A$ we define the energy norm by
\begin{align*}
\left\vert\left\vert \underline{\beta}\right\vert\right\vert^2_{\mathscr{I}_{p,\tilde u,\tilde v}} &\doteq \sup_{0 \leq j \leq 2}\sup_{(u,v) \in \mathcal{P}_{\tilde u,\tilde v}}\Bigg[(-u)^{2q}\int_{{\rm max}\left(-v\underline{v},-\underline{v}^2\right)}^u\int_{\mathbb{S}^2}\Omega^2\left|\slashed{\nabla}^j\left(\Omega^{-1}\underline{\beta}\right)\right|^2v^{3+2j}\left(\frac{-\dot{u}}{v}\right)^{2p}(-\dot{u})^{-2q}\, d\dot{u}\, \mathring{dVol}
\\ \nonumber &\qquad  \qquad + (-u)^{2q}\int_{{\rm max}\left(-v\underline{v},-\underline{v}^2\right)}^u\int_{-u\underline{v}^{-1}}^v\int_{\mathbb{S}^2}\left|\slashed{\nabla}^j\left(\Omega^{-1}\underline{\beta}\right)\right|^2\dot{v}^{2+2j}\left(\frac{-\dot{u}}{\dot{v}}\right)^{-2p}(-\dot{u})^{-2q}\, d\dot{v}\, d\dot{u}\, \mathring{dVol}\Bigg].
\end{align*}
For any renormalized curvature component $\check{\Psi}$ of signature $s$ not equal to $\underline{\beta}_A$, we have
\begin{align*}
\left\vert\left\vert \check{\Psi} \right\vert\right\vert^2_{\mathscr{I}_{p,\tilde u,\tilde v}} &\doteq \sup_{0 \leq j \leq 2}\sup_{(u,v) \in \mathcal{P}_{\tilde u,\tilde v}}\Bigg[(-u)^{2q}\int_{{\rm max}\left(-v\underline{v},-\underline{v}^2\right)}^u\int_{\mathbb{S}^2}\Omega^2\left|\slashed{\nabla}^j\left(\Omega^{-s}\check{\Psi}\right)\right|^2v^{3+2j}\left(\frac{-\dot{u}}{v}\right)^{2p}(-\dot{u})^{-2q}\, d\dot{u}\, \mathring{dVol}
\\ \nonumber &\qquad \qquad \int_{-u\underline{v}^{-1}}^v\int_{\mathbb{S}^2}\left|\slashed{\nabla}^j\left(\Omega^{-s}\check{\Psi}\right)\right|^2\dot{v}^{3+2j}\left(\frac{-u}{\dot{v}}\right)^{2p}\, d\dot{v}\, \mathring{dVol}
\\ \nonumber &\qquad  \qquad + (-u)^{2q}\int_{{\rm max}\left(-v\underline{v},-\underline{v}^2\right)}^u\int_{-u\underline{v}^{-1}}^v\int_{\mathbb{S}^2}\left|\slashed{\nabla}^j\left(\Omega^{-s}\check{\Psi}\right)\right|^2\dot{v}^{3+2j}\left(\frac{-\dot{u}}{\dot{v}}\right)^{2p}(-\dot{u})^{-1}(-\dot{u})^{-2q}\, d\dot{v}\, d\dot{u}\, \mathring{dVol}\Bigg].
\end{align*}
We also introduce the notation 
\[\mathfrak{I}_{p,\tilde u,\tilde v} \doteq \sum_{\check{\Psi}}\left\vert\left\vert \check{\Psi}\right\vert\right\vert_{\mathscr{I}_{p,\tilde u,\tilde v}}.\]
Finally, when it will not cause confusion, we will often suppress a subset of the $\left(p,\tilde u,\tilde v\right)$ indices from the $\mathscr{I}$ or $\mathfrak{I}$ subscript. 
\end{definition}

Next, we define the low-regularity norm for the Ricci Coefficients.
\begin{definition}Let $0 < p \ll 1$. For any Ricci coefficient $\psi \neq \eta_A,\underline{\omega},\underline{\chi}_{AB}$ of signature $s$ we define
\begin{align*}
\left\vert\left\vert \psi\right\vert\right\vert^2_{\mathscr{K}_{p,\tilde u,\tilde v}} &\doteq \sup_{0 \leq j \leq 2}\sup_{(u,v)\in\mathcal{P}_{\tilde u,\tilde v}}\int_{\mathbb{S}^2}\left|\slashed{\nabla}^j\left(\Omega^{-s}\psi\right)^*\right|^2v^{2j+2}\, \mathring{\rm dVol},
\end{align*}
where $\left(\Omega^{-s}\psi\right)^*$ denotes the different of $\Omega^{-s}\psi$ and its Minkowski value.

For $\psi \in \{\eta_A,\underline{\chi}_{AB}\}$ we define
\begin{align*}
\left\vert\left\vert \psi \right\vert\right\vert^2_{\mathscr{K}_{p,\tilde u,\tilde v}} &\doteq \sup_{0 \leq j \leq 2}\sup_{(u,v)\in\mathcal{P}_{\tilde u,\tilde v}}\int_{\mathbb{S}^2}\left|\slashed{\nabla}^j\left(\Omega^{-s}\psi\right)\right|^2v^{2j+2}\left(\frac{-u}{v}\right)^{2p}\, \mathring{\rm dVol},
\end{align*}
where $s$ denotes the signature of $\psi$. 

We also introduce the notation 
\[\mathfrak{K}_{p,\tilde u,\tilde v} \doteq \sum_{\psi\neq \underline{\omega}}\left\vert\left\vert \psi\right\vert\right\vert_{\mathscr{K}_{p,\tilde u,\tilde v}}.\]
Finally, when it will not cause confusion, we will often suppress a subset of the $\left(p,\tilde u,\tilde v\right)$ indices from the $\mathscr{K}$ or $\mathfrak{K}$ subscript. 
\end{definition}

Now we define the high-regularity norm for the Ricci Coefficients.
\begin{definition}Let $0 < q \ll p  \ll 1$. For any Ricci coefficient $\psi \neq \eta_A,\underline{\omega},\hat{\underline{\chi}}_{AB}$ of signature $s$ we define
\begin{align*}
&\left\vert\left\vert \psi\right\vert\right\vert^2_{\mathscr{L}_{p,\tilde u,\tilde v}} \doteq 
\\ \nonumber &\sup_{0 \leq j \leq 3}\sup_{(u,v)\in\mathcal{P}_{\tilde u,\tilde v}}(-u)^{2q}\int_{{\rm max}\left(-v\underline{v},-\underline{v}^2\right)}^u\int_{-u\underline{v}^{-1}}^v\int_{\mathbb{S}^2}\left|\slashed{\nabla}^j\left(\Omega^{-s}\psi\right)^*\right|^2\dot{v}^{1+2j}\left(\frac{-\dot{u}}{\dot{v}}\right)^{2p}(-\dot{u})^{-1}(-\dot{u})^{-2q}\, d\dot{u}\, d\dot{v}\, \mathring{\rm dVol},
\end{align*}
where $\left(\Omega^{-s}\psi\right)^*$ denotes the different of $\Omega^{-s}\psi$ and its Minkowski value.

For $\psi \in \{\eta_A,\hat{\underline{\chi}}_{AB}\}$ we define
\begin{align*}
\left\vert\left\vert \psi\right\vert\right\vert^2_{\mathscr{L}_{p,\tilde u,\tilde v}} &\doteq \sup_{0 \leq j \leq 3}\sup_{(u,v)\in\mathcal{P}_{\tilde u,\tilde v}}(-u)^{2q}\int_{{\rm max}\left(-v\underline{v},-\underline{v}^2\right)}^u\int_{-u\underline{v}^{-1}}^v\int_{\mathbb{S}^2}\left|\slashed{\nabla}^j\left(\Omega^{-s}\psi\right)\right|^2\dot{v}^{2j}\left(\frac{-\dot{u}}{\dot{v}}\right)^{2p}(-\dot{u})^{-2q}\, d\dot{u}\, d\dot{v}\, \mathring{\rm dVol},
\end{align*}
where $s$ denotes the signature of $\psi$.

We also introduce the notation 
\[\mathfrak{L}_{p,\tilde u,\tilde v} \doteq \sum_{\psi \neq \underline{\omega}}\left\vert\left\vert \psi\right\vert\right\vert_{\mathscr{L}_{p,\tilde u,\tilde v}}.\]
Finally, when it will not cause confusion, we will often suppress a subset of the $\left(p,\tilde u,\tilde v\right)$ indices from the $\mathscr{L}$ or $\mathfrak{L}$ subscript. 
\end{definition}

Finally, we define the norm for the metric coefficients. 
\begin{definition}Let $0 < p \ll 1$. For any metric coefficient $\phi \neq \Omega$ we define 
\[\left\vert\left\vert \phi\right\vert\right\vert_{\mathscr{M}_{p,\tilde u,\tilde v}} \doteq \sup_{0 \leq j \leq 3}\sup_{(u,v)\in\mathcal{P}_{\tilde u,\tilde v}}\int_{\mathbb{S}^2}\left|\slashed{\nabla}^j\left(\Omega^{-s}\psi\right)^*\right|^2v^{2j}\, \mathring{\rm dVol},\]
where $\left(\Omega^{-s}\phi\right)^*$ denotes the different of $\Omega^{-s}\phi$ and its Minkowski value.

For the lapse $\Omega$, we define 
\[\left\vert\left\vert \Omega \right\vert\right\vert_{\mathscr{M}_{p,\tilde u,\tilde v}} \doteq \sup_{1 \leq j \leq 3}\sup_{(u,v)\in\mathcal{P}_{\tilde u,\tilde v}}\int_{\mathbb{S}^2}\left|\slashed{\nabla}^j\Omega\right|^2v^{2j}\left(\frac{-u}{v}\right)^{2p}\, \mathring{\rm dVol} + \sup_{(u,v)\in\mathcal{P}_{\tilde u,\tilde v}}\left|\log\Omega\right|^2\left|\log^2\left(\frac{-u\underline{v}}{v}\right)\right|.\]
We also introduce the notation 
\[\mathfrak{M}_{p,\tilde u,\tilde v} \doteq \sum_{\phi}\left\vert\left\vert \phi\right\vert\right\vert_{\mathscr{M}_{p,\tilde u,\tilde v}}.\]
Finally, when it will not cause confusion, we will often suppress a subset of the $\left(p,\tilde u,\tilde v\right)$ indices from the $\mathscr{M}$ or $\mathfrak{M}$ subscript. 
\end{definition}

\subsection{Estimates}\label{estimatesregion3}
A standard argument using Proposition~\ref{itexistsbutforalittlewhile} shows that Theorem~\ref{itisreg3} will follow from the following proposition.
\begin{proposition}\label{tobootornottoboo3}Let $\underline{v} > 0$, and $(\mathcal{M},g_{\mu\nu})$ be a spacetime produced by Proposition~\ref{kdowvappend} which exists in the region rectangle $\mathcal{P}_{\tilde u,\tilde v}$ for some $(\tilde u,\tilde v) \in \mathcal{P}$ and which satisfies the ``bootstrap assumption''

\begin{equation}\label{bootstrapgronwallikeest3}
\mathfrak{I}_{\tilde u,\tilde v} + \mathfrak{K}_{\tilde u,\tilde v}+\mathfrak{L}_{\tilde u,\tilde v} +\mathfrak{M}_{\tilde u,\tilde v}  \leq 2A\epsilon^{1-\delta}.
\end{equation}

We then claim that~\eqref{bootstrapgronwallikeest3} implies
\begin{equation}\label{dkowdovoboto3}
\mathfrak{I}_{\tilde u,\tilde v} + \mathfrak{K}_{\tilde u,\tilde v}+\mathfrak{L}_{\tilde u,\tilde v} +\mathfrak{M}_{\tilde u,\tilde v}\leq A\epsilon^{1-1\delta}.
\end{equation}
\end{proposition}

As usual, the proof will be broken up into a few separate estimates. We start with estimates for the curvature components, then prove estimates for the Ricci coefficients, and finish with the estimates for the metric coefficients.

Throughout the proofs in this section we will use without comment that for any point in $\mathcal{P}$, we have
\[\frac{-u}{v} \leq \underline{v}.\]

We start by observing the Sobolev spaces generated by $\slashed{g}_{AB}$ and $\mathring{\slashed{g}}_{AB}$ are comparable.
\begin{lemma}\label{sosimilararethepscaes3}Let $\left(\mathcal{M},g_{\mu\nu}\right)$ satisfy the hypothesis of Proposition~\ref{tobootornottoboo3}. Then we have that
\[\left\vert\left\vert w\right\vert\right\vert_{\tilde H^i\left(\mathbb{S}^2_{u,v}\right)} \sim_k v^{-k}\left\vert\left\vert w\right\vert\right\vert_{\mathring{H}^i\left(\mathbb{S}^2_{u,v}\right)}\text{ for }i \in \{0,1,2, 3\},\]
\[\left\vert\left\vert w\right\vert\right\vert_{\tilde L^p\left(\mathbb{S}^2_{u,v}\right)} \sim_k v^{-k}\left\vert\left\vert w\right\vert\right\vert_{\mathring{L}^p\left(\mathbb{S}^2_{u,v}\right)},\]
where we recall that $\mathring{H}^i$ and $\mathring{L}^p$ denote the Sobolev and $L^p$ spaces generated by the round metric $\mathring{\slashed{g}}_{AB}$, and the spaces $\tilde{H}^j$ are defined as in Definition~\ref{idkdtildedefdef}.
\end{lemma}
\begin{proof}This is an immediate consequence of Lemma~\ref{comparethespaces}, the bootstrap hypothesis, and the smallness  of $\epsilon$.
\end{proof}

Next, we have an analogue of Lemma~\ref{someelleststs}.
\begin{lemma}\label{someelleststs3}Let $\left(\mathcal{M},g_{\mu\nu}\right)$ satisfy the hypothesis of Proposition~\ref{tobootornottoboo3}. Then for any function $f$, $1$-form $\theta_A$, and symmetric trace-free $2$-tensor $\nu_{AB}$ we have
\begin{equation}\label{pokdwimnfe32}
\left\vert\left\vert f\right\vert\right\vert_{\tilde{H}^{2+i}\left(\mathbb{S}^2_{u,v}\right)} \lesssim (-u)^2\left\vert\left\vert \slashed{\Delta}f\right\vert\right\vert_{\tilde{H}^i\left(\mathbb{S}^2_{u,v}\right)} + \left\vert\left\vert f\right\vert\right\vert_{\tilde{H}^i\left(\mathbb{S}^2_{u,v}\right)},\text{ for }i\in \{0,1\},
\end{equation}
\begin{equation}\label{pokdwimnfe232}
\left\vert\left\vert\theta \right\vert\right\vert_{\tilde H^i\left(\mathbb{S}^2_{u,v}\right)} \lesssim \left(-u\right)\left[\left\vert\left\vert \slashed{\rm div}\theta\right\vert\right\vert_{\tilde H^{i-1}\left(\mathbb{S}^2_{u,v}\right)} +\left\vert\left\vert \slashed{\rm curl}\theta\right\vert\right\vert_{\tilde H^{i-1}\left(\mathbb{S}^2_{u,v}\right)}\right],\text{ for }i\in \{1,2,3\},
\end{equation}
\begin{equation}\label{pokdwimnfe22}
\left\vert\left\vert \nu\right\vert\right\vert_{\tilde H^i\left(\mathbb{S}^2_{u,v}\right)} \lesssim (-u)\left\vert\left\vert \slashed{\rm div}\nu\right\vert\right\vert_{\tilde H^{i-1}\left(\mathbb{S}^2_{u,v}\right)},\text{ for }i \in \{1,2,3\}.
\end{equation}
\end{lemma}
\begin{proof}This is proven in the same fashion as Lemma~\ref{someelleststs}.
\end{proof}

Now we observe that the standard Sobolev inequalities hold for the spaces $H^i$.
\begin{lemma}\label{sosososob3}Let $\left(\mathcal{M},g_{\mu\nu}\right)$ satisfy the hypothesis of Proposition~\ref{tobootornottoboo3}. Then, for any $(0,k)$-tensor $w_{A_1\cdots A_k}$, we have that
\[\left\vert\left\vert w\right\vert\right\vert_{\tilde{L}^p\left(\mathbb{S}^2_{u,v}\right)} \lesssim_{p,k} \left\vert\left\vert w\right\vert\right\vert_{\tilde H^1\left(\mathbb{S}^2_{u,v}\right)},\qquad \left\vert\left\vert w\right\vert\right\vert_{\tilde{L}^{\infty}\left(\mathbb{S}^2_{u,v}\right)}\lesssim_k \left\vert\left\vert w\right\vert\right\vert_{\tilde{H}^2\left(\mathbb{S}^2_{u,v}\right)},\]
and for any $(0,k)$-tensor $w_{A_1\cdots A_k}$ and $(0,k')$-tensor $v_{A_1\cdots A_k'}$ we have
\[\left\vert\left\vert w\cdot v\right\vert\right\vert_{\tilde H^2\left(\mathbb{S}^2_{u,v}\right)} \lesssim_{k,k'} \left\vert\left\vert w\right\vert\right\vert_{\tilde H^2\left(\mathbb{S}^2_{u,v}\right)} \left\vert\left\vert v\right\vert\right\vert_{\tilde H^2\left(\mathbb{S}^2_{u,v}\right)}.\]

\end{lemma}
\begin{proof}This is an immediate consequence of Lemma~\ref{sosimilararethepscaes3} and Lemma~\ref{soblemm}.
\end{proof}
These Sobolev inequalities will be used repeatedly in our estimates of nonlinear terms and we will often do so without explicit comment. 

The following lemma will be frequently used to obtain $L^{\infty}_{u,v}$ estimates. 
\begin{lemma}\label{intlemma} Let $-1 \leq x_0 < x_1 < 0$ and let $f(x) : [x_0,x_1] \to \mathbb{R}$ satisfy
\[\sup_{\tilde x \in [x_0,x_1]}\left(-\tilde x\right)^{2\delta_1}\int_{x_0}^{\tilde x}\left|f(x)\right|^2\left(-x\right)^{2\delta_2}\, dx \doteq Z^2 < \infty,\]
where $\delta_1,\delta_2 > 0$ and $\delta_1+\delta_2 < 1/2$. Then we have  
\[\int_{x_0}^{x_1}\left|f(x)\right|\, dx \lesssim (-x_0)^{1/2-\delta_1-\delta_2}Z,\]
where the implied constant is independent of $f$, $x_0$,  $x_1$, and $Z$. 
\end{lemma}
\begin{proof}Let $j_0,j_1 \in \mathbb{Z}_{\geq 0}$ be defined by the requirement that $x_0 \in [-2^{-j_0},-2^{-j_0-1}]$ and $x_1 \in [-2^{-j_1},-2^{-j_1-1}]$. Then we have
\begin{align*}
 \int_{x_0}^{x_1}\left|f(x)\right|\, dx  &\leq \sum_{j=j_0}^{j_1}\int_{-2^{-j}}^{-2^{-j-1}}\left|f(x)\right|\, dx 
\\ \nonumber &\lesssim \sum_{j=j_0}^{j_1}2^{j\left(\delta_2-\frac{1}{2}\right)}\left(\int_{-2^{-j}}^{-2^{-j-1}}\left|f(x)\right|^2(-x)^{2\delta_2}\, dx\right)^{1/2} 
\\ \nonumber &\lesssim  Z\sum_{j=j_0}^{j_1}2^{j\left(\delta_1+\delta_2-1/2\right)} 
\\ \nonumber &\lesssim (-x_0)^{1/2-\delta_1-\delta_2}Z.
\end{align*}
\end{proof}

\subsubsection{Estimates for Curvature}
In this section we will prove the energy estimates for the null curvature components. We start by re-writing the Bianchi equations in a form which eliminates $\underline{\omega}$ from the equations. 

\begin{lemma}\label{Bianchiit}For any spacetime $\left(\mathcal{M},g_{\mu\nu}\right)$ in a double-null foliation, we have the following equations for the renormalized curvature components:
\begin{align}\label{3alphablahblah}
\Omega\nabla_3\left(\Omega^2\alpha\right)_{AB} +\frac{\Omega^2}{2}\left(\Omega^{-1}{\rm tr}\underline{\chi}\right)\Omega^2\alpha_{AB} &= \Omega^2\Big[\slashed{\nabla}\hat{\otimes}\left(\Omega\beta\right)  -3\left(\left(\Omega\hat{\chi}\right)\rho + {}^*\left(\Omega\hat{\chi}\right)\sigma\right) 
\\ \nonumber &\qquad \qquad \qquad \qquad \qquad+ \left(-\underline{\eta}+ 4\eta\right)\hat{\otimes}\left(\Omega\beta\right)\Big]_{AB},
\\ \label{4betablahblah} \Omega\nabla_4\left(\Omega\beta\right)_A + 2\left(\Omega{\rm tr}\chi\right)\left(\Omega \beta\right)_A &= \slashed{\rm div}\left(\Omega^2\alpha\right)_A - 6\left(\Omega\omega\right)\left(\Omega\beta\right)_A - \left(\underline{\eta}\cdot\left(\Omega^2\alpha\right)\right)_A,
\\ \label{ren133} \Omega\nabla_3\left(\Omega\beta\right)_A + \Omega^2\left(\Omega^{-1}{\rm tr}\underline{\chi}\right)\Omega\beta_A &= \Omega^2\Bigg[\slashed{\nabla}_A\check{\rho} + {}^*\slashed{\nabla}_A\check{\sigma} + 2\left(\hat{\chi}\cdot\underline{\beta}\right)_A + 3\left(\eta\check{\rho} + {}^*\eta\check{\sigma}\right)_A
\\ \nonumber &\qquad +\frac{1}{2}\left(\slashed{\nabla}\left(\hat{\chi}\cdot\hat{\underline{\chi}}\right) + {}^*\slashed{\nabla}\left(\hat{\chi}\wedge \hat{\underline{\chi}}\right)\right)_A + \frac{3}{2}\left(\eta \hat{\chi}\cdot\hat{\underline{\chi}} + {}^*\eta \hat{\chi}\wedge \hat{\underline{\chi}}\right)_A\Bigg],
\\ \label{ren233} \Omega\nabla_4\check{\sigma} + \frac{3}{2}\left(\Omega{\rm tr}\chi\right) \check{\sigma} &= -\slashed{\rm div}{}^*\left(\Omega\beta\right) - \underline{\eta} \wedge \left(\Omega\beta\right) - \frac{1}{2}\left(\Omega\hat{\chi}\right)\wedge \left(\slashed{\nabla}\hat{\otimes}\underline{\eta}\right) - \frac{1}{2}\left(\Omega\hat{\chi}\right)\wedge \left(\underline{\eta}\hat{\otimes}\underline{\eta}\right),
\end{align}
\begin{align}
 \label{ren333} \Omega\nabla_4\check{\rho} + \frac{3}{2}\left(\Omega{\rm tr}\chi\right) \check{\rho} &= \slashed{\rm div}\left(\Omega\beta\right) +\underline{\eta}\cdot\left(\Omega\beta\right)
\\ \nonumber&\qquad - \frac{1}{2}\left(\Omega\hat{\chi}\right)\cdot\left(\slashed{\nabla}\hat{\otimes}\underline{\eta}\right) -\frac{1}{2}\left(\Omega\hat{\chi}\right)\cdot\left(\underline{\eta}\hat{\otimes}\underline{\eta}\right) + \frac{1}{4}\left(\Omega{\rm tr}\underline{\chi}\right)\left|\hat{\chi}\right|^2,
\\ \label{ren433} \Omega\nabla_3\check{\sigma} + \Omega^2\frac{3}{2}\left(\Omega{\rm tr}\underline{\chi}\right) \check{\sigma} &= \Omega^2\Bigg[\slashed{\rm div}{}^*\left(\Omega^{-1}\underline{\beta}\right) -\left(\underline{\eta} + 2\eta \right)\wedge \left(\Omega^{-1}\underline{\beta}\right) 
\\ \nonumber &\qquad +\frac{1}{2}\left(\Omega^{-1}\hat{\underline{\chi}}\right)\wedge \left(\slashed{\nabla}\hat{\otimes}\eta\right) + \frac{1}{2}\left(\Omega^{-1}\hat{\underline{\chi}}\right)\wedge \left(\eta\hat{\otimes}\eta\right)\Bigg],
\\ \label{ren533} \Omega\nabla_3\check{\rho} + \Omega^2\frac{3}{2}\left(\Omega{\rm tr}\underline{\chi}\right) \check{\rho} &= \Omega^2\Bigg[-\slashed{\rm div}\left(\Omega^{-1}\underline{\beta}\right)  -\left(\underline{\eta}+ 2\eta\right)\cdot\left(\Omega^{-1}\underline{\beta}\right) 
\\ \nonumber &\qquad - \frac{1}{2}\left(\Omega^{-1}\hat{\underline{\chi}}\right)\cdot\left(\slashed{\nabla}\hat{\otimes}\eta\right) -\frac{1}{2}\left(\Omega^{-1}\hat{\underline{\chi}}\right)\cdot\left(\eta\hat{\otimes}\eta\right) + \frac{1}{4}\left(\Omega^{-1}{\rm tr}\chi\right)\left|\hat{\underline{\chi}}\right|^2\Bigg],
\\ \label{ren633} \Omega\nabla_4\left(\Omega^{-1}\underline{\beta}\right)_A + \left(\Omega{\rm tr}\chi\right)\left(\Omega^{-1}\underline{\beta}\right)_A &= -\slashed{\nabla}_A\check{\rho} + {}^*\slashed{\nabla}_A\check{\sigma} + 4\left(\Omega\omega\right)\left(\Omega^{-1}\underline{\beta}\right)_A + 2\left(\left(\Omega^{-1}\hat{\underline{\chi}}\right)\cdot\left(\Omega\beta\right)\right)_A
\\ \nonumber &\qquad - 3\left(\underline{\eta}\check{\rho} - {}^*\underline{\eta}\check{\sigma}\right)_A
 -\frac{1}{2}\left(\slashed{\nabla}_A\left(\left(\Omega^{-1}\hat{\underline{\chi}}\right)\cdot\left(\Omega\hat{\chi}\right)\right) + {}^*\slashed{\nabla}_A\left(\left(\Omega^{-1}\hat{\underline{\chi}}\right)\wedge \left(\Omega\hat{\chi}\right)\right)\right) 
\\ \nonumber &\qquad - \frac{3}{2}\left(\underline{\eta} \left(\Omega^{-1}\hat{\underline{\chi}}\right)\cdot\left(\Omega\hat{\chi}\right) + {}^*\underline{\eta} \left(\Omega^{-1}\hat{\underline{\chi}}\right)\wedge \left(\Omega\hat{\chi}\right)\right)_A.
\end{align}
We recall the $\check{\rho}$ and $\check{\sigma}$ are defined by~\eqref{reallyaquitenicerenormalization}.
\end{lemma}

Next, we carry out an estimate for the Gauss curvature $K$.
\begin{proposition}\label{thegaussisundercontrolyay}Let $\left(\mathcal{M},g_{\mu\nu}\right)$ satisfy the hypothesis of  Proposition~\ref{tobootornottoboo3}. Then we have
\[\sup_{i \in \{0,1\}}\sup_{(u,v) \in \mathcal{P}_{\tilde u,\tilde v}}\int_{\mathbb{S}^2_{u,v}}v^{4+2i}\left|\slashed{\nabla}^i\left(K-(v-u)^{-2}\right)\right|^2\mathring{\rm dVol} \lesssim A\epsilon^{2-2\delta}.\]
Recall that the constant $A$ is defined from the bootstrap assumption.
\end{proposition}
\begin{proof}We first note that from~\eqref{ren533} and Lemma~\ref{commlemm}, for $i \in \{0,1\}$, we may derive the following equation $\slashed{\nabla}_{A_1\cdots A_i}^i\check{\rho}$:
\begin{equation}\label{ddkokwasfd}
\Omega\nabla_3\left(v^{2+i}\slashed{\nabla}_{A_1\cdots A_i}^i\check{\rho}\right) = v^{2+i}\mathcal{E},
\end{equation}
where $\mathcal{E}$ is a sum of expressions of the following schematic form:
\[\slashed{\nabla}^i\left[\Omega^2\left(\Omega^{-s_1}\psi_1\right)\cdot\left(\Omega^{-s_2}\check{\Psi}_2,\slashed{\nabla}\left(\Omega^{-s_3}\psi_3\right)\right)\right],\qquad \slashed{\nabla}^{i-1}\left(\left(\Omega^{-s_1}\psi_1\right)\cdot\left(\Omega^{-s_2}\psi_s\right)\cdot \check{\rho}\right),\qquad \slashed{\nabla}^i\left(\Omega^2\slashed{\nabla}\left(\Omega^{-s_1}\check{\Psi}_1\right)\right),\]
where $\psi_i$ is a Ricci coefficient of signature $s_i$ not equal to $\underline{\omega}$ and $\check{\Psi}_i$ denotes a renormalized curvature component of signature $s_i$. Contracting both sides of~\eqref{ddkokwasfd} with $v^{2+i}\slashed{\nabla}_{A_1\cdots A_i}^i\check{\rho}$, integrating by parts, using the bootstrap assumption, and using Proposition~\ref{itstartedoutok2} leads to
\[\sup_{(u,v) \in \mathcal{P}_{\tilde u,\tilde v}}\int_{\mathbb{S}^2_{u,v}}v^{4+2i}\left|\slashed{\nabla}^i\check{\rho}\right|^2\mathring{\rm dVol} \lesssim \left(\int_{{\rm max}\left(-v\underline{v},-\underline{v}^2\right)}^uv^{2+i}\left|\mathcal{E}\right||_{(u,v) = (\dot{u},v)}\, d\dot{u}\, \mathring{\rm dVol}\right)^2 + \epsilon^{2-2\delta}.\]
Lemma~\ref{intlemma} then yields
\begin{align*}
&\sup_{(u,v) \in \mathcal{P}_{\tilde u,\tilde v}}\int_{\mathbb{S}^2_{u,v}}v^{4+2i}\left|\slashed{\nabla}^i\check{\rho}\right|^2\mathring{\rm dVol} \lesssim 
\\ \nonumber &\qquad \underline{v}^{1-100p}\sup_{\mathring{u} \in [-v\underline{v},u]}\left(-\mathring{u}\right)^{100q}\int_{{\rm max}\left(-v\underline{v},-\underline{v}^2\right)}^{\mathring{u}}v^{5+2i}\left|\mathcal{E}\right|^2\left(\frac{-\dot{u}}{v}\right)^{100p}(-\dot{u})^{-100q}|_{(u,v) = (\dot{u},v)}\, d\dot{u}\, \mathring{\rm dVol} + \epsilon^{2-2\delta}.
\end{align*}
Then, using the bootstrap assumptions, Sobolev inequalities, and the smallness of $\underline{v}$, $p$, and $q$, we obtain 
\[\sup_{i \in \{0,1\}}\sup_{(u,v) \in \mathcal{P}_{\tilde u,\tilde v}}\int_{\mathbb{S}^2_{u,v}}v^{4+2i}\left|\slashed{\nabla}^i\check{\rho}\right|^2\mathring{\rm dVol} \lesssim \epsilon^{2-2\delta}.\]
To go from the control of $\check{\rho}$ to the control of $K-1$, we use the following consequence of the Gauss equation~\eqref{genGauss}:
\[K - (v-u)^{-2} = -\check{\rho} - \frac{1}{4}\left(\Omega{\rm tr}\chi\right)^*\Omega^{-1}{\rm tr}\underline{\chi}  -\frac{1}{2(v-u)}\left(\Omega^{-1}{\rm tr}\underline{\chi}\right)^*,\]
from which the proof is immediately concluded. 

\end{proof}

We now carry out the energy estimates for curvature.
\begin{proposition}\label{thisisenegege3}Let $\left(\mathcal{M},g_{\mu\nu}\right)$ satisfy the hypothesis of  Proposition~\ref{tobootornottoboo3}. Then  we have
\[\mathfrak{I} \lesssim \epsilon^{1-\delta}.\]
\end{proposition}
\begin{proof}We start with the equations~\eqref{ren433}-\eqref{ren633} corresponding to the Bianchi pair $\left(\underline{\beta}_A,\check{\rho},\check{\sigma}\right)$ which we write in the following form:
\begin{align}\label{newform}
\Omega\nabla_3\check{\sigma} &= \Omega^2\slashed{\rm div}{}^*\left(\Omega^{-1}\underline{\beta}\right) + \mathcal{E}_1,
\\ \label{newform2} \Omega\nabla_3\check{\rho} &=  -\Omega^2\slashed{\rm div}\left(\Omega^{-1}\underline{\beta}\right) + \mathcal{E}_2,
\\ \label{newform3} \Omega\nabla_4\left(\Omega^{-1}\underline{\beta}\right)_A + \left(\Omega{\rm tr}\chi\right)\left(\Omega^{-1}\underline{\beta}\right)_A &= -\slashed{\nabla}_A\check{\rho} + {}^*\slashed{\nabla}_A\check{\sigma}  + \mathcal{E}_3,
\end{align}
where $\mathcal{E}_1$, $\mathcal{E}_2$ contain terms which are of the following schematic form:
\[\Omega^2\left(\Omega^{-s_1}\psi_1\right)\cdot\left(\left(\Omega^{-s_2}\check{\Psi}_2\right),\slashed{\nabla}\left(\Omega^{-s_2}\psi_2\right),\left(\Omega^{-s_2}\psi_2\right)\cdot\left(\Omega^{-s_3}\psi_3\right)\right),\]
and $\mathcal{E}_3$ contain terms of the schematic form:
\[\left(\Omega^{-s_1}\psi_1\right)\cdot\left(\left(\Omega^{-s_2}\check{\Psi}_2\right),\slashed{\nabla}\left(\Omega^{-s_2}\psi_2\right),\left(\Omega^{-s_2}\psi_2\right)\cdot\left(\Omega^{-s_3}\psi_3\right)\right),\]
where $\check{\Psi}_i$ denotes a renormalized curvature component of signature $s_i$ and $\psi_i$ denotes a Ricci coefficient not equal to $\underline{\omega}$ of signature $s_i$, and in the cubic term at most of one of the $\psi$ can be equal to ${\rm tr}\underline{\chi}$ or ${\rm tr}\chi$. The terms in $\mathcal{E}_3$ have the additional constraint that $\psi_i \not\in \{{\rm tr}\chi,{\rm tr}\underline{\chi},\eta_A\}$. Now, for $i \in \{0,1,2\}$ we commute~\eqref{newform} and~\eqref{newform2} with $v^{3/2}\left(\frac{-u}{v}\right)^p(-u)^{-q}\slashed{\nabla}_{A_1\cdots A_i}^i$ and commute~\eqref{newform3} with $\Omega\slashed{\nabla}_{A_1\cdots A_i}^i$. Using Lemma~\ref{commlemm}, we then end up with (suppressing the indices on $\slashed{\nabla}^i$ for typographical reasons)
\begin{align}\label{newform}
&\Omega\nabla_3\left(v^{3/2+i}\left(\frac{-u}{v}\right)^p(-u)^{-q}\slashed{\nabla}^i\check{\sigma}\right) + \frac{p-q}{-u}\left(v^{3/2+i}\left(\frac{-u}{v}\right)^p(-u)^{-q}\slashed{\nabla}^i\check{\sigma}\right)  = 
\\ \nonumber &\qquad \qquad \Omega^2\slashed{\rm div}{}^*\left[v^{3/2+i}\left(\frac{-u}{v}\right)^p(-u)^{-q}\slashed{\nabla}^i\left(\Omega^{-1}\underline{\beta}\right)\right] + v^{3/2+i}\left(\frac{-u}{v}\right)^p(-u)^{-q}\mathcal{F}_1,
\\ \label{newform2} &\Omega\nabla_3\left(v^{3/2+i}\left(\frac{-u}{v}\right)^p(-u)^{-q}\slashed{\nabla}^i\check{\rho} \right) + \frac{p-q}{-u}\left(v^{3/2+i}\left(\frac{-u}{v}\right)^p(-u)^{-q}\slashed{\nabla}^i\check{\rho} \right)= 
\\ \nonumber &\qquad  \qquad -\Omega^2\slashed{\rm div}\left[v^{3/2+i}\left(\frac{-u}{v}\right)^p(-u)^{-q}\slashed{\nabla}^i\left(\Omega^{-1}\underline{\beta}\right)\right] + v^{3/2+i}\left(\frac{-u}{v}\right)^p(-u)^{-q}\mathcal{F}_2,
\\ \label{newform3} &\Omega\nabla_4\left(v^{3/2+i}\left(\frac{-u}{v}\right)^p(-u)^{-q}\Omega\slashed{\nabla}^i\left(\Omega^{-1}\underline{\beta}\right)\right)_A 
\\ \nonumber &\qquad + \left(\frac{-3/2-i}{v} + \frac{2+i}{2}\Omega{\rm tr}\chi\right)v^{3/2+i}\left(\frac{-u}{v}\right)^p(-u)^{-q}\Omega\slashed{\nabla}^i\left(\Omega^{-1}\underline{\beta}\right)_A = 
\\ \nonumber &\qquad \qquad -\Omega\left(\slashed{\nabla}\left[v^{3/2+i}\left(\frac{-u}{v}\right)^p(-u)^{-q}\slashed{\nabla}^i\check{\rho}\right] + {}^*\slashed{\nabla}\left[v^{3/2+i}\left(\frac{-u}{v}\right)^p(-u)^{-q}\slashed{\nabla}^i\check{\sigma} \right]\right)_A 
\\ \nonumber &\qquad \qquad + v^{3/2+i}\left(\frac{-u}{v}\right)^p(-u)^{-q}\mathcal{F}_3,
\end{align}
where $\mathcal{F}_1$ and $\mathcal{F}_2$ contain terms of the schematic form
\[\underbrace{\slashed{\nabla}^i\left[\Omega^2\left(\Omega^{-s_1}\psi_1\right)\cdot\left(\left(\Omega^{-s_2}\check{\Psi}_2\right),\slashed{\nabla}\left(\Omega^{-s_2}\psi_2\right),\left(\Omega^{-s_2}\psi_2\right)\cdot\left(\Omega^{-s_3}\psi_3\right)\right)\right]}_I,\]\[ \underbrace{\slashed{\nabla}^{i-1}\left(\Omega^2\left(\Omega^{-s_1}\psi_1\right)\cdot\left(\Omega^{-s_2}\psi_2\right)\cdot\left(\Omega^{-s_3}\check{\Psi}_3\right)\right)}_{II},\qquad \underbrace{\slashed{\nabla}^{i-1}\left(\Omega^2 K\cdot\Omega^{-s_1}\check{\Psi}_1\right)}_{III},\]
and $\mathcal{F}_3$ contains terms of the schematic form
\[\underbrace{\Omega\slashed{\nabla}^i\left[\left(\Omega^{-s_1}\psi_1\right)\cdot\left(\left(\Omega^{-s_2}\check{\Psi}_2\right),\slashed{\nabla}\left(\Omega^{-s_2}\psi_2\right),\left(\Omega^{-s_2}\psi_2\right)\cdot\left(\Omega^{-s_3}\psi_3\right)\right)\right]}_{IV},\]\[ \underbrace{\Omega\slashed{\nabla}^{i-1}\left(\left(\Omega^{-s_1}\psi_1\right)\cdot\left(\Omega^{-s_2}\psi_2\right)\cdot\left(\Omega^{-s_3}\check{\Psi}_3\right)\right)}_{V},\qquad \underbrace{\Omega\slashed{\nabla}^{i-1}\left(K\cdot\Omega^{-s_1}\check{\Psi}_1\right)}_{VI},\]
where the terms with the Gaussian curvature $K$ cannot occur for $i = 0$, $\check{\Psi}_i$ denotes a renormalized curvature component of signature $s_i$ and $\psi_i$ denotes a Ricci coefficient not equal to $\underline{\omega}$ of signature $s_i$, and in the cubic term at most of one of the $\psi$ can be equal to ${\rm tr}\underline{\chi}$ or ${\rm tr}\chi$. The terms in $\mathcal{F}_3$ have the following additional constraints:
\begin{enumerate}
 \item  We have $\psi_i \not\in \{{\rm tr}\underline{\chi},\eta_A\}$.
 \item The only place where ${\rm tr}\chi$ can appear, without an angular derivative acting on it, is as exactly one of the $\psi$'s in a cubic term.
 \item We can have at most one of $\{\underline{\beta}_A,\hat{\underline{\chi}}_{AB}\}$ in any of the nonlinear expressions.
 \end{enumerate}

Next, using that $3/2 < 2$, we observe the following consequence of the  bootstrap assumptions:
\begin{equation}\label{thismakesitwork}
\left(\frac{-3/2-i}{v} + \frac{2+i}{2}\Omega{\rm tr}\chi\right) \gtrsim v^{-1}.
\end{equation}
Keeping~\eqref{thismakesitwork} in mind, we contract~\eqref{newform} with $v^{3/2+i}\left(\frac{-u}{v}\right)^p(-u)^{-q}\slashed{\nabla}^i\check{\sigma}$,~\eqref{newform2} with $v^{3/2+i}\left(\frac{-u}{v}\right)^p(-u)^{-q}\slashed{\nabla}^i\check{\rho}  $,~\eqref{newform3} with $v^{3/2+i}\left(\frac{-u}{v}\right)^p(-u)^{-q}\Omega\slashed{\nabla}^i\left(\Omega^{-1}\underline{\beta}\right)_A$, add the resulting equations together, integrate by parts, and multiply the final result by $(-u)^{2q}$ to obtain:
\begin{align}\label{wellisupposewehaveachievedthis}
&\sup_{(u,v) \in \mathcal{P}_{\tilde u,\tilde v}}\Bigg[(-u)^{2q}\int_{{\rm max}\left(-v\underline{v},-\underline{v}^2\right)}^u\int_{\mathbb{S}^2}\Omega^2\left|\slashed{\nabla}^j\left(\Omega^{-1}\underline{\beta}\right)\right|^2v^{3+2j}\left(\frac{-\dot{u}}{v}\right)^{2p}(-\dot{u})^{-2q}\, d\dot{u}\, \mathring{dVol}
\\ \nonumber &\qquad  \qquad + (-u)^{2q}\int_{{\rm max}\left(-v\underline{v},-\underline{v}^2\right)}^u\int_{-u\underline{v}^{-1}}^v\int_{\mathbb{S}^2}\Omega^2\left|\slashed{\nabla}^j\left(\Omega^{-1}\underline{\beta}\right)\right|^2\dot{v}^{2+2j}\left(\frac{-\dot{u}}{\dot{v}}\right)^{-2p}(-\dot{u})^{-2q}\, d\dot{v}\, d\dot{u}\, \mathring{\rm dVol}\Bigg]
\\ \nonumber &\qquad + \sup_{(u,v) \in \mathcal{P}_{\tilde u,\tilde v}}\Bigg[\int_{-u\underline{v}^{-1}}^v\left|\slashed{\nabla}^j\left(\check{\rho},\check{\sigma}\right)\right|^2\dot{v}^{3+2j}\left(\frac{-u}{\dot{v}}\right)^{2p}\, d\dot{v}\, \mathring{dVol}
\\ \nonumber &\qquad  \qquad + (-u)^{2q}\int_{{\rm max}\left(-v\underline{v},-\underline{v}^2\right)}^u\int_{-u\underline{v}^{-1}}^v\int_{\mathbb{S}^2}\left|\slashed{\nabla}^j\left(\check{\rho},\check{\sigma}\right)\right|^2\dot{v}^{3+2j}\left(\frac{-\dot{u}}{\dot{v}}\right)^{-2p}(-\dot{u})^{-2q}(-\dot{u})^{-1}\, d\dot{v}\, d\dot{u}\, \mathring{\rm dVol}\Bigg] \lesssim 
\\ &\sup_{(u,v) \in \mathcal{P}_{\tilde u,\tilde v}} (-u)^{2q}\int_{{\rm max}\left(-v\underline{v},-\underline{v}^2\right)}^u\int_{-u\underline{v}^{-1}}^v\int_{\mathbb{S}^2}\Bigg[ v^{3+2j}\left(\frac{-\dot{u}}{\dot{v}}\right)^{2p}(-\dot{u})^{1-2q}\left|\left(\mathcal{F}_1,\mathcal{F}_2\right)\right|^2 
\\ \nonumber &\qquad \qquad \qquad \qquad \qquad \qquad \qquad \qquad + v^{4+2j}\left(\frac{-\dot{u}}{\dot{v}}\right)^{2p}(-\dot{u})^{-2q}\left|\mathcal{F}_3\right|^2\Bigg]\, d\dot{u}\, d\dot{v}\, \mathring{\rm dVol}
\\ \nonumber &\qquad +\sup_{u \in [-\underline{v}^2,0)}(-u)^{2q}\int_{-\underline{v}^2}^u\int_{\mathbb{S}^2}\left((-s)^{3+2j}(-s)^{-2q}\left|\slashed{\nabla}^j\left(\Omega^{-1}\underline{\beta},\check{\rho},\check{\sigma}\right)\right|^2\right)|_{(u,v) = (s,-\underline{v}^{-1}s)}\, d\hat{s}\, \mathring{\rm dVol}
\\ \nonumber &\qquad +\underline{v}^{-2p}\int_{\underline{v}}^{\infty}\int_{\mathbb{S}^2}\left(\dot{v}^{3+2j-2p}\left|\slashed{\nabla}^j\left(\Omega^{-1}\underline{\beta},\check{\rho},\check{\sigma}\right)\right|^2\right)|_{(u,v) = (-\underline{v}^2,\dot{v})}\, d\dot{v}\, \mathring{\rm dVol}.
\end{align}

Observing that the left hand side of~\eqref{wellisupposewehaveachievedthis} already controls a good spacetime term integral of $\check{\rho}$ and $\check{\sigma}$, one may inductively repeat this analysis for the remaining Bianchi pairs $\left(\check{\sigma},\check{\rho},\beta_A\right)$ and $\left(\beta_A,\alpha_{AB}\right)$, and arrive at 
\begin{align}\label{wellisupposewehaveachievedthis2}
&\mathfrak{I} \lesssim \sup_{(u,v) \in \mathcal{P}_{\tilde u,\tilde v}} (-u)^{2q}\int_{{\rm max}\left(-v\underline{v},-\underline{v}^2\right)}^u\int_{-u\underline{v}^{-1}}^v\int_{\mathbb{S}^2}\Bigg[ \dot{v}^{3+2j}\left(\frac{-\dot{u}}{\dot{v}}\right)^{2p}(-\dot{u})^{1-2q}\left|\tilde{\mathcal{F}}\right|^2 
\\ \nonumber &\qquad \qquad \qquad \qquad \qquad \qquad \qquad \qquad + \dot{v}^{4+2j}\left(\frac{-\dot{u}}{\dot{v}}\right)^{2p}(-\dot{u})^{-2q}\left|\mathcal{F}_3\right|^2\Bigg]\, d\dot{u}\, d\dot{v}\, \mathring{\rm dVol}
\\ \nonumber &\qquad +\sum_{\check{\Psi}}\sup_{u \in [-\underline{v},0)}(-u)^{2q}\int_{-\underline{v}^2}^u\int_{\mathbb{S}^2}\left((-\dot{s})^{3+2j}(-\dot{s})^{-2q}\left|\slashed{\nabla}^j\check{\Psi}\right|^2\right)|_{(u,v) = (\dot{s},-\underline{v}^{-1}\dot{s})}\, d\dot{s}\, \mathring{\rm dVol}
\\ \nonumber &\qquad +\underline{v}^{-2p}\int_{\underline{v}}^{\infty}\int_{\mathbb{S}^2}\left(\dot{v}^{3+2j-2p}\left|\slashed{\nabla}^j\check{\Psi}\right|^2\right)|_{(u,v) = (-\underline{v}^2,\dot{v})}\, d\dot{v}\, \mathring{\rm dVol}.
\end{align}
where $\tilde{\mathcal{F}}$ has the same schematic form as $\mathcal{F}_1$ and $\mathcal{F}_2$. We now turn to an analysis of the various terms on the right hand side of~\eqref{wellisupposewehaveachievedthis2}.

First of all, it follows from Proposition~\ref{itstartedoutok2} that 
\[\sum_{\check{\Psi}}\sup_{u \in [-\underline{v},0)}(-u)^{2q}\int_{-\underline{v}}^u\int_{\mathbb{S}^2}\left((-\dot{s})^{3+2j}(-\dot{s})^{-2q}\left|\slashed{\nabla}^j\check{\Psi}\right|^2\right)|_{(u,v) = (\dot{s},-\underline{v}\dot{s})}\, d\dot{s}\, \mathring{\rm dVol} \lesssim \epsilon^{2-2\delta},\]
\[\underline{v}^{-2p}\int_{\underline{v}}^{\infty}\int_{\mathbb{S}^2}\left(\dot{v}^{3+2j-2p}\left|\slashed{\nabla}^j\check{\Psi}\right|^2\right)|_{(u,v) = (-\underline{v}^2,\dot{v})}\, d\dot{v}\, \mathring{\rm dVol} \lesssim \epsilon^{2-2\delta}.\]
(After possibly slightly increasing $\delta$.)

Next we turn to the $\mathcal{F}_3$ term, and consider first expressions which do not contain the Gaussian curvature $K$. There are no $\eta_A$'s or $\underline{\omega}$'s and  each expression is genuinely quadratic in that, using the bootstrap assumption, there are at least two terms in each expression which are controlled, in a suitable norm, by $\epsilon^{1-\delta}$, and finally each expression can only contain at most one of $\{\hat{\underline{\chi}}_{AB},\underline{\beta}_A\}$.  Thus it follows from Sobolev inequalities and the bootstrap assumptions that 
\begin{align}
\sup_{(u,v) \in \mathcal{P}_{\tilde u,\tilde v}}(-u)^{2q}\int_{{\rm max}\left(-v\underline{v},-\underline{v}^2\right)}^u\int_{-u\underline{v}^{-1}}^v\int_{\mathbb{S}^2}\dot{v}^{4+2j}\left(\frac{-\dot{u}}{\dot{v}}\right)^{2p}(-\dot{u})^{-2q}\left|\left(IV,V\right)\right|^2\, d\dot{u}\, d\dot{v}\, \mathring{\rm dVol} \lesssim \epsilon^{2-2\delta}.
\end{align}
For the terms $I$ and $II$ in $\tilde{\mathcal{F}}$, we do not need to exploit the absence of $\eta_A$ or a limit on the appearance of $\hat{\underline{\chi}}_{AB}$ or $\underline{\beta}_A$ because the $u$-weight is more favorable. We  thus obtain
\begin{align}
\sup_{(u,v) \in \mathcal{P}_{\tilde u,\tilde v}}(-u)^{2q}\int_{{\rm max}\left(-v\underline{v},-\underline{v}^2\right)}^u\int_{-u\underline{v}^{-1}}^v\int_{\mathbb{S}^2}\dot{v}^{3+2j}\left(\frac{-\dot{u}}{\dot{v}}\right)^{2p}(-\dot{u})^{1-2q}\left|\left(I,II\right)\right|^2\, d\dot{u}\, d\dot{v}\, \mathring{\rm dVol} \lesssim \epsilon^{2-2\delta}.
\end{align}

Now we turn to the terms $III$ and $VI$. For these, we simply argue as in the derivation of~\eqref{okdkowdqqqpkbnwif} and use Proposition~\ref{thegaussisundercontrolyay} to obtain that 
\begin{align}
\sup_{(u,v) \in \mathcal{P}_{\tilde u,\tilde v}}(-u)^{2q}\int_{{\rm max}\left(-v\underline{v},-\underline{v}^2\right)}^u\int_{-u\underline{v}^{-1}}^v\int_{\mathbb{S}^2}\Bigg[&\dot{v}^{4+2j}\left(\frac{-\dot{u}}{v}\right)^{2p}(-\dot{u})^{-2q}\left|III\right|^2
\\ \nonumber &+\dot{v}^{3+2j}\left(\frac{-\dot{u}}{\dot{v}}\right)^{2p}(-\dot{u})^{1-2q}\left|VI\right|^2\Bigg]\, d\dot{u}\, d\dot{v}\, \mathring{\rm dVol} \lesssim \epsilon^{2-2\delta}.
\end{align}
This concludes the proof. 
\end{proof}
\subsubsection{Estimates for the Ricci Coefficients}
Now we turn to the estimates for the Ricci coefficients. We start with the low-regularity estimates.
\begin{proposition}Let $\left(\mathcal{M},g_{\mu\nu}\right)$ satisfy the hypothesis of  Proposition~\ref{tobootornottoboo3}. Then  we have
\[\mathfrak{K}_{\tilde u,\tilde v}  \lesssim \epsilon^{1-\delta}.\]
\end{proposition}
\begin{proof}If $\psi \neq \underline{\omega},\eta_A,\hat{\underline{\chi}}, {\rm tr}\underline{\chi},{\rm tr}\chi$, then after applying Lemma~\ref{commlemm}, one easily establishes that for $i \in \{0,1,2\}$, $\psi$ will satisfy an equation of the form
\begin{equation}\label{Dwdw2sxc}
\Omega\nabla_3\left(v^{1+i}\slashed{\nabla}_{A_1\cdots A_i}^i\left(\Omega^{-s}\psi\right)\right) = v^{1+i}\mathcal{E},
\end{equation}
where $s$ is the signature of $\psi$ and $\mathcal{E}$ is a sum of terms of the schematic form:
\[\slashed{\nabla}^i\left(\Omega^2\Omega^{-s_1}\check{\Psi}_1\right),\qquad \slashed{\nabla}^i\left(\Omega^2\slashed{\nabla}\left(\Omega^{-s_1}\psi_1\right)\right),\]
\[\slashed{\nabla}^i\left(\Omega^2 \left(\Omega^{-s_1}\psi_1\right)\cdot\left(\Omega^{-s_2}\psi_2\right)\right),\qquad \Omega^2\slashed{\nabla}^{i-1}\left(\left(\Omega^{-s_1}\psi_1\right)\cdot\left(\Omega^{-s_2}\psi_2\right)\cdot\left(\Omega^{-s_3}\psi_3\right)\right).\]
where $\psi_i$ is a Ricci coefficient of signature $s_i$ not equal to $\underline{\omega}$ and $\check{\Psi}_i$ is a renormalized null curvature component of signature $s_i$. 

Now we contract~\eqref{Dwdw2sxc} with $v^{1+i}\slashed{\nabla}^i\Omega^{-s}\psi$, integrate by parts, use Proposition~\ref{itstartedoutok2}, and use Lemma~\ref{intlemma} to obtain
\begin{align}
&\sup_{(u,v) \in \mathcal{P}_{\tilde u,\tilde v}}\int_{\mathbb{S}^2_{u,v}}v^{2+2i}\left|\slashed{\nabla}^i\left(\Omega^{-s}\psi\right)\right|^2\mathring{\rm dVol} 
\\ \nonumber &\lesssim \left(\int_{{\rm max}\left(-v\underline{v},-\underline{v}^2\right)}^uv^{1+i}\left|\mathcal{E}\right||_{(u,v) = (\dot{u},v)}\, d\dot{u}\, \mathring{\rm dVol}\right)^2 + \epsilon^{2-2\delta}
\\ \nonumber &\lesssim \underline{v}^{1-100p}\sup_{\mathring{u} \in [-v\underline{v},u]}\left(-\mathring{u}\right)^{100q}\int_{{\rm max}\left(-v\underline{v},-\underline{v}^2\right)}^{\mathring{u}}v^{3+2i}\left|\mathcal{E}\right|^2\left(\frac{-\dot{u}}{v}\right)^{100p}(-\dot{u})^{-100q}|_{(u,v) = (\dot{u},v)}\, d\dot{u}\, \mathring{\rm dVol} + \epsilon^{2-2\delta}.
\end{align}
Then, using the bootstrap assumptions, Sobolev inequalities, and the smallness of $\underline{v}$, $p$, and $q$, we obtain 
\[\sup_{i \in \{0,1,2\}}\sup_{(u,v) \in \mathcal{P}_{\tilde u,\tilde v}}\int_{\mathbb{S}^2_{u,v}}v^{2+2i}\left|\slashed{\nabla}^i\left(\Omega^{-s}\psi\right)\right|^2\mathring{\rm dVol}  \lesssim \epsilon^{2-2\delta}.\]

For $\psi \in \{{\rm tr}\underline{\chi},{\rm tr}\chi\}$ we have an equation of the form 
\begin{equation}\label{Dwdw2sxc234}
\Omega\nabla_3\left(v^{1+i}\slashed{\nabla}^i\left(\Omega^{-s}\psi\right)^*\right) = v^{1+i}\mathcal{W},
\end{equation}
where $\mathcal{W}$ contains terms of the schematic form 
\[\slashed{\nabla}^i\left(\Omega^2\slashed{\nabla}\left(\Omega^{-s_1}\psi_1\right)^*\right),\qquad \slashed{\nabla}^i\left(\Omega^2 \left(u-v\right)^{-1}\cdot\left(\Omega^{-s_2}\psi_2\right)^*\right),\]
\[\slashed{\nabla}^i\left(\Omega^2 \left(\Omega^{-s_1}\psi_1\right)\cdot\left(\Omega^{-s_2}\psi_2\right)^*\right),\qquad \Omega^2\slashed{\nabla}^{i-1}\left(\left(\Omega^{-s_1}\psi_1\right)\cdot\left(\Omega^{-s_2}\psi_2\right)\cdot\left(\Omega^{-s_3}\psi_3\right)^*\right).\]
where $\psi_i$ is a Ricci coefficient of signature $s_i$ not equal to $\underline{\omega}$ and $\check{\Psi}_i$ is a renormalized null curvature component of signature $s_i$. Then one my repeat the above analysis to obtain
\[\sup_{i \in \{0,1,2\}}\sup_{(u,v) \in \mathcal{P}_{\tilde u,\tilde v}}\int_{\mathbb{S}^2_{u,v}}v^{2+2i}\left|\slashed{\nabla}^i\left(\Omega^{-s}\psi\right)^*\right|^2\mathring{\rm dVol}  \lesssim \epsilon^{2-2\delta}.\]

It remains to estimate $\eta_A$. For this we must use the corresponding $\nabla_4$ equation. Using Lemma~\ref{commlemm} we obtain, for $i \in \{0,1,2\}$ (suppressing the indices on $\slashed{\nabla}^i$):
\begin{equation}\label{foretaetatea}
\Omega\nabla_4\slashed{\nabla}^i\eta_A + \frac{1+i}{2}\left(\Omega{\rm tr}\chi\right)\slashed{\nabla}^i\eta_A = \mathcal{F},
\end{equation}
where $\mathcal{F}$ is a sum of expressions of the schematic form:
\[\slashed{\nabla}^i\left(\Omega\beta\right)_A,\qquad \slashed{\nabla}^i\left(\left(\Omega^{s_1}\psi_1\right)\cdot\left(\Omega^{s_2}\psi_2\right)\right),\qquad \slashed{\nabla}^{i-1}\left(\left(\Omega^{s_1}\psi_1\right)\cdot\left(\Omega^{s_2}\psi_2\right)\cdot\left(\Omega^{s_3}\psi_3\right)\right),\]
where
\begin{enumerate}
	\item Each $\psi_i$ denotes a Ricci coefficient of signature $s_i$ which is not equal to $\underline{\omega}$, ${\rm tr}\underline{\chi}$. or $\hat{\underline{\chi}}_{AB}$.
	\item Each expression can have at most one $\psi_i$ which is equal to $\eta_A$.
	\item If $\eta_A$ shows up in the one of the terms making up $\mathcal{F}$, then at least one of the other terms which it is contracted with must be $\hat{\chi}_{AB}$ or $\underline{\eta}_A$, or  ${\rm tr}\chi$ with an angular derivative applied. 
\end{enumerate}
Now we conjugate~\eqref{foretaetatea} by $v^{1+i}\left(-\frac{u}{v}\right)^p$ to obtain
\begin{equation}\label{foretaetatea2}
\Omega\nabla_4\left(v^{1+i}\left(\frac{-u}{v}\right)^p\slashed{\nabla}^i\eta\right)_A +\left(\frac{p-1-i}{v}+ \frac{1+i}{2}\left(\Omega{\rm tr}\chi\right)\right)v^{1+i}\left(\frac{-u}{v}\right)^p\slashed{\nabla}^i\eta_A = v^{1+i}\left(\frac{-u}{v}\right)^p\mathcal{F}.
\end{equation}
We now note that the bootstrap assumptions imply that 
\[\left(\frac{p-1-i}{v}+ \frac{1+i}{2}\left(\Omega{\rm tr}\chi\right)\right) \gtrsim \frac{p}{v}.\]
Thus we can contract~\eqref{foretaetatea2} with $v^{1+i}\left(\frac{-u}{v}\right)^{-p}\slashed{\nabla}^i\eta_A$, integrate by parts, and use Proposition~\ref{itstartedoutok2} to obtain
\begin{align*}
\sup_{i \in \{0,1,2\}}\sup_{(u,v) \in \mathcal{P}_{\tilde u,\tilde v}}\int_{\mathbb{S}^2_{u,v}}v^{2+2i}\left(\frac{-u}{v}\right)^{2p}\left|\slashed{\nabla}^i\eta\right|^2\mathring{\rm dVol}  &\lesssim \sup_{i \in \{0,1,2\}}\sup_{(u,v) \in \mathcal{P}_{\tilde u,\tilde v}}\int_{-u\underline{v}^{-1}}^v\int_{\mathbb{S}^2}\dot{v}^{3+2i}\left(\frac{-u}{\dot{v}}\right)^{2p}\left|\mathcal{F}\right|^2 + \epsilon^{2-2\delta}
\\ \nonumber &\lesssim \epsilon^{2-2\delta}.
\end{align*}
Here the $\mathcal{F}$ term is controlled via Sobolev inequalities, the bootstrap inequalities, and Proposition~\ref{thisisenegege3}.

It remains to estimate $\hat{\underline{\chi}}_{AB}$. For this we first note that by using the $\nabla_4$ equation for $\underline{\beta}_A$ and arguing as we have just done for $\eta_A$ (one uses the weight $v^{2+i}\left(\frac{-u}{v}\right)^{-p}$), one may establish 
\begin{align*}
\sup_{i \in \{0,1\}}\sup_{(u,v) \in \mathcal{P}_{\tilde u,\tilde v}}\int_{\mathbb{S}^2_{u,v}}v^{4+2i}\left(\frac{-u}{v}\right)^{2p}\left|\slashed{\nabla}^i\underline{\beta}\right|^2\mathring{\rm dVol}   &\lesssim \epsilon^{2-2\delta}.
\end{align*}
Then the desired bound for $\hat{\underline{\chi}}_A$ follows from the Codazzi equation~\eqref{tcod1} as well as the elliptic estimate~\eqref{pokdwimnfe22}.

\end{proof}

Next, we turn to the high regularity estimate for the Ricci coefficients.
\begin{proposition}\label{dwlpdwlpdwspacetime}Let $\left(\mathcal{M},g_{\mu\nu}\right)$ satisfy the hypothesis of  Proposition~\ref{tobootornottoboo3}. Then  we have
\[\mathfrak{L}_{\tilde u,\tilde v}  \lesssim \epsilon^{1-\delta}.\]
\end{proposition}
\begin{proof}As is well known, in order to obtain these highest order estimates for the Ricci coefficients we will need to re-write some of the null structure equations in a way which reduces their top-order dependence on curvature. In order to do this we first note the following consequences of the Bianchi equations from Lemma~\ref{Bianchiit} and the commutation formulas from Lemma~\ref{commlemm}:
\begin{align}\label{aqple1}
\slashed{\nabla}_{AB}^2\slashed{\rm div}\left(\Omega\beta\right) &= \Omega\nabla_4\left(\slashed{\nabla}_{AB}^2\check{\rho}\right) + \mathcal{E}_1,
\\ \label{aqple2} \Omega^2 \slashed{\nabla}_{AB}^2\slashed{\rm div}\left(\Omega^{-1}\underline{\beta}\right) &= -\Omega\nabla_3\left(\slashed{\nabla}_{AB}^2\check{\rho}\right) + \mathcal{F}_1,
\\ \label{aqple3} \Omega^2 \slashed{\nabla}_A\slashed{\Delta}\check{\rho} &= \Omega \nabla_3\left(\slashed{\nabla}_A\slashed{\rm div}\left(\Omega\beta\right)\right) + \mathcal{F}_2,
\end{align}
where $\mathcal{E}_1$ is a sum of terms of the following schematic form:
\[\slashed{\nabla}^2\left[\left(\Omega^{-s_1}\psi_1\right)\cdot\left(\left(\Omega^{-s_2}\check{\Psi}_2\right),\slashed{\nabla}\left(\Omega^{-s_2}\psi_2\right),\left(\Omega^{-s_2}\psi_2\right)\cdot\left(\Omega^{-s_3}\psi_3\right)\right)\right],\]\[ \slashed{\nabla}\left(\left(\Omega^{-s_1}\psi_1\right)\cdot\left(\Omega^{-s_2}\psi_2\right)\cdot\left(\Omega^{-s_3}\check{\Psi}_3\right)\right),\qquad \slashed{\nabla}\left(K\cdot\Omega^{-s_1}\check{\Psi}_1\right),\]
and $\mathcal{F}_1$ and $\mathcal{F}_2$ are a sum of terms of the followings schematic form:
\[\slashed{\nabla}^2\left[\Omega^2\left(\Omega^{-s_1}\psi_1\right)\cdot\left(\left(\Omega^{-s_2}\check{\Psi}_2\right),\slashed{\nabla}\left(\Omega^{-s_2}\psi_2\right),\left(\Omega^{-s_2}\psi_2\right)\cdot\left(\Omega^{-s_3}\psi_3\right)\right)\right],\]\[ \slashed{\nabla}\left(\Omega^2\left(\Omega^{-s_1}\psi_1\right)\cdot\left(\Omega^{-s_2}\psi_2\right)\cdot\left(\Omega^{-s_3}\check{\Psi}_3\right)\right),\qquad \slashed{\nabla}\left(\Omega^2 K\cdot\Omega^{-s_1}\check{\Psi}_1\right),\]
where $\check{\Psi}_i$ denotes a renormalized curvature component of signature $s_i$ and $\psi_i$ denotes a Ricci coefficient not equal to $\underline{\omega}$ of signature $s_i$, and in the any term at most of one of the $\psi$ can be equal to an undifferentiated ${\rm tr}\underline{\chi}$ or ${\rm tr}\chi$. The terms in $\mathcal{E}_1$ have the following additional constraints:
\begin{enumerate}
 \item  We have $\psi_i \not\in \{{\rm tr}\underline{\chi},\eta_A\}$.
 \item The only place where ${\rm tr}\chi$ can appear, without an angular derivative acting on it, is as exactly one of the $\psi$'s in a cubic term or multiplying a $\check{\Psi}$.
 \item We can have at most one of $\{\underline{\beta}_A,\hat{\underline{\chi}}_{AB}\}$ in any of the nonlinear expressions.
  \end{enumerate}
 
 Using~\eqref{aqple1}-\eqref{aqple3}, the null structure equations, and Lemma~\ref{commlemm} we may then derive the following:
 \begin{align}\label{forthehighersodrccicekw}
 \Omega\nabla_3\left(\slashed{\nabla}_{ABC}^3\left(\Omega^{-1}{\rm tr}\underline{\chi}\right)^*\right) &= \mathcal{H}_1,
\\ \label{forthehighersodrccicekw2} \Omega\nabla_3\left(\slashed{\nabla}_A\slashed{\Delta}\left(\Omega\omega\right) - \frac{1}{2}\slashed{\nabla}_A\slashed{\rm div}\left(\Omega\beta\right)\right) &= \mathcal{F}_3,
\\ \label{forthehighersodrccicekw3} \Omega\nabla_3\left[\slashed{\nabla}_{AB}^2\slashed{\rm div}\underline{\eta} + \slashed{\nabla}^2_{AB}\check{\rho}\right] &= \mathcal{F}_4,
\\ \label{forthehighersodrccicekw4} \Omega\nabla_4\left(\slashed{\nabla}_{ABC}^3\left(\Omega{\rm tr}\chi\right)^*\right) + \frac{5}{v-u}\left(\slashed{\nabla}_{ABC}^3\left(\Omega{\rm tr}\chi\right)^*\right) &= \mathcal{H}_2,
\\ \label{forthehighersodrccicekw5} \Omega\nabla_4\left(\slashed{\nabla}_{AB}^2\slashed{\rm div}\eta+\slashed{\nabla}_{AB}^2\check{\rho}\right) + 2\Omega{\rm tr}\chi\left(\slashed{\nabla}_{AB}^2\slashed{\rm div}\eta+\slashed{\nabla}_{AB}^2\check{\rho}\right) &= \mathcal{E}_2,
 \end{align}
 where each $\mathcal{F}_i$  contains terms of the same type as in $\mathcal{F}_1$ and $\mathcal{F}_2$. $\mathcal{E}_3$ contains terms of the same type as $\mathcal{E}_1$ except that we also have terms where ${\rm tr}\chi$ multiplies $\underline{\eta}_A$. In $\mathcal{H}_1$ we have terms of the following schematic type:
 \[\slashed{\nabla}^3\left[\Omega^2\left(\Omega^{-s_1}\psi_1\right)^*\cdot\left(\left(v-u\right)^{-2},\left(\Omega^{-s_2}\psi_2\right)^*\right)\right],\qquad \slashed{\nabla}^2\left[\Omega^2\left(\Omega^{-s_1}\psi_1\right)^*\cdot\left(\Omega^{-s_2}\psi_2\right)\cdot\left(\Omega^{-s_3}\psi_3\right)\right],\]
 where $\psi_i$ denotes a Ricci coefficient not equal to $\underline{\omega}$ of signature $s_i$. Finally, in $\mathcal{H}_2$ we have terms of the following schematic type:
  \[\slashed{\nabla}^3\left(\Omega\omega \left(v-u\right)^{-2}\right),\qquad \slashed{\nabla}^3\left[\Omega^2\left(\Omega^{-s_1}\psi_1\right)^*\cdot\left(\Omega^{-s_2}\psi_2\right)^*\right],\qquad \slashed{\nabla}^2\left[\Omega^2\left(\Omega^{-s_1}\psi_1\right)^*\cdot\left(\Omega^{-s_2}\psi_2\right)^*\cdot\left(\Omega^{-s_3}\psi_3\right)\right],\]
  where $\psi_i$ denotes a Ricci coefficient not equal to $\underline{\omega}$ of signature $s_i$.
    
Letting $X \in \{\slashed{\nabla}_{ABC}^3\left(\Omega^{-1}{\rm tr}\underline{\chi}\right)^*,\slashed{\nabla}_A\slashed{\Delta}\left(\Omega\omega\right) - \frac{1}{2}\slashed{\nabla}_A\slashed{\rm div}\left(\Omega\beta\right),\slashed{\nabla}_{AB}^2\slashed{\rm div}\underline{\eta} + \slashed{\nabla}_{AB}^2\check{\rho}\}$, we can write any of the $\nabla_3$ equations~\eqref{forthehighersodrccicekw},~\eqref{forthehighersodrccicekw2},~\eqref{forthehighersodrccicekw3} as
\[\Omega\nabla_3\left(v^{7/2}\left(\frac{-u}{v}\right)^p(-u)^{-2q}X\right) + \frac{p}{-u}v^{7/2-q}\left(\frac{-u}{v}\right)^p(-u)^{-2q}X = v^{7/2-q}\left(\frac{-u}{v}\right)^p(-u)^{-2q}\left(\mathcal{H}_1,\mathcal{F}_3,\mathcal{F}_4\right).\]
 Contracting with  $v^{7/2}\left(\frac{-u}{v}\right)^pX$, integrating by parts, using Proposition~\ref{itstartedoutok2}, using the smallness of $\underline{v}$, using Sobolev inequalities, and appealing to the bootstrap assumption leads to 
\begin{align*}
&\sup_{(u,v)\in\mathcal{P}_{\tilde u,\tilde v}}(-u)^{2q}\int_{{\rm max}\left(-v\underline{v},-\underline{v}^2\right)}^u\int_{-u\underline{v}^{-1}}^v\int_{\mathbb{S}^2}\left|X\right|^2\dot{v}^7\left(\frac{-\dot{u}}{\dot{v}}\right)^{2p}(-\dot{u})^{-1}(-\dot{u})^{-2q}\, d\dot{u}\, d\dot{v}\, \mathring{\rm dVol} \lesssim 
\\ \nonumber &\qquad \sup_{(u,v)\in\mathcal{P}_{\tilde u,\tilde v}}(-u)^{2q}\int_{{\rm max}\left(-v\underline{v},-\underline{v}^2\right)}^u\int_{-u\underline{v}^{-1}}^v\int_{\mathbb{S}^2}\left|\left(\mathcal{H}_1,\mathcal{F}_3,\mathcal{F}_4\right)\right|^2\dot{v}^7\left(\frac{-\dot{u}}{\dot{v}}\right)^{2p}(-\dot{u})(-\dot{u})^{-2q}\, d\dot{u}\, d\dot{v}\, \mathring{\rm dVol} + \epsilon^{2-2\delta}
\\ \nonumber &\qquad \lesssim \epsilon^{2-2\delta}.
\end{align*}
Using Proposition~\ref{thisisenegege3} and Lemma~\ref{someelleststs3} we thus obtain
\begin{align}\label{pldpwlko}
\sup_{(u,v)\in\mathcal{P}_{\tilde u,\tilde v}}(-u)^{2q}\int_{{\rm max}\left(-v\underline{v},-\underline{v}^2\right)}^u\int_{-u\underline{v}^{-1}}^v\int_{\mathbb{S}^2}\left|\slashed{\nabla}^3\left(\Omega{\rm tr}\chi,\Omega\omega,\underline{\eta}\right)\right|^2\dot{v}^7\left(\frac{-\dot{u}}{\dot{v}}\right)^{2p}(-\dot{u})^{-1}(-\dot{u})^{-2q}\, d\dot{u}\, d\dot{v}\, \mathring{\rm dVol} \lesssim \epsilon^{2-2\delta}.
\end{align}

Next, we re-write~\eqref{forthehighersodrccicekw4} and~\eqref{forthehighersodrccicekw5} as  
\begin{align}\label{dawed}
& \Omega\nabla_4\left(v^{7/2}\left(\frac{-u}{v}\right)^p\left(\frac{-u}{v}\right)^{-1/2}(-u)^{-2q}\slashed{\nabla}_{ABC}^3\left(\Omega{\rm tr}\chi\right)^*\right) 
\\ \nonumber &\qquad + \left(\frac{-4+p}{v} + \frac{5}{v-u}\right)\left(v^{7/2}\left(\frac{-u}{v}\right)^p\left(\frac{-u}{v}\right)^{-1/2}(-u)^{-2q}\slashed{\nabla}_{ABC}^3\left(\Omega{\rm tr}\chi\right)^*\right) 
\\ \nonumber &\qquad \qquad =v^{7/2}\left(\frac{-u}{v}\right)^p\left(\frac{-u}{v}\right)^{-1/2}(-u)^{-2q}\mathcal{H}_2,
\\ \label{dawed2}&\Omega\nabla_4\left(v^{7/2}\left(\frac{-u}{v}\right)^p(-u)^{-2q}\left(\slashed{\nabla}_{AB}^2\slashed{\rm div}\eta+\slashed{\nabla}_{AB}^2\check{\rho}\right)\right) 
\\ &\qquad + \left(\frac{-7/2+p}{v}+2\Omega{\rm tr}\chi\right)\left(v^{7/2}\left(\frac{-u}{v}\right)^p(-u)^{-2q}\left(\slashed{\nabla}_{AB}^2\slashed{\rm div}\eta+\slashed{\nabla}_{AB}^2\check{\rho}\right)\right) = 
\\ \nonumber &\qquad \qquad v^{7/2}\left(\frac{-u}{v}\right)^p(-u)^{-2q}\mathcal{E}_2.
\end{align}
The key point is that using the bootstrap assumptions, we have that 
\[\frac{-4+p}{v} + \frac{5}{v-u} \gtrsim v^{-1},\qquad \frac{-7/2+p}{v}+2\Omega{\rm tr}\chi\gtrsim v^{-1}.\]
Thus, contracting~\eqref{dawed} with  $v^{7/2}\left(\frac{-u}{v}\right)^p(-u)^{-2q}\left(\frac{-u}{v}\right)^{-1/2}\slashed{\nabla}_{ABC}^3\left(\Omega{\rm tr}\chi\right)^*$ and~\eqref{dawed2} with \\ $v^{7/2}\left(\frac{-u}{v}\right)^p(-u)^{-2q}\left(\slashed{\nabla}^2_{AB}\slashed{\rm div}\eta+\slashed{\nabla}^2_{AB}\check{\rho}\right)$, integrating by parts, using Proposition~\ref{itstartedoutok2}, using~\eqref{pldpwlko}, using Sobolev inequalities, and appealing to the bootstrap assumption leads to 
\begin{align}\label{pldpwlko}
\sup_{(u,v)\in\mathcal{P}_{\tilde u,\tilde v}}(-u)^{2q}\int_{{\rm max}\left(-v\underline{v},-\underline{v}^2\right)}^u\int_{-u\underline{v}^{-1}}^v\int_{\mathbb{S}^2}&\left[\left(\frac{-\dot{u}}{\dot{v}}\right)^{-1}\left|\slashed{\nabla}^3\Omega{\rm tr}\chi\right|^2 + \left|\slashed{\nabla}^2\slashed{\rm div}\eta+\slashed{\nabla}^2\check{\rho}\right|^2\right]
\\ \nonumber &\qquad \times\dot{v}^6\left(\frac{-\dot{u}}{\dot{v}}\right)^{2p}(-\dot{u})^{-2q}\, d\dot{u}\, d\dot{v}\, \mathring{\rm dVol} \lesssim \epsilon^{2-2\delta}.
\end{align}
Using Proposition~\ref{thisisenegege3} we thus obtain
\begin{align}\label{pldpwlko}
\sup_{(u,v)\in\mathcal{P}_{\tilde u,\tilde v}}(-u)^{2q}\int_{{\rm max}\left(-v\underline{v},-\underline{v}^2\right)}^u\int_{-u\underline{v}^{-1}}^v\int_{\mathbb{S}^2}&\left[\left(\frac{-\dot{u}}{\dot{v}}\right)^{-1}\left|\slashed{\nabla}^3\Omega{\rm tr}\chi\right|^2 + \left|\slashed{\nabla}^2\slashed{\rm div}\eta\right|^2\right]
\\ \nonumber &\qquad \times\dot{v}^6\left(\frac{-\dot{u}}{\dot{v}}\right)^{2p}(-\dot{u})^{-2q}\, d\dot{u}\, d\dot{v}\, \mathring{\rm dVol} \lesssim \epsilon^{2-2\delta}.
\end{align}

It remains to estimate $\hat{\chi}_{AB}$ and $\hat{\underline{\chi}}_{AB}$. However, the desired estimates for these follow from the already established Ricci coefficient estimates, the Codazzi equations~\eqref{tcod1} and~\eqref{tcod2}, Proposition~\ref{thisisenegege3}, and the elliptic estimates from Lemma~\ref{someelleststs3}.
\end{proof}

\subsubsection{Estimates for the metric coefficients}
Lastly, we come to the estimates for the metric coefficients.
\begin{proposition}Let $\left(\mathcal{M},g_{\mu\nu}\right)$ satisfy the hypothesis of  Proposition~\ref{tobootornottoboo3}. Then  we have
\[\mathfrak{M}_{\tilde u,\tilde v}  \lesssim \epsilon^{1-\delta}.\]
\end{proposition}
\begin{proof}We first observe that by a mild adaption of the proof of Proposition~\ref{dwlpdwlpdwspacetime} one may establish that
\begin{equation}\label{comesfromamildadaption}
\sup_{0 \leq j \leq 3}\sup_{(u,v)\in\mathcal{P}_{\tilde u,\tilde v}}(-u)^{2q}\int_{{\rm max}\left(-v\underline{v},-\underline{v}^2\right)}^u\int_{\mathbb{S}^2}\left|\slashed{\nabla}^j\left(\Omega^{-1}\chi^*,\zeta\right)\right|^2v^{1+2j}\left(\frac{-\dot{u}}{v}\right)^{2p}(-\dot{u})^{-2q}\, d\dot{u}\, \mathring{\rm dVol} \lesssim \epsilon^{2-2\delta}.
\end{equation}
Given~\eqref{comesfromamildadaption}, it is straightforward to use Lemma~\ref{intlemma}, integrate the equations 
\[\mathcal{L}_{e_3}\slashed{g}_{AB} = 2\Omega \underline{\chi}_{AB},\qquad \mathcal{L}_{\partial_u}\mathring{b}^A = 4\Omega^2\zeta^A\]
in the $e_3$ direction, and obtain the desired bounds for $\slashed{g}_{AB}$ and $\mathring{b}^A$. 

However, for the lapse $\Omega$ we must integrate in the $\nabla_4$ direction to obtain an estimate for $\Omega$ from $\omega$:
\[\partial_v\log\Omega + \left(b\cdot\slashed{\nabla}\right)\log\Omega = -2\Omega\omega.\]
The desired estimates follows immediately by integrating along integral curves of $\partial_v + b^A\slashed{\nabla}_A$.
\end{proof}

This concludes the proof of Proposition~\ref{tobootornottoboo3}, and hence also Theorem~\ref{itisreg3}.

\section{Incompleteness of Future Null Infinity, the $\left(u,\dot{v},\theta^A\right)$ Coordinates, and the Hawking Mass}\label{incompsec}
In this section we will use Theorem~\ref{itisreg3} to complete the proof of our main result Theorem~\ref{themainextresult}.

Let $(\mathcal{M},g_{\mu\nu})$ be the spacetime produced by Theorem~\ref{itisreg3}. We will start by showing the that the hypersurface $\{u = -\underline{v}^2\}$ is an asymptotically flat null hypersurface.
\begin{lemma}The hypersurface $\{u = -\underline{v}^2\}$ is asymptotically flat in the sense that there exists a function $\tilde v : \{u = -\underline{v}^2\} \to \mathbb{R}$ with the following properties:
\begin{enumerate}
	\item The hypersurface $\{u = -\underline{v}^2\}$ is diffeomorphic to $\{(\tilde v,\theta^A) \in [0,\infty) \times \mathbb{S}^2\}$.
	\item Let $\mathcal{S}_{\tilde v}$ denote a surface of constant $\tilde v$. Then, when $\tilde v$ is sufficiently large, we have that the induced metric on $\mathcal{S}_{\tilde v}$ is $\tilde v^2\mathring{\slashed{g}}_{AB}$, where $\mathring{\slashed{g}}_{AB}$ denotes a Lie-propagated round metric on $\mathbb{S}^2$. 
	
	\end{enumerate}
\end{lemma}
\begin{proof}
 It is a consequence of Proposition~\ref{kdowvappend} that when $v \gg 1$ and $u = -\underline{v}^2$ we have that 
\[\Omega|_{\{v \gg 1\} \cap \{u=-\underline{v}^2\}} = 1,\qquad b^A|_{\{v \gg 1\} \cap \{u=-\underline{v}^2\}} = 0,\]
and  that $\slashed{g}_{AB}|_{\{v \gg 1\} \cap \{u=-\underline{v}^2\}} = \varphi^2\mathring{\slashed{g}}_{AB},$ where $\mathring{\slashed{g}}_{AB}$ is a Lie-propagated round metric and there exists a constant $C$, independent of $\epsilon$, so that
\[e^{-C\epsilon^{1-\delta}}\left(v+\underline{v}\right) \leq \varphi_{\{v \gg 1\} \cap \{u=-\underline{v}^2\}} \leq e^{C\epsilon^{1-\delta}}\left(v+\underline{v}\right).\]
For $v \gg 1$, the above facts imply that the $e_4$-Raychaudhuri equation becomes
\begin{equation}\label{ldwlpwdlpdwlp}
\partial_v{\rm tr}\chi + \frac{1}{2}\left({\rm tr}\chi\right)^2 = 0.
\end{equation}
Since Proposition~\ref{kdowvappend} also implied that
\[\left|{\rm tr}\chi - \frac{2}{v+\underline{v}}\right||_{\{v \gg 1\} \cap \{u=-\underline{v}^2\}} \lesssim \frac{\epsilon^{1-\delta}}{v},\]
one can easily solve~\eqref{ldwlpwdlpdwlp} to obtain that for large $v$, ${\rm tr}\chi = \frac{2}{v+Q(\theta)}$ where $\sup_{\theta}\left|Q(\theta)-\underline{v}\right| \lesssim \epsilon^{1-\delta}$. For $v \gg 1$, we have that $\partial_v\log\varphi = \frac{1}{2}{\rm tr}\chi$. Thus, for some  $H(\theta)$ with $\sup_{\theta}\left|H(\theta) - 1\right| \lesssim \epsilon^{1-\delta}$ we have that  
\[\slashed{g}_{AB} = \left(v+Q\right)^2H(\theta)\mathring{\slashed{g}}_{AB}.\]
Now it suffices to define the function $\tilde{v} :  \{u = -\underline{v}^2\} \cap \{v \gg 1\} \to \mathbb{R}$ by
\[\tilde v\left(v,\theta^A\right) \doteq \left(v+Q\right)\sqrt{H(\theta)}.\]
\end{proof}

Next, we truncate $\left(\mathcal{M},g_{\mu\nu}\right)$ to the region $\{\dot{v} \geq 0\} \cap \{-\underline{v}^2 \leq u < 0\}$ and thus obtain a globally hyperbolic spacetime. For convenience, we continue to refer to the truncated spacetime as $\left(\mathcal{M},g_{\mu\nu}\right)$. We now check that our spacetime contains a naked singularity.
\begin{lemma}The spacetime $\left(\mathcal{M},g_{\mu\nu}\right)$ contains a naked singularity in the sense of Definition~\ref{nakeddef}.
\end{lemma}
\begin{proof}We take $\{u = -\underline{v}^2\}$ as our asymptotically flat hypersurface and $\partial_v$ as our geodesic normal $L'$. Then $\underline{L'} = \partial_u$ and, since $\partial_u + b^A\slashed{\nabla}_A$ is geodesic, we immediately see that every geodesic starting on $\{u = -\underline{v}^2\}$ with initial tangent vector $\partial_u$ leaves the spaces after affine time $\underline{v}^2$. 
\end{proof}

Next, we observe that given Theorem~\ref{itisreg3}, a straightforward argument, in the spirit of Lemma~\ref{shifttheshift},  allows us to define global coordinates $\left(u,v,\theta\right)$ in the region $\{v > 0\} \cap \{u < 0\}$ with the shift in the $e_3$ direction. We then define global $\left(u,\hat{v},\theta\right)$ coordinates by setting $\hat{v} = v^{1-2\kappa}$ (keep in mind Definition~\ref{kapselfseim}). The regularity statements for $g_{\mu\nu}$ in Theorem~\ref{themainextresult} now follow easily by using the established estimates and arguing as in the proof of Lemma~\ref{dkwowsss322}, we omit the details. Lastly, we need to compute the Hawking mass of the spheres $\mathbb{S}^2_{u,0}$ along $\{\hat{v} = 0\}$.

\begin{lemma}\label{hawkcalccalc}Recall that the Hawking mass of a sphere $\mathbb{S}^2_{u,\dot{v}}$ is defined by
\[m\left(\mathbb{S}^2_{u,\dot{v}}\right) \doteq \frac{r}{8\pi}\int_{\mathbb{S}^2_{u,\dot{v}}}\left(-\check{\rho}\right)\, \slashed{dVol},\]
where
\[r \doteq \sqrt{\frac{{\rm Area}\left(\mathbb{S}^2_{u,\dot{v}}\right)}{4\pi}},\qquad \check{\rho} \doteq \rho - \frac{1}{2}\hat{\chi}\cdot\underline{\hat{\chi}}.\]
It will be useful to keep in mind the fact, which  follows easily from~\eqref{genGauss}, that $\check{\rho}$ is invariant under the change of coordinates $\left(u,v,\theta^A\right) \mapsto \left(u,\hat{v},\theta^A\right)$.

We have that
\begin{enumerate}
	\item For any self-similar solution with $\kappa = 0$, then $m\left(\mathbb{S}^2_{u,0}\right) = 0$.
	\item For the spacetime $\left(\mathcal{M},g_{\mu\nu}\right)$ produced by Theorem~\ref{itisreg3} we will have $m\left(\mathbb{S}^2_{u,0}\right) \sim \epsilon^2|u| > 0$.
\end{enumerate}
\end{lemma}
\begin{proof} It follows from Proposition~\ref{formalcalculationsscaleinv} that when $\kappa = 0$ we have that  both $\rho$ and $\hat{\underline{\chi}}_{AB}$ vanish along $\{\hat{v} = 0\}$. This clearly implies that $m$ vanishes when $\{\hat{v}=0\}$. 

Next we consider the case when $\kappa > 0$. Keeping in mind that $\check{\rho}$ is invariant under the change of coordinates $\left(u,v,\theta^A\right) \mapsto \left(u,\hat{v},\theta^A\right)$, we will work in the $\left(u,v,\theta^A\right)$ coordinates for $v > 0$ and then take the limit as $v \to 0$. We recall the propagation equations for $\check{\rho}$:
\begin{align}\label{nabla3checkrho}
\nabla_3\check{\rho} + \frac{3}{2}{\rm tr}\underline{\chi} \check{\rho} &= -\slashed{\rm div}\underline{\beta} + \zeta \cdot \underline{\beta} - 2\eta\cdot\underline{\beta} - \frac{1}{2}\hat{\underline{\chi}}\cdot\slashed{\nabla}\hat{\otimes}\eta - \frac{1}{2}\hat{\underline{\chi}}\cdot\left(\eta\hat{\otimes}\eta\right) + \frac{1}{4}{\rm tr}\chi\left|\hat{\underline{\chi}}\right|^2,
\end{align}

We know from Lemma~\ref{dkwowsss322} that $\check{\rho}$ has a regular limit as $v\to 0$ and hence that $m\left(\mathbb{S}^2_{u,v}\right)$ will also have a regular limit as $v\to 0$. Next, we multiply~\eqref{nabla3checkrho} by $\Omega$ and take the limit as $v\to 0$. When $v = 0$, self-similarity implies that $\partial_u\check{\rho} = -u^{-1}\check{\rho}$. Thus using also the identities from Lemma~\ref{scalrelations2}  we obtain in the $v\to 0$ limit that 
\begin{align}\label{nabla3checkrho2}
&\mathcal{L}_b\check{\rho} -2u^{-1}\check{\rho}+ \frac{3}{2}\left(\frac{2}{u} + \slashed{\rm div}b\right) \check{\rho} = 
\\ \nonumber &\qquad -\Omega\slashed{\rm div}\underline{\beta} + \zeta \cdot\left( \Omega\underline{\beta}\right) - 2\eta\cdot\left(\Omega\underline{\beta}\right) - \frac{1}{2}\Omega\hat{\underline{\chi}}\cdot\slashed{\nabla}\hat{\otimes}\eta - \frac{1}{2}\Omega\hat{\underline{\chi}}\cdot\left(\eta\hat{\otimes}\eta\right) + \frac{1}{16}\Omega^{-1}{\rm tr}\chi\left|\slashed{\nabla}\hat{\otimes}b\right|^2.
\end{align}
The left hand side of~\eqref{nabla3checkrho2} simplifies to $\left(u^{-1}+\frac{3}{2}{\rm div}b\right)\check{\rho} + \mathcal{L}_b\check{\rho}$, 
On the other hand, using the Codazzi equation, the right hand side of~\eqref{nabla3checkrho2} is equal to
\begin{align*}
&-\slashed{\nabla}^A\left(\Omega\underline{\beta}_A\right) - \eta^A\left(\Omega\underline{\beta}\right)_A -\frac{1}{2}\hat{\underline{\chi}}_{AB}\left(\slashed{\nabla}\hat{\otimes}\eta\right)^{AB} -\frac{1}{2}\hat{\underline{\chi}}_{AB}\left(\eta\hat{\otimes}\eta\right)^{AB}+ \frac{1}{16}\Omega^{-1}{\rm tr}\chi\left|\slashed{\nabla}\hat{\otimes}b\right|^2 =
\\ \nonumber &\qquad -\slashed{\nabla}^A\left(\Omega\underline{\beta}_A\right) -\eta^A\left(\Omega\nabla^B\hat{\underline{\chi}}_{BA} - \frac{1}{2}\Omega\slashed{\nabla}_A{\rm tr}\underline{\chi} - \zeta^B\left(\Omega\hat{\underline{\chi}}_{AB}\right) + \frac{1}{2}\zeta_A\Omega{\rm tr}\underline{\chi}\right)
\\ \nonumber &\qquad -\frac{1}{2}\Omega\hat{\underline{\chi}}_{AB}\left(\slashed{\nabla}\hat{\otimes}\eta\right)^{AB} -\frac{1}{2}\Omega\hat{\underline{\chi}}_{AB}\left(\eta\hat{\otimes}\eta\right)^{AB}+ \frac{1}{16}\Omega^{-1}{\rm tr}\chi\left|\slashed{\nabla}\hat{\otimes}b\right|^2 =
\\ \nonumber &\qquad -\slashed{\nabla}^A\left(\Omega\underline{\beta}_A\right) - \slashed{\nabla}^B\left(\eta^A\left(\Omega\hat{\underline{\chi}}\right)_{AB}\right)  + \frac{1}{2}\eta^A\slashed{\nabla}_A\left(\Omega{\rm tr}\underline{\chi}\right) - \frac{1}{2}\left|\eta\right|^2\Omega{\rm tr}\underline{\chi}+ \frac{1}{16}\Omega^{-1}{\rm tr}\chi\left|\slashed{\nabla}\hat{\otimes}b\right|^2.
\end{align*} 

Thus, integrating~\eqref{nabla3checkrho2} over $\mathbb{S}^2_{u,0}$ and applying the divergence theorem leads to
\begin{align}\label{hawkmassform}
&u^{-1}\int_{\mathbb{S}^2_{u,0}}\left[\left(1+\frac{u}{2}\slashed{\rm div}b\right)\check{\rho}\right]\slashed{dVol} = 
\\ \nonumber &\qquad \int_{\mathbb{S}^2_{u,0}}\left[\frac{1}{2}\eta^A\slashed{\nabla}_A\slashed{\rm div}b - \frac{1}{2}\left|\eta\right|^2\left(\frac{2}{u} +\slashed{\rm div}b\right) + \frac{1}{16}\Omega^{-1}{\rm tr}\chi\left|\slashed{\nabla}\hat{\otimes} b\right|^2\right]\slashed{dVol}.
\end{align}

Now, combining~\eqref{hawkmassform} with the characteristic initial data estimates from Section~\ref{setitupayya} leads to
\[\int_{\mathbb{S}^2_{u,0}}\left(-\check{\rho}\right)\slashed{dVol} = (-u)\left(O\left(\epsilon^3\right) + \int_{\mathbb{S}^2_{u,0}}\left[\left|\eta\right|^2 + \frac{K}{16}\left|\slashed{\nabla}\hat{\otimes}b\right|^2\right]\slashed{dVol}\right) \sim (-u)\epsilon^2.\]

\end{proof}

\appendix
\section{Examples of $\left(\epsilon,\gamma,\delta,N_0,M_0,M_1\right)$-Regular Data}\label{constructthetupleregul}
In this section we construct examples of $\left(\epsilon,\gamma,\delta,N_0,M_0,M_1\right)$-regular $4$-tuples $\left(\slashed{g}_{AB},b^A,\kappa,\Omega\right)$ in the sense of Definition~\ref{Mreg}.
\begin{proposition}\label{solvethedivergenceequationforb}Let $0 < \epsilon \ll \gamma \ll 1$, $\check{b}^A$ be a choice of seed data in the sense of Definition~\ref{areallynicedefinitionofseeddata}, and $\left(N_0,M_0,M_1\right) \in \left(\mathbb{Z}_{> 0}\right)^3$ satisfy $N_0 \gg 1$, $M_0 \gg N_0$, and $M_1 \gg N_0$. Then there exists $\delta > 0$ satisfying $\gamma \ll \delta \ll 1$, a vector field $b^A = \check{b}^A + \mathring{\nabla}^Af$, and $\kappa > 0$ solving~\eqref{rayconstr} with $\slashed{g}_{AB} = \mathring{\slashed{g}}_{AB}$ and $\Omega = 1$ such that $\left(\mathring{\slashed{g}}_{AB},b^A,\kappa,1\right)$ is a $\left(\epsilon,\gamma,\delta,N_0,M_0,M_1\right)$-regular $4$-tuple in the sense of Definition~\ref{Mreg}.
\end{proposition}
\begin{proof}We will look for a solution $b$ of the form
\[b^A = \check{b}^A + \mathring{\nabla}^Af,\]
and construct our solution by an iteration procedure. We  define sequences $\{D_i\}_{i=0}^{\infty}$, $\{\kappa_i\}_{i=0}^{\infty}$, and $\{f_i\}_{i=0}^{\infty}$ of functions $D_i$, constants $\kappa_i$, and functions $f_i$ as follows. First of all, we set $D_0 = 0$, $\kappa_0 = 0$, and $f_0 = 0$ and also define
\[b_i^A \doteq  \check{b}^A + \mathring{\nabla}^Af_i,\qquad \forall i \geq 0.\]
Now we will explain our inductive construction. Thus we assume that $\left(D_{i-1},\kappa_{i-1},f_{i-1}\right)$ have been constructed. For every constant $\tilde \kappa$, we use Proposition~\ref{qualitativeexistence} to define a function $\tilde D_i\left(\theta^A,\tilde \kappa\right)$ by requiring that
\begin{equation}\label{eqndefDi}
\tilde D_i + \mathcal{L}_{b_{i-1}}\tilde D_i = \frac{1}{2}\left(D_{i-1}\right)^2 + \frac{1}{4}\left|\mathring{\nabla}\hat{\otimes}b_{i-1}\right|_{\mathring{\slashed{g}}}^2 - 4\tilde\kappa + 2\kappa_{i-1}D_{i-1}.
\end{equation}
We then choose $\kappa_i$ by requiring that 
\begin{equation}\label{kappaeqn}
\int_{\mathbb{S}^2}\tilde D_i\left(\theta^A, \kappa_i\right)\mathring{{\rm dVol}} = 0,
\end{equation}
and then set $D_i\left(\theta^A\right) \doteq \tilde D_i\left(\theta^A,\kappa_i\right)$. Finally, we require that $f_i$ satisfy
\begin{equation}\label{feqn}
\mathring{\Delta}f_i  = D_i,\qquad \int_{\mathbb{S}^2}f_i\, \mathring{{\rm dVol}} = 0.
\end{equation}
Note that the condition~\eqref{kappaeqn} is a necessary and sufficient orthogonality condition to solve~\eqref{feqn}.

Now we will show by induction on $i$ that the sequences $\{D_i\}_{i=1}^{\infty}$, $\{\kappa_i\}_{i=1}^{\infty}$ and $\{f_i\}_{i=1}^{\infty}$ are well defined and satisfy (for a suitably large $A > 0$ and suitably small $0 < \tilde\delta \ll 1$) the estimates:
\begin{equation}\label{thisishteestforthesequenceofstuff}
\left\vert\left\vert D_i\right\vert\right\vert_{\mathring{H}^{N_0}} \leq A\epsilon^{2-2\tilde\delta},\qquad \left\vert\left\vert f_i\right\vert\right\vert_{\mathring{H}^{N_0+2}} \leq A\epsilon^{2-3\tilde\delta},\qquad |\kappa_i| \leq A\epsilon^{2-\tilde\delta},
\end{equation}
We start with the base case $i = 1$. The equation for $\tilde{D}_1$ may be written (keeping Remark~\ref{comp} in mind) as 
\begin{equation}\label{ford1}
\tilde D_1 +4\tilde \kappa + \mathcal{L}_{\check{b}}\tilde D_1= \frac{1}{4}\left|\mathring{\nabla}\hat{\otimes}\check{b}\right|_{\mathring{\slashed{g}}}^2 = \frac{1}{4}\epsilon^2a^2(\theta) + E,
\end{equation}
where 
\[E \doteq \frac{\epsilon}{2}\mathring{\slashed{g}}\left(\mathring{\nabla}\hat{\otimes}\tilde b,z\right) + \frac{1}{4}\left|z\right|^2_{\mathring{\slashed{g}}}\]
satisfies
\[\left\vert\left\vert E\right\vert\right\vert_{\mathring{H}^{30N}} \lesssim \epsilon^{100}.\]

We allow $\tilde\kappa$ to be an arbitrary constant satisfying  the bound 
\begin{equation}\label{boundtobjqjqmqmq}
\left|\tilde\kappa \right| \leq A\epsilon^{2-\tilde\delta}.
\end{equation}
Then, if $\epsilon$ is sufficiently small and $A$ is sufficiently large, we can apply Proposition~\ref{qualitativeexistence} to obtain a unique $\tilde D_1$ solving~\eqref{ford1} and also satisfying the bound
\begin{equation}\label{hmmthisreallyisnice}
\left\vert\left\vert \tilde D_1\right\vert\right\vert_{\mathring{H}^{N_0}} \leq A\epsilon^{2-2\tilde\delta}.
\end{equation}
Furthermore, re-writing the equation for $\tilde D_1$ as 
\begin{equation}\label{ford122}
\left(\tilde D_1 +4\tilde \kappa - \frac{1}{4}\left|\mathring{\nabla}\hat{\otimes}\check{b}\right|^2\right) +  \mathcal{L}_{\check{b}}\left(\tilde D_1 +4\tilde\kappa- \frac{1}{4}\left|\mathring{\nabla}\hat{\otimes}\check{b}\right|_{\mathring{\slashed{g}}}^2\right)= -\frac{1}{4}\mathcal{L}_{\check{b}}\left(\left|\mathring{\nabla}\hat{\otimes}\check{b}\right|_{\mathring{\slashed{g}}}^2\right),
\end{equation}
we can appeal again to Proposition~\ref{qualitativeexistence} and Sobolev inequalities to show that 
\begin{equation}\label{okmnbvcdsertyuijjuj}
\left\vert\left\vert \tilde D_1 - \frac{1}{4}\left|\mathring{\nabla}\hat{\otimes}\check{b}\right|_{\mathring{\slashed{g}}}^2 + 4\tilde\kappa \right\vert\right\vert_{L^{\infty}} \lesssim \epsilon^3,
\end{equation}
and also that
\begin{equation}\label{uniquekknjkasdfavav}
\tilde{D}_1\left(\theta^A,\tilde\kappa^{(1)}\right) - \tilde{D}_1\left(\theta^A,\tilde \kappa^{(2)}\right) = \tilde \kappa^{(1)}-\tilde \kappa^{(2)}.
\end{equation}
Next, we will show that we can pick $\kappa_1$ satisfying the bound from~\eqref{thisishteestforthesequenceofstuff} so that 
\begin{equation}\label{forkappa1}
\int_{\mathbb{S}^2}\tilde D_1\left( \theta^A,\kappa_1\right)\mathring{{\rm dVol}} = 0.
\end{equation}
To see this, it suffices to note that we can use~\eqref{uniquekknjkasdfavav} to write
\begin{align}\label{knbvcdftyuioqqqqq}
\int_{\mathbb{S}^2}\tilde D_1\left(\theta^A,\tilde \kappa\right)\mathring{{\rm dVol}} = \frac{\tilde \kappa}{4\pi} +\int_{\mathbb{S}^2}\tilde D_1\left(\theta^A,0\right)\mathring{{\rm dVol}},
\end{align}
and then note that~\eqref{okmnbvcdsertyuijjuj} implies that
\[\left|\int_{\mathbb{S}^2}D_1\left(\theta,0\right)\mathring{\rm dVol}\right| \lesssim \epsilon^2,\]
Having picked $\kappa_1$ we then set $D_1\left(\theta^A\right) \doteq \tilde D_1\left(\theta^A,\kappa_1\right)$. It follows then that $D_1$ satisfies the estimate from~\eqref{thisishteestforthesequenceofstuff}. Finally, it follows immediately from elliptic theory that we can uniquely define $f_1$ by solving
\[\mathring{\Delta}f_1 = D_1 ,\qquad \int_{\mathbb{S}^2}f_1\mathring{\rm dVol} = 0,\]
and that this $f_1$ will satisfy the bound from~\eqref{thisishteestforthesequenceofstuff}.

Having established the base case, we turn to the inductive step. Thus, we assume that $\{D_i\}_{i=1}^{j-1}$, $\{f_i\}_{i=1}^{j-1}$, and $\{\kappa_i\}_{i=1}^{j-1}$ have been constructed and satisfy~\eqref{thisishteestforthesequenceofstuff}. The equation~\eqref{eqndefDi} (with $i = j$) which defines $\tilde D_j$ in terms of the constant $\tilde \kappa$ may be written as
\begin{equation}\label{eqndefDi222}
\tilde D_j + \mathcal{L}_{b_{j-1}}\tilde D_j = \frac{1}{2}\left(D_{j-1}\right)^2 + \frac{1}{4}\left|\mathring{\nabla}\hat{\otimes}b_{j-1}\right|_{\mathring{\slashed{g}}}^2 - 4\tilde\kappa + 2\kappa_{j-1}D_{j-1}.
\end{equation}
By the inductive hypothesis, we have 
\[\left\vert\left\vert b_{j-1}\right\vert\right\vert_{\mathring{H}^{N_0+1}} \lesssim \left\vert\left\vert \check{b}\right\vert\right\vert_{\mathring{H}^{N_0+1}} + \left\vert\left\vert f_{j-1}\right\vert\right\vert_{\mathring{H}^{N_0+2}} \lesssim \epsilon^{1-\tilde\delta}.\]
Thus, we can apply apply Proposition~\ref{qualitativeexistence}, the induction hypothesis, Lemma~\ref{interlemm}, and take $\epsilon^{\tilde\delta}A \ll 1$ to obtain, for any $\tilde\kappa$ satisfying~\eqref{boundtobjqjqmqmq}, the following:
\begin{align}\label{atthisstepweobtainedthisestiamtefordj}
\left\vert\left\vert \tilde D_j\right\vert\right\vert_{\mathring{H}^{N_0}} &\lesssim \left\vert\left\vert \left(D_{j-1}\right)^2\right\vert\right\vert_{\mathring{H}^{N_0}} + \left\vert\left\vert \left|\mathring{\nabla}\check{b}\right|^2\right\vert\right\vert_{\mathring{H}^{N_0}} +  \left\vert\left\vert \left|\mathring{\nabla}^2f\right|^2\right\vert\right\vert_{\mathring{H}^{N_0}} + |\tilde \kappa| +
 |\kappa_{j-1}|\left\vert\left\vert D_{j-1}\right\vert\right\vert_{\mathring{H}^{N_0}}
  \\ \nonumber &\lesssim A^2\epsilon^{4-4\tilde\delta} + \epsilon^{2-\tilde\delta} + A^2\epsilon^{4-3\tilde\delta} + 
 A^2\epsilon^{4-3\tilde\delta}
  \\ \nonumber &\leq  A\epsilon^{2-2\tilde\delta}.
\end{align}
In particular, for any choice of $\tilde\kappa$ satisfying the bound~\eqref{boundtobjqjqmqmq}, then $\tilde D_j$ satisfies the desired bound from~\eqref{thisishteestforthesequenceofstuff}. Now we can re-write~\eqref{eqndefDi222} as 
\begin{align}\label{eqndefDi2222222}
&\left(\tilde D_j -\frac{1}{4}\left|\mathring{\nabla}\hat{\otimes}b_{j-1}\right|_{\mathring{\slashed{g}}}^2+4\tilde \kappa\right) + \mathcal{L}_{b_{j-1}}\left(\tilde D_j-\frac{1}{4}\left|\mathring{\nabla}\hat{\otimes}b_{j-1}\right|_{\mathring{\slashed{g}}}^2+4\tilde\kappa\right) =
\\ \nonumber &\qquad  -\frac{1}{4}\mathcal{L}_{b_{j-1}}\left(\left|\mathring{\nabla}\hat{\otimes}b_{j-1}\right|_{\mathring{\slashed{g}}}^2\right) +\frac{1}{2}\left(D_{j-1}\right)^2 + 2\kappa_{j-1}D_{j-1}.
\end{align}
Applying Proposition~\ref{qualitativeexistence} and a Sobolev inequality and arguing as in~\eqref{atthisstepweobtainedthisestiamtefordj} and using $\epsilon^{\tilde\delta}A \ll 1$ leads to
\begin{equation}\label{9olkmnbgfr42567w8s9o}
\left\vert\left\vert \tilde D_i - \frac{1}{4}\left|\mathring{\nabla}\hat{\otimes}b_{j-1}\right|_{\mathring{\slashed{g}}}^2 + 4\tilde \kappa\right\vert\right\vert_{L^{\infty}} \lesssim \left\vert\left\vert \tilde D_i - \frac{1}{4}\left|\mathring{\nabla}\hat{\otimes}b_{j-1}\right|_{\mathring{\slashed{g}}}^2 + 4\tilde \kappa\right\vert\right\vert_{\mathring{H}^2} \lesssim \epsilon^{3-\tilde\delta}.
\end{equation}
Since the right had side of~\eqref{eqndefDi2222222} does not depend on $\tilde\kappa$, we also have 
\begin{equation}\label{kappadldldkdkdk}
\tilde D_j\left(\theta^A,\tilde\kappa^{(1)}\right) - \tilde D_j\left(\theta^A,\tilde\kappa^{(2)}\right) = \tilde\kappa^{(1)} - \tilde\kappa^{(2)}.
\end{equation}
Thus, we have
\begin{equation}\label{tobepiajsnsjsk}
\int_{\mathbb{S}^2}\tilde D_j\left(\theta^A,\tilde\kappa\right) \mathring{{\rm dVol}} = \frac{\tilde\kappa}{4\pi} + \int_{\mathbb{S}^2}\tilde D_j\left(\theta^A,0\right)\mathring{\rm dVol}.
\end{equation}
From the bound~\eqref{9olkmnbgfr42567w8s9o} we obtain
\[\left|\int_{\mathbb{S}^2}\tilde D_j\left(\theta^A,0\right)\mathring{\rm dVol}\right| \lesssim \epsilon^2.\]
In particular, we can pick $\tilde\kappa$ so that the left hand side of~\eqref{tobepiajsnsjsk} vanishes and so that the bound~\eqref{boundtobjqjqmqmq} is satisfied. We set this choice of $\tilde\kappa$ to be $\kappa_j$ and then set $D_j \doteq \tilde D_j\left(\theta^A,\kappa_j\right)$. Finally, elliptic theory and the induction hypothesis uniquely define $f_j$ solving
\[\mathring{\Delta}f_j = D_j,\qquad \int_{\mathbb{S}^2}f_j\mathring{\rm dVol} = 0,\]
and satisfying the desired bounds from~\eqref{thisishteestforthesequenceofstuff}. This completes the induction step and thus~\eqref{thisishteestforthesequenceofstuff} holds for all $j \geq 0$. 

Arguing in a similar fashion with equations derived for the difference of the quantities, one may show that the sequences $\{D_i\}$, $\{f_i\}$, and $\{\kappa_i\}$ are Cauchy and converge to $D_{\infty}$, $f_{\infty}$, and $\kappa_{\infty}$ all satisfying the bound~\eqref{thisishteestforthesequenceofstuff}. Finally, one sets $b^A = \check{b}^A + \mathring{\nabla}^Af_{\infty}$ and $\kappa = \kappa_{\infty}$.  We will then have that $D_{\infty} = \slashed{\rm div}b$ and it follows that $b$ and $\kappa$ solve the equation~\eqref{rayconstr}. By passing to the limit in the bounds~\eqref{thisishteestforthesequenceofstuff}, we have
\begin{equation}\label{finalfbodjkslslsl}
\left\vert\left\vert f_{\infty}\right\vert\right\vert_{\mathring{H}^{N_0+2}} \lesssim \epsilon^{2-3\tilde\delta},\qquad \left|\kappa_{\infty}\right| \lesssim \epsilon^{2-\tilde\delta}.
\end{equation}
Thus, after defining $\delta \doteq 3\tilde\delta$, we have a $4$-tuple $\left(\mathring{\slashed{g}}_{AB},b^A,\kappa,1\right)$ for which we have verified all of the conditions of being an $\left(\epsilon,\gamma,\delta,N_0,M_0,M_1\right)$-regular $4$-tuple except for~\eqref{somebound3ssssss}. However,~\eqref{somebound3ssssss} is easily proven by commuting the relevant equations with $\mathcal{L}_{\partial_{\phi}}$ and using that $a(\theta)$ is axisymmetric and~\eqref{theotherrequirementsbutheyareimportant}. We omit the details.
\end{proof}

\section{Useful Tensorial Identities}
We start with two useful preliminary lemmas.
\begin{lemma}\label{identitytwotrfere}Let $\slashed{g}_{AB}$ be a Riemannian metric on $\mathbb{S}^2$, and let $\mu_{AB}$ and $\nu_{AB}$ be trace-free symmetric tensors. Then
\begin{equation}\label{toprovelemmaaa}
\frac{1}{2}\left(\slashed{\nabla}\left(\mu\cdot\nu\right) + {}^*\slashed{\nabla}\left(\mu\wedge \nu\right)\right) = \mu \slashed{\rm div}\nu + \left(\nu\cdot\slashed{\nabla}\right)\cdot\mu.
\end{equation}
\end{lemma}
\begin{proof}Let $e_1$ be an arbitrary unit vector and then choose $e_2$ so that $\left(e_1,e_2\right)$ is positively oriented. Since $e_1$ is arbitrary, it suffices to establish the identity~\eqref{toprovelemmaaa} when evaluated on $e_1$. 

We have
\begin{equation}\label{okmdw981n1}
\mu\cdot\nu = \mu_{11}\nu_{11} + 2\mu_{12}\nu_{12} + \mu_{22}\nu_{22} = 2\mu_{11}\nu_{11} + 2\mu_{12}\nu_{12},
\end{equation}
\begin{equation}\label{okmdw981n2}
\mu\wedge \nu = \mu_{11}\nu_{12} + \mu_{12}\nu_{22} - \mu_{12}\nu_{11} - \mu_{22}\nu_{12} = 2\mu_{11}\nu_{12} - 2\mu_{12}\nu_{11}.
\end{equation}
Using~\eqref{okmdw981n1} and~\eqref{okmdw981n2} we may now calculate 
\begin{align}\label{okmdw981n3}
\frac{1}{2}\left(\slashed{\nabla}_1\left(\mu\cdot\nu\right) + {}^*\slashed{\nabla}_1\left(\mu\wedge \nu\right)\right) &= \left(\slashed{\nabla}_1\mu_{11}\right)\nu_{11} + \mu_{11}\left(\slashed{\nabla}_1\nu_{11}\right) + \left(\slashed{\nabla}_1\mu_{12}\right)\nu_{12} + \mu_{12}\left(\slashed{\nabla}_1\nu_{12}\right)
\\ \nonumber &\qquad +\left(\slashed{\nabla}_2\mu_{11}\right)\nu_{12} + \mu_{11}\left(\slashed{\nabla}_2\nu_{12}\right) - \left(\slashed{\nabla}_2\mu_{12}\right)\nu_{11} - \mu_{12}\left(\slashed{\nabla}_2\nu_{11}\right).
\end{align}

Next, we calculate
\begin{align}\label{okmdw981n4}
\mu_{1B}\slashed{\nabla}_C\nu^{BC} + \left(\slashed{\nabla}_B\mu_{C1}\right)\nu^{BC} &= \mu_{11}\left(\slashed{\nabla}_1\nu_{11}\right) + \mu_{11}\left(\slashed{\nabla}_2\nu_{12}\right) + \mu_{12}\left(\slashed{\nabla}_1\nu_{12}\right) - \mu_{12}\slashed{\nabla}_2\nu_{11} 
\\ \nonumber &\qquad + \left(\slashed{\nabla}_1\mu_{11}\right)\nu_{11} + \left(\slashed{\nabla}_2\mu_{11}\right)\nu_{12} + \left(\slashed{\nabla}_1\mu_{12}\right)\nu_{12} - \left(\slashed{\nabla}_2\mu_{12}\right)\nu_{11}.
\end{align}
Finally, by inspection we see that~\eqref{okmdw981n3} and~\eqref{okmdw981n4} are the same.

\end{proof}

\begin{lemma}\label{iguessthisisusefulwww}Let $\slashed{g}_{AB}$ be a Riemannian metric on $\mathbb{S}^2$,  let $\mu_{AB}$ and $\nu_{AB}$ be trace-free symmetric tensors, and let $\vartheta_A$ be a $1$-form. Then
\[\mu\cdot\left(\vartheta\cdot\nu\right) - \nu\cdot\left(\vartheta\cdot\mu\right) = {}^*\vartheta\left(\mu\wedge\nu\right),\]
\[\mu\cdot\left(\vartheta\cdot\nu\right) + \nu\cdot\left(\vartheta\cdot\mu\right) = \vartheta\left(\mu\cdot\nu\right),\]
\[\mu\cdot\left(\vartheta\cdot\nu\right) = \frac{1}{2}\mu\cdot\left(\vartheta\cdot\nu\right)  - \frac{1}{2}\nu\cdot\left(\vartheta\cdot\mu\right)+  \frac{1}{2}\mu\cdot\left(\vartheta\cdot\nu\right)+\frac{1}{2}\nu\cdot\left(\vartheta\cdot\mu\right) =  \frac{1}{2}\vartheta\left(\mu\cdot\nu\right) + \frac{1}{2}{}^*\vartheta\left(\mu\wedge\nu\right).\]
\end{lemma}
\begin{proof}The proof of the first two identities is similar to the proof of Lemma~\ref{identitytwotrfere}. One simply writes out both sides of the identity in an orthonormal oriented basis. We omit the details. The final identity is an immediate consequence of the first two.
\end{proof}

\begin{lemma}\label{bettereta}Let $\left(\mathcal{M},g_{\mu\nu}\right)$ be a spacetime satisfying the Einstein vacuum equations in the double-null gauge. Then we have 
\begin{equation}\label{plwo2}
\Omega\nabla_3\eta_A + \left(\Omega{\rm tr}\underline{\chi}\right)\eta_A + 2\left(\left(\Omega\hat{\underline{\chi}}\right)\cdot\eta\right)_A = -\Omega\underline{\beta}_A - 4\slashed{\nabla}_A\left(\Omega\underline{\omega}\right),
\end{equation}
\begin{equation}\label{plwo3}
\Omega\underline{\beta}_A = \slashed{\rm div}\left(\Omega\hat{\underline{\chi}}\right)_A - \frac{1}{2}\slashed{\nabla}_A\left(\Omega{\rm tr}\underline{\chi}\right) - \left(\eta\cdot\left(\Omega\hat{\underline{\chi}}\right)\right)_A + \frac{1}{2}\eta_A \left(\Omega{\rm tr}\underline{\chi}\right),
\end{equation}
\begin{equation}\label{plwo4}
\Omega^{-1}\beta_A = -\slashed{\rm div}\left(\Omega^{-1}\hat{\chi}\right)_A + \frac{1}{2}\slashed{\nabla}_A\left(\Omega^{-1}{\rm tr}\chi\right) - \left(\eta\cdot\left(\Omega^{-1}\hat{\chi}\right)\right)_A + \frac{1}{2}\eta_A\left(\Omega^{-1}{\rm tr}\chi\right).
\end{equation}
\end{lemma}
\begin{proof}
Using~\eqref{com2} we have
\begin{align}\label{plwo}
\Omega\nabla_3\underline{\eta}_A &= \Omega\nabla_3\left(-\eta_A+2\slashed{\nabla}_A\log\Omega\right)
\\ \nonumber &= -\Omega\nabla_3\eta_A + 2\left[\Omega\nabla_3,\slashed{\nabla}_A\right]\log\Omega - 4\slashed{\nabla}_A\left(\Omega\underline{\omega}\right)
\\ \nonumber &= -\Omega \nabla_3\eta_A - \left(\Omega{\rm tr}\underline{\chi}\right)\left(\slashed{\nabla}_A\log\Omega\right) - 2\Omega\left(\hat{\underline{\chi}}\cdot\slashed{\nabla}\right)_A\log\Omega - 4\slashed{\nabla}_A\left(\Omega\underline{\omega}\right).
\end{align}

Plugging~\eqref{plwo} into~\eqref{3ueta} then yields~\eqref{plwo2}.

Finally,~\eqref{plwo3} and~\eqref{plwo4} follow in a straightforward fashion from~\eqref{tcod1} and~\eqref{tcod2}.

\end{proof}

\section{Proof of Lemma~\ref{codisnice}}\label{dwlpdwlp}
\begin{proof}

For any trace-free symmetric $2$-tensor $\nu_{AB}$, let us set $F_{AB} \doteq \Omega\nabla_4\nu_{AB}$. Then, using Lemma~\ref{commlemm}, we find that
\begin{align}\label{ddwwd}
&\Omega\nabla_4\slashed{\rm div}\nu_A =
\\ \nonumber & \slashed{\rm div}F_A + \Omega^2\left[2\left(\Omega^{-1}\beta\right)\cdot\nu - \left(\Omega^{-1}\chi\cdot\nu\right)\underline{\eta}+ \Omega^{-1}\chi\cdot\left(\underline{\eta}\cdot\nu\right) - \nu\cdot \left(\underline{\eta}\cdot\Omega^{-1}\chi\right)\right]_A - \Omega^2\left(\Omega^{-1}\chi\right)^{BC}\slashed{\nabla}_B\nu_{CA}.
\end{align}
Integrating~\eqref{ddwwd} in the $v$-direction with $\nu_{AB} = \overline{\Omega\hat{\underline{\chi}}}_{AB} +\widetilde{\hat{\underline{\chi}}}_{AB}$, using~\eqref{hatchirenroma}, and using Lemma~\ref{iguessthisisusefulwww} we obtain that
\begin{align}\label{lpwlpdkocmienco12456gvcd}
&\slashed{\rm div}\left(\Omega\hat{\underline{\chi}}\right)_A = \slashed{\rm div}\left(\widetilde{\hat{\underline{\chi}}}^{(1)}\right)_A+\overline{\slashed{\rm div}\Omega\hat{\underline{\chi}}}_A\\ \nonumber &\qquad +\int_0^v\Bigg[\overline{\slashed{\nabla}^B}\Bigg(\left(\frac{\hat{v}}{-u}\right)^{-2\kappa}\left(\overline{\left(\frac{\hat{v}}{-u}\right)^{\kappa}\Omega}\right)^2\Bigg(\overline{-\frac{1}{2}\left(\Omega^{-1}{\rm tr}\chi\right)\left(\Omega\hat{\underline{\chi}}_{AB}\right) + \left(\slashed{\nabla}\hat{\otimes}\underline{\eta}\right)_{AB} + \left(\underline{\eta}\hat{\otimes}\underline{\eta}\right)_{AB}}
\\ \nonumber &\qquad -\frac{1}{2} \overline{\Omega{\rm tr}\underline{\chi}}\left(\overset{\triangleright}{\Omega^{-1}\hat{\chi}}_{AB}\right)\Bigg)\Bigg)+\Bigg(\left(\frac{\hat{v}}{-u}\right)^{-2\kappa}\left(\overline{\left(\frac{\hat{v}}{-u}\right)^{\kappa}\Omega}\right)^2\Bigg(2\left(\overset{\triangleright}{\Omega^{-1}\beta}\right)^B\overline{\Omega\hat{\underline{\chi}}}_{AB} 
\\ \nonumber &\qquad -\left(\left(\overset{\triangleright}{\Omega^{-1}\hat{\chi}}\right)\cdot\overline{\Omega\hat{\underline{\chi}}}\right)\overline{\underline{\eta}}_A +  \overline{{}^*\underline{\eta}}_A\left(\left(\overset{\triangleright}{\Omega^{-1}\hat{\chi}}\right)\wedge \overline{\Omega\hat{\underline{\chi}}}\right) 
 -\frac{1}{2}\overline{\left(\Omega^{-1}{\rm tr}\chi\right)\slashed{\rm div}\left(\Omega\hat{\underline{\chi}}\right)_A}
 \\ \nonumber &\qquad - \left(\overset{\triangleright}{\Omega^{-1}\hat{\chi}}\right)^{BC}\overline{\slashed{\nabla}_B\left(\Omega\hat{\underline{\chi}}\right)_{CA}}\Bigg)+\left(\frac{\hat{v}}{-u}\right)^{-2\kappa}\left(\overline{\left(\frac{\hat{v}}{-u}\right)^{\kappa}\Omega}\right)^2\left(\overset{\triangleright}{\Omega^{-1}\hat{\chi}}\cdot\overline{{\rm div}\left(\Omega\hat{\underline{\chi}}\right)} + \frac{1}{2}\overline{\Omega^{-1}{\rm tr}\chi {\rm div}\left(\Omega\hat{\underline{\chi}}\right)} \right) \Bigg]\, d\hat{v} + \mathcal{H}^{(2)}.
\end{align}

Next, using Lemma~\ref{commlemm} and~\eqref{trchiundaplwd}, we find that
\begin{align}\label{lpdwldpkq}
&\Omega\nabla_4\slashed{\nabla}_A\left(\overline{\Omega{\rm tr}\underline{\chi}} + \widetilde{{\rm tr}\underline{\chi}}^{(0)}\right) =
\\ \nonumber & \left(\frac{v}{-u}\right)^{-2\kappa}\slashed{\nabla}_A\left[\left(\overline{\left(\frac{v}{-u}\right)^{\kappa}\Omega}\right)^2\left(\overline{-\frac{1}{2}\left(\Omega^{-1}{\rm tr}\chi\right)\left(\Omega{\rm tr}\underline{\chi}\right) + 2\left(\rho - \frac{1}{2}\left(\Omega^{-1}\hat{\chi}\right)\cdot\left(\Omega\hat{\underline{\chi}}\right)\right) + 2\slashed{\rm div}\underline{\eta} + \left|\underline{\eta}\right|^2}\right)\right]
\\ \nonumber &-\frac{1}{2}\Omega^2\left(\Omega^{-1}{\rm tr}\chi\right)\slashed{\nabla}_A\left(\overline{\Omega{\rm tr}\underline{\chi}} + \widetilde{{\rm tr}\underline{\chi}}^{(0)}\right) -\Omega^2\left(\Omega^{-1}\hat{\chi}\right)_A^{ \ \ B}\slashed{\nabla}_B\left(\overline{\Omega{\rm tr}\underline{\chi}} + \widetilde{{\rm tr}\underline{\chi}}^{(0)}\right).
\end{align}

Thus we obtain the following analogue of~\eqref{lpwlpdkocmienco12456gvcd}
\begin{align}\label{afwfvkpw}
&\slashed{\nabla}_A\left(\Omega{\rm tr}\underline{\chi}\right)  = \overline{\slashed{\nabla}_A\left(\Omega{\rm tr}\underline{\chi}\right)} + \slashed{\nabla}_A\widetilde{{\rm tr}\underline{\chi}}^{(1)} + \int_0^v\left(\frac{\hat{v}}{-u}\right)^{-2\kappa}\left(\overline{\left(\frac{\hat{v}}{-u}\right)^{\kappa}\Omega}\right)^2\left(\overset{\triangleright}{\Omega^{-1}\hat{\chi}} + \frac{1}{2}\overline{\Omega^{-1}{\rm tr}\chi}\right)\overline{\slashed{\nabla}_A\Omega{\rm tr}\underline{\chi}}\, d\hat{v}
\\ \nonumber &+\int_0^v\Bigg(\left(\frac{\hat{v}}{-u}\right)^{-2\kappa}\slashed{\nabla}_A\left[\left(\overline{\left(\frac{\hat{v}}{-u}\right)^{\kappa}\Omega}\right)^2\left(\overline{-\frac{1}{2}\left(\Omega^{-1}{\rm tr}\chi\right)\left(\Omega{\rm tr}\underline{\chi}\right) + 2\left(\rho - \frac{1}{2}\left(\Omega^{-1}\hat{\chi}\right)\cdot\left(\Omega\hat{\underline{\chi}}\right)\right) + 2\slashed{\rm div}\underline{\eta} + \left|\underline{\eta}\right|^2}\right)\right]
\\ \nonumber &\left(\frac{\hat{v}}{-u}\right)^{-2\kappa}\left(\overline{\left(\frac{\hat{v}}{-u}\right)^{\kappa}\Omega}\right)^2\left[\overline{-\frac{1}{2}\left(\Omega^{-1}{\rm tr}\chi\right)\slashed{\nabla}_A\left(\Omega{\rm tr}\underline{\chi} \right)} -\left(\overset{\triangleright}{\Omega^{-1}\hat{\chi}}\right)_A^{ \ \ B}\overline{\slashed{\nabla}_B\left(\Omega{\rm tr}\underline{\chi}\right)}\right]\Bigg)\, d\hat{v} + \mathcal{H}^{(2)}.
\end{align}

Finally the proof concludes by substituting in~\eqref{lpwlpdkocmienco12456gvcd},~\eqref{afwfvkpw},~\eqref{betamowfa}, and~\eqref{etaffwffww} into~\eqref{plwo3}, using Lemmas~\ref{identitytwotrfere},~\ref{iguessthisisusefulwww}, ~\ref{kpdwdddd}, and~\ref{bettereta}, and carrying out all of the possible cancellations. We omit the straightforward if tedious calculation.

\end{proof}

\section{Proof of Lemma~\ref{thatissuchaconst}}\label{askf}
In this section we give the proof of Lemma~\ref{thatissuchaconst}.

\begin{proof}
Let's set
\begin{equation}\label{sffsfsfs}
\check{\Omega}\left(\theta\right) \doteq \lim_{v\to 0}\left(\frac{v}{-u}\right)^{\kappa}\Omega\left(v,u,\theta\right),\qquad u^{-1}\check{b}^A(\theta) \doteq \lim_{v\to 0}b^A\left(v,u,\theta\right),\qquad u^2\check{\slashed{g}}_{AB}(\theta) \doteq \lim_{v\to 0}\slashed{g}_{AB}\left(u,v,\theta\right).
\end{equation}
Multiplying the $e_3$-Raychaudhuri equation~\eqref{3trchi} by $\Omega^2$ leads to the following equation:
\begin{equation}\label{rewritetrchiund2}
\left(\partial_u + \mathcal{L}_b\right)\left(\Omega{\rm tr}\underline{\chi}\right) + \frac{1}{2}\left(\Omega{\rm tr}\underline{\chi}\right)^2 = -\left|\Omega\hat{\underline{\chi}}\right|^2 - 4\left(\Omega\underline{\omega}\right)\left(\Omega{\rm tr}\underline{\chi}\right).
\end{equation}
Next, note that
\begin{equation}\label{omegalimitv0}
\lim_{v\to 0}\frac{v}{u}\Omega\omega = \frac{\kappa}{2u}
\end{equation}
Observe that Lemma~\ref{scalrelations2} and~\eqref{omegalimitv0} imply the following relations:
\begin{equation}\label{onv0form2}
\Omega{\rm tr}\underline{\chi}|_{v=0} = \frac{2}{u} +\slashed{\rm div}b,\qquad \Omega\hat{\underline{\chi}}_{AB}|_{v=0} = \frac{1}{2}\left(\slashed{\nabla}\hat{\otimes}b\right)_{AB},\qquad \Omega\underline{\omega}|_{v=0} = -\frac{\kappa}{2u}-\frac{1}{2}\mathcal{L}_b\log\check{\Omega}.
\end{equation}
Recalling the definition of $\check{b}$ and $\check{\slashed{g}}$ from~\eqref{sffsfsfs}, we may plug in~\eqref{onv0form2} into~\eqref{rewritetrchiund2} and simplify to obtain
\begin{align}\label{raychaubecomesthisscalinv2}
&u^{-2}\check{\slashed{g}}^{AB}\slashed{\nabla}_A\check{b}_B + u^{-2}\mathcal{L}_{\check{b}}\left(\check{g}^{AB}\slashed{\nabla}_A\check{b}_B\right) + \frac{1}{2}u^{-2}\left(\check{\slashed{g}}^{AB}\slashed{\nabla}_A\check{b}_B \right)^2 =
\\ \nonumber &\qquad  -\frac{1}{4}u^{-2}\check{\slashed{g}}^{AC}\check{\slashed{g}}^{BD}\left(\slashed{\nabla}\hat{\otimes}\check{b}\right)_{AB}\left(\slashed{\nabla}\hat{\otimes}\check{b}\right)_{CD} +\left(\frac{2\kappa}{u}+2u^{-1}\check{\Omega}^{-1}\left(\mathcal{L}_{\check{b}}\check{\Omega}\right)\right)\left(\frac{2}{u} +u^{-1}\check{\slashed{g}}^{AB}\slashed{\nabla}_A\check{b}_B\right).
\end{align}

It  immediately follows that~\eqref{basicconstr2} holds.
\end{proof}

\end{document}